\documentclass[10pt,letterpaper]{article}

\usepackage{latexsym,
amsthm,
color,
amssymb,
url,
bm,
amssymb,
microtype,
amsmath,
amsfonts,
mathtools}
\usepackage{verbatim,listings}
\usepackage{todonotes}
\usepackage[letterpaper, top=1in, bottom=1in, left=1in, right=1in, marginparwidth=2.3cm, marginparsep=1mm, includefoot]{geometry}

\usepackage{caption}
\usepackage{setspace}
\usepackage{subcaption}
\usepackage{setspace}

 \usepackage[pdftex, plainpages = false, pdfpagelabels, 
                 bookmarks = true,
                 bookmarksopen = true,
                 bookmarksnumbered = true,
                 breaklinks = true,
                 linktocpage,
                 pagebackref,
                 colorlinks = true,  
                 linkcolor = blue,
                 urlcolor  = blue,
                 citecolor = red,
                 anchorcolor = green,
                 hyperindex = true,
                 hyperfigures
                 ]{hyperref} 
                     
\usepackage[inline]{enumitem}
\usepackage[mathscr]{euscript}
\usepackage{amsthm}
\usepackage{nicefrac}
\usepackage{dsfont}
\usepackage{tikz}
\usepackage[ruled,linesnumbered]{algorithm2e}
\usepackage{nicefrac}
\usepackage[mathscr]{euscript}

\RequirePackage[T1]{fontenc}
\RequirePackage[utf8]{inputenc}
\usepackage[russian,greek,english]{babel}
\usepackage{alphabeta}

\usetikzlibrary{math}

\makeatletter
\input{colordvi}

\usepackage{aliascnt}

\hypersetup{
    colorlinks,
    linkcolor={red!75!black},
    citecolor={green!75!black},
    urlcolor={blue!75!black},
}

\definecolor{MidnightBlack}{rgb}{0.1,0.1,.34}
\definecolor{MidnightBlue}{rgb}{0.1,0.1,0.43}
\definecolor{Black}{rgb}{0,0, 0}
\definecolor{Blue}{rgb}{0, 0 ,1}
\definecolor{Red}{rgb}{1, 0 ,0}
\definecolor{White}{rgb}{1, 1, 1}
\definecolor{grey}{rgb}{.6, .6, .6}
\definecolor{Mygreen}{rgb}{.0, .7, .0}
\definecolor{Yellow}{rgb}{.55,.55,0}
\definecolor{Mustard}{rgb}{1.0, 0.86, 0.35}
\definecolor{applegreen}{rgb}{0.55, 0.71, 0.0}
\definecolor{darkturquoise}{rgb}{0.0, 0.81, 0.82}
\definecolor{celestialblue}{rgb}{0.29, 0.59, 0.82}
\definecolor{green_yellow}{rgb}{0.68, 1.0, 0.18}
\definecolor{crimsonglory}{rgb}{0.75, 0.0, 0.2}
\definecolor{darkmagenta}{rgb}{0.30, 0.0, 0.30}
\definecolor{magenta}{rgb}{0.50, 0.0, 0.50}
\definecolor{internationalorange}{rgb}{1.0, 0.31, 0.0}
\definecolor{darkorange}{rgb}{1.0, 0.55, 0.0}
\definecolor{ao}{rgb}{0.0, 0.5, 0.0}
\definecolor{awesome}{rgb}{1.0, 0.13, 0.32}
\definecolor{darkcyan}{rgb}{0.0, 0.50, 0.50}
\definecolor{violet}{rgb}{0.93, 0.51, 0.93}
\definecolor{brown}{rgb}{0.65, 0.16, 0.16}
\definecolor{orange}{rgb}{1.0, 0.65, 0.0}
\definecolor{cornflowerblue}{rgb}{0.39, 0.58, 0.93}

\newcommand{\darkmagenta}[1]{{\color{darkmagenta}#1}}
\newcommand{\magenta}[1]{{\color{magenta}#1}}

\newcommand{\cref}[1]{\autoref{#1}}

\newcommand{\remove}[1]{}

\newcounter{func}
\newcommand{\funref}[1]{\hyperref[#1]{f_{\ref*{#1}}}} % print a
\tikzset{black node/.style={draw, circle, fill = black, minimum size = 5pt, inner sep = 0pt}}
\tikzset{white node/.style={draw, circlternary_treese, fill = white, minimum size = 5pt, inner sep = 0pt}}
\tikzset{normal/.style = {draw=none, fill = none}}
\tikzset{lean/.style = {draw=none, rectangle, fill = none, minimum size = 0pt, inner sep = 0pt}}
\usetikzlibrary{decorations.pathreplacing}
\usetikzlibrary{arrows.meta}
\usetikzlibrary{calc}

\usetikzlibrary{shapes}
\tikzset{diam/.style={draw, diamond, fill = black, minimum size = 7pt, inner sep = 0pt}}
\tikzset{
	position/.style args={#1:#2 from #3}{
		at=($(#3)+(#1:#2)$)
	}
}

\tikzset{
  v:main/.style = {draw, circle, scale=0.8, thick,fill=black,inner sep=0.7mm},
  v:ghost/.style = {inner sep=0pt,scale=1},
  v:marked/.style = {circle, scale=1.3, fill=DarkGoldenrod,opacity=0.4},
  >={latex},
  e:main/.style = {line width=1pt}
}

\newcommand{\Bcal}{\mathcal{B}}
\newcommand{\Ccal}{\mathcal{C}}
\newcommand{\Dcal}{\mathcal{D}}
\newcommand{\Ecal}{\mathcal{E}}

\newcommand{\Gcal}{\mathcal{G}}
\newcommand{\Hcal}{\mathcal{H}}
\newcommand{\Ical}{\mathcal{I}}

\newcommand{\Kcal}{\mathcal{K}}
\newcommand{\Lcal}{\mathcal{L}}
\newcommand{\Mcal}{\mathcal{M}}

\newcommand{\Ocal}{\mathcal{O}}
\newcommand{\Pcal}{\mathcal{P}}
\newcommand{\Qcal}{\mathcal{Q}}
\newcommand\Rcal{\mathcal{R}}
\newcommand{\Scal}{\mathcal{S}}
\newcommand{\Tcal}{\mathcal{T}}
\newcommand{\Ucal}{\mathcal{U}}
\newcommand{\Vcal}{\mathcal{V}}
\newcommand{\Wcal}{\mathcal{W}}
\newcommand{\Xcal}{\mathcal{X}}

\newcommand{\Zcal}{\mathcal{Z}}

\newcommand{\Bbbb}{\mathbb{B}}

\newcommand{\Hbbb}{\mathbb{H}}

\newcommand{\Kbbb}{\mathbb{K}}

\newcommand{\Nbbb}{\mathbb{N}}

\newcommand{\Pbbb}{\mathbb{P}}

\newcommand{\Rbbb}{\mathbb{R}}
\newcommand{\Sbbb}{\mathbb{S}}

\newcommand{\Vbbb}{\mathbb{V}}

\RequirePackage{stmaryrd}
\usepackage{textcomp}
\DeclareUnicodeCharacter{2286}{\subseteq}
\DeclareUnicodeCharacter{2192}{\ifmmode\to\else\textrightarrow\fi}
\DeclareUnicodeCharacter{2203}{\ensuremath\exists}
\DeclareUnicodeCharacter{183}{\cdot}
\DeclareUnicodeCharacter{2200}{\forall}
\DeclareUnicodeCharacter{2264}{\leq}
\DeclareUnicodeCharacter{2265}{\geq}
\DeclareUnicodeCharacter{8614}{\mathbin{\mapsto}}
\DeclareUnicodeCharacter{8656}{\Leftarrow}
\DeclareUnicodeCharacter{8657}{\Uparrow}
\DeclareUnicodeCharacter{8658}{\Rightarrow}
\DeclareUnicodeCharacter{8659}{\Downarrow}
\DeclareUnicodeCharacter{8669}{\rightsquigarrow}
\newcommand{\eqdef}{\stackrel{{\scriptsize\rm def}}{=}}
\DeclareUnicodeCharacter{8797}{\eqdef}
\DeclareUnicodeCharacter{8870}{\vdash}
\DeclareUnicodeCharacter{8873}{\Vdash}
\DeclareUnicodeCharacter{22A7}{\models}
\DeclareUnicodeCharacter{9121}{\lceil}
\DeclareUnicodeCharacter{9123}{\lfloor}
\DeclareUnicodeCharacter{9124}{\rceil}
\DeclareUnicodeCharacter{2208}{\in}
\DeclareUnicodeCharacter{9126}{\rfloor}
\DeclareUnicodeCharacter{9655}{\triangleright}
\DeclareUnicodeCharacter{9665}{\triangleleft}
\DeclareUnicodeCharacter{9671}{\diamond}
\DeclareUnicodeCharacter{9675}{\circ}
\DeclareUnicodeCharacter{10178}{\bot}
\DeclareUnicodeCharacter{10214}{} % needs stmaryrd
\DeclareUnicodeCharacter{10215}{} % needs stmaryrd
\DeclareUnicodeCharacter{10229}{\longleftarrow}
\DeclareUnicodeCharacter{10230}{\longrightarrow}
\DeclareUnicodeCharacter{10231}{\longleftrightarrow}
\DeclareUnicodeCharacter{10232}{\Longleftarrow}
\DeclareUnicodeCharacter{10233}{\Longrightarrow}
\DeclareUnicodeCharacter{10234}{\Longleftrightarrow}
\DeclareUnicodeCharacter{10236}{\longmapsto}
\DeclareUnicodeCharacter{10238}{\Longmapsto} % needs stmaryrd
\DeclareUnicodeCharacter{10503}{\Mapsto}    % needs stmaryrd
\DeclareUnicodeCharacter{10971}{\mathrel{\not\hspace{-0.2em}\cap}}
\DeclareUnicodeCharacter{65294}{\ldotp}
\DeclareUnicodeCharacter{65372}{\mid}

\definecolor{Red}{rgb}{1, 0 ,0}
\definecolor{Blue}{rgb}{0, 0 ,1}

\newtheorem{theorem}{Theorem}[section]

\newaliascnt{question}{theorem}

\aliascntresetthe{question}

\newaliascnt{lemma}{theorem}
\newtheorem{lemma}[lemma]{Lemma}
\aliascntresetthe{lemma}

\newaliascnt{claim}{theorem}
\newtheorem{claim}[claim]{Claim}
\aliascntresetthe{claim}

\newaliascnt{invariant}{theorem}

\aliascntresetthe{invariant}

\newaliascnt{proposition}{theorem}
\newtheorem{proposition}[proposition]{Proposition}
\aliascntresetthe{proposition}

\newaliascnt{observation}{theorem}
\newtheorem{observation}[observation]{Observation}
\aliascntresetthe{observation}

\newaliascnt{corollary}{theorem}
\newtheorem{corollary}[corollary]{Corollary}
\aliascntresetthe{corollary}

\newaliascnt{definition}{theorem}

\aliascntresetthe{definition}

\newaliascnt{conjecture}{theorem}
\newtheorem{conjecture}[conjecture]{Conjecture}
\aliascntresetthe{conjecture}

\newaliascnt{counterexample}{theorem}

\aliascntresetthe{counterexample}

\newcommand{\hh}{\end{document}}

\newcommand{\p}{{\sf p}}
\newcommand{\sobs}{{\sf sobs}}
\newcommand{\obs}{{\sf obs}}

\newcommand{\cobs}{\mbox{\rm \textsf{cobs}}}

\newcommand{\gall}{\mathcal{G}_{{\text{\rm  \textsf{all}}}}}

\newcommand{\tw}{{\sf tw}\xspace}%treewdith
\newcommand{\pack}{{\sf pack}\xspace}%Pack
\newcommand{\cover}{{\sf cover}\xspace}%covering number
\newcommand{\cupall}{{\pmb{\bigcup}}}

\newenvironment{cproof}{\proof[Proof of claim]}{\endproof}
\newcommand*\samethanks[1][\value{footnote}]{\footnotemark[#1]}
\newcommand{\ground}{\ensuremath{\mathsf{ground}}}

\newcommand{\vcells}{\text{$\mathsf{vcells}$}\xspace}

\newcommand{\poly}{\text{$\mathsf{poly}$}\xspace}
\newcommand{\simple}{\text{$\mathsf{simple}$}\xspace}
\newcommand{\exceptional}{\text{$\mathsf{exceptional}$}\xspace}

\newcommand{\inG}{\text{$\mathsf{inner}$}\xspace}
\newcommand{\outG}{\text{$\mathsf{outer}$}\xspace}
\newcommand{\FPT}{\textsf{FPT}\xspace}
\newcommand{\nonplanar}{\mathsf{npl}}
\newcommand{\eg}{{\sf eg}\xspace}%Euler genus
\newcommand{\Perimeter}{{\sf Perimeter}\xspace}%Perimeter
\newcommand{\bd}{{\sf bd}\xspace}%Boundary
\newcommand{\major}{{host}\xspace}%major
\newcommand{\majors}{{hosts}\xspace}%major
\newcommand{\FFcal}{\mathcal{M}}

\newcommand{\numen}[1]{\ifthenelse{\not\equal{#1}{1}}{#1}{}}
\usepackage{color}

\definecolor{vagelisColour}{RGB}{0, 65, 130}

\title{Delineating Half-Integrality of the Erdős-Pósa
\\ Property for Minors: the Case of Surfaces%
\thanks{Emails of authors: 
\texttt{christophe.paul@lirmm.fr}, \texttt{evangelos.protopapas@lirmm.fr},~~\texttt{sedthilk@thilikos.info}, \texttt{sebastian.wiederrecht@gmail.com}\ .}
\thanks{All authors were supported by the French-German Collaboration ANR/DFG Project UTMA (ANR-20-CE92-0027). Dimitrios M. Thilikos was also supported by the Franco-Norwegian project PHC AURORA 2024. }}
\author{Christophe Paul\thanks{LIRMM, Univ Montpellier, CNRS, Montpellier, France.} \and Evangelos Protopapas\samethanks \and Dimitrios M. Thilikos\samethanks \and    Sebastian Wiederrecht\thanks{Discrete Mathematics Group, Institute for Basic Science, Daejeon, South Korea (IBS-R029-C1).}}
\date{\empty}

\begin{document}

\maketitle

\begin{abstract}
\noindent In 1986 Robertson and Seymour proved a generalization of the seminal result of Erdős and Pósa on the duality of packing and covering cycles: A graph has the \textsl{Erdős-Pósa property} for minors if and only if it is planar.
In particular, for every non-planar graph $H$ they gave examples showing that the Erdős-Pósa property does not hold for $H.$
Recently, Liu confirmed a conjecture of Thomas and showed that every graph has the half-integral Erdős-Pósa property for minors.
Liu's proof   is non-constructive and to this date, with the exception of a small number of examples, no constructive proof is known.

In this paper, we initiate the delineation of the half-integrality of the Erdős-Pósa property for minors.
We conjecture that for every graph $H,$ there exists a \textsl{unique} (up to a suitable equivalence relation) graph parameter $\text{\scriptsize\textsf{EP}}_H$ such that $H$ has the Erdős-Pósa property in a minor-closed graph class $\mathcal{G}$ if and only if $\sup\{ \text{\scriptsize\textsf{EP}}_H(G) \mid G\in\mathcal{G} \}$ is finite.
We prove this conjecture for the class $\mathcal{H}$ of Kuratowski-connected shallow-vortex minors by showing that, for every non-planar $H\in\mathcal{H},$ the parameter $\text{\scriptsize\textsf{EP}}_H(G)$ is \textsl{precisely} the maximum order of a Robertson-Seymour counterexample to the Erdős-Pósa property of $H$ which can be found as a minor in $G.$ 
Our results are constructive and imply, for the first time, parameterized algorithms that find either a packing, or a cover, or one of the Robertson-Seymour counterexamples, certifying the existence of  a half-integral packing for the graphs in $\mathcal{H}.$

\end{abstract}
\medskip


\newpage
\tableofcontents
\newpage

\section{Introduction}\label{@hollingdale}

In 1965 Erdős and Pósa published a paper \cite{ErdosPosaOriginal} proving the following min-max duality theorem.
\begin{eqnarray*}
\begin{minipage}{14cm}
\textsl{For every positive integer $k$ and every graph $G,$ either $G$ contains $k$ pairwise vertex-disjoint cycles, or there exists a set $S\subseteq V(G)$ with $|S|=\Ocal(k\cdot\log(k))$ such that $G-S$ has no cycles.}
\end{minipage}\label{@exaltation}
\end{eqnarray*} 

 This result has since become central in both, graph theory and algorithm design  \cite{robertson1986graph,bodlaender1994disjoint,kakimura2011packing,van2019tight,liu2022packing}.
A collection of pairwise vertex-disjoint cycles is called a (cycle) \emph{packing}, while a set $S$ as above is commonly referred to as a (cycle) \emph{cover} or \emph{transversal}.
In a more general context, one may consider any family $\FFcal$ of graphs and define $\mathsf{pack}_{\FFcal}(G)$ to be the largest size of a \textsl{packing} of members of $\FFcal$ in $G,$ while $\mathsf{cover}_{\FFcal}(G)$ is the minimum size of a set $S\subseteq V(G)$ such that $G-S$ contains\footnote{At this point we consider  containment to be defined through the subgraph relation.} no member of $\FFcal.$
Clearly $\mathsf{pack}_{\FFcal}(G)\leq\mathsf{cover}_{\FFcal}(G).$
We say that $\FFcal$ has the \emph{Erdős-Pósa property} (EP-property) in some graph class $\Gcal$ if there exists some function $f$ such that, for every $G\in\Gcal,$ it holds that $\mathsf{cover}_{\FFcal}(G)\leq f(\mathsf{pack}_{\FFcal}(G))$.

If we now fix some graph $H$ and select $\FFcal_H$ to be the class of all   graphs containing $H$ as a minor, we enter the realm of the \textsl{Graph Minors Series} of Robertson and Seymour.
In Graph Minors V. \cite{robertson1986graph}, as an implication of their  min-max duality between the treewidth of a graph and its largest grid-minor, they prove that
\begin{eqnarray}
\begin{minipage}{13cm}
\textsl{For every graph $H,$  $\FFcal_{H}$ has the EP-property in the class of all graphs if and only if $H$ is planar.}
\end{minipage}\label{@reiteration}
\end{eqnarray} 

The tools and ideas of Erdős-Pósa-type dualities have since found many applications and interpretations \cite{bruce1999mangoes,marx2012obtaining,huynh2019unified,BasteST20acomp,fomin2020hitting1,RaymondT16}. 
Moreover, the study of   Erdős-Pósa dualities has led to important advances in structural graph theory.
As an example, the proof for the directed version of Erdős and Pósa's result \cite{reed1996packing}, known as \textsl{Younger's Conjecture} has paved the way for proving the \textsl{Directed Grid Theorem} \cite{kawarabayashi2015directed}.

\paragraph{Half-integral Erdős-Pósa.} We call a collection  $\Ccal$ of subgraphs of $G$  a \emph{half-integral packing} of $\FFcal$ in $G$
if every graph in $\Ccal$ belongs to $\FFcal$ and no vertex of $G$ belongs to more than two of them. We define $\nicefrac{1}{2}\text{-}\mathsf{pack}_{\FFcal}(G)$ to be the maximum size of such a half-integral packing. Accordingly, a graph $H$ has the $\nicefrac{1}{2}$EP-property in some graph class $\Gcal$ if  there exists some function $f$ such that, for every $G\in\Gcal,$ it holds that $\mathsf{cover}_{\FFcal}(G)\leq f(\nicefrac{1}{2}\text{-}\mathsf{pack}_{\FFcal}(G)).$

Attempting to generalize Robertson and Seymour's seminal result on planar graphs, Robin Thomas conjectured the following relaxation of the EP-property (see \cite{kawarabayashi2007half,liu2022packing}).
\begin{eqnarray}
\begin{minipage}{13cm}
\textsl{For every graph $H,$ $\FFcal_{H}$ has the $\nicefrac{1}{2}$EP-property in the class of all graphs.}
\end{minipage}\label{@empirically}
\end{eqnarray} 

The above conjecture was recently proven by Liu \cite{liu2022packing}.
As before, it is apparent from the definition that $\nicefrac{1}{2}\text{-}\mathsf{pack}_{\FFcal}(G)\leq 2\cdot\mathsf{cover}_{\FFcal}(G).$
Hence, Liu's theorem reveals  a min-max duality between half-integral packing and covering in all graphs.
Moreover, it is a consequence of the Graph Minors Theorem \cite{robertson2004graph} that for each graph $H$ and every graph parameter $\mathsf{p}\in\{\mathsf{pack}_{\FFcal_H},\nicefrac{1}{2}\text{-}\mathsf{pack}_{\FFcal_H},\mathsf{cover}_{\FFcal_H}\}$ one can decide in time $f_{H,\mathsf{p}}(k)|V(G)|^3$ if $\mathsf{p}(G)\geq k$ (or $\mathsf{p}(G)\leq k$ in the  case where $\mathsf{p}=\mathsf{cover}_{\FFcal_H}$) \cite{fellows1988nonconstructive} for some function $f_{H,\mathsf{p}}.$

In light of the above results, it appears that the story of the Erdős-Pósa property in the regime of graph minors, from both a structural and an algorithmic perspective, is quite complete.
However, we should stress the following two points.

\smallskip
\noindent \textbf{First:} The algorithm from \cite{fellows1988nonconstructive} is inherently \textsl{non-constructive}.
Indeed, while for $\mathsf{pack}_{\FFcal_H}$ and $\mathsf{cover}_{\FFcal_H}$ constructive algorithms are known \cite{SauST23apices,SauST22apicesalg,KawarabayashiKR12Thedisjoint}, with the exception of some small special cases \cite{kawarabayashi2007half}, \textsl{no} such results exist for $\nicefrac{1}{2}\text{-}\mathsf{pack}_{\FFcal_H},$ not even approximation algorithms.

\smallskip
\noindent \textbf{Second:} Let $\mathcal{C}$ be a graph class and let $\mathsf{p}$ be a graph parameter. We say that $\p$ is \emph{bounded} in $\Gcal$ if there exists $c\in \Nbbb$ such that, for every $G\in\Gcal$, it holds that  $\p(G)≤c$.
The proof of the ``if'' direction of \eqref{@reiteration} was based on the fact that, for every $H,$ $\FFcal_{H}$ has the EP-property in every graph class of bounded treewidth.
This leads  to the following \textsl{intermediate} question:
\textsl{For which graph parameters $\mathsf{p}$  it holds that $\FFcal_H$ has the EP-property in every class where $\mathsf{p}$ is bounded?}
To be specific, if we fix some graph $H,$ is it 
possible to find a
graph parameter $\text{\scriptsize\textsf{EP}}_H$ such that $\FFcal_H$ has the EP-property in some minor-closed\footnote{A graph class is \emph{minor-closed} if it contains all minors of its graphs.} graph class $\Gcal$ \textsl{if and only if} 
$\text{\scriptsize\textsf{EP}}_H$ is bounded in $\Gcal$?
Indeed, we conjecture that for every graph $H,$ such a graph parameter exists and \emph{precisely} delineates the half-integrality of the Erdős-Pósa property of $\FFcal_H.$

\begin{conjecture}\label{@physiologist}
For every graph $H$ there exists a minor-monotone graph 
parameter $\text{\scriptsize\textsf{EP}}_H$ such that $\FFcal_H$ has the Erdős-Pósa property in a 
minor-closed graph class $\mathcal{G}$ if and only if $\text{\scriptsize\textsf{EP}}_H$ is bounded in $\mathcal{G}.$
\end{conjecture}

Notice that for any planar graph $P,$ we can simply set $\text{\scriptsize\textsf{EP}}_P$ to be the constant zero-function and thus, \cref{@physiologist} trivially holds for all planar  graphs, because of  \eqref{@reiteration}.
However, for non-planar graphs, the existence of such a parameter does not follow  from any known results.
Even {if} $\text{\scriptsize\textsf{EP}}_H$ would exist for some particular non-planar graph $H,$ it would be desirable to have some  constructive, and ideally  canonical, characterization of  $\text{\scriptsize\textsf{EP}}_H.$
That is, we aim at a description  of $\text{\scriptsize\textsf{EP}}_H$ that allows for algorithmic applications.
There are reasons to believe that $\text{\scriptsize\textsf{EP}}_H$ exists 
and moreover has some canonical representation. It has  recently been shown in \cite{paul2023graph}, that this   assertion is tied to the conjecture that graphs are $ω^2$-well quasi ordered by minors, which is a wide open question in order theory (see the classic result of  Thomas in  \cite{Thomas1989wellquasi} for the most advanced result on this conjecture).

The contribution of this paper is \textsl{resolving} \cref{@physiologist} for an infinite family of non-planar graphs.
Moreover, our results are constructive and provide a canonical  representation of $\text{\scriptsize\textsf{EP}}_H$ yielding parameterized approximation algorithms\footnote{This means that our algorithms run in time $f(k)\cdot |V(G)|^{\Ocal(1)}$ for some computable function $f$ where $k$ is the size of the half-integral packing we are looking for.} for $\nicefrac{1}{2}\text{-}\mathsf{pack}_H$ for any $H$ in our family.

\subsection{The threshold of half-integrality}

In Graph Minors V \cite{robertson1986graph}, towards proving the ``only if'' direction of  \eqref{@reiteration}, Robertson and Seymour gave counterexamples of   graphs where non-planar graphs cannot have the EP-property.
Let us investigate such an example for the graph $K_5.$
One may embed $K_5$ in both the projective plane and the torus, but it is impossible to have two disjoint drawings of $K_5$ in either of them. 

\begin{figure}[ht]
  \begin{center}
  \scalebox{.7}{\includegraphics{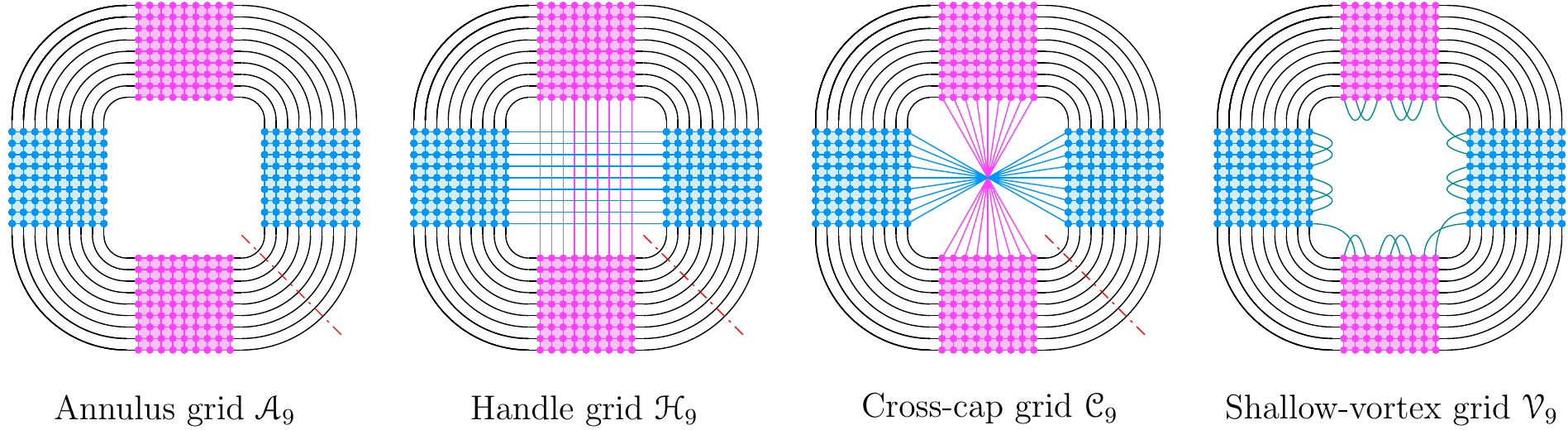}}
  \end{center}
    \caption{The parametric graphs representing the annulus grid $\mathscr{A}_{k},$ the handle grid $\mathscr{H}_{k},$ the cross-cap grid $\mathscr{C}_{k},$ and the shallow-vortex grid $\mathscr{V}_{k}.$}
  \label{@exorbitant}
\end{figure}

Consider the two graphs in the middle of \cref{@exorbitant} and notice that the number of cycles and paths can be scaled.
We call the infinite sequences defined by such ``scalable graphs'' \emph{parametric graphs}\footnote{We postpone the formal definition of parametric graphs to a later point. See \cref{sec_nots_defs}.}.
These parametric graphs are the \emph{handle grid} $\mathscr{H}$ and the \emph{cross-cap grid} $\mathscr{C}$ and represent the torus and the projective plane respectively (see \cref{@headstrong}). 
None of them contains two disjoint copies of graphs from $\FFcal_{K_5},$  both have  a half-integral packing of $Ω(k)$ members of $\FFcal_{K_5},$ and any minimum-size cover of all $\FFcal_{K_5}$  has $Ω(k)$ vertices.

The seminal theorem of Reed \cite{bruce1999mangoes} on the $\nicefrac{1}{2}$EP-property of odd cycles exhibits exactly this kind of behaviour.
Reed showed that odd cycles have the EP-property in every \emph{odd-minor}\footnote{\emph{Odd-minors} are a variant of the minor relation that preserves the parity of cycles. For example, bipartite graphs are exactly the $K_3$-odd-minor-free graphs. We refer the interested reader to \cite{geelen2009odd} for a formal definition.}-closed graph class excluding an \textit{Escher-wall}, while the Escher-wall itself is a counterexample to the EP-property of odd cycles.
Here, the $k$-Escher-wall is obtained by taking exactly the bipartite graphs from the parametric graph $\mathscr{C},$  representing the projective plane, and subdividing each of the ``crossing'' edges once.
The result is a non-bipartite graph where every odd cycle must use an odd number of these subdivided edges.
In the realm of odd-minors, this establishes a positive instance of \cref{@physiologist}:    pick 
$\text{\scriptsize\textsf{EP}}_{K_{3}}$ as the maximum $k$ for which $G$ contains the $k$-Escher-wall as an odd minor.

It is tempting to suspect that Reed's strategy can  apply for the Erdős-Pósa property for minors.
That is, for $K_5$, the two parametric graphs $\mathscr{H}$ and $\mathscr{C}$ are essentially the {only} counterexamples for the EP-property of $\FFcal_{K_5}$ and excluding both of them as minors always yields a class in which $\FFcal{K_5}$ exhibits the EP-property.
Notice that this would imply that  the $\subseteq$-minimal minor-closed classes where 
the EP-property fails for $\FFcal_{K_5}$ are precisely two: the class of graphs embeddable in the projective plane and the class of graphs embeddable in the torus.
Clearly, both these two classes have bounded Euler genus.
Our next step is to observe that this is not true in general.

\paragraph{Kuratowski-connectivity.}
We say that a graph $G$ is \emph{Kuratowski-connected} if for every separation $(A,  B)$ of $G$ of order at most $3,$ 
if there is a component $C$ of $G[A \setminus B]$ and a component $D$ of $G[B \setminus A],$ such that every vertex in $A \cap B$ has a neighbour in $V(C)$ and a neighbour in $V(D),$ 
then one of $G[A],$ $G[B]$ can be drawn in a disc $Δ$ with $A \cap B$ drawn in the boundary of $Δ.$ 
We denote by $\mathcal{K}$ the set of all Kuratowski-connected graphs.
This definition was introduced by Robertson, Seymour, and Thomas as a tool for their characterization of \textsl{linklessly embeddable graphs} via a finite set of minimal obstructions \cite{robertson1995sachs} (see also \cite{VANDERHOLST2009512,NORIN2023184}).

\begin{figure}[ht]
  \begin{center}
  \scalebox{.7}{\includegraphics{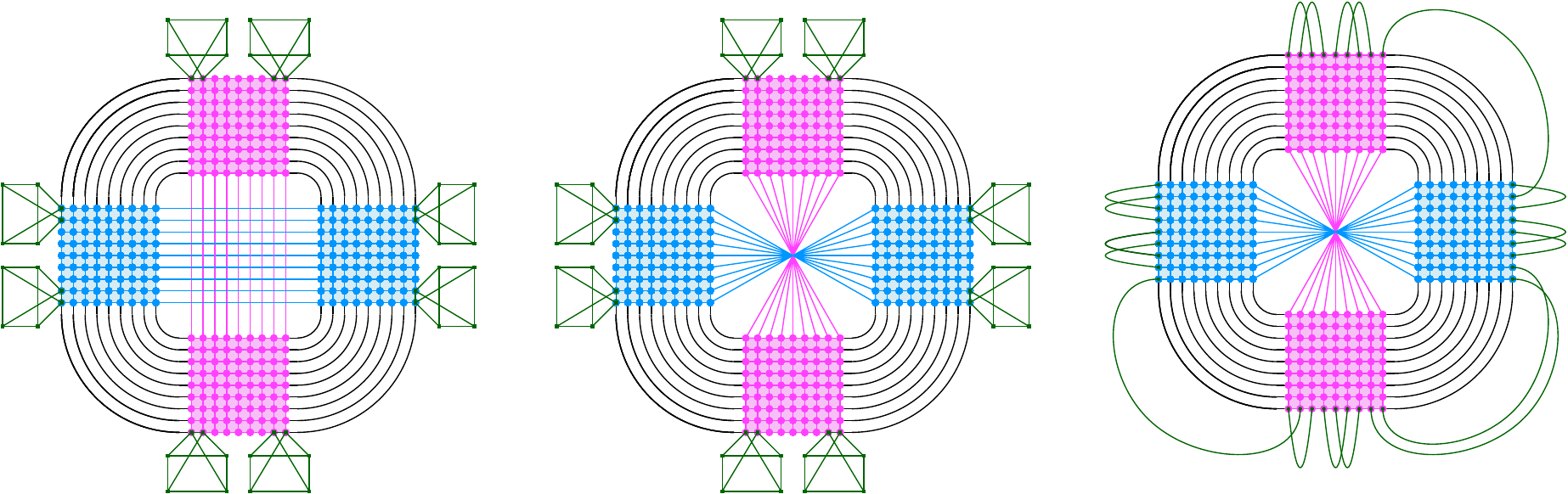}}
  \end{center}
    \caption{The two first parametric graphs serve as  counterexamples for the Erdős-Pósa property of the graph $J$.
    The third  parametric graph is a counterexample for the Erdős-Pósa property of $K_8.$ All three parametric graphs have unbounded Euler-genus.
    For the first two this is witnessed by a large packing of $K_{3,3}$ while the last one can be observed to contain $K_{3,r}$ as a minor.}
  \label{@collection}
\end{figure}

Consider the graph $J$ obtained by identifying two adjacent vertices of $K_{3,3}$ with two vertices of $K_5$ and observe that $J$ is not Kuratowski-connected.
Similar to $K_5,$ there cannot be two disjoint drawings of $K_{3,3}$ on the torus.
So, if we take the parametric graph representing the torus ($\mathscr{H}_{k}$) or the projective plane ($\mathscr{C}_{k}$) from \cref{@exorbitant} and paste ``many'' copies of $K_{3,3}$ around the ``outer cycle'', we obtain a parametric graph without two disjoint $J$-minors but where no small vertex-set can hit all $J$-minors (see the two first graphs in \cref{@collection}).

\paragraph{Shallow-vortex minors.}
There is a second property, that poses a similar issue.
In \cite{ThilikosW22Killingconf} Thilikos and Wiederrecht introduced the parametric graph $\mathscr{V}$ of \textsl{shallow-vortex grids}
where $\mathscr{V}_k$ is obtained from the annulus grid $\mathscr{A}_{k}$ by adding $k$ consecutive crossings in its internal cycle (see the fourth graph in \cref{@exorbitant} for an illustration of $\mathscr{V}_{9}$). The class $\Vcal$ of \emph{shallow-vortex minors} was defined in \cite{ThilikosW22Killingconf}  as the class containing all minors of $\mathcal{V}_{k},$ for all $k\in\Nbbb.$
Notice that $K_8$ is a Kuratowski-connected graph.
It was shown by Curticapean and Xia \cite{curticapean2022parameterizing} that $K_8$ is not a shallow-vortex minor.
However, this is the case for $K_{3,r},$  for every $r\in\mathbb{N},$ which implies that the  parametric graph $\mathscr{V}_{k}$ has unbounded Euler-genus.
If we now paste the $k$ extra crossings of $\mathscr{V}_{k}$ to the ``outer cycle'' of $\mathscr{C}_{k},$ we obtain a parametric graph that is a counterexample for the EP-property of $K_8$ but which is of unbounded Euler-genus (see the last graph in \cref{@collection}).
These observations indicate that, if we want to understand the graphs for which the counterexamples of Robertson and Seymour precisely define the boundary to the $\nicefrac{1}{2}$EP-property, we have to consider the graphs in $\mathcal{K}\cap \mathcal{V}.$

\paragraph{Our contribution.}
The main combinatorial result (stated in \cref{@admiration} in its full generality) is  that \cref{@physiologist} holds, for every graph $H$ that is Kuratowski-connected and a shallow vortex minor.
Moreover, for every such non-planar $H,$ $\text{\scriptsize\textsf{EP}}_H(G)$ is equivalent to the exclusion of the parametric graphs representing some particular set of surfaces   where $H$ embeds.
Therefore, for the non-planar graphs $H\in\mathcal{K}\cap\mathcal{V},$ the boundary between the Erdős-Pósa property and its half-integral relaxation is drawn \textsl{precisely} by a set of surfaces, depending on $H.$
Notice, that the class $\mathcal{K} \cap \mathcal{V}$ encompasses, apart from planar graphs, several important graphs such as  $K_5,K_{3,3},K_{4,4}$, $K_6,$ $K_7,$
and the entire \textsl{Petersen family}.
These last observations imply that our results extend, both algorithmically and combinatorially, to the half-integral packing of \textsl{links} and \textsl{knots}.

\subsection{Notation and definitions}\label{sec_nots_defs}

Let us introduce some notation in order to present our results in full generality.
A \emph{minor antichain} is a family $\mathcal{A}$ of graphs such that no graph $G_1\in \mathcal{A}$ is a minor of another graph $G_2\in\mathcal{A}\setminus\{ G_1\}.$
Since we focus on the minor relation, we refer to minor antichains simply as \emph{antichains}.
Let us denote by $\mathbb{K}$ the collection of all antichains $\mathcal{A}$ where \textsl{every} member of $\mathcal{A}$ is Kuratowski-connected.
Moreover, let us denote by $\mathbb{V}$ the collection of all antichains containing \textsl{at least} one shallow-vortex minor.
Finally, let $\mathbb{P}$ be the collection of all antichains  containing at least one planar graph.
We define $\mathbb{H} \coloneqq \mathbb{K} \cap \mathbb{V},$ $\mathbb{K}^{-} \coloneqq \mathbb{K} \setminus \mathbb{P}$ and $\mathbb{H}^{-} \coloneqq \mathbb{H} \setminus \mathbb{P}.$

\subparagraph*{The Erdős-Pósa property for antichains.}
Let $H$ and $G$ be graphs. A subgraph $H'\subseteq G$ is an \emph{$H$-host} in $G$ if $H$ is a minor of $H'.$
An \emph{$H$-packing} in $G$ is a collection of pairwise vertex-disjoint $H$-hosts in $G.$
An \emph{$H$-cover} is a set $S\subseteq V(G)$ such that $G-S$ is $H$-minor-free.
A \emph{half-integral $H$-packing} is a collection of $H$-hosts in $G$ such that no vertex of $G$ belongs to more than two of them. 

Given an antichain $\Zcal,$ we say that a subgraph $H'\subseteq G$ is a $\Zcal $-host in $G$ if it is an $H$-host for some $H\in\Zcal .$
A $\Zcal$-\emph{packing} is an $H$-packing of \textsl{some} $H\in\Zcal $ and a $\Zcal$-\emph{cover} is an $H$-cover for \textsl{all} $H\in\Zcal ,$ finally a \emph{half-integral $\Zcal$-packing} is a half-integral $H$-packing for some $H\in\Zcal .$
We define the two graph parameters $\cover_{\Zcal}$ and $\pack_{\Zcal}$ as follows.

$$\cover_{\Zcal}(G) \coloneqq \min\{k \mid \mbox{$G$ has an $\Zcal$-cover of size $k$}\}\mbox{~and}$$
$$\pack_{\Zcal}(G) \coloneqq \max\{k\mid \mbox{$G$ has an $\Zcal$-packing of size $k$}.$$

We say that $\Zcal $ has the \emph{Erdős-Pósa property} in a graph class $\mathcal{G}$ if there exists some function $f\colon\mathbb{N}\to\mathbb{N}$ such that $\cover_{\Zcal}(G)\leq f(\pack_{\Zcal}(G)),$ for all $G\in\mathcal{G}.$

\paragraph{Equivalence of graph parameters.}
We use $\gall$ for the class of all graphs.
Given two graph parameters $\mathsf{p},\mathsf{q}\colon\mathcal{G}_{\mathsf{all}}\to\mathbb{N},$ we say that $\mathsf{p}$ and $\mathsf{q}$ 
are \emph{equivalent}, and write $\mathsf{p}\sim\mathsf{q},$ if there exists a function $f\colon\mathbb{N}\to\mathbb{N}$ 
such that,  for every graph  $G,$ $\mathsf{p}(G) \leq f(\mathsf{q}(G))$ and 
$\mathsf{q}(G) \leq f(\mathsf{p}(G)).$ We refer to the function $f$ as the \emph{gap} of this equivalence.

Our result is the identification of a graph parameter $\text{\scriptsize\textsf{EP}}$ such that $\Zcal$ has the Erdős-Pósa property in a minor-closed graph class $\Gcal$ with single-exponential gap if and only if $\text{\scriptsize\textsf{EP}}$ is bounded in $\Gcal$, for every $\Zcal \in \mathbb{H}$.

\paragraph{Surfaces and embeddability.}
We consider a containment relation $\preceq$ between surfaces where we write $Σ\preceq Σ'$ if the surface $Σ'$ can be obtained by adding handles or cross-caps to the surface $Σ.$
The \emph{empty surface} will be denoted by $Σ^{\varnothing}$ and the surface obtained by adding $h$ handles and $c$ cross-caps to the sphere  $Σ^{(0,0)}$ is denoted by $Σ^{(h,c)}.$
Its \emph{Euler-genus} is defined to be $2h+c.$ 
Notice that, by Dyck's Theorem \cite{Dyck1888Beitrage}, we may assume that $c\leq 2$ for all surfaces.
Let $\mathbb{S}$ be a set of surfaces. We say that $\mathbb{S}$ is  \emph{closed}, if $Σ\in\mathbb{S}$ and $Σ'\preceqΣ$ imply that $Σ'\in\mathbb{S}$ and  that it  
 is  \emph{proper}, if it does not contain all surfaces.
If $\mathbb{S}$ is closed and proper we define the ``surface obstruction set''  $\mathsf{sobs}(\mathbb{S})$ as the set of all $\preceq$-minimal surfaces which do not belong to $\mathbb{S}.$
It is easy to observe that  $\mathsf{sobs}(\mathbb{S})$ always consists of one or two surfaces \cite{thilikos2023excluding}.
Notice that 
$\mathsf{sobs}(\emptyset)=\{Σ^{\varnothing}\},$ 
$\mathsf{sobs}(\{Σ^{\varnothing}\})=\{Σ^{(0,0)}\},$ 
$\mathsf{sobs}(\{Σ^{\varnothing},Σ^{(0,0)}\})=\{Σ^{(1,0)},Σ^{(0,1)}\},$ 
and, for a more complicated example, 
$\sobs(\{Σ^{\varnothing},Σ^{(0,0)},Σ^{(0,1)},Σ^{(0,2)}\})=\{Σ^{(1,0)}\}.$

We say that a graph $G$ is \emph{embeddable} in a surface $Σ$ (or \emph{$Σ$}-embeddable) if it has a drawing in $Σ$ without crossings.
The \emph{Euler genus} of a graph $G,$ denoted by $\mathsf{eg}(G),$ is the smallest Euler genus of a surface where $G$ is embeddable.

\paragraph{Parametric graphs and Dyck-grids.}
A \emph{parametric graph} is a sequence $\mathscr{G}=\langle \mathscr{G}_i\rangle_{i\in\mathbb{N}}$ of graphs indexed by non-negative integers.
We say that $\mathscr{G}$ is \emph{minor-monotone} if for every $i\in\mathbb{N}$ we have that $\mathscr{G}_i$ is a minor of $\mathscr{G}_{i+1}.$
All parametric graphs considered in this paper are minor-monotone.
We write $\mathscr{G}^{(1)}\lesssim\mathscr{G}^{(2)}$ for two minor-monotone parametric graphs $\mathscr{G}^{(1)}$ and $\mathscr{G}^{(2)}$ if there exists a function $f\colon\mathbb{N}\to\mathbb{N}$ such that for every $i\in\mathbb{N}$ it holds that $\mathscr{G}^{(1)}_i$ is a minor of $\mathscr{G}^{(2)}_{f(i)}.$
A \emph{minor-monotone parametric family} is a finite collection of $\mathfrak{G}=\{ \mathscr{G}^{(j)} \mid j\in[r] \}$ of minor-monotone parametric graphs such that for distinct $i,j\in[r]$ it holds that $\mathscr{G}^{(i)}\not\lesssim\mathscr{G}^{(j)}$ and $\mathscr{G}^{(j)}\not\lesssim\mathscr{G}^{(i)}.$
We define the minor-monotone parameter
\begin{eqnarray}
&\mathsf{p}_{\mathfrak{G}}(G)\coloneqq \max\{ k \mid  \text{there exists }i\in[r]\text{ such that }G\text{ contains }\mathscr{G}^{(i)}_k\text{ as a minor}\}.\label{@liberating}
\end{eqnarray}
 
The three parametric graphs $\mathscr{A}=\langle \mathscr{A}_{k}\rangle_{k\in \mathbb{N}},$ $\mathscr{H}=\langle \mathscr{H}_{k}\rangle_{k\in \mathbb{N}},$  and $\mathscr{C}=\langle \mathscr{C}_{k}\rangle_{k\in \mathbb{N}}$ are defined as follows:
The  \emph{annulus grid} $\mathscr{A}_{k}$ is the $(4k,k)$-cylindrical grid\footnote{An $(n \times  m)$-cylindrical grid is a Cartesian product of a cycle on $n$ vertices and a path on $m$ vertices.} depicted in the far left of \cref{@exorbitant}.
The \emph{handle grid} $\mathscr{H}_{k}$ (resp. \emph{cross-cap grid} $\mathscr{C}_{k}$) is obtained by adding in $\mathscr{A}_{k}$ edges as indicated in the middle left (resp. middle right) part of \cref{@exorbitant}. We refer to the added edges as \emph{transactions} of the handle  grid $\mathscr{H}_{k}$ or of  the cross-cap grid $\mathscr{C}_{k}.$

Let now $h\in\mathbb{N}$ and $c\in[0,2].$
We define the parametric graph $\mathscr{D}^{(h,c)}=\langle \mathscr{D}_{k}^{(h,c)}\rangle_{k\in \mathbb{N}}$ by taking one copy of $\mathscr{A}_{k},$ $h$ copies of $\mathscr{H}_{k},$ and  $c\in[0,2]$ copies of $\mathscr{C}_{k},$ then ``cut'' them along the dotted {red} line, as in \cref{@exorbitant}, and join them together in the cyclic order $\mathscr{A}_{k},\mathscr{H}_{k},\ldots,\mathscr{H}_{k},\mathscr{C}_{k},\ldots,\mathscr{C}_{k},$ as indicated in \cref{@freemasonry}.

\begin{figure}[ht]
  \begin{center}
  \scalebox{0.88}{\includegraphics{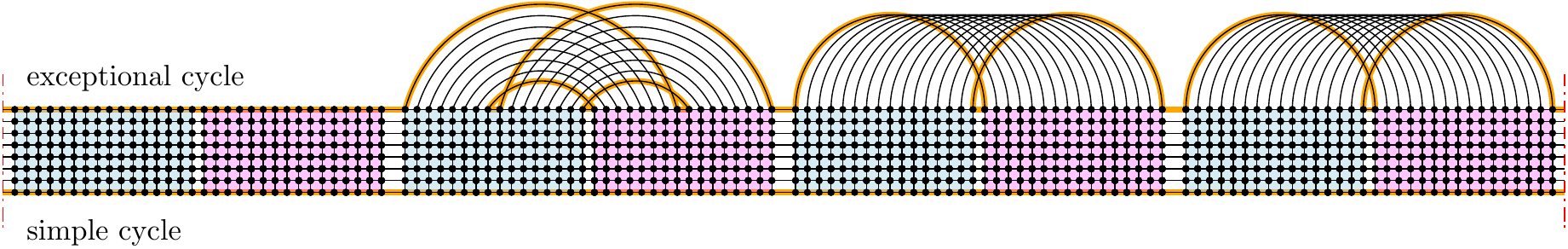}}
  \end{center}
    \caption{The Dyck-grid  $\mathscr{D}_{8}^{1,2}$. The simple and the exceptional cycles are drawn in orange.}
  \label{@freemasonry}
\end{figure}

We call the graph $\mathscr{D}_{k}^{(h,c)}$ the \emph{Dyck-grid} of \emph{order} $k$ \emph{with} $h$ \emph{handles and} $c$ \emph{cross-caps}.
Given some surface $Σ=Σ^{(h,c)},$ we say that the graph $D$ is the \emph{$(Σ;d)$-Dyck-grid} if $D=\mathscr{D}^{(h,c)}_d$ and we use $\mathscr{D}^{Σ}$ to denote the parametric graph $\langle \mathscr{D}^{Σ}_i\rangle_{i\in\mathbb{N}},$ where $\mathscr{D}^{Σ}_i$ is the $(Σ;i)$-Dyck-grid.

Let us return to our antichain $\Zcal \in\mathbb{H}^{-}.$
We denote by $\mathbb{S}_{\Zcal}$ the set of surfaces where none of the graphs in $\Zcal $ can be embedded.
Notice  that $\mathbb{S}_{\Zcal}$ is closed and proper and, for every $Σ\in\mathsf{sobs}(\mathbb{S}_{\Zcal})$, there exists some $H\in\Zcal $ such that $H$ embeds in $Σ.$

\subsection{Our results}

We associate with $\Zcal$ the  parametric family $\mathfrak{D}_{\Zcal}\coloneqq \{ \mathscr{D}^{Σ} \mid Σ\in\mathsf{sobs}(\mathbb{S}_{\Zcal}) \}$. 
Let $ \text{\scriptsize\textsf{EP}}_{\Zcal}:= \mathsf{p}_{\mathfrak{D}_{\Zcal}}$.
Our combinatorial result determines precisely when a member in $\mathbb{H}^{-}$ has the Erdős-Pósa property in some minor-closed graph class.

\begin{theorem}\label{@admiration}
For every $\Zcal\in\mathbb{H}^{-},$ for every minor-closed graph class $\mathcal{G}$, $\Zcal$ has the Erdős-Pósa property in $\mathcal{G}$ if and only if $\text{\scriptsize\textsf{EP}}_{\Zcal}$ is bounded in $\mathcal{G}$.
\end{theorem}

For every antichain $\Zcal$ we define $h_{\Zcal}\coloneqq \max\{ |V(H)| \mid H\in\Zcal \}$ and $\gamma_{\Zcal}\coloneqq \max\{ \mathsf{eg}(H) \mid H \in \Zcal  \}.$
The engine that drives the proof of \cref{@admiration} and which represents our first main algorithmic result is the following. 

\begin{theorem}
\label{@indescribably}
There exists a function  $f_{\ref{@indescribably}}:\mathbb{N}^4\to\mathbb{N}$ such that, for every antichain $\Zcal \in\Hbbb^-,$ there exists an algorithm such  that, given $k,t\in\mathbb{N}$ and a graph $G,$ outputs one of the following:
\begin{itemize}
\item  a $\mathscr{D}^{Σ}_{t}$-\major in $G,$ for some $Σ\in \sobs(\mathbb{S}_{\Zcal}),$ or
\item  an  $\Zcal$-packing of size at least $k$ in $G,$ or
\item  an $\Zcal$-cover of size at most $f_{\ref{@indescribably}}(γ_{\Zcal},h_{\Zcal},t,k)$ in $G$.
\end{itemize}
Moreover, the algorithm runs in time $2^{2^{\Ocal_{γ_{\Zcal}}(\mathsf{poly}(t)) + \Ocal_{h_{\Zcal}}(k)}}\cdot |V(G)|^4 \cdot \log(|V(G)|)$
and $f_{\ref{@indescribably}}(γ_{\Zcal},h_{\Zcal},t,k) = 2^{\Ocal_{γ_{\Zcal}}(\mathsf{poly}(t)) + \Ocal_{h_{\Zcal}}(k)}.$
\end{theorem}

By a recent result of Gavoille and Hilaire \cite{gavoille2023minoruniversal} (\cref{@headstrong}) it holds that there exists some constant $c$ such that for every $\Zcal \in\mathbb{H}^{-}$ and $Σ\in \mathsf{sobs}(\mathbb{S}_{\Zcal})$ there exists some $H\in\Zcal $ such that $H$ is a minor of $\mathscr{D}^{Σ}_{c\gamma_{\Zcal}^4h_{\Zcal}^2}.$
Moreover, as observed in \cite{thilikos2023excluding},  every Dyck-grid of big enough order contains a large half-integral packing of itself of smaller order (\cref{@monopolism}).
Combining these two results with \cref{@indescribably}, yields the following (constructive) parameterized approximation algorithm for  $\nicefrac{1}{2}\text{-}\mathsf{pack}_{\Zcal}.$

\begin{theorem}\label{@irrespective}
There exists a function  $f_{\ref{@irrespective}}:\mathbb{N}^2\to\mathbb{N}$ such that, for every antichain $\Zcal \in\Hbbb^-,$ there exists an algorithm such  that, given $k\in\mathbb{N}$ and a graph $G,$ outputs one of the following:
\begin{enumerate}
  \item a half-integral $\Zcal$-packing of size at least $k$ in $G,$ or
  \item an $\Zcal$-cover of size at most $f_{\ref{@irrespective}}(h_{\Zcal},k)$ in $G$.
\end{enumerate}
Moreover, the algorithm runs in time\footnote{Given two functions $\chi,\psi\colon \mathbb{N}\rightarrow \mathbb{N},$ we write $\chi(n)=\mathcal{O}_{x}(\psi(n))$ in order to denote that there exists a computable function $f\colon\mathbb{N} \rightarrow \mathbb{N}$
such that $\chi(n)=\mathcal{O}( f(x)\cdot \psi(n)).$ We also use 
$\chi(n)=\poly_{x}(ψ(n))$  instead of   $\chi(n)=\Ocal_{x}((ψ(n))^c),$ for some $c\in \Nbbb$.
} {$2^{2^{\poly_{h_{\Zcal}}(k)}}\cdot |V(G)|^4 \cdot \log(|V(G)|)$ and $f_{\ref{@irrespective}}(h_{\Zcal},k)=2^{\poly_{h_{\Zcal}}(k)}.$}
\end{theorem}
We wish to stress that, given the combinatorial bounds of \cref{@indescribably},  we may directly apply the minor-checking algorithm of \cite{KawarabayashiKR12Thedisjoint} for the two first outcomes of \cref{@indescribably} and the algorithm of \cite{MorelleSST23faste}  for its third outcome.
Both these algorithms are quadratic on $|V(G)|$ and this implies alternative quadratic algorithms to those in \cref{@indescribably} and \cref{@irrespective}. However, this would come with the cost of enormous parametric dependencies on $k$.

\subsection{Some implications of our results}

\paragraph{Half-integral Erdős-Pósa for linked pairs and knots.}
As mentioned above, $\Hbbb = \Kbbb \cap \Vbbb$ contains several antichains of particular interest.
A first example is the \textsl{Petersen family}, which is exactly the (minor) obstruction set\footnote{The obstruction set of some minor-closed class $\Gcal$ is the set $\obs(\Gcal)$ of the minor-minimal graphs that are not in $\Gcal.$} for the so-called \textsl{linklessly embeddable} graphs (in short, \emph{link-less} graphs).
Indeed, the origin of the definition of Kuratowski-connectivity comes from the paper of Robertson, Seymour, and Thomas \cite{robertson1995sachs}, where this  obstruction was found.
All obstructions for link-less graphs  
as well as those for knot-less  graphs are Kuratowski-connected.
Moreover, as the shallow-vortex minor $K_{6}$  (resp. $K_{7}$) is a member of the obstruction set of link-less  (resp. knot-less) graphs, we also have that both these obstruction sets 
belong to $\Hbbb^-.$
This insight allows us to apply \cref{@irrespective} to topological objects such as \textsl{links} and \textsl{knots}.

Let $G$ be a graph and let  $\Ccal=\{C_1,\ldots,C_{k}\}$ be a 
 collection of subgraphs of $G.$
 The \emph{intersection graph} of $\Ccal$ is the graph $I(\Ccal)=(\Ccal,E_{\Ccal})$ where $CC'\in E_{\Ccal}$ if and only if $C\cap C'$ is not the empty graph.  
 We say that $\Ccal$ is a \emph{collection of double cycles} (resp. \emph{cycles}) 
 if each $C_{i}$ is  union of two disjoint cycles (resp. a cycle).
 
Given a collection  $\Ccal$ of double cycles (resp. cycles) of $G,$ 
and some $\Rbbb^3$-embedding of $G,$ 
we say that $\Ccal$ is a \emph{$\nicefrac{1}{2}$-packing} of links (resp. knots)    if  
for every $i\in[k],$ the two components of $C_i$ are linked (resp. the cycle $C_{i}$ is knotted) in this particular embedding (see \cite{adams1997knot} for more on links and knots). 
The \emph{half-integral linked pair  (resp. knot) packing number of a graph} $G,$ denoted by 
$\nicefrac{1}{2}$-$\mathsf{lppack}(G)$ (resp. 
$\nicefrac{1}{2}$-$\mathsf{knpack}(G)$), is the maximum $k$ such that, 
for every $\Rbbb^3$-embedding of $G,$ 
there exists a $\nicefrac{1}{2}$-packing of links (resp. knots) in $G$ of size $k.$
Both $\nicefrac{1}{2}$-$\mathsf{lppack}(G)$ and $\nicefrac{1}{2}$-$\mathsf{knpack}(G)$ are minor-monotone parameters, therefore we know (non-constru\-ctively) that there is an algorithm that for checking whether $\nicefrac{1}{2}$-$\mathsf{lppack}(G) ≥ k$ ($\nicefrac{1}{2}$-$\mathsf{knpack}(G) ≥ k$) in time $f(k)\cdot |V(G)|^{2}.$
Up to now, no constructive (on $k$) algorithm is known for these problems.
Our results imply the following.

\begin{theorem}
\label{@inescapably}
There exists a function $f: \Nbbb\to\Nbbb$ and  algorithms that, given 
a graph $G$ and a $k\in\Nbbb,$ outputs either  that $\nicefrac{1}{2}$-$\mathsf{lppack}(G)≥k$
(resp. $\nicefrac{1}{2}$-$\mathsf{knpack}(G)≥k$)
or   a vertex set $A$
of at most $f(k)$ vertices such that  $G-A$ has a link-less (knot-less) $\Rbbb^3$-embedding. Moreover, both algorithms run in time $2^{2^{\poly(k)}}\cdot |V(G)|^3\cdot \big(\log(|V(G)|)\big)^2$ and $f(k)=2^{\poly(k)}.$
\end{theorem}

 \cref{@inescapably}  implies
that both parameters $\nicefrac{1}{2}$-$\mathsf{lppack}$ and $\nicefrac{1}{2}$-$\mathsf{knpack}$
admit \FPT-approximation algorithms with exponential 
approximation gap. Moreover, in case the output is that  $\nicefrac{1}{2}$-$\mathsf{lppack}(G)≥k$ (resp. $\nicefrac{1}{2}$-$\mathsf{knpack}(G)≥k$), 
the algorithms output a $\nicefrac{1}{2}$-packing of $k$ graphs
certifying that,   every  $\Rbbb^3$-embedding of $G$  contains  a $\nicefrac{1}{2}$-packing $\Ccal$ of $k$ links (resp. knots)  such that  $I(\Ccal)$ is either edgeless or a clique. We stress that the above algorithms become constructive (on $k$)
of we know the obstructions of link-less/knot-less graphs
or at least of some upper bound to their size.
For the later class not such bound is known.

Other implications of our results, related to canonical approimate characterizations  of  the parameters we study, are discussed in the conclusion section (\cref{@regressing}).

\subsection{Outline of the proof}

We begin the  description of the main ideas of our proof with the definition of a tree decomposition.

\paragraph{Tree decompositions.}
Let $G$ be a graph. A \emph{tree decomposition} of a graph $G$ is a pair $(T, β)$ where $T$ is a tree and $β : V(T) \to 2^{V(G)}$ is a function, whose images are called the \emph{bags} of $\mathcal{T},$ such that  $\bigcup_{t \in V(T)} β(t) = V(G),$  for every $e=xy \in E(G),$ there exists $t \in V(T)$ with $\{x,y\} \subseteq β(t),$ and
 for every $v \in V(G),$ the set $\{ t \in V(T) \mid v \in β(t) \}$ induces a subtree of $T.$
We refer to the vertices of $T$ as the \emph{nodes} of the tree decomposition $\mathcal{T}.$ The \emph{width} of $\mathcal{T}$ is the value $\max_{t \in V(T)} |β(t)| - 1.$ The \emph{treewidth} of $G,$ denoted by $\tw(G),$ is the minimum width over all tree decompositions of $G.$

\paragraph{The classic approach.}
In order to facilitate the presentation of our proof, let us briefly explain the two main ideas of the proof that planar graphs enjoy the Erdős-Pósa property in the set of all graphs.
The key ingredient is that every planar graph is a minor of a graph of sufficiently large treewidth. The proof follows in two steps.

\smallskip
\noindent
\textbf{Step 1.} Assuming that $\pack_{H}(G) \leq k$, based on the grid theorem by Robertson and Seymour, we may assume that the treewidth of $G$ is bounded by some function of $k.$

\smallskip
\noindent
\textbf{Step 2.} With the tree decomposition $(T,β)$ of $G$ at hand, we build an  $H$-cover $A$ of $G$  by adding to it an adhesion (if any)  $D_{xy}=β(x) \cap β(y)$ such that both $G_{x}\coloneqq G[β(V(T_{x}))\setminus D_{xy}]$ and 
$G_{y}\coloneqq G[β(V(T_{y}))\setminus D_{xy}]$  contain $H$ as a minor
(here $T_{x}$ and $T_{y}$  are the two components of $T-xy$)
and then recursing on the corresponding tree decompositions of $G_{x}$ and $G_{y}.$ If $\pack_{H}(G)≤k,$ eventually  this procedure returns an $H$-cover of  size at most $k\cdot (\tw(G)+1)$.

\smallskip
Throughout the present outline we describe arguments that can be paralleled to the two steps above.
Moreover, in each step we explain the challenges that are met and the way we deal with them in our proof.

For simplicity, instead of an antichain $\Zcal,$ we consider a non-planar graph $H$ that is Kuratowski-connected and a shallow-vortex minor.
We denote by $\Sbbb_{H}$ the set of all surfaces where $H$ cannot be embedded and by $\Sbbb'_{H}\coloneqq\sobs(\Sbbb_{H})$ the corresponding surface obstruction set. 
We stress that the graphs in $\frak{D}_{H}=\{\mathscr{D}^{Σ}\mid Σ\in  \Sbbb'_{H}\}$  can be seen as ``generators of half-integrality''.
Indeed, it is possible to prove that, for every $t \in \Nbbb$, $\pack_{H}(\mathscr{D}_{t}^{Σ}) \leq 1,$ and  $\cover_{H}(\mathscr{D}_{t}^{Σ})=Θ(\nicefrac{1}{2}\mbox{-}\pack_{H}(\mathscr{D}_{t}^{Σ}))=Ω(t).$
This already proves the easy direction of \cref{@admiration}.

Let $\mathcal{T} = (T, β)$ be a tree decomposition of a graph $G.$ 
For each $t \in V(T),$ we define the \emph{adhesions} of $t$ as the sets in $\{ β(t) \cap β(d) \mid \textrm{$d$ adjacent with $t$}\}$ and the maximum size of them is called the \emph{adhesion} of $t.$
The \emph{adhesion} of $\mathcal{T}$ is the maximum adhesion of a node of $\mathcal{T}.$ 
The \emph{torso} of $\mathcal{T}$ on a node $t$ is the graph, denoted by $G_{t},$ obtained by adding edges between every pair of vertices of $\beta(t)$ which belongs to a common adhesion of $t.$
  
We now consider a graph $G$ where $\pack_{H}(G) \leq k$ and we assume that $G$ excludes as a minor the Dyck grid $\mathscr{D}_{t}^{Σ},$ for every $Σ\in  \Sbbb'_{H}.$
Under these circumstances, our aim is to find an $H$-cover whose size is bounded by some function of $t$ and $k.$
  
\paragraph{Graphs excluding Dyck grids.}
As a first step, we need a deeper understanding of how the graphs excluding $\mathscr{D}_{t}^{Σ}$ look like.
In general, the structure of graphs excluding a given graph as a minor is given by the Graph Minors Structure theorem (in short GMST).
However, the formal definition of GMST involves complicated concepts which we prefer not to introduce in this brief outline.
Instead we give a more compact statement, proved in \cite{thilikos2023excluding}.

Given a graph $H$ and a set $A\subseteq V(G),$ we say that $H$ is an   \emph{$A$-minor of $G$}
if there is a collection $\mathcal{S}=\{S_{v}\mid v\in V(H))\}$ of pairwise vertex-disjoint {connected}\footnote{A set $X\subseteq V(G)$ is \emph{connected} in $G$ if the induced subgraph $G[X]$ is a connected graph.} subsets of $V(G),$ each containing at least one vertex of $A$ and such that, for every edge $xy\in E(H),$ the set $S_{x}\cup S_{y}$ is connected in $G.$
Given an annotated graph $(G,A)$ where $G$ is a graph and $A\subseteq V(G),$ we define $\tw(G,A)$ as the maximum treewidth of an $A$-minor of $G.$ A streamlined way to restate the GMST is the following. 

\begin{proposition}[\!\! \cite{thilikos2023excluding}]\label{@mutilation}
There exists a function $f:\Nbbb\to \Nbbb$ such that 
every graph  $G$ excluding a graph on $k$ vertices as a minor, 
has a tree decomposition $(T,β)$ where, for every $t\in V(T),$ 
the torso $G_{t}$ contains some set $A_{t}$ where $\tw(G_{t},A_{t})≤f(k)$
and such that $G_{t}-A_{t}$ can be embedded in a surface of Euler genus at most $f(k).$ 
\end{proposition}

To deal with the exclusion of Dyck graphs (corresponding to surfaces), we need a more refined version of \cref{@mutilation} that works for every (closed and proper) set of surfaces $\Sbbb.$  In this direction,  Thilikos and Wiederrecht defined in \cite{thilikos2023excluding}  an extension of treewidth, namely 
$\Sbbb\mbox{-}\tw,$ where for a graph $G,$ 
\begin{eqnarray}
\begin{minipage}{13cm}
\textsl{$\Sbbb\mbox{-}\tw(G)$ is the minimum $k$ for which $G$ 
has a tree decomposition $(T,β)$ where, for every $t\in V(T),$ 
the torso   $G_{t}$ contains some set $A_t$ where $\tw(G_t,A_t)≤k$
and $G_t-A_t$ is embeddable in a surface in $\Sbbb.$ }
\end{minipage}\label{@translated}
\end{eqnarray} 

The main result of \cite{thilikos2023excluding} is that in order to  exclude the graphs in $\frak{D}_{H}=\{\mathscr{D}_{t}^{Σ}\mid Σ\in  \cobs(\Sbbb)\},$ we have to fix  the surface of 
\cref{@mutilation} to be one of the surfaces in $\Sbbb.$
\begin{proposition}
\label{@industrialism}
For every closed and proper set of surfaces $\Sbbb,$ there exists some function $f:\Nbbb\to\Nbbb$ such that, for every graph $G,$ 
if $G$ excludes all graphs in $\{\mathscr{D}_{t}^{Σ}\mid Σ\in  \sobs(\Sbbb)\}$ as minors, then $\Sbbb\mbox{-}\tw(G)≤f(t).$
\end{proposition}

Notice that the above already gives us the grid theorem when applied 
for the set $\Sbbb_{\emptyset}$ containing  the empty surface $Σ^{\varnothing}.$ It is easy to verify that $\tw+1=\Sbbb_{\emptyset}\mbox{-}\tw.$ As  $\sobs(\Sbbb_{\emptyset})=\{Σ^{(0,0)}\},$ \cref{@industrialism} implies that   graphs excluding  $\mathscr{D}_{t}^{Σ^{(0,0)}}=\mathscr{A}_{t}$ have bounded treewidth (see \cref{@exorbitant} for an example of an annulus grid).

\paragraph{From small treewidth modulators to small size modulators.}
\cref{@industrialism} gives valuable information on the structure of the graphs that exclude  the ``half-integrality generators'' in $\frak{D}_{H}=\{\mathscr{D}_{t}^{Σ}\mid Σ\in  \sobs(\Sbbb_{H})\}.$ Therefore, we can assume that $\Sbbb\mbox{-}\tw(G) \leq f(k),$ which provides a  tree decomposition as the one in \eqref{@translated}.
In order to make progress, we need to further refine this decomposition as small treewidth modulators are not particularly helpful in finding an $H$-cover of small size.
For this we exploit the assumption that $H$ is a shallow-vortex minor.

To elaborate, we need some additional information, analogous to the exclusion of a planar graph in \textbf{Step 1}.
This corresponds to the assumption that $H$ is a minor of the shallow-vortex grid $\mathscr{V}_{h'}$ for some $h'$ depending on $H.$
One can observe that $\mathscr{V}_{3(k+1)h'}$ contains a $\mathscr{V}_{h'}$-packing of size $(k+1)$.
Therefore, the assumption that $H$-$\pack(G) \leq k$ gives us the right to additionally assume that $G$ also excludes the shallow-vortex minor $\mathscr{V}_{3(k+1)h'}.$
Using this and the fact that $\Sbbb\mbox{-}\tw(G) \leq f(k),$ we are able to further restrict the decomposition of \eqref{@translated}.
To quantify this, we introduce a new graph parameter $\Sbbb\mbox{-}\tw_{\mathsf{apex}}$ defined as follows.
\begin{eqnarray}
\begin{minipage}{13cm}
\textsl{$\Sbbb\mbox{-}\tw_{\mathsf{apex}}(G)$  is the minimum $k$ for which $G$ 
has a tree decomposition $(T,β)$ where, for every $t\in V(T),$ 
the torso   $G_{t}$ contains some set $A_t$ where $|A_t|≤k$
and $G_t-A_t$ is embeddable in a surface in $\Sbbb.$ }
\end{minipage}\label{@ambivalent}
\end{eqnarray} 

Notice that the only difference between \eqref{@translated} and \eqref{@ambivalent} is the measure defined on the ``modulator'' $A_{t}.$
While in  \eqref{@translated} it is the treewidth of the annotated graph $(G_{t},A_{t}),$ in \eqref{@ambivalent} it is the \textsl{size} of $A_{t}.$
The first ingredient of our proof is that, 
under the absence of some shallow-vortex minor, the two parameters  $\Sbbb\mbox{-}\tw_{\mathsf{apex}}$ and  $\Sbbb\mbox{-}\tw_{\mathsf{}}$ are equivalent. This is proved in \cref{@reciprocal} by combining the 
results of \cite{thilikos2023excluding} with the results of \cite{ThilikosW22Killingconf} on the structure of the graphs 
excluding a shallow-vortex grid.

As a consequence, we may now assume that we have a tree decomposition $(T,β)$ as the one in \eqref{@ambivalent}.
This decomposition is not yet in position to play the role of the tree decomposition in \textbf{Step 2}, as its torsos may have unbounded size.
To circumvent this issue, in \cref{@reciprocal} we instead prove a local structure theorem for the exclusion of $\frak{D}_{H}\cup\{\mathscr{V}\}$  in the form of \cref{@enlightened}, that can be extended to a global one (the desired tree decomposition), using standard balanced separator arguments.
In the main part of our proof (\cref{@ilemigoddesscs}) we only use the local version of this structural result, while we comment on its global version in \cref{@persecuting}.
The general approach is to consider some big enough wall $W_{t}$ and locally focus on a torso $G_{t}$ that contains most of the essential part of $W_{t}.$

\paragraph{Torsos with Dyck walls.}

According to \cref{@industrialism}, \eqref{@ambivalent}, and the equivalence of $\Sbbb\mbox{-}\tw_{\mathsf{apex}}$ and $\Sbbb\mbox{-}\tw_{\mathsf{}},$  $G_{t}$ comes together with a set $A_{t}$ such that the graph $G_{t}'\coloneqq G_{t}-A_{t}$  is accompanied by some $Σ$-embedding for a surface $Σ\in\Sbbb_{H},$ where $H$ cannot be embedded.
However, we require some additional infrastructure in $G_{t}$ that will come in the form of a large wall-like object that is controlled by our $\Sigma$-embedding.

Notice that  every adhesion $β(t)\cap β(t')$ of $t$ defines 
a {separation $(X_{t'},Y_{t'})$ of $G-A_{t}$ of order at most $3$ where $G[X_{t'}\cap Y_{t'}]$  is drawn in $Σ$ as a clique}.
We fix the orientation $(X_{t'},Y_{t'})$ such that $V(G_t)\subseteq Y_{t'}$, thereby indicating that $Y_{t'}$ is the ``important'' part of the separation.
Due to the results in \cite{thilikos2023excluding},  $G_t$  contains a $(Σ;d)$-\emph{Dyck wall}%
\footnote{Here a $(Σ;d)$-Dyck wall is certifying the existence of the Dyck grid $\mathscr{D}_{d}^{Σ}$ as a minor (see \cref{@reexpressing}). For the precise definition, see \cref{@crystallize}.} 
$D_t,$ which is highly linked to the wall $W_t$ above.
Here $d$ will be chosen ``big enough'' so to ensure the applicability of the next steps of our proof.  
Also, we may assume that the ``essential'' part of $D_{t}$ is drawn ``inside'' $G_{t}$ in the sense that, for each $(X_{t'},Y_{t'}),$ at most one branch vertex of $D_{t}$ is in $Y_{t'}\setminus X_{t'}.$
The wall $W_{t}$ is chosen large enough to represent some \textsl{tangle}, that is an orientation of the separations of $G$ of some suitably bounded order.  
The way to algorithmically detect such a big wall $W_{t}$ is given in \cite{thilikos2023excluding} in the form of \cref{thm_algogrid}.

\paragraph{The role of Kuratowki-connectivity.}

We next make some observations on how hosts of $H$ can behave with respect to the $\Sigma$-embedding of $G_{t}$.
These observations will play a key role in understanding how to ``attack'' and later ``kill'' copies of $H$ in our graph.

The first comes from the non-$Σ$-embeddability property of $H$:
``minimal'' $H$-hosts in $G,$ called \emph{$H$-inflated copies}, \textsl{cannot be entirely inside $G_{t}$}, otherwise we would be able to embed $H$ in a surface where it cannot be embedded.
Another important feature comes from the fact that $H$ is Kuratowski-connected: every $H$-inflated copy $M$ in $G$ is ``well oriented'' with respect to the adhesions of $t$ in the sense that,  when $M$  traverses 
some adhesion  $X_{t'}\cap Y_{t'}=β(t)\cap β(t')$ of $G,$ \textsl{exactly one} of the two parts  of $M$ induced by $X_{t'}$ and $Y_{t'}$ should \textsl{not} be embeddable in the disk bounding $β(t)\cap β(t')$ with the vertices of $β(t)\cap β(t')$ on its boundary. 
This implies that the ``non-disk embeddable'' part will always lie inside the set $X_{t'}$ of the separation $(X_{t'},Y_{t'})$ above.
Given now some adhesion $β(t)\cap β(t'),$ we say that it is \emph{$H$-red} if it is intersected by the (unique, due to Kuratowski-connectivity) non-disk embeddable part of some $H$-inflated copy $M$ in $G.$ 
That way, it is convenient to visualize $H$-red adhesions as the ``entrances'' from which the $H$-inflated copies of $G$ ``invade'' $G_{t}.$

\paragraph{Updating the $Σ$-embedding.}

From our previous observations it follows that to eliminate all copies of $H$ locally in $G_{t}$ it suffices to deal with all $H$-inflated copies that invade $G_{t}$ through $H$-red adhesions.
Therefore, our next objective is to update  $A_{t},$ $G_{t}'=G_{t}-A_{t},$  and the $Σ$-embedding of $G_{t}'$ in a way that the remaining part of $G_{t}'$ will not contain any $H$-red adhesions, i.e., in a way that no invading $H$-inflated copy survives.

During our proof, this updating procedure will focus on some closed disk $Δ$ containing some collection of $H$-red adhesions (these disks will be gathered together in what we call \textsl{$H$-red railed flat vortices}) and detect some  separation $(X,Y)$ of $G$ where $X\setminus Y$ contains the vertices of the Dyck wall $D_{t}$ and $Y$ contains all $H$-red adhesions in $Δ$. We call such a separation a \emph{carving separation}.
Each time we find such a separation, we move $X\cap Y$ to $A$ and also move $Y\setminus X$ ``outside'' $G_{t}.$ As the set $X\cap Y$  adds up to the size of $A$ we also need that $X\cap Y$ has ``small'' order.
We refer to this operation as \textsl{taking a carving} of our $Σ$-embedding at the carving separation $(X,Y).$ 
When the whole procedure terminates,  none of the adhesions of the updated $G_{t}'$ is $H$-red. This implies that  $(V(G_{t}),A_{t})$ is what we call an \emph{$H$-local cover} of $G,$ that is: \textsl{if  the non-disk embeddable  
part of some $H$-inflated copy in $G$ intersects $V(G_{t})$ then it also intersects $A_{t}$}.

\medskip
To achieve the previously described objective we adopt the following strategy.
Recall that in the $Σ$-embedding of $G_{t},$ $H$-red adhesions are cliques of size at most three that may be drawn all around $Σ.$
Our first step is to show that $H$-red adhesions can be cornered in the ``interior'' of less than $k$ pairwise-disjoint territories of $Σ,$ each maintaining a large enough ``buffer''  around a disk where the $H$-red adhesions reside.
This step is materialized in \cref{@abstractness}.
Afterwards we refine these territories in order to bound their complexity in the sense that there is no large flow in $G_{t}$ that crosses through these territories.
This step is the subject of \cref{@conjecture}.
Through this refinement step we obtain some some additional structural information so that in the last part of the proof that is formalized in \cref{@translation}, these territories along with their infrastructure will allow us to finally eliminate all $H$-inflated copies by removing a bounded number of vertices from their interiors.

\paragraph{Redrawing $H$-inflated copies inside a railed flat vortex.}

To formalize the aforementioned territories that will encapsulate the $H$-red adhesions of our embedding, we utilize the concept of a \emph{railed nest} $(\Ccal,\Pcal)$ of $G$ \emph{around} some closed disk $Δ^{\rm int}$ of $Σ.$
Here $\Ccal=\langle C_{1},\ldots,C_{\ell}\rangle$ is a sequence of $\ell$ disjoint cycles of $G_{t},$ where each $C_{i}$ bounds some closed disk $Δ_{i}$ in $Σ,$ where $Δ^{\rm int} \subseteq Δ_{1} \subsetneq   \dots \subsetneq Δ_{\ell},$ 
along with a set of paths $\Pcal=\langle P_{1},\ldots,P_{\ell}\rangle,$ drawn in $Δ^{\rm ext}\coloneqq Δ_{\ell},$ not traversing the interior of $Δ^{\rm int},$ joining vertices of $C_{1}$ with  vertices of $C_{\ell},$  and traversing the cycles in $\Ccal$ \emph{orthogonally}, that is $P_{i}\cap C_{j}$ is connected for every $(i,j)\in[\ell]^2.$
We refer to such a railed nest, as a \emph{railed flat vortex} and  we refer to the disk $Δ^{\rm int}$ (resp. $Δ^{\rm ext}$) as its \emph{internal} (resp. \emph{external disk}).   
Moreover, if all $H$-red adhesions drawn in $Δ^{\rm ext}$ are also drawn inside $Δ^{\rm int},$ then we call it an \emph{$H$-red railed flat vortex}.
An important ingredient of our proof is to show that we may use the infrastructure of the cycles and the paths in $(\Ccal,\Pcal)$ in order to \textsl{redraw} inside  $Δ^{\rm ext}$ every $H$-inflated copy $M$ that invades $G_{t}$ via  an  $H$-red adhesion of $Δ^{\rm ext}.$
Even if the part of $M$ that is embedded inside $Δ^{\rm ext}$ is not necessarily a disk embedding, we can make this redrawing possible by using disk embedability properties emerging from the Kuratowski-connectivity of $H$ and the ``linkage combing'' lemma from \cite{GolovachST19Hitting,GolovachST22Combing,GolovachST20Hitting}.
We refer to this as \textsl{the redrawing lemma}  that is \cref{@horkheimer} in \cref{@enthroning}.

\paragraph{Gathering $H$-adhesions in railed flat vortices.}

The next step of our strategy, is to corner all $H$-red adhesions in the interior of less than $k$ $H$-red railed flat vortices.
Towards this, we take advantage of the infrastructure provided by the $(Σ;d)$-\emph{Dyck wall} $D_{t}.$ A brick of $D_{t}$ is called $H$-red if it ``contains'' an $H$-red adhesion.
More precisely, this is formalized by the notion of the \textsl{influence} of a brick defined in \cref{@abstractness}, which roughly corresponds to a set of $H$-red adhesions that are intersected or contained by a closed disk in $\Sigma$ that bounds the ``area'' that is enclosed by the corresponding brick.
This assigns each $H$-red adhesion to the influence of at most three neighbouring $H$-red bricks and defines a notion of distance between $H$-red adhesions expressed by the distance of the corresponding $H$-red bricks in $D_{t}.$
Next, we prove that under this distance notion, no scattered enough set of $H$-red bricks of size $k$ can exist.
For this, we use the fact that each $H$-red brick $B$ implies the existence of an $H$-inflated copy in $G$ that, due to the aforementioned ``redrawing lemma'', can be redrawn in a small radius around $B.$
This radius is bounded but also big enough so as to permit the redrawing. 
Likewise, we prove that there are few $H$-red bricks away from the exceptional and the simple cycle of $D_{t}$ (see \cref{@reexpressing} for a visualization of these two cycles).
Next we use a greedy procedure in order to group together this bounded number of bricks and maintain enough railed nest infrastructure around them to cluster them into less than $k$ railed flat vortices.
The construction is completed by creating two more railed flat vortices, one for the simple cycle of $D_{t}$ and one for the exceptional one.

\paragraph{Refining $H$-red railed flat vortices.}

We are now in the position where we have defined a set of less than $k$ many $H$-red railed flat vortices whose internal disks contain all $H$-red adhesions and whose external disks are pairwise disjoint.
The next step is to further refine these flat vortices which is done in \cref{@conjecture}.

In our proof, we treat what is drawn in the external disk $Δ^{\rm ext}$ as a vortex in the classic sense and our goal is to bound their \textsl{depth}, that is to ensure that no large \textsl{transaction} goes through the society defined by each railed flat vortex.
Each of them consists of a subgraph $G_{Δ^{\rm ext}}$ of $G$ (the one that is drawn in $Δ^{\rm ext}$) where the vertices in the boundary of the external disk $Δ^{\rm ext}$ are  arranged in some cyclic ordering $Ω_{Δ^{\rm ext}}.$
A \emph{segment} of~$Ω_{Δ^{\rm ext}}$ is a set~${S \subseteq V(\Omega_{Δ^{\rm ext}})}$ such that there do 
not exist~${s_1,s_2 \in S}$ and~${t_1,t_2 \in V(\Omega_{Δ^{\rm ext}}) \setminus S}$ such that~${s_1,t_1,s_2,t_2}$ occur in~$\Omega_{Δ}$ in the order listed. 
A \emph{transaction} in~${(G_{Δ^{\rm ext}},\Omega_{Δ^{\rm ext}})}$ is a set of pairwise disjoint paths, drawn in $Δ,$ between two disjoint  segments~$A,B$ of~$\Omega_{Δ^{\rm ext}}.$ 
The \emph{depth} of~${(G_{Δ^{\rm ext}},\Omega_{Δ^{\rm ext}})}$ is the maximum size of a transaction in~${(G_{Δ^{\rm ext}},\Omega_{Δ^{\rm ext}})}.$
Our next objective is to refine each of our $H$-red railed flat vortices so that, in the end, some disk $Δ'\subseteq Δ^{\rm ext}$  defines a vortex $(G_{Δ'},Ω_{Δ'})$ of \textsl{bounded} depth and, moreover, the vertices in the boundary of $Δ'$ are all connected with disjoint paths to the  boundary of the external disk $Δ^{\rm ext}.$
We do this as follows: If there is no transaction in~${(G_{Δ^{\rm ext}},\Omega_{Δ^{\rm ext}})}$  where a big part of its paths also traverse $Δ^{\rm int},$ we make use of the ``nest tightening''-lemma from  \cite{ThilikosW22Killingconf}   in order to either update the nest to a ``tighter'' one (which allows us to recurse), or find the disk $Δ'$ claimed above, or find a small-order carving separation $(X,Y)$ (again defined by some closed disk) at which we may take a carving of our $Σ$-embedding.
If there is a transaction in~${(G_{Δ^{\rm ext}},\Omega_{Δ^{\rm ext}})}$  where a big enough part of its paths also traverse $Δ^{\rm int},$ then we use this transaction in order to split the vortex into two vortices and recurse.
This split is performed using the path infrastructure offered by the transaction, along with the cycles of the railed nest and may result in either a ``tighter'' $H$-red railed flat vortex around $Δ^{\rm int}$ or in two $H$-red railed flat vortices.
In both cases, this allows us to recurse.
As we know by the redrawing lemma, that $k$ such $H$-red railed flat vortices may give an $H$-packing, this procedure will end and will produce less than $k$ $H$-red railed flat vortices, each with some closed disk $Δ'$ defining a bounded depth vortex, as above.

\paragraph{Killing $H$-red flat  vortices.}

In the next and final step we exploit all the additional structure we obtained via the refinement step and ``attack''  each of the obtained $H$-red railed flat vortices separately. For each of them we ``kill'' all $H$-red adhesions residing in its internal disk $Δ^{\rm int} \subseteq Δ'$ by identifying a bounded set of vertices drawn within $Δ^{\rm int}$.
This is performed in \cref{@translation}.

Towards this, recall that the refinement step ensures that the vortex $(G_{Δ'},Ω_{Δ'})$ has bounded depth. 
Using a known result of \cite{kawarabayashi2020quickly}, we  construct a \textsl{bounded width} \textsl{linear decomposition} of $G_{Δ'},$ that is a path decomposition $\langle X_1,X_2,\dots,X_n\rangle$ where every bag $X_{i}$ contains some vertex $x_{i}$ of the boundary of $Δ'$ in a way that these $x_{1},\ldots,x_{n}$ are the vertices of $V(Ω_{Δ'}),$ appearing in the same order as they appear in $Ω_{Δ'}.$
We next partition $\langle X_1,X_2,\dots,X_n\rangle$ into $r$ segments $\{\langle X_{p_{i-1}},\ldots,X_{p_i-1},X_{p_i}\rangle, i\in[r]\}$ each ``minimally capable'' to host some $H$-red adhesion from which an $H$-inflated copy invades $G_{t}.$
Likewise, we find equally many $H$-inflated copies in $G$ where the parts drawn inside $Δ'$ are disjoint.
Then we bound the number of these segments by proving that they may be extended to an  $H$-packing of size $r,$  inside $Δ^{\rm ext}.$
For this, we use the full power of the redrawing lemma (\cref{@horkheimer}) along with the infrastructure offered by the railed nest.
As long as there are less than $k$ segments in $\{\langle X_{p_{i-1}},\ldots,X_{p_i}\rangle, i\in[r]\}$ we define  a carving separation $(X,Y)$ of $G$ where $Y$ contains the union of all $X_{p_{i-1}}\cup X_{p_{i}},$ $i\in[r]$ and $X\cap Y$ contains the union of all $(X_{p_{i-1}}\cap X_{p_{i}})\cup (X_{p_{i}}\cap X_{p_{i+1}}),$ $i\in[r].$
As the size of $X\cap Y$ depends on $k$ and the width of the decomposition (that is bounded), we have that $(X,Y)$ has bounded order.
Therefore, we may {take a carving} of our $Σ$-embedding at the carving separation $(X,Y).$ When this is done for all 
$H$-red flat vortices, we know that what remains from $G_{t}'$ has a $Σ$-embedding that has no $H$-red adhesions.

\paragraph{From local to global.}

Recall that all above steps were applied to an initial torso $G_{t}$ and, in particular, to the corresponding $Σ$-embedding of $G_{t}' = G_{t} - A_{t}.$
In the end, what we obtained is a new $G'_{t}$ and $A_{t}$ and a $Σ$-embedding of  $G'_{t}$ with no $H$-red adhesions.
The elimination of $H$-red adhesions was done by taking successive carvings of the $Σ$-embedding of $G_{t}'$ at a bounded number of carving separations $(X,Y),$ each of bounded order.
This came at some cost: By taking these carvings, we added all $X \cap Y$'s  to $A_t$ and, moreover, removed all $Y \setminus X$'s from $G'$.
As we already mentioned above, the resulting pair $(V(G_{t}),A_{t})$ is an \emph{$H$-local cover} of $G,$ which means that  if the non-disk-embeddable part of some $H$-inflated copy in $G$ intersects $V(G_{t}),$  then it also intersects $A_{t}.$ 
At this point we should forget the initial tree decomposition and just keep in mind that we started with a wall $W_{t}$ of some torso $G_{t}$ and we finally computed an $H$-local cover $(X_t,A_t)$ of $G$ where $X_{t}$ still maintains a big part of the Dyck grid $D_{t}$ that is the ``essential'' part of $W_{t}.$ 
This  constitutes the proof of a ``local structure theorem'' that, apart from excluding the Dyck grids in $\{\mathscr{D}_{t}^{Σ}\mid \Sigma\in  \sobs(\Sbbb_{H})\},$ assumes that $G$ has no $H$-packing of size $k,$ and, given a big enough wall $W,$ returns an $H$-local cover $(X,A)$ of $G$ where the ``essential part'' of $W$ is intact in $X.$
What we need now is to bring this result to the form of a global structure theorem, that is a new tree decomposition $(T,β)$ where each node $t$ is accompanied by a set $α(t)\subseteq β(t)$ where $(β(t),α(t))$ is an $H$-local cover of $G.$
This decomposition may serve as the analogue of the tree decomposition in \textbf{Step 2}.
The proof, presented in \cref{@liabilities}, follows along standard balanced separator arguments.

\paragraph{From connected to disconnected.}
Given  the decomposition $(T,β)$ from above, we may now delete adhesions, as it was performed in \textbf{Step 2}.
After this, we obtain a $\Zcal$-local cover $(X,A)$ of $G$ such that $G-X$ is $H$-minor-free and $|A|$ is bounded.
With some more preprocessing, explained in \cref{@analogistic} and \cref{@restrained}, this decomposition may be used to obtain a separation $(X,Y)$ of $G$ of bounded order where $(X,X\cap Y)$ is an $\Zcal$-local cover and {$G[Y\setminus X]$} is $H$-minor-free.
Notice that at this point, with $H$ being connected, we may conclude.
However, in order to obtain a fully general result, in \cref{@expurgated}, we deal with the case that $H$ is not connected by setting up a recursive algorithm which uses the connected case as the base case and each time it is called, it is called for the union of a smaller number of connected components of $H.$
The final outcome is an $H$-cover of $G$ whose size depends single-exponentially on the size of the excluded Dyck grids from $\frak{D}_{H}$ and the size of the maximum $H$-packing of $G.$ 

\subsection{Organization of the paper}

In \cref{@reformation} we present all concepts of graphs minors that are necessary for our proof. 
This includes the presentation of the local structure theorem on the exclusion of the Dyck wall~\cite{thilikos2023excluding}.
Preliminary results are presented in \cref{@deprivation}. 
They concern two sets of results. 
The first concerns the notion of a $\Zcal$-red cell and a series of lemmata about how an $H$-inflated copy may be drawn relatively to it (\cref{@proclaimed}). 
The second one contains the redrawing lemma and its proof (\cref{@enthroning}). 
\cref{@ilemigoddesscs} contains the main part of the proofs that aims at the proof of our local structure theorem.
This includes the additional exclusion of shallow-vortex minors (\cref{@reciprocal}), the gathering of the railed flat vortices (\cref{@abstractness}), the refining of flat red vortices (\cref{@conjecture}), and finally the ``elimination'' of the flat red vortices (\cref{@translation}).
The last section of the proof, \autoref{@persecuting}, is dedicated to the proof of \cref{@indescribably}. 
This includes the proof of a global structure theorem, 
the extraction from the derived tree decomposition of the three outcomes of \cref{@indescribably} for the case of a connected graph in $\Zcal$ (\cref{@analogistic} and  \cref{@restrained}), and, finally, the treatment of the disconnected case (\cref{@expurgated}). 
The proof of the final result is given in \cref{@homerische} and includes the lower bound \cref{@instigated} and the final proof of \cref{@admiration} (\cref{@prophetesses}).
In \autoref{@regressing}, we give some conclusions, consequences, open problems, and conjectures.

\section{Preliminaries}\label{@reformation}

In this section we introduce most of the required notions we need to proceed with the proofs in the following sections.

\subsection{Basic concepts}

In this subsection we present some basic concepts about sets, integers, and graphs.

\paragraph{Sets and integers.} We denote by $\mathbb{N}$ the set of non-negative integers, by $\Nbbb_{\geq n}$, $n > 1$, to be the set $\Nbbb \setminus \{ m \in \Nbbb \mid m < n \}$, and by $\Nbbb^\mathsf{even}$ the set of even numbers in $\Nbbb.$
Given two integers $p, q,$ where $p \leq q,$ we denote by $[p, q]$ the set $\{p, \dots, q\}.$
For an integer $p \geq 1,$ we set $[p] = [1, p]$ and $\mathbb{N}_{\geq p} = \mathbb{N} \setminus [0, p - 1].$
For a set $S,$ we denote by $2^{S}$ the set of all subsets of $S$ and by $\binom{S}{2}$ the set of all subsets of $S$ of size $2.$
If $\mathcal{S}$ is a collection of objects where the operation $\cup$ is defined, then we denote $\cupall S = \bigcup_{X \in \mathcal{S}} X.$
Also, given a function $f\colon A\to B$ we always consider its extension $f:2^A\to B$ such that for every $X\subseteq A,$ $f(X)=\{f(a)\mid a \in X\}.$

\paragraph{Basic concepts on graphs.} A graph $G$ is a pair $(V, E)$ where $V$ is a finite set and $E \subseteq \binom{V}{2},$ i.e. all graphs in this paper are undirected, finite, and without loops or multiple edges. When we denote 
an edge $\{x,y\},$ we use instead the simpler notation $xy$ (or $yx$). 
We write $\gall$ for the set of all graphs.
We also define $V(G) = V$ and $E(G) = E.$

We say that a pair $(A, B) \in 2^{V(G)} \times 2^{V(G)}$ is a \emph{separation} of $G$ if $A \cup B = V(G)$ and there is no edge in $G$ between $A \setminus B$ and $B \setminus A.$
We call $|A \cap B|$ the \emph{order} of $(A, B).$

If $H$ is a subgraph of $G,$ that is $V(H)\subseteq V(G)$ and $E(H)\subseteq E(G)$, we denote this by $H\subseteq G.$
Given two graphs $G_{1}$ and $G_{2},$ we denote $G_{1}\cup G_{2}=(V(G_{1})\cup V(G_{2}),E(G_{1})\cup E(G_{2})).$ We also use $G_{1}+G_{2}$ to denote the disjoint union of $G_{1}$ and $G_{2}.$
Also, given a $t \in \Nbbb,$ we denote by $t\cdot G$ the disjoint union of $t$ copies of $G.$

Given a vertex $v \in V(G),$ we denote by $N_{G}(v)$ the set of vertices of $G$ that are adjacent to $v$ in $G.$
Also, given a set $S \subseteq V(G),$ we set $N_{G}(S) = \bigcup_{v \in S} N_{G}(v) \setminus S.$

For $S \subseteq V(G),$ we set $G[S] = (S, E \cap \binom{S}{2})$ and use $G - S$ to denote $G[V(G) \setminus S].$ We say that $G[S]$ is an \emph{induced (by $S$) subgraph} of $G.$

Given an edge $e = uv \in E(G),$ we define the \emph{subdivision} of $e$ to be the operation of deleting $e,$ adding a new vertex $w,$ and making it adjacent to $u$ and $v.$ 
Given two graphs $H$ and $G,$ we say that $H$ is  \emph{a subdivision} of $G$ if $H$ can be obtained from $G$ by subdividing edges.

The \emph{contraction} of an edge $e = uv \in E(G)$ results in a graph $G'$ obtained from $G \setminus \{ u, v \}$ by adding a new vertex $w$ adjacent to all vertices in the set $(N_{G}(u) \cup N_{G}(v)) \setminus \{ u, v \}.$
A graph $H$ is a \emph{minor} of a graph $G$ if $H$ can be obtained from a subgraph of $G$ after a series of edge contractions.
We denote this relation by $\leq.$
Given a set $\Hcal$ of graphs, we use $\Hcal≤G$ in order 
to denote that at least one of the graphs in $\Hcal$ is a minor of $G.$
We refer the reader to~\cite{diestel2016graph} for any undefined terminology on graphs.

\subsection{Drawings in surfaces}

In this subsection we introduce a series of notions on surfaces and the ways to draw and decompose graphs with respect to them.
We largely use the notation from \cite{kawarabayashi2020quickly}. 

\paragraph{Surfaces.}
Given a pair $(\mathsf{h},\mathsf{c}) \in \mathbb{N} \times [0,2]$ we define $\Sigma^{(\mathsf{h}, \mathsf{c})}$ to be the two-dimensional surface without boundary created from the sphere by adding $\mathsf{h}$ handles and $\mathsf{c}$ crosscaps (for a more detailed definition see~\cite{MoharT01Graphs}).
If $\mathsf{c} = 0,$ the surface $\Sigma^{(\mathsf{h},\mathsf{c})}$ is an \emph{orientable} surface, otherwise it is a \emph{non-orientable} one.
By Dyck's theorem \cite{Dyck1888Beitrage,Francis99ConwayZIP}, two crosscaps are equivalent to a handle in the presence of a (third) crosscap.
This implies that the notation $\Sigma^{(\mathsf{h}, \mathsf{c})}$ is sufficient to denote all two-dimensional surfaces without boundary.

\paragraph{Societies.}

Let $\Omega$ be a cyclic permutation of the elements of some set which we denote by $V(\Omega).$
A \emph{society} is a pair $(G,\Omega),$ where $G$ is a graph and $\Omega$ is a cyclic permutation with $V(\Omega)\subseteq V(G).$
A \emph{cross} in a society $(G,\Omega)$ is a pair $(P_1,P_2)$ of disjoint paths\footnote{When we say two paths are \emph{disjoint} we mean that their vertex sets are disjoint.} in $G$ such that $P_i$ has endpoints $s_i,t_i\in V(\Omega)$ and is otherwise disjoint from $V(\Omega),$ and the vertices $s_1,s_2,t_1,t_2$ occur in $\Omega$ in the order listed.

\paragraph{Drawing a graph in a surface.}

Let $\Sigma$ be a surface, possibly with boundary.
A \emph{drawing} (with crossings) in $\Sigma$ is a triple $\Gamma=(U,V,E)$ such that
\begin{itemize}
\item $V$ and $E$ are finite, 
\item $V\subseteq U \subseteq \Sigma,$ 
\item $V\cup\bigcup_{e\in E}e=U$ and $V\cap (\bigcup_{e\in E}e)=\emptyset,$ 
\item for every $e\in E,$ either $e=h((0,1)),$ where $h\colon[0,1]_{\mathbb{R}}\to U$ is a homeomorphism onto its image with $h(0),h(1)\in V,$ or $e = h(\mathbb{S}^{2} - (1, 0)),$ where $h : \mathbb{S}^{2} \to U$ is a homeomorphism onto its image with $h(0,1) \in V,$ and
\item if $e,e'\in E$ are distinct, then $|e\cap e'|$ is finite.
\end{itemize}
We call the set $V,$ sometimes denoted by $V(\Gamma),$ the \emph{vertices of $\Gamma$} and the set $E,$ denoted by $E(\Gamma),$ the \emph{edges of $\Gamma$}. We also denote $U(Γ)=U.$
If $G$ is a graph and $\Gamma=(U,V,E)$ is a drawing with crossings in a surface $\Sigma$ such that $V$ and $E$ naturally correspond to $V(G)$ and $E(G)$ respectively, we say that $\Gamma$ is a \emph{drawing} of $G$ in $\Sigma$ (possibly with crossings).
In the case where no two edges in $E(\Gamma)$ have a common point, we say that $\Gamma$ is a \emph{drawing} of $G$ in $Σ$ \emph{without crossings}.
In this last case, the connected components of $Σ\setminus U,$ are the \emph{faces} of $\Gamma.$

\paragraph{$Σ$-decompositions.}
Let $\Sigma$ be a surface, possibly with boundary.
If $\Sigma$ has a boundary, then we denote it by $\bd(Σ).$ Also we refer to $Σ\setminus\bd(Σ)$ as the \emph{interior} of $Σ.$
A \emph{$Σ$-decomposition} of a graph $G$ is a pair $δ=(\Gamma,\mathcal{D}),$ where $\Gamma$ is a drawing of $G$ in $Σ$ with crossings, and $\mathcal{D}$ is a collection of closed disks, each a subset of $Σ$ such that
\begin{enumerate}
\item the disks in $\mathcal{D}$ have pairwise disjoint interiors, 
\item the boundary of each disk in $\mathcal{D}$ intersects $\Gamma$ in vertices only, 
\item if $Δ_1,Δ_2\in\mathcal{D}$ are distinct, then $Δ_1\capΔ_2\subseteq V(\Gamma),$ and 
\item every edge of $\Gamma$ belongs to the interior of one of the disks in $\mathcal{D}.$ 
\end{enumerate} 
  
A $Σ$-\emph{embedding} of a graph $G,$ is a $Σ$-decomposition $δ = (Γ,\Dcal)$ where  $\Dcal$ is a collection of closed disks such that, for any disk in $\Dcal,$ only a single edge of $Γ$ is drawn in its interior.  
For simplicity, we make the convention, that when we refer to a $Σ$-embedding we just refer to the drawing of $Γ,$ as the choice of $\Dcal$ is obvious in this case.

For a $Σ$-decomposition $δ = (Γ, \Dcal)$, let $N$ be the set of all vertices of $\Gamma$ that do not belong to the interior of the disks in $\mathcal{D}.$ 
We refer to the elements of $N$ as the \emph{nodes} of $δ.$
If $Δ \in \mathcal{D},$ then we refer to the set $Δ - N$ as a \emph{cell} of $δ.$
We denote the set of nodes of $δ$ by $N(δ)$ and the set of cells by $C(δ).$

For a cell $c \in C(δ)$ the set of nodes that belong to the closure of $c$ is denoted by $\tilde{c}.$
Given a cell $c \in C(δ),$ we define the \emph{disk} of $c$ as $\Delta_{c} \coloneqq \bd(c) \cup c.$
For a cell $c \in C(δ)$ we define the graph $\sigma_{\delta}(c),$ or $\sigma(c)$ if $\delta$ is clear from the context, to be the subgraph of $G$ consisting of all vertices and edges drawn in $Δ_{c}.$ 
We define $\pi_{δ}\colon N(δ)\to V(G)$ to be the mapping that assigns to every node in $N(δ)$ the corresponding vertex of $G.$
We also define the set of \emph{ground vertices} in $\delta$ as $\ground(\delta) \coloneqq \pi_{\delta}(N(\delta)).$

Let $G$ be a graph, $\Sigma$ be a surface, and $\delta = (\Gamma, \mathcal{D})$ be a $\Sigma$-decomposition of $G.$
A cell $c \in C(δ)$ is called a \emph{vortex} if $|\tilde{c}| \geq 4.$
Moreover, we call $\delta$ \emph{vortex-free} if no cell in $C(\delta)$ is a vortex.
  
Given a set $X \subseteq V(G),$ we denote by $\delta - X$ the $\Sigma$-decomposition $δ' = (\Gamma', \mathcal{D}')$ of $G - X$ where $\Gamma'$ is obtained by the drawing $Γ$ after removing all points in $\pi_{\delta}^{-1}(X)$ and all drawing of edges with an endpoint in $X.$
For every point $x \in \pi_{\delta}^{-1}(X)$ we pick $\Delta_{x}$ to be an open disk containing $x$ and not containing any point of some remaining vertex or edge and such that no two such disks intersect.
We also set $\Delta_{X} = \bigcup_{x \in \pi_{\delta}^{-1}(X)} \Delta_x$ and we define $\mathcal{D}' = \{ D \setminus \Delta_{X} \mid D \in \Dcal \}$.
Clearly, there is a one to one correspondence between the cells of $\delta$ and the cells of $\delta'.$
If a cell $c$ of $\delta$ corresponds to a cell $c'$ of $\delta',$ then we call $c'$ the \emph{heir} of $c$ in $\delta'$ and we call $c$ the \emph{precursor} of $c'$ in $\delta.$

\paragraph{$\delta$-aligned disks.}

Let $\delta = (\Gamma, \mathcal{D}),$ $\Gamma = (U,V,E),$ be a $\Sigma$-decomposition of a graph $G.$
We say that closed disk $\Delta$ in $\Sigma$ is $\delta$-\emph{aligned} if its boundary intersects $\Gamma$ only in nodes of $\delta$.
We denote by $\Omega_{\Delta}$ one of the cyclic orderings of the vertices on the boundary of $\Delta.$
We define the \emph{inner graph} of a $\delta$-aligned disk $\Delta$ as 
$$\inG_{δ}(Δ) \coloneqq \bigcup_{\textrm{$c \in C(δ)$ and $c \subseteq Δ$}} σ(c)$$ 
and the \emph{outer graph of $Δ$} as 
$$\outG_{δ}(Δ) \coloneqq \bigcup_{\textrm{$c \in C(δ)$ and $c \cap Δ \subseteq \ground(δ)$}} σ(c).$$

Given an arc-wise connected set $\Delta \subseteq \Sigma$ such that $U \cap \Delta \subseteq N$, we use $G \cap \Delta$ to denote the subgraph of $G$ consisting of the vertices and edges of $G$ that are drawn in $\Delta.$
Note that, in particular, the above definition applies in the case that $\Delta$ is a $\delta$-aligned disk of $\Sigma$ and $G \cap \Delta = \inG_{\delta}(\Delta).$
We also define $\Gamma\cap \Delta \coloneqq (U \cap \Delta, V \cap \Delta, \{ e \in E \mid e \subseteq \Delta \})$ and observe that $\Gamma \cap \Delta$ is a drawing of $G \cap \Delta$ in $\Delta$.

\paragraph{Renditions.}

A \emph{rendition} of a society $(G, \Omega)$ in a disk $\Delta$ is a $\Delta$-decomposition $\rho$ of $G$ such that such that $\pi_{\rho}(N(\rho)\cap \bd(Δ)) = V(\Omega)$, mapping one of the two cyclic orders (clockwise or counter-clockwise) of $\bd(Δ)$ to the order of $\Omega.$

Given a $\Sigma$-decomposition $\delta$ of a graph $G$ and a $\delta$-aligned disk $\Delta$ of $\Sigma,$ we denote by $\delta \cap \Delta$ the rendition $(\Gamma \cap \Delta, \{ \Delta_{c} \in \mathcal{D} \mid \Delta_{c} \subseteq \Delta \})$ of the society $(\inG_{\delta}(\Delta), \Omega_{\Delta})$ in $\Delta.$

\paragraph{Traces.}

Let $δ$ be a $Σ$-decomposition of a graph $G$ in a surface $Σ.$ 
For every cell $c \in C(δ)$ with $|\tilde{c}| = 2$ we select one of the components of $\bd(c) - \tilde{c}.$  
This selection is called a \emph{tie-breaker} in $δ,$ and we assume every $Σ$-decomposition to come equipped with a tie-breaker. 
Let $Q \subseteq G$ be either a cycle or a path that uses no edge of $σ(c)$ for every vortex $c \in C(δ).$ 
We say that $Q$ is \emph{grounded in $δ$} if either $Q$ is a non-zero length path with both endpoints in $π_{δ}(N(δ)),$ or $Q$ is a cycle, and it uses edges of $σ(c_{1})$ and $σ(c_{2})$ for two distinct cells $c_{1}, c_{2} \in C(δ).$
A $2$-connected subgraph $H$ of $G$ is said to be \emph{grounded in $δ$} if every cycle in $H$ is grounded in $δ.$

If $Q$ is grounded in $δ$, we define the \emph{trace} of $Q$ as follows.
Let $P_{1}, \dots, P_{k}$ be distinct maximal subpaths of $Q$ such that $P_{i}$ is a subgraph of $σ(c)$ for some cell $c.$ Fix an index $i.$
The maximality of $P_{i}$ implies that its endpoints are $π_δ(n_{1})$ and $π_δ(n_{2})$ for distinct $δ$-nodes $n_{1}, n_{2} \in N(δ).$
If $|\tilde{c}| = 2,$ define $L_{i}$ to be the component of $\bd(c) - \{ n_{1}, n_{2} \}$ selected by the tie-breaker, and if $|\tilde{c}| = 3,$ define $L_{i}$ to be the component of $\bd(c) - \{ n_{1}, n_{2} \}$ that is disjoint from $\tilde{c}.$
Finally, we define $L'_{i}$ by slightly pushing $L_{i}$ to make it disjoint from all cells in $C(δ).$ We define such a curve $L'_{i}$ for all $i$ while ensuring that the curves intersect only at a common endpoint.
The \emph{trace} of $Q$ is defined to be $\bigcup_{i \in [k]} L'_{i}.$
So the trace of a cycle is the homeomorphic image of the unit circle, and the trace of a path is an arc in $Σ$ with both endpoints in $N(δ).$

\paragraph{Grounded graphs.}
Let $\delta$ be a $\Sigma$-decomposition of a graph $G.$ 
Let $Q \subseteq G$ be either a cycle or a path that uses no edge of a vortex cell in $\delta.$
We say that $Q$ is $\delta$-\emph{grounded} if either $Q$ is a non-trivial path with both endpoints in $\pi_{\delta}(N(\delta))$ or $Q$ is a cycle that contains edges of $\sigma(c_{1})$ and $\sigma(c_{2})$ for at least two distinct cells $c_{1}, c_{2} \in C(\delta).$
A $2$-connected subgraph $H$ of $G$ is said to be $\delta$-\emph{grounded} if every cycle in $H$ is grounded in $\delta.$

\subsection{More about vortices}

We now introduce several concepts and results around vortices. 
Among them, the most important is the one of a linear decomposition of a vortex and its width and the notion of a (railed) nest around a disk that will be important 
in the proofs of the lemmata of \cref{@ilemigoddesscs}.

\paragraph{Paths.}
If~$P$ is a path and~$x$ and~$y$ are vertices on~$P,$ we denote by~${xPy}$ the subpath of~$P$ with endpoints~$x$ and~$y.$
Moreover, if~$s$ and~$t$ are the endpoints of~$P,$ and we order the vertices of~$P$ by traversing~$P$ from~$s$ to~$t,$ then~${xP}$ denotes the path~${xPt}$ and~${Px}$ denotes the path~${sPx}.$
Let~$P$ be a path from~$s$ to~$t$ and~$Q$ be a path from~$q$ to~$p.$
If~$x$ is a vertex in~${V(P) \cap V(Q)}$ such that~$Px$ and~$xQ$ intersect only in $x$, then~${PxQ}$ is the path obtained from the union of~$Px$ and~$xQ.$
Let~${X,Y \subseteq V(G)}.$
A path is an \emph{$X$-$Y$-path} if it has one endpoint in $X$ and the other in $Y$ and is internally disjoint from~${X \cup Y},$
Whenever we consider~$X$-$Y$-paths we implicitly assume them to be ordered starting in $X$ and ending in $Y,$ except if stated otherwise.
An \emph{$X$-path} is an $X$\nobreakdash-$X$\nobreakdash-path of length at least one.
In a society~$(G,\Omega),$ we write~$\Omega$-path as a shorthand for a~$V(\Omega)$-path.

\paragraph{Segments.}
Let~$(G,\Omega)$ be a society.
A \emph{segment} of~$\Omega$ is a set~${S \subseteq V(\Omega)}$ such that there do not exist~${s_1,s_2 \in S}$ and~${t_1,t_2 \in V(\Omega) \setminus S}$ such that~${s_1,t_1,s_2,t_2}$ occur in~$\Omega$ in the order listed. 
A vertex~${s \in S}$ is an \emph{endpoint} of the segment~$S$ if there is a vertex~${t \in V(\Omega) \setminus S}$ which immediately precedes or immediately succeeds~$s$ in the order~$\Omega.$
For vertices~${s,t\in V(\Omega)},$ if~$t$ immediately precedes~$s$,
 we define~$s\Omega t$ to be the \emph{trivial segment}~$V(\Omega),$
and otherwise we define~$s\Omega t$ to be the uniquely determined segment with first vertex~$s$ and last vertex~$t.$

\paragraph{Linkages.}\label{@mediterranean}
Let~$G$ be a graph.
A \emph{linkage} in~$G$ is a set of pairwise vertex-disjoint paths.
In slight abuse of notation, if~$\mathcal{L}$ is a linkage, we use~$V(\mathcal{L})$ and~$E(\mathcal{L})$ to denote~${\bigcup_{L\in\mathcal{L}}V(L)}$ and~${\bigcup_{L\in\mathcal{L}}E(L)}$ respectively.
Given two sets~$A$ and~$B$, we say that a linkage~$\mathcal{L}$ is an \emph{$A$-$B$-linkage} if every path in~$\mathcal{L}$ has one endpoint in~$A$ and one endpoint in~$B.$
We call $|\mathcal{L}|$ the \emph{order} of $\mathcal{L}.$
  
\paragraph{Transactions.} Let~${(G,\Omega)}$ be a society. 
A \emph{transaction} in~${(G,\Omega)}$ is an $A$-$B$-linkage for disjoint segments~$A,B$ of~$\Omega.$ 
We define the \emph{depth} of~${(G,\Omega)}$ as the maximum order of a transaction in~${(G,\Omega)}.$

Let $\mathcal{T}$ be a transaction in a society $(G, \Omega).$ 
We say that $\mathcal{T}$ is \emph{planar} if no two members of $\mathcal{T}$ form a cross in $(G, \Omega).$ 
An element $P \in \mathcal{T}$ is \emph{peripheral} if there exists a segment $X$ of $\Omega$ containing both endpoints of $P$ and no endpoint of another path in $\mathcal{T}.$ 
A transaction is \emph{crooked} if it has no peripheral element.

\paragraph{Vortex societies and linear decompositions.}
  Let~$Σ$ be a surface and~$G$ be a graph.
  Let~${δ = (\Gamma,\mathcal{D})}$ be a $Σ$-decomposition of~$G.$
  Every vortex~$c$ defines a society~${(σ(c),\Omega)},$ called the \emph{vortex society} of~$c,$ by saying that~$\Omega$ consists of the vertices in $\pi_{δ}(\tilde{c})$   in the order given by~$\Gamma.$
  (There are two possible choices of~$\Omega,$ namely~$\Omega$ and its reversal. 
  Either choice gives a valid vortex society.)

Let $(G,\Omega)$ be a society.
A \emph{linear decomposition} of $(G,\Omega)$ is a sequence  $$\mathcal{L}= \langle X_1,X_2,\dots,X_n,v_1,v_2,\dots,v_n \rangle$$ where $v_1,v_2,\dots,v_n$ is a labelling $V(\Omega)$ such that $v_1,v_2,\dots,v_n$ occur in that order on $\Omega$ and $X_1,X_2,\dots,X_n$ are subsets of $V(G)$ such that
  \begin{itemize}
   
    \item  for all $i\in[n],$ $v_i\in X_i,$
    \item $\bigcup_{i\in[n]}X_i=V(G),$ 
    \item for every $uv\in E(G),$ there exists $i\in [n]$ such that $u,v\in X_i,$ and 
    \item for all $x\in V(G),$ the set $\{ i\in[n] \mid x\in X_i \}$ forms an interval in $[n].$
  \end{itemize}
The \emph{adhesion} of a linear decomposition is $\max_{i\in [n-1]}|X_i\cap X_{i+1}|,$ its \emph{width} is $\max_{i\in [n]}|X_i|.$

\begin{proposition}[\!\! \cite{kawarabayashi2020quickly}]
\label{@prefascist}
If a society $(G,Ω)$ has depth at most $θ,$ then it has 
a linear decomposition of adhesion at most $2θ.$
Moreover there exists an algorithm that either finds a transaction of order $>θ$ or outputs such a decomposition in time $\Ocal(θ|V(G)|^2).$
\end{proposition}

Let $δ = (\Gamma,\mathcal{D})$ be a $Σ$-decomposition of a graph $G$ in a surface $Σ.$
Let $C$ be a cycle in $G$ that is grounded in $δ,$ such that the trace $T$ of $C$ bounds a closed disk $Δ_{C}$ in $Σ.$
We define the \emph{outer} (resp. \emph{inner}) \emph{graph} of $C$ in $δ$ as the graph $\outG_{δ}(C) := \outG_{δ}(Δ_C)$ (resp. $\inG_{δ}(C) := \inG_{δ}(Δ_C)$).

\paragraph{Nests and rails.} \label{@fingerprint} 
Let $δ = (\Gamma,\mathcal{D})$ be a $Σ$-decomposition of a graph $G$ in a surface $Σ$ and let $Δ^{*} \subseteq Σ$ be an arcwise connected set.
A \emph{nest} in $δ$ around $Δ^{*}$ of order $s$ is a sequence $\mathcal{C} = \langle C_1, \dots,C_s \rangle$ of disjoint cycles in $G$ such that each of them is grounded in $δ$ and the trace of $C_i$ bounds a closed disk $Δ_{i}$ in such a way that $Δ^{*} \subseteq Δ_{1} \subsetneq Δ_{2} \subsetneq \dots \subsetneq Δ_{s} \subseteq Σ$ and $Δ_{s}$ does not intersect the boundary of $Σ$.
We call $C_{1}$ (resp. $C_{s}$) the \emph{internal} (resp. \emph{external}) cycle of $\Ccal.$
We call the sequence $\langle Δ_{1}, \dots,  Δ_{s} \rangle$ the \emph{disk sequence} of $\mathcal{C}$.

Let $\rho$ be a rendition of a society $(G, \Omega)$ in a disk $\Delta,$ let $\Delta^* \subseteq \Delta$ be an arcwise connected set and let $\mathcal{C} = \langle C_1, \dots, C_s \rangle$ be a nest in $\rho$ around $\Delta^*$ of order $s.$
We say that a family of pairwise vertex-disjoint paths $\Pcal$ in $G$ is a \emph{radial linkage} if each path in $\mathcal{P}$ has at least one endpoint in $V(\Omega)$ and the other endpoint of $\mathcal{P}$ is drawn in $\Delta^*.$
Moreover, we say that $\mathcal{P}$ is \emph{orthogonal} to $\mathcal{C}$ if for every $P \in \mathcal{P}$ and every $i \in [s],$ $C_i \cap P$ consists of a single component.
Similarly, if $\mathcal{P}$ is a transaction in $(G, \Omega)$ then $\mathcal{P}$ is said to be \emph{orthogonal} to $\mathcal{C}$ if for every $i \in [s]$ and every $P \in \mathcal{P},$ $C_i \cap P$ consists of exactly two components.

Moreover, given a radial linkage $\Pcal$ of size $r$, let $\Qcal$ contain the minimal $A$-$B$ subpaths of the paths in $\Pcal$, where $A = V(C_{1}) \cap π_{δ}(N(δ)),$ $B = V(C_{s}) \cap π_{δ}(N(δ)).$
We call the pair $(\Ccal, \Pcal)$ a \emph{railed nest} in $δ$ around $Δ^{*}$ of order $\min\{s, r\}$.
Notice that $\cupall \Pcal$ is disjoint from $\inG_{δ}(C_{1}) - V(C_{1})$ and $\outG_{δ}(C_{s}) - V(C_{s}).$

\paragraph{Renditions of nests.}

Let $(G, Ω)$ be a society, $ρ = (Γ, \mathcal{D})$ be a rendition of $(G, Ω)$ in a disk, and let $c_{0} \in C(ρ)$ be such that no cell in $C(ρ) \setminus \{ c_{0} \}$ is a vortex.
We say that the triple $(Γ, \mathcal{D}, c_{0})$ is a \emph{cylindrical rendition} of $(G, Ω)$ around $ c_{0}.$

Let $(G, Ω)$ be a society and $\rho = (\Gamma, \mathcal{D}, c_{0})$ be a cylindrical rendition of $(G, Ω)$ around $c_{0}$.
Let $\Delta$ be the disk of $c_{0}$ in $\rho.$
Let $\mathcal{C} = \langle C_{1}, \dots, C_{s} \rangle$ be a nest in $\rho$ around $\Delta$ of order $s \geq 1.$
We say that $(ρ, G, Ω)$ is the \emph{rendition of $\mathcal{C}$ around $c_{0}$ in $ρ$}.

\subsection{Grids, walls, and shallow-vortex minors}
\label{@crystallize}

In this subsection we introduce all necessary definitions around grids, walls, and shallow-vortex minors.
This includes the formal description of Dyck grids, Dyck walls and shallow-vortex grids.
Moreover, we present a local structure result, proved in~\cite{thilikos2023excluding}, that will form the basis of our results.

\paragraph{Walls.}
An \emph{$(n\times m)$-grid} is the graph $G_{n,m}$ with vertex set $[n]\times[m]$ and edge set \[\{(i,j)(i,j+1) \mid i\in[n],j\in[m-1]\}\cup\{(i,j)(i+1,j) \mid i\in[n-1],j\in[m]\}.\]
We call the path where vertices appear as $(i,1),(i,2),\dots, (i,m)$ the \emph{$i$th row} and the path where vertices appear as $(1,j),(2,j),\dots, (n,j)$ the \emph{$j$th column} of the grid.
An \emph{elementary $k$-wall}~$W_k$ for~${k \geq 3},$ is obtained from the ${(k\times 2k)}$-grid $G_{k,2k}$ by deleting every odd edge in every odd column and every even edge in every even column, and then deleting all degree-one vertices.
The \emph{rows} of~$W_k$ are the subgraphs of~$W_k$ induced by the rows of~$G_{k,2k},$ while the \emph{$j$th column} of~$W_k$ is the subgraph induced by the vertices of columns~${2j-1}$ and~${2j}$ of~$G_{k,2k}.$
We define the perimeter of $W_k$ to be the subgraph induced by~${\{ (i,j) \in V(W_k) \mid j \in \{1,2,2k,2k-1\} \text{ and } i \in [k] \textnormal{, or } i \in \{1,k\} \text{ and } j \in [2k] \}}.$
A \emph{$k$-wall}~$W$ is a graph isomorphic to a subdivision of~$W_k.$
The vertices of degree three in $W$ are called the \emph{branch vertices}.
In other words, $W$ is obtained from a graph $W'$ isomorphic to $W_k$ by subdividing each edge of~$W'$ an arbitrary (possibly zero) number of times.
The \emph{perimeter} of~$W,$ denoted by~$\Perimeter(W),$ is the subgraph isomorphic to the subgraph of~$W'$ induced by the vertices of the perimeter of~$W_k$ together with the subdivision vertices of the edges of the perimeter of~$W_k.$
We define rows and columns of $k$-walls analogously to their definition for elementary walls.
A \emph{wall} is a $k$-wall for some~$k.$

An $h$-wall $W'$ is a \emph{subwall} of some $k$-wall $W$ where $h\leq k$ if every row (column) of $W'$ is contained in a row (column) of $W.$

Notice that, as $k \geq 3,$ an elementary $k$-wall is a planar graph that has a unique (up to topological isomorphism) embedding in the plane $\mathbb{R}^{2}$ such that all its finite faces are incident to exactly six edges. The perimeter of an elementary $r$-wall is the cycle bounding its infinite face, while the cycles bounding its finite faces are called \emph{bricks}. A cycle of a wall $W,$ obtained from the elementary wall $W',$ is a \emph{brick} (resp. the perimeter) of $W$ if its branch vertices are the vertices of a brick (resp. the perimeter) of $W'.$ A brick of $W$ is \emph{internal} if it is disjoint from $\Perimeter(W).$

\paragraph{Dyck grids.} The next step is to introduce several classes of grid-like parametric graphs that capture the behaviour of surfaces.

Given $k\in\Nbbb_{≥1}$ and $h,c\in\Nbbb\times[0,2],$ consider the \emph{$(k,4k(1+h+c))$-cylindrical grid}, $\mathscr{A}_{k,4k(1+2h+c)},$ i.e.,
 the Cartesian product of a path on $k$ vertices and a cycle on $4k(1+h+c)$ vertices. We can see $\mathscr{A}_{k,4k(1+h+c)}$ as the union of
 $k$ cycles $C_i,i\in[k]$ and $4k(1+h+c)$ paths $P_j,j\in[4k(1+h+c)]$
where, for $i\in[k],$ the vertices in $C_{i}$ appear, in order, as $v^i_1,\dots,v^i_{4k(1+h+c)}$ and, for $j\in[4k(1+h+c)],$ the vertices of $P_{j}$ 
appear in order as  $v^1_j,\dots, v^{4k(1+h+c)}_j.$
$\mathscr{A}_{k,4k(1+h+c)}$ will be a common spanning subgraph of most parametric graphs that we define. In all of them, we refer to the paths $P_j$, where $j \in [4k(1+h+c)]$, as \emph{tracks}.

Let $p\in [1+h+c].$ 
By \emph{adding a handle at position $p$}, we mean adding the edges
\begin{align*}
\{ v^1_{4k(p-1)+h}v^1_{4k(p-1)+3k-h+1} \mid h\in[k] \}\cup\{ v^1_{4k(p-1)+k+h}v^1_{4k(p-1)+4k-h+1} \mid h\in[k]\},
\end{align*}
and by \emph{adding a cross-cap at position $p$}, we mean adding the edges
\begin{align*}
\{ v^1_{4k(p-1)+h}v^1_{4k(p-1)+2k+h} \mid h\in[2k] \}.
\end{align*}
  
  The   
    \emph{$(h,c)$-Dyck grid} of order $k\in \mathbb{N}$  is the graph  $\mathscr{D}_{k}^{(h,c)}$
    obtained from $\mathscr{A}_{k,4k(1+h+c)}$
after  adding a handle at every position $i\in [2,h+1]$ 
and adding a cross-cap at every position in $[h+2,h+1+c].$
We define the parametric graph of Dyck grids as $\mathscr{D}^{(h,c)}=\langle\mathscr{D}^{(h,c)}_{k}\rangle_{k\in\Nbbb_{≥3}}.$

Given a surface $Σ = Σ^{(h,c)},$ we define the parametric graph $\mathscr{D}^{Σ} = \langle \mathscr{D}^{Σ}_{k} \rangle_{k \in \mathbb{N}}$ as the Dyck grid $\mathscr{D}^{(h,c)}.$ Notice that $\mathscr{D}^{(h,c)}_{k}$ has a drawing in $Σ^{(h,c)}$ where all faces are squares expect from the face bounded by the 
cycle $C_{k},$ of length $4k(1+h+c)$ that we call \emph{simple} cycle of $\mathscr{D}^{(h,c)}_{k}$
and a face bounding a cycle of length $4(2h+c)+4k$ that we call  \emph{exceptional} cycle of $\mathscr{D}^{(h,c)}_{k}.$

For a drawing of $\mathscr{D}_{8}^{(1,2)}$ see \autoref{@freemasonry}. 
$\mathscr{D}_{8}^{(1,2)}$ can be seen as a cyclic concatenation of one copy 
of the annulus grid $\mathscr{A}_{k},$ $h$ copies of the handle grid 
$\mathscr{H}_{k}$ and $c$ copies of the cross-cap grid $\mathscr{C}_{k}.$

The following result indicates that there is a one to one correspondence between graphs embeddable in surfaces and minors of Dyck-grids.

\begin{proposition}[\!\! \cite{gavoille2023minoruniversal}]
\label{@headstrong}
There exists a function $f_{\ref{@headstrong}}:\Nbbb^2\to\Nbbb$
such that a graph $H$ on $h$ vertices is embeddable in a surface $Σ$ if and only if it is a minor of $\mathscr{D}^{Σ}_{f_{\ref{@headstrong}}(\eg(Σ),h)}.$
Moreover $f_{\ref{@headstrong}}(\eg(Σ),h)=\Ocal((\eg(Σ))^4h^2)$
\end{proposition}

We moreover require the following observation from~\cite{thilikos2023excluding}, which certifies that a Dyck-grid half-integrally packs itself.

 \begin{proposition}[\!\! \cite{thilikos2023excluding}]
 \label{@monopolism}
For every surface $Σ$ and every $t,k\in\Nbbb_{≥1},$ 
${\mathscr{D}}_{tk}^{Σ}$ contains a half-integral
${\mathscr{D}}_{t}^{Σ}$-packing of size $k.$ 
 \end{proposition}

\paragraph{Dyck walls.} Next, we define \textsl{Dyck walls} from Dyck grids in a way similar to the one we used to define walls from grids. Let $h\in\Nbbb, c \in [0, 2],$ and $t\in\Nbbb_{≥1}$ and let $Σ=Σ^{(c,h)}.$ The \emph{elementary} $(Σ;t)$-\emph{Dyck wall} is obtained from $\mathscr{D}^{h, c}_{2t}$ by deleting the cycles $C_{t+1}, \dots, C_{2t}$ and by deleting the edge $v^{i}_{j}v^{i+1}_{j}$ for every odd $i \in [t-1]$ and every odd $j \in [8t]$ and for every even $i \in [t-1]$ and even $j \in [8t].$ Moreover, for each handle and cross-cap that was added to create $\mathscr{D}^{h, c}_{2t}$, delete every edge starting on $v_{l}^{1}$ when $l$ is even. That is, we delete every second edge of a handle or a cross-cap. See~\autoref{@reexpressing} for an illustration of the elementary $(Σ^{1,1};6)$-Dyck wall.

\begin{figure}[ht]
  \begin{center}
  \scalebox{0.8}{\includegraphics{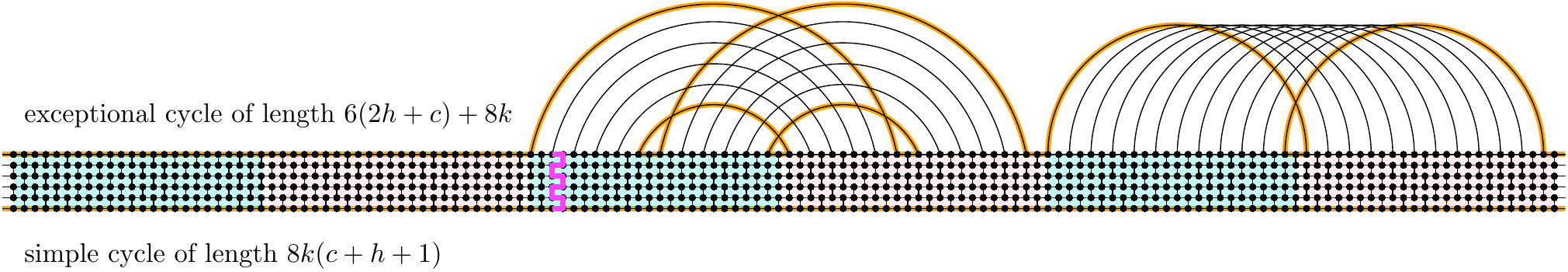}}
  \end{center}
    \caption{The elementary $(h,c;k)$-Dyck wall, where $h=1,$ $c=1,$ and $k=6.$ We draw in \magenta{magenta} one of the tracks of the $(h,c;k)$-Dyck wall.}
  \label{@reexpressing}
\end{figure}

A $(Σ;t)$-\emph{Dyck wall} $D$ is a subdivision of the elementary $(Σ;t)$-Dyck wall $D'.$ 
The vertices of degree three in $D$ are called the \emph{branch} vertices. 
Consider the $Σ$-embedding of $D$.
We call a cycle of $D$ a \emph{face-cycle} if it is the boundary of a face of the $Σ$-embedding of $D$.
We classify the face-cycles of $D$ as follows. If a face-cycle has length six, then it is a \emph{brick} of $D.$
There are exactly two face-cycles that are not bricks,
the one bounded by $C_{t}$ is called the \emph{simple face} and has length $8k(c+h+1).$
We denote by $\simple(D)$ the cycle $C_{t}.$
The other is called the \emph{exceptional face} and has length $6(2h+c)+8k$ (see~\cref{@reexpressing}). 
We denote by $\exceptional(D)$ the cycle of $D$ that bounds this face.
A brick of $D$ is \emph{internal} if it is disjoint from $\simple(D)$ and $\exceptional(D).$ The \emph{cycles} of $D$ are the (subdivided) cycles 
of the  cycles of the underlying Dijk grid, each containing $8k(c+h+1)$
branch vertices. Notice also that these cycles are crossed by $8k(c+h+1)$
pairwise disjoint paths each on $2k$ vertices (one of them is depicted in  \magenta{magenta} in \cref{@reexpressing}) that will we call \emph{tracks} of $D.$

If $D$ is a $(Σ;d)$-Dyck wall for some surface $Σ$ and some integer $d,$ we say that a $(Σ,d')$-Dyck wall $D'\subseteq D$ is a \emph{Dyck subwall} of $D$ if every cycle of $D'$ is a cycle of $D$ and every track of $D'$ is contained in a single track of $D.$

The following is a statement of the Grid Theorem.
While we will not explicitly use the Grid Theorem, we will, later on, make use of the existence of a function that forces the existence of a large wall in a graph with large enough treewidth.

\begin{proposition}[Grid Theorem \cite{robertson1986graph,chuzhoy2021towards}]\label{@reestablish}
	There exists a universal constant $c\geq 1$ such that for every $k\in\mathbb{N}$ and every graph $G,$ if $\mathsf{tw}(G)\geq ck^{10},$ then $G$ contains the $(k\times k)$-grid as a minor. 
	\end{proposition}
Notice that any graph that contains a $(2k\times 2k)$-grid as a minor contains a $k$-wall as a subgraph.

\paragraph{Shallow-vortex minors.}

The next parametric graph that we defined is the \emph{shallow-vortex grid}  $\mathscr{V}=\langle\mathscr{V}_{k}\rangle_{k\in \Bbb{N}_{≥1}}$
where  $\mathscr{V}_{k}$ is the  \emph{shallow-vortex grid of order} $k$ and is  obtained by $\mathscr{D}_{k}^{(0,0)}$
(that is the $(k,4k)$-cylindrical grid)
after adding the edges 
$$\{v^{1}_{4(i-1)+1}v^{1}_{4(i-1)+3},v^{1}_{4(i-1)+2}v^{1}_{4(i-1)+4}\mid i\in [k]\}.$$
See \cref{@bankruptcy} for a partial drawing of the shallow-vortex grid $\mathscr{V}_{8}.$

\begin{figure}[ht]
  \begin{center}
  \scalebox{1.26}{\includegraphics{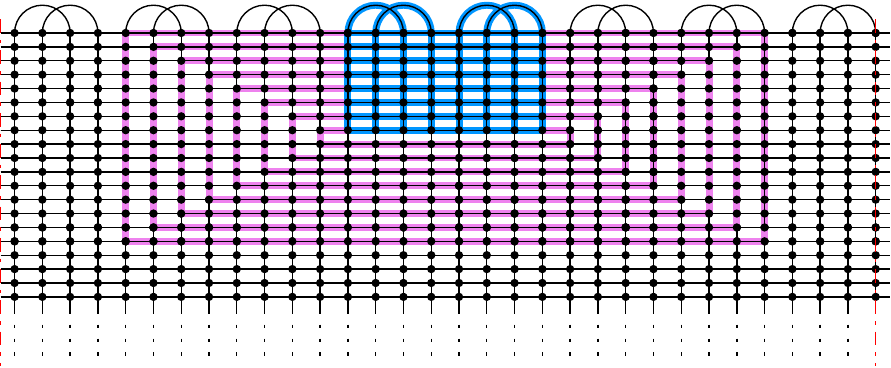}}
  \end{center}
    \caption{The (partial) drawing of the shallow-vortex grid $\mathscr{V}_{8}$ 
   where  20 or its 32 cycles are drawn along with 
    the drawing of a subdivision of $\mathscr{V}_{2}$ in it. Notice that $\mathscr{V}_{2}$  is a minor of $\mathscr{V}_{6}.$  In general, $\mathscr{V}_{k}$ can be seen as a minor of $\mathscr{V}_{3k},$ using (parts of)  $6\times 4$ tracks, because we need paths from $k$ more tracks on the left along with paths from 
   $k$ tracks on the right,  and paths from $k$
    more cycles below (all drawn in violet) in order to simulate the $k$ missing edges completing the $k$ cycles of $\mathscr{V}_{2}.$}
  \label{@bankruptcy}
\end{figure}

We say that a graph is a \emph{shallow-vortex minor} if it is the minor of some shallow-vortex grid $\mathscr{V}_{k},$ for some $k\in\Nbbb_{≥1}.$
We need the following easy observation.
See \cref{@bankruptcy}  for an illustration of the construction.

\begin{observation}\label{@childishness}
For every $k,t\in\Nbbb_{≥1},$  $\mathscr{V}_{3kt}$ contains a $\mathscr{V}_{k}$-packing of size $t.$ 
\end{observation}

\paragraph{Ring blowups.} A handy way to prove that a graph is a shallow-vortex minor 
is to see it as a minor of a \textsl{ring blowup} graph.
Consider a graph $G$ and a cycle $C\subseteq G.$
We say that $C$ is a \emph{facial} cycle of $G$ if the 
graph obtained from $G$ 
after making a new vertex $v_{\rm new}$ adjacent with 
all the vertices in $C$ is planar.
A \emph{facial pair} $(G,C)$ be a pair where $G$ is a graph and $C$ is a facial cycle of $G.$
The \emph{ring blowup} of a facial pair $(G,C)$ is the graph obtained from $G-V(C)$ if we add the vertices in $\{v^1,v^2\mid v\in V(C)\}$ and the edges in
$$\{v^1v^2\mid v\in V(C)\} \cup$$
$$\{v^1u^1,v^2u^1,v^1u^2,v^2u^2\mid vu\in E(C)\}\cup$$
$$\{xv^1,xv^2\mid x\in V(G)\setminus V(C)\text{~and~}xv\in E(C)\}.$$
A \textsl{ring blowup graph} is any ring blowup of some facial pair.
As observed in \cite{curticapean2022parameterizing,ThilikosW22Killingconf, thilikos2022killing}, every ring blowup graph is a shallow-vortex minor. For instance, 
$K_{7}$ is a ring blowup of $(K_{4},K_{3})$ where $K_{3}$ is any triangle of $K_{4}.$ It is also easy to see that every  ring blowup is  a Kuratowski-connected graph. Therefore, all the results of our paper 
are applicable for every $\Zcal$ containing a ring blowup graph.

\subsection{Tangles and $\Sigma$-schemata}

In this subsection we introduce tangles, balanced separators, and well-linked sets, as well as the notion of a $\Sigma$-schema, that will be a central concept in the proofs that follow.

\paragraph{Tangles.} Let $G$ be a graph and $k$ be a positive integer. 
We denote by $\mathcal{S}_{k}$ the collection of all separations $(A, B)$ of order less than $k$ in $G$.
An \emph{orientation} of $\mathcal{S}_{k}$ is a set $\mathcal{O}$ such that for all $(A, B) \in \mathcal{S}_{k}$ exactly one of $(A, B)$ and $(B, A)$ belongs to $\mathcal{O}.$
A \emph{tangle} of order $k$ in $G$ is an orientation $\mathcal{T}$ of $\mathcal{S}_{k}$ such that for all $(A_{1}, B_{1}), (A_{2}, B_{2}), (A_{3}, B_{3}) \in \mathcal{T}$, 
we have $A_{1} \cup A_{2} \cup A_{3} \neq \emptyset.$

Let $D$ be a $(Σ; k)$-Dyck wall in a graph $G$, for some surface $Σ$, and $(A, B)$ be a separation of order $< k$. 
Then exactly one of $A$ or $B$ contains a cycle of $D.$
Let $\mathcal{T}_{D}$ be the orientation of $\mathcal{S}_{k}$ where $(A, B) \in \mathcal{T}_{D}$ if and only if $B\setminus A$ contains a cycle and a track of $D.$
Then it is easy to observe that $\mathcal{T}_{D}$ is a tangle, which we call the \emph{tangle induced by $D$}.
In the same way we may also define tangles \emph{induced by} walls.

Let $\mathcal{T}'\subsetneq\mathcal{T}$ be a tangle which is {properly contained} in the tangle $\mathcal{T}.$
We say that $\mathcal{T}'$ is a \emph{truncation} of $\mathcal{T}.$

\paragraph{Well-linked sets and tangles.} Let $α \in [2/3, 1),$ $G$ be a graph and $S \subseteq V(G)$ be a vertex set. 
A set $X \subseteq V(G)$ is said to be an \emph{$α$-balanced separator} for $S$ if for every component $C$ of $G - X$ it holds that $|V(C) \cap S| \leq α|S|.$ 
Let $q \in \Nbbb.$ We say that $S$ is a \emph{$(q, α)$-well-linked set} of $G$ if there is no $α$-balanced separator of size at most $q$ for $S$ in $G.$ 
Given a $(q, α)$ well-linked set of $G$ we define $$\mathcal{T}_{S} := \{ (A, B) \in \mathcal{S}^{q+1} \mid |S \cap B| > α|S| \}.$$ 
It is not hard to see that $\mathcal{T}_{S}$ is a tangle of order $q$ in $G$.

Given an $(s,\alpha)$-well-linked set $S$ for large enough $s$ we can find a large wall $W$ in FPT-time such that $\mathcal{T}_W$ is a truncation of $\mathcal{T}_S.$
This will give us the right to equip our local structure theorems with an additional property that will allow us to localise it with respect to some tangle which will be encoded using the set $S$ instead of the collection of separations which can be very large.

\begin{theorem}[\cite{thilikos2023excluding}]\label{thm_algogrid}
Let $k\geq 3$ be an integer, $\alpha\in [2/3,1).$
There exist universal constants $c_1,c_2\in\mathbb{N}\setminus\{ 0\},$ and an algorithm that, given a graph $G$ and an $(36c_1k^{20}+3,\alpha)$-well-linked set $S\subseteq V(G)$ computes in time $2^{\mathcal{O}(k^{c_2})}|V(G)|^2|E(G)|\log(|V(G)|)$ a $k$-wall $W\subseteq G$ such that $\mathcal{T}_W$ is a truncation of $\mathcal{T}_S.$
\end{theorem}

\paragraph{$Σ$-schemata and their carvings.}
Let $G$ be a graph, $Σ$ be a surface, and $D\subseteq G$ be a $(Σ;d)$-Dyck wall.
We say that the triple $(Α,δ,D)$ is a \emph{$Σ$-schema} of $G$ if 
\begin{enumerate}
\item $A\subseteq V(G),$
\item $δ=(Γ,\Dcal)$ is a vortex-free $Σ$-decomposition of $G' \coloneqq G-A,$ and
\item $D$ is grounded in $δ.$
\end{enumerate}

Given some separation $(X_{1},X_{2})\in \Tcal_{D}$,
we define the \emph{carving} of $(Α,δ,D)$ by $(X_{1},X_{2})$
as the $Σ$-schema $(Α',δ',D')$ defined as follows. We first set  
$A'=A\cup(X_{1}\cap X_{2}).$ Then, we define the $Σ$-decomposition $δ'$ of $G-A'$ by taking the $Σ$-decomposition  $(Γ^*, \Dcal'):=δ-X_{1}$  of $G-(A\cup X_{1})$ and then  completing it to a $Σ$-decomposition  $δ'=(Γ', \Dcal')$ of $G-A'$
by considering any drawing $Γ^*$
of $G[X_1\setminus X_{2}]$ inside  any of the cells of $\Dcal'$ and taking $Γ':=Γ'\cup Γ^*.$ Finally, let $D'$ be a maximum order $(Σ;d)$-Dyck wall that is a subgraph of $D-X_{1}.$

\subsection{Excluding Dyck grids} 

In this subsection, we present a local structure theorem, proven in~\cite{thilikos2023excluding}, that will be the departing point for the proof of our local structure theorem (see \cref{@antiauthoritarian}), whose proof is the subject of \cref{@ilemigoddesscs}.

\medskip
Before we progress any further we formally discuss ``minimal'' hosts.

\paragraph{Expansions.}
An \emph{expansion} is a pair $(M,T)$ such that $M$ is a graph, $T\subseteq V(M)$ and all vertices of $V(M)\setminus T$ have degree two.
Given a graph $H,$ an expansion $(M,T)$ is an $H$-\emph{expansion} if the graph obtained if we dissolve all vertices not in $T$ contains $H$ as a minor. 
An $H$-\emph{expansion} $(M,T)$ is \emph{minimal} if there is no $H$-expansion $(M',T)$ where $M'$ is a subgraph of $M.$
We say that a graph $M$ is an \emph{$H$-inflated copy} if there exists a choice of vertices $T\subseteq V(M)$ such that $(M,T)$ is a minimal $H$-expansion.

Note that if $(M,T)$ is a minimal $H$-expansion, then  $|V(T)|≤|V(H)|^{2}.$
Given a graph $G,$ an ($H$-)expansion $(M,T)$ is an ($H$-)expansion \emph{in} $G$ if $M$ is a subgraph of $G.$
In this case we also say that $M$ is an ($H$-)inflated copy \emph{in} $G.$

\begin{proposition}[\!\! \cite{thilikos2023excluding}]
\label{@duplicating}
Let $\mathbb{S}$ be a closed set of surfaces, each of Euler genus at most $\darkmagenta{γ}.$
There exist functions 
$f_{\ref{@duplicating}}^{1} : \mathbb{N}^3 \to \mathbb{N},$ 
$f_{\ref{@duplicating}}^{2} : \mathbb{N}^{4} \to \mathbb{N},$ 
$f_{\ref{@duplicating}}^{3} : \mathbb{N}^{3} \to \mathbb{N},$ 
$f_{\ref{@duplicating}}^{4} : \mathbb{N}^{2} \to \mathbb{N},$ 
$f_{\ref{@duplicating}}^{5} : \mathbb{N}^{3} \to \mathbb{N}$ 
and an algorithm that,  given a graph $G,$  {four} integers $t,q,d,r\in \mathbb{N},$ where $q≥ f_{\ref{@duplicating}}^{1}(\darkmagenta{γ},t,d),$ and an $f_{\ref{@duplicating}}^{2}(\darkmagenta{γ},t,q,r)$-wall $W$ of $G,$ outputs one of the following:
\begin{itemize}
   
\item Either an inflated copy  $M'$ of the Dyck grid $\mathscr{D}_{t}^{Σ'},$ for some $Σ'\in\sobs(\mathbb{S})$ or
\item a triple $(A,δ,D)$ such that
\begin{itemize}
   
\item $A$ is a vertex set of $G$ where $|A|≤f_{\ref{@duplicating}}^{3}(\darkmagenta{γ},t,q),$ 
\item $δ$ is a $Σ$-decomposition   of $G'\coloneqq G-A,$ for some $Σ\in\mathbb{S},$  such that 
\begin{itemize}
\item[$\triangleright$] $δ$  has at most $f_{\ref{@duplicating}}^{4}(\darkmagenta{γ},t)$ vortex cells each of depth at most $f_{\ref{@duplicating}}^{5}(\darkmagenta{γ},t,q),$
\item[$\triangleright$] for every  vortex cell $c$ of $δ,$ there exists a nest $(C_{1}^c,\ldots,C_{q}^c)$ or order $q$ around the disk of $c,$  with disk sequence $\langle Δ_{C^c_{1}}, \dots,  Δ_{C^c_{q}}\rangle,$
such that,  for each pair of distinct vortex cells $c,c'$ of $d,$ it holds  that $Δ_{C^c_{q}}\cap Δ_{C^{c'}_{q}}=\emptyset,$ 
\end{itemize}
\item  $D$ is a $(Σ;d)$-Dyck wall that is grounded in $δ,$
\item there exists an $r$-subwall $W'\subseteq W$ which is vertex-disjoint from $D,$ grounded in $δ,$ drawn entirely within the disk defined by the trace of the simple face of $D,$
\item no vertex of $D$ or $W'$ belongs to $\bigcup_{c\in \vcells(δ)} Δ_{c},$ and
\item $\mathcal{T}_D$ and $\mathcal{T}_{W'}$ are truncations of $\mathcal{T}_W.$
\end{itemize}
\end{itemize}
Moreover, it holds that 
\begin{eqnarray*}
f_{\ref{@duplicating}}^{1}(\darkmagenta{γ},t,d) & = &  2^{Θ(γ)}(t+d), \\
f_{\ref{@duplicating}}^{2}(\darkmagenta{γ},t,q) & = & 2^{\poly(2^{Θ(γ)}t)}\cdot q+ \mathsf{poly}(2^{Θ(γ)}t)\cdot r, \\
f_{\ref{@duplicating}}^{3}(\darkmagenta{γ},t,q) & = & 2^{\poly(2^{Θ(γ)}t)}\cdot q, \\
f_{\ref{@duplicating}}^{4}(\darkmagenta{γ},t) & = & 2^{Θ(γ)}t^2, \\
f_{\ref{@duplicating}}^{5}(\darkmagenta{γ},t,q)&  = &  2^{\poly(2^{Θ(γ)}t)}\cdot q  \end{eqnarray*}
 and $\mathbf{A}_\darkmagenta{γ}$ runs in time $q^2\cdot  2^{\mathsf{poly}(t\cdot 2^{\Ocal(γ)})}|V(G)|^{2}.$
\end{proposition}

\section{Preliminary results}
\label{@deprivation}

In this section we present a series of definitions and results regarding the behaviour of hosts of connected non-planar Kuratowski-connected graphs in a $\Sigma$-decomposition of a given graph (\cref{@proclaimed}).
A key result in this section is the ``Redrawing Lemma'' (\cref{@enthroning}), that shows how we can exploit the property of being Kuratowski-connected in order to redraw ``minimal'' hosts of such graphs ``invading'' through cells deep within the interior of a railed nest of sufficiently large order in our decomposition, ``locally'' within the are bounded by the outermost cycle of the railed nest.
This result is one of the key components that allows for obtaining large packings locally in different steps throughout our proof or otherwise allows us to further refine our initial $\Sigma$-decomposition until we reach our final local structure theorem.

\subsection{Hosts of non-planar Kuratowski-connected graphs}
\label{@proclaimed}

In this subsection we investigate the different ways hosts (and therefore also inflated copies) of a connected non-planar Kuratowski-connected graph $H$ may interact with a vortex-free $\Sigma$-decomposition of a given graph.
We introduce the notion of a $c$-\textsl{invading} $H$-host for a given cell $c$ of our $\Sigma$-decomposition.
We moreover introduce the notion of $H$-\textsl{red} cells in relation to $c$-invading $H$-hosts and we argue that with respect to a fixed $\Sigma$-decomposition of our graph it is sufficient to only cover such $H$-hosts in our local structure theorem in order to later globally cover all $H$-hosts in $G.$

\paragraph{Minimal separations and Kuratowski-connectivity.}
Consider a separation $(A, B)$ of a graph $G.$
We call $(A, B)$ a \emph{minimal separation} if there exists a component $C$ of $G[A \setminus B]$ and a component $D$ of $G[B \setminus A],$ such that every vertex in $A \cap B$ has a neighbour in $V(C)$ and a neighbour in $V(D).$

Given a graph $G$ and a set $X \subseteq V(G)$ of vertices of $G,$ we say that $G$ is $X$-\emph{disk embeddable} if $G$ can be embedded in a disk $\Delta$ with the vertices of $X$ drawn in the boundary of $\Delta.$

We repeat the definition of Kuratowski-connectivity, introduced in~\cref{@hollingdale}, following the terminology introduced in this section.
We say that a graph $G$ is \emph{Kuratowski-connected} if for every minimal separation $(A,  B)$ of $G$ of order at most three, one of $G[A]$ or $G[B]$ is $A \cap B$-disk embeddable.

Let $H$ be a non-planar Kuratowski-connected graph.
Let $(A, B)$ be a minimal separation of $H$ of order at most three.
Observe that since $H$ is non-planar and Kuratowski-connected exactly one of $H[A]$ and $H[B]$ is not $A \cap B$-disk embeddable.
Indeed, otherwise we would be able to combine the two corresponding disk embeddings of $H[A]$ and $H[B]$ into a plane embedding of $H.$

\paragraph{Boundaried graphs.}

A \emph{boundaried graph} is a tuple $\mathbf{G} = \langle G, v_{1}, \ldots, v_{r} \rangle,$ $r \in \Nbbb,$ where $G$ is a graph and $\{ v_{1}, \ldots, v_{r} \} \subseteq V(G).$
We say that two boundaried graphs $\mathbf{G} = \langle G, v_{1}, \ldots,v_{r} \rangle$ and $\mathbf{G}' = \langle G', v_{1}',\ldots,v_{r'}'\rangle$ are \emph{compatible} if $r = r'.$
Given two compatible boundaried graphs $\mathbf{G} = \langle G, v_{1}, \ldots, v_{r} \rangle$ and $\mathbf{G}' = \langle G', v_{1}', \ldots, v_{r}' \rangle$ we define $\mathbf{G} \oplus \mathbf{G}'$ as the graph obtained if we take the disjoint union of $G$ and $G'$ and then, for every $i \in [r],$ we identify vertices $v_{i}$ and $v_{i}'$ (by identifying $v_{i}$ and $v_{i}'$ the newly created vertex is adjacent to all neighbours of $v_{i}$ in $G$ and of $v_{i}'$ in $G'$).

\medskip
To study the behaviour of an $H$-host for a connected non-planar Kuratowski-connected graph $H$ within a $\Sigma$-decomposition $\delta$ of a graph $G,$ we need to look at how minor models of $H$ within an $H$-host traverse cells of $\delta.$
Any cell of $\delta$ corresponds to a separation of $G$ of order at most three and may naturally induces a separation of $H$ of order at most three, given a minor model of $H$ in an $H$-host in $G$ that traverses a given cell.
However this separation is not necessarily a minimal separation of $H$ and hence we cannot directly evoke the properties of Kuratowski-connected graphs.
The following lemma deals with non-minimal separations of order three in connected non-planar Kuratowski-connected graphs.

\medskip
Let $H$ be a connected non-planar Kuratowski-connected graph.
We say that a separation $(A, B)$ of $H$ is \emph{trivial} if either $A \setminus B$ or $B \setminus A$ is empty.
Given a non-trivial separation $(A, B)$ of $H$ of order at most three we define $\Scal_{(A, B)}$ to be the set of all minimal separations $(X, Y)$ of $H$ of order at most three such that $X \cap Y \subseteq A \cap B.$
Moreover, we call a connected component $C$ of $H - (A \cap B)$ the \emph{core component} of $(A, B)$ if the intersection of the non-embeddable sides for all $(X, Y) \in \Scal_{(A, B)}$ minus $A \cap B$ is exactly $V(C).$
As we demonstrate in the following lemma the core component of a non-trivial separation $(A, B)$ of $H$ order at most three always exists and moreover is unique.

\begin{figure}[ht]
  \begin{center}
  \scalebox{0.1920}{\includegraphics{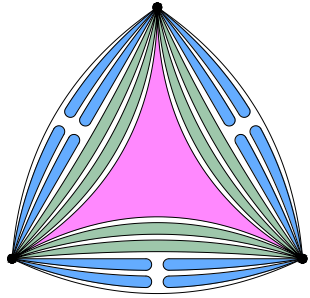}}
  \end{center}
    \caption{An illustration for the proof of \cref{@inteuigence}. The $3$ vertices are the vertices in $A \cap B$.}
  \label{@backwoodsmen}
\end{figure}

\begin{lemma}\label{@inteuigence} Let $H$ be a connected non-planar Kuratowski-connected graph and $(A, B)$ be a non-trivial separation of $H$ of order at most three.
Then there exists a connected component $C$ of $H - (A \cap B)$ such that $C$ is the unique core component of $(A, B)$ and $H \setminus C$ is $(N_{H}(C) \cap (A \cap B))$-disk embeddable.
\end{lemma}
\begin{proof}
Let $\Ccal_{1}$ (resp. $\Ccal_{2}$) be the set of connected components in $H[A \setminus B]$ (resp. $H[B \setminus A]$).
By assumption both $\Ccal_{1}$ and $\Ccal_{2}$ are non-empty.
Clearly the vertex set of any connected component in $\Ccal_{1}$ and $\Ccal_{2}$ is a subset of either $X$ or $Y$ for any $(X, Y) \in \Scal_{(A, B)}.$

First notice that as we have already observed, by definition of Kuratowski-connectivity and since $H$ is non-planar, for every $(X, Y) \in \Scal_{(A, B)}$ exactly one of $G[X]$ or $G[Y]$ is $X \cap Y$-disk embeddable.
Let $(X, Y)$ and $(Z, W)$ be two minimal separations in $\Scal_{(A, B)}.$
We say that $(X, Y)$ and $(Z, W)$ \emph{agree} if their non $X \cap Y$-disk embeddable sides minus $A \cap B$ have a non-empty intersection and that they \emph{disagree} otherwise.
We first claim that all separations in $\Scal_{(A, B)}$ agree.

\begin{claim} All separations in $\Scal_{(A, B)}$ agree.
\end{claim}
\begin{cproof} Assume that there exists $(X, Y)$ and $(Z, W)$ in $\Scal_{(A, B)}$ that disagree.
Assume without loss of generality that $H[X]$ and $H[Z]$ are the non-disk embeddable sides of $(X, Y)$ and $(Z, W)$ respectively.
Since $(X, Y)$ and $(Z, W)$ disagree we have that $X \cap Z \subseteq A \cap B.$
This implies that $Z \setminus (A \cap B) \subseteq Y \setminus (A \cap B)$ which in turn, since $(X, Y)$ is a minimal separation of $H,$ implies that $Z \cap W \subseteq X \cap Y$ and therefore that $Z \subseteq Y.$
Now, since $H[Y]$ is $X \cap Y$-disk embeddable, it is implied that $H[Z]$ is $Z \cap W$-disk embeddable, since $Z \cap W \subseteq X \cap Y$ and $Z \subseteq Y,$ which contradicts our assumption.
\end{cproof}

Now, let $\sigma \colon \Scal \to V(H)$ be a function mapping every minimal separation in $\Scal_{(A, B)}$ to its non-disk embeddable side.
Let $\Phi \coloneqq (\bigcap_{S \in \Scal} \sigma(S)) \setminus (A \cap B).$
We next claim that $\Phi$ is a non-empty subset of vertices of $H$ which induces the unique core component of $(A, B).$

\begin{claim} $\Phi$ is non-empty.
\end{claim}
\begin{cproof} We call a component $C \in \Ccal_{1} \cup \Ccal_{2}$ a \emph{trivial} component of $(X, Y) \in \Scal_{(A, B)}$ if every vertex in $X \cap Y$ has a neighbour in $C$ while no vertex of $(A \cap B) \setminus (X \cap Y)$ does.
Observe that for every trivial component $C$ of $(X, Y) \in \Scal_{(A, B)}$ the pair $(C, (X \cup Y) \setminus C)$ is also in $\Scal_{(A, B)}.$
Then, since all separations in $\Scal_{(A, B)}$ agree, if $H[C]$ is the non-disk embeddable side of the separation $(C, (X \cup Y) \setminus C),$ it must be that $\Phi = C.$
Therefore, we assume that every trivial component of a separation in $\Scal_{(A, B)}$ belongs to its disk embeddable side.

Let $(X, Y)$ and $(Z, W)$ be two minimal separations in $\Scal_{(A, B)}.$
We say that $(X, Y)$ and $(Z, W)$ are \emph{parallel} if either $X \subseteq Z$ and $W \subseteq Y$ or $Z \subseteq X$ and $Y \subseteq W.$
We define a subset $\Scal^{\mathsf{min}}_{(A, B)}$ of $\Scal$ that contains all separations $S$ in $\Scal_{(A, B)}$ that are inclusion-minimal with respect to $\sigma(S).$
First note that $\Scal^{\mathsf{min}}_{(A, B)} \neq \emptyset.$
Observe that, for any two parallel separations $(X, Y)$ and $(Z, W)$ in $\Scal_{(A, B)}$, since $(X, Y)$ and $(Z, W)$ agree, we have that either $\sigma((Z, W)) \subseteq \sigma((X, Y))$ or $\sigma((X, Y)) \subseteq \sigma((Z, W)).$
Furthermore, if $X \cap Y \subseteq Z \cap W$ we have that $\sigma((Z, W)) \subseteq \sigma((X, Y)).$
Moreover notice that, since all separations in $\Scal_{(A, B)}$ agree, if $\Scal_{(A, B)}$ contains at most two minimal separations, we can immediately conclude that $\Phi$ is non-empty.
It follows that there is one remaining case to examine where $\Scal^{\mathsf{min}}_{(A, B)} = \{ (I, J), (K, L), (M, N) \},$ none of the separations in $\Scal^{\mathsf{min}}_{(A, B)}$ are parallel, $I \cap J,$ $K \cap L,$ and $M \cap N$ are pairwise incomparable, and $\sigma((I, J)) \setminus (I \cap J)$ (resp. $\sigma((K, L)) \setminus (K \cap L)$) (resp. $\sigma((M, N)) \setminus (M \cap N)$) consists of the vertices of the unique non-trivial component of $(I, J)$ (resp. $(K , L)$) (resp. $(M, N)$).
Moreover, observe that every trivial component of any of the three separations in $\Scal^{\mathsf{min}}_{(A, B)}$ is either a trivial component or a subset of of the non-trivial component of one of the other two separations.
To conclude, assume towards contradiction that $\sigma((I, J)) \cap \sigma((K, L)) \cap \sigma((M, N)) = \emptyset.$
It follows that every component in $\Ccal_{1} \cup \Ccal_{2}$ it a trivial component of one of the three separations in $\Scal^{\mathsf{min}}_{(A, B)}.$
This implies that the entire graph $H$ is $A \cap B$-disk embeddable which contradicts our assumptions (see \cref{@backwoodsmen} for an illustration of how to embed all trivial components in disks bounded by their corresponding neighbourhoods in $A \cap B$).
\end{cproof}

Now, we argue that the vertices in $\Phi$ correspond to the vertices of a single component in $\Ccal_{1} \cup \Ccal_{2}.$
Indeed, assume that $\Phi$ contains two vertices $x$ and $y$ from distinct components $C_{x}$ and $C_{y}$ in $\Ccal_{1} \cup \Ccal_{2}$ respectively.
It is easy to see that there exists a minimal separation $(X, Y)$ in $\Scal_{(A, B)}$ such that without loss of generality $x \in X$ and $y \in Y.$
However, this contradicts the definition of $\Phi.$

Now, by definition, $H[\Phi]$ is the unique core component of $(A, B).$
To conclude with our proof, it follows by definition of the core component $\Phi,$ that the graph $H - \Phi$ is $(N_{H}(\Phi) \cap (A \cap B))$-disk embeddable.
To see this, observe that the connected components of $H - \Phi$ are all components that belong to the disk embeddable side of some separation in $\Scal^{\mathsf{min}}_{(A, B)}.$
\end{proof}

We require one more lemma which states that given two non-trivial ``parallel'' separations of a connected non-planar Kuratowski-connected graph of order three, with some additional properties, the core component of one is a subgraph of the core component of the other.

\begin{lemma}\label{core_comp_agree} Let $H$ be a connected non-planar Kuratowski-connected graph.
Let $(X, Y)$ be a non-trivial separation of $H$ of order at most three and $C$ be the core component of $H - (X \cap Y)$ where $V(C) \subseteq X \setminus Y.$
Let $(Z, W)$ be a non-trivial separation of $H$ of order at most three such that $X \subseteq W$ and $D$ be the core component of $H - (Z \cap W).$
Then $C \subseteq D.$
\end{lemma}
\begin{proof} First, since $X \subseteq W,$ every connected component of $H - (X \cap Y)$ whose vertex set is a subset of $X \setminus Y$ is a subgraph of some connected component of $H - (Z \cap W)$ that is a subset of $W \setminus X.$
Assume towards contradiction that our claim is false.
From our previous observation, it must be that $V(D) \subseteq Z \setminus W.$

Now, observe that there exists a minimal separator $S$ in $H$ that separates $C$ and $D$ such that $S \subseteq X \cap Y.$
Also, since $D$ is the core component of $(Z, W),$ by an application of \autoref{host_model_core_emb}, it follows that the graph $H \setminus D$ is $(N_{H}(V(D)) \cap (Z \cap W))$-disk embeddable.
Let $H'$ be the graph obtained from $H$ by contracting $D$ into a single vertex $x_{D}.$
Since $H \setminus D$ is $(N_{H}(V(D)) \cap (Z \cap W))$-disk embeddable it follows that $H'$ is a planar graph.
It also follows that $S$ is a minimal separator in $H'$ between $C$ and $x_{D}.$
Now, notice that since $S$ is a minimal separator in $H',$ there exists a minimal separation $(A_{1}, B_{1})$ in $H'$ such that $A_{1} \cap B_{1} = S$ and moreover, without loss of generality, $V(C)$ is a subset of $A_{1} \setminus B_{1}$ and $x_{D}$ belongs to $B_{1} \setminus A_{1}.$
Also, since $S \subseteq X \cap Y,$ by definition of $H',$ there exists a separation $(A_{2}, B_{2}) \in \Scal_{(X, Y)}$ in $H$ such that $A_{2} = A_{1}$ and $A_{2} \cap B_{2} = S.$
However, by standard arguments about minimal separators in planar graphs, there mush exist a closed curve in an embedding of $H'$ intersecting $H'$ only in the vertices in $S$ which witnesses that $H'[A_{1}] = H[A_{2}]$ is $S$-disk embeddable.
This contradicts the fact that $C$ is the core component of $H - (X \cap Y)$ and we conclude.
\end{proof}

We are now in the position to examine the behaviour of separations of order at most three of graphs that contain a given connected non-planar Kuratowski-connected graph $H$ as a minor.
To do this we need to consider $H$-\textsl{minor models} within
graphs that contain $H$ as a minor.

\paragraph{Models.}

Let $H$ and $G$ be two graphs.
Let $\Xcal = \{ X_{u} \mid u \in V(H) \},$ where
\begin{itemize}
\item  for every $u \in V(H),$ $X_{u} \subseteq V(G),$ $G[X_{u}]$ is connected,
\item  for every distinct $u,v \in V(G),$ $X_{u} \cap X_{v} = \emptyset,$ and \item for every edge $uv \in E(H)$ there is an edge $xy \in E(G)$ such that $x \in X_{u}$ and $y \in X_{v}.$
\end{itemize}
We call $\Xcal$ an $H$-\emph{minor model in} $G$ or, simply, an $H$-\emph{model} in $G.$
We write $V(\Xcal)$ for the set of vertices $\cupall \Xcal.$
Moreover, we call the sets $X_v,$ $v\in V(H),$ the \emph{branch sets} of $\mathcal{X}.$ 
We write $G[\Xcal]$ instead of $G[V(\Xcal)].$ 
Observe that $H$ is a minor of $G$ if and only if there exists an $H$-model in $G.$
Moreover notice that for any $H$-minor-model $\Xcal$ in an $H$-inflated copy $M$ it holds that $M[V(\Xcal)] = M.$

\medskip
We now extend the previous result to graphs that contain a given connected non-planar Kuratowski-connected graph as a minor.
More specifically, let $H$ be a connected non-planar Kuratowski-connected graph and let $M$ be a graph that contains $H$ as a minor.
Let $\Xcal = \{ X_{u} \mid u \in V(H) \}$ be an $H$-model in $M$ and $(A, B)$ be a separation of $M$ of order at most three.

We define the sets $A_{\Xcal} \coloneqq \{ u \in V(H) \mid X_{u} \cap A \neq \emptyset \}$ and $B_{\Xcal} \coloneqq \{ u \in V(H) \mid X_{u} \cap B \neq \emptyset \}.$
Observe that $(A_{\Xcal}, B_{\Xcal})$ is a separation of $H$ where $A_{\Xcal} \cap B_{\Xcal}$ contains every vertex $u \in V(H)$ such that $X_{u} \cap (A \cap B) \neq \emptyset.$
Notice that the order of $(A_{\Xcal},B_{\Xcal})$ is at most the order of $(A,B).$

Then, we say that $A$ (resp. $B$) is the $\Xcal$-\emph{core side} of $(A, B)$ if either
\begin{itemize}
\item $(A_{\Xcal}, B_{\Xcal})$ is a trivial separation of $H$ and $B_{\Xcal} \setminus A_{\Xcal}$ (resp. $A_{\Xcal} \setminus B_{\Xcal}$) is empty or
\item $(A_{\Xcal}, B_{\Xcal})$ is a non-trivial separation of $H$ and $A_{\Xcal}$ (resp. $B_{\Xcal}$) contains the core component of $(A_{\Xcal}, B_{\Xcal}).$
\end{itemize}

Moreover, let $A \cap B = \{ m_{1}, \ldots, m_{\ell}\},$ $\ell \leq 3,$ and $A_{\Xcal} \cap B_{\Xcal} = \{ s_{1}, \ldots, s_{\ell'} \},$ where $\ell' \leq \ell.$
We say that $M[A]$ (resp. $M[B]$) has an $(A \cap B)$-\emph{disk embeddable complement} $\langle Q, \tilde{Q} \rangle$ if there exists a graph $Q$ on at most $|V(H)|$ vertices and a set $\tilde{Q} = \{ q_{1}, \ldots, q_{\ell} \} \subseteq V(Q)$ such that $Q$ is $\tilde{Q}$-disk embeddable and $M' = \langle M[A], m_{1}, \ldots, m_{\ell} \rangle \oplus \langle Q, q_{1}, \ldots, q_{\ell} \rangle$ (resp. $M' = \langle M[B], m_{1}, \ldots, m_{\ell} \rangle \oplus \langle Q, q_{1}, \ldots, q_{\ell} \rangle$) contains $H$ as a minor.

Also, let $(A', B')$ be the separation of $M'$ such that $A' = A$ (resp. $B' = B$) and $A' \cap B' = A \cap B.$
We moreover say that $M[A]$ (resp. $M[B]$) has an $\Xcal$-respectful $\{ m_{1}, \ldots, m_{\ell} \}$-\emph{disk embeddable complement} if there exists an $H$-model $\Xcal'$ in $M'$ such that one of the following statements is true:
\begin{itemize}
\item If $(A_{\Xcal}, B_{\Xcal})$ is trivial then
\begin{itemize}
\item $(A'_{\Xcal'}, B'_{\Xcal'})$ is trivial and
\item $\Xcal' = \{ X_{u} \cap A' \mid X_{u} \in \Xcal \}$ (resp. $\Xcal' = \{ X_{u} \cap B' \mid X_{u} \in \Xcal \}$);
\end{itemize}
\item If $(A_{\Xcal}, B_{\Xcal})$ is non-trivial then
\begin{itemize}
\item $(A'_{\Xcal'}, B'_{\Xcal'})$ is non-trivial and the core component $C$ of $(A'_{\Xcal'}, B'_{\Xcal'})$ is the same as of $(A_{\Xcal}, B_{\Xcal})$ and
\item $\Xcal' = \{ X'_{u} \mid u \in V(H) \},$ where for every $u \in V(C),$ $X_{u} = X'_{u}.$
\end{itemize}
\end{itemize}
Notice that then, $A'$ (resp. $B'$) is the $\Xcal'$-core side of $(A', B').$

\begin{lemma}\label{host_model_core_emb}
Let $H$ be a connected non-planar Kuratowski-connected graph and $M$ be a graph that contains $H$ as a minor.
Then for every separation $(A, B)$ of $M$ of order at most three and every $H$-model $\Xcal$ in $M,$ if $A$ (resp. $B$) is the $\Xcal$-core side of $(A, B)$ then $M[A]$ (resp. $M[B]$) has an $\Xcal$-respectful $(A \cap B)$-disk embeddable complement.
\end{lemma}
\begin{proof} Let $\Xcal$ be an $H$-model in $M.$
Recall the definition of the separation $(A_{\Xcal}, B_{\Xcal})$ of $H.$
Clearly one of $A_{\Xcal} \setminus B_{\Xcal}$ or $B_{\Xcal} \setminus A_{\Xcal}$ is non-empty.
Assume without loss of generality that is $A_{\Xcal} \setminus B_{\Xcal}.$
There are two cases.

If $B_{\Xcal} \setminus A_{\Xcal}$ is empty then we are in the case where $(A_{\Xcal}, B_{\Xcal})$ is trivial and therefore by assumption $A$ is the $\Xcal$-core side of $(A, B).$
Then, we can simply conclude with $Q$ being a graph on $\ell$ vertices $\{ q_{1}, \ldots, q_{\ell} \},$ where $q_{i}$ and $q_{j}$ are adjacent if and only if $s_{i}$ and $s_{j}$ belong to the same branch set of $\Xcal.$
Then, by construction $Q$ is $\{ q_{1}, \ldots, q_{\ell} \}$-disk embeddable and $\langle M[A], m_{1}, \ldots, m_{\ell} \rangle \oplus \langle Q, q_{1}, \ldots, q_{\ell} \rangle$ contains $H$ as a minor, i.e. $M[A]$ has an $(A \cap B)$-disk embeddable complement which is trivially $\Xcal$-respectful.
Therefore we may assume that $(B_{\Xcal} \setminus A_{\Xcal})$ is non-empty.

Let $A \cap B = \{ m_{1}, \ldots, m_{\ell}\},$ where $\ell \leq 3,$ and $A_{\Xcal} \cap B_{\Xcal} = \{ s_{1}, \ldots, s_{\ell'} \}$ where $\ell' \leq \ell.$
By calling upon \cref{@inteuigence}, we obtain the unique core component $C$ of $(A_{\Xcal}, B_{\Xcal})$ and we moreover have that the graph $H \setminus C$ is $N_{H}(V(C)) \cap (A_{\Xcal} \cap B_{\Xcal})$-disk embeddable.
Now, we define a graph $H'$ by starting from a fresh copy of the graph $H \setminus C,$ adding the vertices $\{ t_{1}, \ldots, t_{\ell'} \},$ and then for each vertex $s_{i}$ in $N_{H}(V(C)) \cap \{ s_{1}, \ldots, s_{\ell'} \},$ we identify $t_{i}$ to the vertex that corresponds to the copy of $s_{i}$ in $H \setminus C.$

Assume without loss of generality that $V(C) \subseteq A_{\Xcal} \setminus B_{\Xcal}.$
Observe that by definition of $H',$ we have that $\langle H[A_{\Xcal}], s_{1}, \ldots, s_{\ell'} \rangle \oplus \langle H', t_{1}, \ldots, t_{\ell'} \rangle$ contains $H$ as a minor.
It also follows by definition that $A$ is the $\Xcal$-core side of $(A, B).$
We show that in this case $M[A]$ has an $\Xcal$-respectful $\{ m_{1}, \ldots, m_{\ell} \}$-disk embeddable complement.
Let $\{ q_{1}, \ldots, q_{\ell} \}$ be a fresh set of vertices and let $Q^{-} \coloneqq H' - \{ t_{1}, \ldots, t_{\ell'}\} + \{ q_{1}, \ldots, q_{\ell} \}.$
We get the desired $Q$ from $Q^{-}$ as follows.
We make $q_{i}$ and $q_{j}$ adjacent if any only if $m_{i}$ and $m_{j}$ belong to the same branch set of $\Xcal.$
Moreover we define an injection $\phi \colon \{ t_{1}, \ldots, t_{\ell'}\} \to \{ q_{1}, \ldots, q_{\ell} \}$ such that $\phi(t_{j}) = q_{i}$ if $m_{j}$ is one of the vertices that belongs to the branch set of $\Xcal$ that models the vertex $s_{j}$ of $H.$
Then, for every edge $ut_{j}$ of $H',$ we add the edge $u\phi(t_{j}).$
It follows by construction that $Q$ is $\{ q_{1}, \ldots, q_{\ell} \}$-disk embeddable.

Now, let $M' = \langle M[A], m_{1}, \ldots, m_{\ell} \rangle \oplus \langle Q, q_{1}, \ldots, q_{\ell} \rangle$ and $(A', B')$ be the separation of $M'$ such that $A' = A$ and $A' \cap B' =  \{ m_{1}, \ldots, m_{\ell} \}.$
Notice that for every vertex $u \in V(C)$ the branch set $X_{u} \in \Xcal$ that models $u$ is a subset of $A.$
Let $\Xcal'$ contain the sets $X_{u} \in \Xcal,$ where $u \in V(C),$ the sets $X_{v} \cap A,$ where $v \in N_{H}(V(C)) \cap (A_{\Xcal} \cap B_{\Xcal}),$ and a singleton $\{ v \},$ for every vertex $v \in V(Q) \setminus \{ q_{1}, \ldots, q_{\ell} \}.$
It follows by construction that $\Xcal'$ is an $H$-model in $M'$ that satisfies the desired properties.
Therefore $M[A]$ has an $\Xcal$-respectful $\{ m_{1}, \ldots, m_{\ell} \}$-disk embeddable complement.
\end{proof}

We require one more lemma which lifts the result of \autoref{core_comp_agree} to graphs that contain a connected non-planar Kuratowski-connected graph as a minor.

\begin{lemma}\label{invading_host_parallel_sep} Let $H$ be a connected non-planar Kuratowski-connected graph.
$M$ be a graph that contains $H$ as a minor and $\Xcal$ be an $H$-model in $M.$
Let $(X, Y)$ be a separation of $M$ of order at most three such that $X$ is the $\Xcal$-core side of $(X, Y).$
Let $(Z, W)$ be a separation of $M$ of order at most three such that $X \subseteq W$ and $Z \cap W = \{ m_{1}, \ldots, m_{\ell} \}.$
Then $M[W]$ has an $(X \cap Y)$-disk embeddable complement $\langle Q, q_{1}, \ldots, q_{\ell} \rangle.$

Moreover let $M' = \langle M[W], m_{1}, \ldots, m_{\ell} \rangle \oplus \langle Q, q_{1}, \ldots, q_{\ell} \rangle$ and $(X', Y')$ be the separation of $M'$ such that $X' = X.$
Then there exists an $H$-model $\Xcal'$ of $M'$ such that $X'$ is the $\Xcal'$-core side of $(X', Y').$
\end{lemma}
\begin{proof}
First, since $X$ is the $\Xcal$-core side of $(X, Y)$ we have that $X_{\Xcal} \setminus Y_{\Xcal} \neq \emptyset$ and therefore since $X \subseteq W$ we also have that $W_{\Xcal} \setminus Z_{\Xcal}$ is not empty.
Therefore, in the case that $(Z, W)$ is a trivial separation of $M,$ clearly $W$ is the $\Xcal$-core side of $(Z, W).$
Otherwise, if $(Z, W)$ is non-trivial, it is implied that $(X, Y)$ is non-trivial, and then an application of \autoref{core_comp_agree} implies that $C \subseteq D$ and therefore that $W$ is the $\Xcal$-core side of $(Z, W),$ where $C$ is the core component of $(X_{\Xcal}, Y_{\Xcal})$ and $D$ is the core component of $(Z_{\Xcal}, W_{\Xcal}).$
In any case, \autoref{core_comp_agree} implies that $M[W]$ has an $\Xcal$-respectful $(X \cap Y)$-disk embeddable complement $\langle Q, q_{1}, \ldots, q_{\ell} \rangle.$

Now, in the case that $(Z, W)$ is trivial, the second scale of our claim follows directly from the fact that $\langle Q, q_{1}, \ldots, q_{\ell} \rangle$ is $\Xcal$-respectful.
Therefore assume that $(Z, W)$ and therefore also $(X, Y)$ is non-trivial.
Then, since $\langle Q, q_{1}, \ldots, q_{\ell} \rangle$ is $\Xcal$-respectful, there exists an $H$-model $\Xcal'$ in $M'$ such that $(X'_{\Xcal'}, Y'_{\Xcal'})$ is non-trivial and the core component of $(X'_{\Xcal'}, Y'_{\Xcal'})$ is $D$ and $\Xcal' = \{ X'_{u} \mid u \in V(H) \},$ where for every $u \in V(C),$ $X_{u} = X'_{u}.$
Then, since $C \subseteq D$ and $X' = X$ our claim follows.
\end{proof}

Having set the lemmas above in place we are now ready to establish the notions of invading hosts and red cells in a $\Sigma$-decomposition of our graph of interest.

\medskip
Let $H$ be a graph. We denote by $\nonplanar(H)$ the union of all its non-planar components.
Observe that if $H$ is Kuratowski-connected, then $\nonplanar(H)$ is connected and moreover $\nonplanar(H)$ is a Kuratowski-connected graph itself.

Let $H$ be a (not necessarily connected) non-planar Kuratowski-connected graph and $\delta$ be a vortex-free $\Sigma$-decomposition of a graph $G$ in a surface $\Sigma.$

For every cell $c \in C(\delta)$ let $(A^{c}, B^{c})$ be the separation of $G$ of order at most three such that $A^{c} \coloneqq V(\sigma_{\delta}(c))$ and $B^{c} \coloneqq V(G) \setminus (V(\sigma_{\delta}(c)) \setminus \pi_{\delta}(\tilde{c})).$
Furthermore, for every $\nonplanar(H)$-host $M$ in $G$ and every cell $c \in C(\delta),$ we define the separation $(A^{c}_{M}, B^{c}_{M})$ of $M$ of order at most three, where $A^{c}_{M} \coloneqq A^{c} \cap V(M)$ and $B^{c}_{M} \coloneqq B^{c} \cap V(M).$
Given a cell $c \in C(\delta)$ we say that an $\nonplanar(H)$-host $M$ is $c$-\emph{invading} if $V(M) \cap \pi_{\delta}(\tilde{c}) \neq \emptyset$ and there exists an $\nonplanar(H)$-model $\Xcal$ in $M$ such that $A^{c}_{M}$ is the $\Xcal$-core side of $(A^{c}_{M}, B^{c}_{M}).$

Moreover, we say that a cell $c \in C(\delta)$ is $H$-\emph{red} if there exists a $c$-invading $\nonplanar(H)$-inflated copy in $G.$

\medskip
We are now ready to prove the main lemma of this subsection which states that every $\nonplanar{H}$-host in a graph $G$ is either $c$-invading or does not contain any ground vertex of a vortex-free $\Sigma$-decomposition of $G$ in a surface $\Sigma,$ assuming that $\Sigma$ belongs to $\Sbbb_{H}.$

\begin{lemma}\label{@primitives} Let $H$ be a non-planar Kuratowski-connected graph.
Let $\delta$ be a vortex-free $\Sigma$-decomposition of a graph $G$ in a surface $\Sigma \in \Sbbb_{H}.$
Then for every $\nonplanar(H)$-host $M$ in $G$ either $V(M) \cap \ground(\delta)$ is empty or there exists a cell $c \in C(\delta)$ such that $M$ is $c$-invading.
\end{lemma}
\begin{proof} Assume towards contradiction that there exists an $\nonplanar(H)$-host $M$ that does not satisfy our claim.
Let $I$ be the set of all cells $c \in C(\delta)$ such that $V(M) \cap \pi_{δ}(\tilde{c}) \neq \emptyset.$
Then, by assumption, it must be that $I$ is not empty and moreover, for every cell $c \in I,$ it must be that $M$ is not $c$-invading.
Our goal it to show that, under these assumptions, we can find an artificial graph that still contains $\nonplanar(H)$ as a minor that can be embedded in $\Sigma$ and therefore obtain our contradiction.
Notice that this is possible since $\Sbbb_{H} = \Sbbb_{\nonplanar(H)}$ as all other components of $H$ are planar and therefore have Euler genus zero.

Consider $c \in I$ and let $V(M) \cap \pi_{\delta}(\tilde{c}) = \{ m^{c}_{1}, \ldots, m^{c}_{\ell_{c}}\},$ $\ell_{c} \leq |\tilde{c}|.$
Since $M$ is not $c$-invading, it must be that $A^{c}_{M}$ is not the $\Xcal$-core side of $(A^{c}_{M}, B^{c}_{M})$ and therefore $B^{c}_{M}$ is the $\Xcal$-core side of $(A^{c}_{M}, B^{c}_{M}),$ for every $\nonplanar(H)$-model $\Xcal$ in $M.$
Fix an $\nonplanar(H)$-model $\Xcal$ in $M.$
Let $A^{c}_{\Xcal} \cap B^{c}_{\Xcal} = \{ s^{c}_{1}, \ldots, s^{c}_{\ell'_{c}}\},$ where $\ell'_{c} \leq \ell_{c}.$
By an application of \autoref{host_model_core_emb} is must be that $M[B^{c}_{M}]$ has an $\Xcal$-respectful $\{ m^{c}_{1}, \ldots, m^{c}_{\ell_{c}} \}$-disk embeddable complement.
Then, by definition, there exists a graph $Q^{c}$ and a set $\{ q^{c}_{1}, \ldots, q^{c}_{\ell_{c}} \} \subseteq V(Q^{c})$ such that $Q^{c}$ is $\{ q^{c}_{1}, \ldots, q^{c}_{\ell_{c}} \}$-disk embeddable, $\langle M' = M[B^{c}_{M}], m^{c}_{1}, \ldots, m^{c}_{\ell_{c}} \rangle \oplus \langle Q^{c}, q^{c}_{1}, \ldots, q^{c}_{\ell_{c}} \rangle$ contains $\nonplanar(H)$ as a minor.

We now consider the graph $G' = G[B^{c}], m^{c}_{1}, \ldots, m^{c}_{\ell_{c}} \rangle \oplus \langle Q^{c}, q^{c}_{1}, \ldots, q^{c}_{\ell_{c}} \rangle$ and the $\Sigma$-decomposition $\delta'$ of $G'$ which we obtain from $\delta$ by replacing 
$\sigma_{\delta}(c) - \{ m_{1}, \ldots, m_{\ell}\}$ and its drawing by the $\{ m^{c}_{1}, \ldots, m^{c}_{\ell_{c}} \}$-disk embeddable complement of $M[B^{c}_{M}].$
By construction, $M'$ is an $\nonplanar(H)$-host in $G'.$
Now, let $I' = I \setminus \{ c \}.$
Since, the $\{ m^{c}_{1}, \ldots, m^{c}_{\ell_{c}} \}$-disk embeddable complement of $M[B^{c}_{M}]$ is also $\Xcal$-respectful, it is implied that there exists an $\nonplanar(H)$-model $\Xcal'$ in $M'$ satisfying the definition of $\Xcal$-respectful.
By an application of (the proof of) \autoref{invading_host_parallel_sep}, it now follows that, for every cell $c \in I',$ $B^{c}_{M'}$ is the $\Xcal'$-core side of $(A^{c}_{M'}, B^{c}_{M'})$ (where $(A^{c}, B^{c})$ is interpreted as a separation in $G'$).
It now follows that, if we iteratively repeat the process above for every cell, whose boundary is intersected by $M$, we can construct a graph that contains $\nonplanar(H)$ as a minor than can be embedded in $\Sigma,$ which contradicts our assumptions.
\end{proof}

Finally, for every $\Zcal \in \Kbbb^{-},$ if $\delta$ is a vortex-free $\Sigma$-decomposition of a graph $G$ in a surface $\Sigma,$ then we say that a cell $c \in C(\delta)$ is $\Zcal$-red if it is $H$-red for some $H \in \Zcal.$

\medskip
The next and final observation of this subsection follows by definition of $\Zcal$-red cells and the minor-checking algorithm of Kawarabayashi, Kobayashi, and Reed \cite{KawarabayashiKR12Thedisjoint}  that takes $\Ocal_{h_{\Zcal}}(|V(G)|^{2})$ time.

\begin{observation}\label{@calculations} For every $\Zcal \in \Kbbb^{-},$ there exists an algorithm that, given a graph $G$ and a $Σ$-schema $(A, δ, D)$ of $G,$ detects all $\Zcal$-red cells of $δ.$
Moreover this algorithm runs in time $\Ocal_{h_{\Zcal}}(|V(G)|^{3}).$
\end{observation}

\subsection{The redrawing lemma}
\label{@enthroning}

In this subsection we prove the ``Redrawing Lemma'' which as we have previously discussed will allow us to ``locally'' redraw the disk-embeddable part of inflated copies of connected non-planar Kuratowski-connected graphs that invade through red cells that reside deep within a well insulated area of our $\Sigma$-decomposition.

\medskip
We first introduce mixed packings of graphs for a specific antichain in a given graph.
Let $G$ be a graph and $\Zcal$ be an antichain of graphs.
A \emph{mixed} (\emph{half-integral}) $\Zcal$-{packing} in $G$ is a collection $M_{1}, \ldots, M_{k}$ of pairwise disjoint $\Zcal$-\majors in $G.$

\begin{observation}
Given an antichain $\Zcal$ of graphs, if a graph $G$ has a (half-integral) mixed $\Zcal$-packing of size $k \cdot |\Zcal|,$ then it also has a (half-integral) $\Zcal$-packing of size $k.$
\end{observation}

We next prove the ``Redrawing Lemma'' which we will be used several times throughout our proofs.
It receives a $\Sigma$-schema $(A, \delta, D)$ of a graph $G$ in a surface $\Sigma$ and a collection of inflated copies in $G - A$ of non-planar components of the graphs in $\Zcal,$ invading through a set of $\Zcal$-red cells in $G - A$ and explains that, under certain circumstances, this collection can be rerouted an enhanced with the planar components that are missing so as to form a mixed $\Zcal$-packing.
The existence of a railed nest around these cells will provide the infrastructure for this rerouting.
\cref{@horkheimer} is used for $k=1$ in \cref{@abstractness} and \cref{@conjecture} and is applied in its full generality in \cref{@translation}.

\begin{lemma}[Redrawing lemma]\label{@horkheimer} There exists a function $f_{\ref{@horkheimer}} : \Bbb{N}^2 \to \Bbb{N}$ such that, for every $\Zcal \in \Kbbb^{-},$ if 
\begin{itemize}
\item $(A, \delta, D)$ is a $Σ$-schema of a graph $G$ in a surface $\Sigma,$
\item $(\Ccal, \Pcal)$ is a railed nest of $G - A$ around an arcwise connected set in $\Sigma$ of order at least $f_{\ref{@horkheimer}}(k, h_{\Zcal}),$
\item $\Delta^{*}$ (resp. $\Delta$) is the disk bounded by the trace of the internal (resp. external) cycle of $\Ccal$, and
\item $\Mcal = \{ M_{i} \mid i \in [k] \}$ is a set of subgraphs of $G - A$ such that, for every $i \in [k],$ $M_{i}$ is an $\nonplanar(H_{i})$-inflated copy in $G - A,$ where $H_{i} \in \Zcal,$ $M_{i}$ is $c_{i}$-invading, where $c_{i} \in C(\delta)$ is a subset of $\Delta^{*},$ and the graphs in $\{ M_{i} \cap \inG_{\delta}(\Delta) \mid i \in [k] \}$ are pairwise disjoint,
\end{itemize}
then $\inG_{δ}(Δ)$ contains a mixed $\Zcal$-packing of size $k.$
Moreover, for every $\Zcal \in \Kbbb^{-},$ there exists an algorithm that, given the above, computes such a mixed $\Zcal$-packing in time $2^{\Ocal(k) \cdot \poly(h_{\Zcal})} |V(G)|.$
Also $f_{\ref{@horkheimer}}$ is of order $2^{\Ocal(k) \cdot \poly(h_{\Zcal})}.$
\end{lemma}
\begin{proof} Let $ξ = |\Ccal|,$ $\Ccal = \langle C_{1}, \ldots, C_{\xi} \rangle,$ and $\Pcal = \langle P_{1}, \ldots, P_{\xi} \rangle.$

Since $M_{i}$ is an $\nonplanar(H_{i})$-inflated copy, by definition there is a set $T_{i}$ of at most $|V(\nonplanar(H_{i}))|^{2}$ vertices of $M_{i}$ such that $(M_{i}, T_{i})$ is a minimal $\nonplanar(H_{i})$-expansion in $G$.
Moreover, since $M_{i}$ is $c_{i}$-invading, let $\Xcal_{i}$ be an $\nonplanar(H_{i})$-model in $M_{i}$ such that $A^{c_{i}}_{M_{i}}$ is the $\Xcal_{i}$-core side of the separation $(A^{c_{i}}_{M_{i}}, B^{c_{i}}_{M_{i}})$ of $M_{i}.$

For every $i \in [k],$ we define the graph $\tilde{M}_{i} = M_i \cap (\inG_{\delta}(\Delta) \cap \outG_{\delta}(c_{i})).$
Notice that $G - A$ contains a minimal cut $S_{i} = \{s^i_{1}, \ldots, s^i_{\ell_i} \}$ in $\tilde{M}_{i},$ between $V(M_{i}) \cap \pi_{\delta}(\tilde{c_{i}})$ and $V(M_{i}) \cap V(Ω_{\Delta}),$ where $\ell_{i} \leq 3,$ and with the property that there is an $S_{i}$-$V(Ω_{\Delta})$ linkage $\Lcal_{i}$ in $G - A$ of order $\ell_{i}$ in $\inG_{\delta}(\Delta)$ and a separation $(Z^{\mathsf{in}}_{i}, Z^{\mathsf{out}}_{i})$ of $M_i$ such that $A^{c_{i}}_{M_{i}} \subseteq Z^{\mathsf{in}}_{i}$ and $Z^{\mathsf{in}}_{i} \cap Z^{\mathsf{out}}_{i} = S_{i}.$
Moreover, by the minimality of $S_{i},$ we may also assume that all vertices in $S_{i}$ are ground vertices in $\delta.$
Let $\{d_{1}^{i}, \ldots, d_{\ell_{i}}^{i} \}$ be the endpoints of the paths in $\Lcal_{i}$ such that, for every $j \in [\ell_{i}],$ $s_{i}$ and $d_{i}$ are endpoints of the same path in $\Lcal_{i}.$

Since $A^{c_{i}}_{M_{i}} \subseteq Z^{\mathsf{in}}_{i}$ and $A^{c_{i}}_{M_{i}}$ is the $\Xcal_{i}$-core side of $(A^{c_{i}}_{M_{i}}, B^{c_{i}}_{M_{i}}),$ let $\langle Q_{i}, q_{1}^{i}, \ldots, q_{\ell_{i}}^{i} \rangle$ be the $S_{i}$-disk embeddable complement of $M_{i}[Z^{\mathsf{in}}_{i}]$ as implied by \cref{invading_host_parallel_sep}.
Let $\tilde{Q_{i}} = \{ q_{1}^{i}, \ldots, q_{\ell_{i}}^{i} \} \subseteq V(Q_{i}).$
In particular, we have that
\begin{itemize}
\item $Q_i$ is $\tilde{Q_i}$-disk embeddable, $|V(Q_i)| ≤ |V(H_i)|,$ and $M''_{i} = \langle M_{i}[Z^{\mathsf{in}}_{i}], s_1^i, \ldots, s_{\ell_i}^i \rangle \oplus \langle Q_i, q_1^i, \ldots, q_{\ell_i}^i \rangle$ contains $\nonplanar(H_{i})$ as a minor and
\item there exists an $\nonplanar(H_{i})$-model $\Xcal''_{i}$ in $M''_{i}$ such that $A^{c_{i}}_{M''_{i}}$ is the $\Xcal''_{i}$-core side of $(A^{c_{i}}_{M''_{i}}, B^{c_{i}}_{M''_{i}}),$ where $(A^{c_{i}}_{M''_{i}}, B^{c_{i}}_{M''_{i}})$ is the separation of $M''_{i}$ where $A^{c_{i}}_{M''_{i}} = A^{c_{i}}_{M_{i}}.$
\end{itemize}

We define
$$M^*_{i} \coloneqq \langle M_{i}[Z^{\mathsf{in}}_{i}], s_1^i, \ldots, s_{\ell_i}^i \rangle \oplus \langle \cupall \Lcal_{i}, s_{1}^i, \ldots, s_{\ell_i}^i \rangle.$$
Observe that $M_{1}^*, \ldots, M_{k}^*$ are pairwise disjoint subgraphs of $G.$
Moreover, for every $i \in [k],$ we define $M_{i}' \coloneqq \langle M_{i}^*, d_1^i, \ldots, d_{\ell_i}^i \rangle \oplus \langle Q_i, q_1^i, \ldots, q_{\ell_i}^i \rangle.$
Let 
\begin{eqnarray*}
G' & = & \langle (G - A) \cap Δ, d_{1}^1, \ldots, d_{\ell_{1}}^1,\ \ldots\ , d_{1}^k, \ldots, d_{\ell_{k}}^k \rangle \oplus \langle \bigcup_{i\in[k]} Q_{i}, q_{1}^1, \ldots, q_{\ell_{1}}^1, \ \ldots\ , q_{1}^k, \ldots, q_{\ell_{k}}^k \rangle
\end{eqnarray*}
In other words, the graph $G'$ is obtained from $(G-A) \cap Δ$ after considering its disjoint union with the graphs $Q_{1}, \ldots, Q_{k}$ and, for every $i \in [k],$ identifying, for every $h \in [\ell_{i}],$ the $h$-th vertex of $\{d_1^i,\ldots, d_{\ell_i}^i \}$ with the $h$-th vertex of $\{q_1^i, \ldots, q_{\ell_i}^i\}.$
Notice now that, for every $i \in [k],$ $(M^{\prime}_{i}, (T_{i} \cap Z^{\mathsf{in}}_{i}) \cup S_{i} \cup V(Q_{i}))$ is an $\nonplanar(H_{i})$-expansion in $G'.$

Keep in mind that the graph $G'$ where we have found the set $\Mcal' = \{ M'_{i} \mid i \in [k] \}$ of pairwise disjoint $\nonplanar(H_{i})$-hosts in $G',$ $i \in [k],$ is an ``artificial'' graph and neither $G'$ is a subgraph of $G$ nor the graphs in $\Mcal'$ are subgraphs of $G.$
Our objective is to use the graph $G'$ and $\Mcal'$ to find a $\Zcal$-mixed packing in $G-A.$
This will be done by revising $\Mcal'$ in a way that every graph in $\Mcal'$ is a subgraph of $G' \cap \Delta_{\xi},$ where $\Delta_{\xi}$ is the disk bounded by the trace of $C_{\xi},$ (that is a subgraph of $G$), i.e., avoids the additional artificial vertices of $G'$ that are drawn outside $\Delta_{\xi}$ and the completing it to a $\Zcal$-mixed packing in $G' \cap \Delta_{\xi}$ by showing how to also find the ``missing'' planar components in $G' \cap \Delta_{\xi},$ that are necessary to complete each $M'_{i}$ into an $H_{i}$-host, $i \in [k].$

\medskip
As a next step, we modify the surface $Σ$ and the decomposition $δ$ as follows: we first consider the sphere $\Sigma^{(0,0)}$ obtained if we glue the boundary of $\Delta$ with the boundary of another closed disk $\Delta^{\mathsf{out}}.$
We then obtain a $Σ^{(0,0)}$-decomposition $δ'$ from $δ$ by removing from it all cells that are not subsets of $\Delta$ as well as the drawing of vertices and edges that are contained in these cells, except from the vertices in $V(Ω_{\Delta}).$
Next we add $k$ cells $c^+_{1}, \ldots, c^{+}_{k}$ that are subsets of $\Delta^{\mathsf{out}}$ so that, for each $i \in [k],$ $\pi_{\delta'}(\tilde{c}^+_{i}) \cap \Delta^{\mathsf{out}} = \{d_{1}^i, \ldots, d_{\ell_{i}}^i\}.$
The construction of $\delta' = (\Gamma', \Dcal')$ is completed by drawing inside each $c_{i}^{+}$ the graph $Q_{i}.$
This is possible as $Q_{i}$ is $\tilde{Q}_i$-disk embeddable.

By construction, $\delta'$ is a $\Sigma^{(0,0)}$-decomposition of the graph $G'.$
Furthermore, notice that it is straightforward to define an $\nonplanar(H_{i})$-model $\Xcal'_{i}$ in $M'_{i}$ from the $\nonplanar(H_{i})$-model $\Xcal''_{i}$ in $M''_{i},$ that preserves the additional property of $\Xcal''_{i}$, namely such that $A^{c_{i}}_{M'_{i}}$ is the $\Xcal'_{i}$-core side of $(A^{c_{i}}_{M'_{i}}, B^{c_{i}}_{M'_{i}}),$ where $(A^{c_{i}}_{M'_{i}}, B^{c_{i}}_{M'_{i}})$ is the separation of $M'_{i}$ where $A^{c_{i}}_{M'_{i}} = A^{c_{i}}_{M_{i}}.$
Notice that $(A^{c_{i}}_{M'_{i}}, B^{c_{i}}_{M'_{i}}) = (A^{c_{i}} \cap V(M'_{i}), B^{c_{i}} \cap V(M'_{i})),$ where $c_{i}$ is interpreted as a cell in $\delta'$ and $(A^{c_{i}}, B^{c_{i}})$ as the corresponding separation in $G'.$
Therefore, for every $i \in [k],$ $M'_{i}$ remains $c_{i}$-invading (with respect to $\delta'$).
We set $T_{i}' = (T_{i} \cap Z^{\mathsf{in}}_{i}) \cup S_{i} \cup V(Q_{i}), i \in [k].$

Given a cell $c$ of $\delta'$ and an $i \in [k],$ we say that $c$ is $M'_i$-\emph{proper} if $V(\inG_{\delta'}(c)) \setminus \pi_{\delta'}(\tilde{c})$ does not contain any vertex of $T'_i.$
If the cell $c$ is not $M'_i$-proper, then we call it $M'_i$-\emph{improper}.
We define
$$\mbox{$M' = \bigcup_{i \in [k]} M'_{i},$ ~~ ${T}' = \bigcup_{i \in [k]} T_{i}',$ ~~ $H=\bigcup_{i\in[k]} \nonplanar(H_i)$}$$
and notice that $(M', T')$ is an $H$-expansion of $G'$ where
$$|T'| ≤ k\bar{h},$$
agreeing that $\bar{h} = (h_{\Zcal}^2 + h_\Zcal + 3).$

Given a cell $c$ of $\delta',$ we say that it is $M'$-\emph{proper} if it is $M'_{i}$-\emph{proper} for some $ i \in [k]$ or, equivalently, if $V(\inG_{\delta'}(c)) \setminus \pi_{\delta'}(\tilde{c})$ does not contain any vertex of $T'.$
If the cell $c$ is not $M'$-proper, then we call it $M'$-\emph{improper}.
Clearly, the cells $c_{1}, \ldots, c_{k}$ are all $M'$-improper and every $M'$-improper cell should contain at least one vertex in $T'.$ 
Therefore there at most $|T'| ≤ k\bar{h}$ $M'$-improper cells in $\delta'.$

We now define an auxiliary graph $\hat{G}'$ from $G'$ as follows:
First, for every $M'$-improper cell $c$ of $\delta',$ we remove from $G'$ all vertices in $\inG_{\delta'}(c)\setminus \pi_{\delta'}(\tilde{c}).$
Next, for every $M'$-proper cell $c$ of $\delta',$ we consider the minimal subgraph $R_{c}$ of $\inG_{\delta'}(c)$ that contains the vertices in $\pi_{\delta'}(\tilde{c})$ and has the property that for every $x, x' \in \pi_{\delta'}(\tilde{c}),$ $x, x'$ are connected by a path in $\inG_{\delta'}(c)$ if and only if they are also connected in $R_{c}.$
Notice that $R_{c}$ is either the graph $(\tilde{c}, \emptyset)$ or is a graph consisting of two or three paths each joining some vertices of $\pi_{\delta'}(\tilde{c})$ with some common vertex $v_{c}.$

Notice that the choice of $R_{c}$ can be done so that $\hat{G}'$ is a subgraph of $G'$ and that, for each $M'$-proper cell of $δ',$ the graph $M' \cap \inG_{\delta'}(c)$ is a subgraph of $R_{c}.$ 
Moreover, we may consider a $\Sigma^{(0,0)}$-decomposition $\hat{\delta}' = (\hat{\Gamma}', \hat{\Dcal}')$ of $\hat{G}',$ where $U(\hat{\Gamma}') \subseteq U(\Gamma'),$ where $\hat{\Dcal}'$ is the same as $\Dcal,$ and where every cell of $\hat{\delta}'$ that corresponds to an $M'$-proper cell of $\delta',$ contains $R_{c}.$
We may also insist, by possibly revising the drawing of the cells of $\delta,$ that the drawings of each $R_{c}$ inside $c,$ are without crossings.
These adaptations do not alter the fact that $(M', T')$ is an $H$-expansion of $G'.$

Similarly, to $\hat{G}',$ we define 
\begin{eqnarray*}
\hat{M}'& = & M' - \cupall\{V(\inG_{\delta'}(c)) \setminus \pi_{\delta'}(\tilde{c})\mid \text{$c$ is an $M'$-improper cell of $\delta'$}\}\mbox{~and~}\\
\hat{T}' & = & (T' \cap V(\hat{M}')) \cup \{\pi_{\delta'}(\tilde{c})\mid \text{$c$ is an $M'$-improper cell of $\delta'$}\}.
\end{eqnarray*}
Observe that $(\hat{M}',\hat{T}')$ is an expansion in $\hat{G}'$ and therefore in $G'$ as well.
Also, since for every $M'$-proper cell $c$ of $\delta',$ in its corresponding cell $c$ of $\hat{\delta}',$ $R_{c}$ is drawn without crossings, we deduce that $\hat{δ}' = (\hat{Γ}', \Dcal')$ is a $Σ^{(0,0)}$-embedding, therefore $\hat{G}'$ is a planar graph.
Also it holds that
$$|\hat{T}'| ≤ |T'| + |T'| ≤ 2k \bar{h}.$$

Let $\hat{H} = (\hat{T}', E_{\hat{H}})$ be the graph obtained if in $\hat{M}'$ we dissolve all vertices that do not belong to $\hat{T}'.$
As $(\hat{M}', \hat{T}')$ is a $\hat{H}$-expansion of $\hat{G}',$ there is a function $\sigma$ mapping each edge $e = xy \in E_{\hat{H}}$ to a path $\sigma(e)$ of $\hat{G}'$ such that $\hat{M}' = \bigcup_{e \in E_{\hat{H}}} \sigma(e).$

\medskip
Notice now that we may detect two cycles $C_{\zeta}$ and $C_{\eta}$ of $\Ccal$ where $\zeta \leq \eta$ and such that the annulus $A_{\zeta, \eta}: = (\Delta_{\eta} \setminus \Delta_{\zeta}) \cup \bd(\Delta_{C_{\zeta}}),$ where $\Delta_{\eta}$ (resp. $\Delta_{\zeta}$) is the disk bounded by the trace of $C_{\eta}$ (resp. $C_{\zeta}$), does not contain any $M'$-improper cell of $\delta'.$
We may also assume that $\eta - \zeta$ is even and we set $\mu = (\eta - \zeta) / 2.$
All these are possible if $\eta - \zeta + 1 = 2 \mu + 1 \leq \xi / (|\hat{T}'|+1)$ or, equivalently, 
\begin{eqnarray}
\xi & \geq & (2 \mu + 1) \cdot (2k \bar{h} + 1).\label{@occasionally}
\end{eqnarray}
We also denote $\overline{A}_{\zeta, \eta} = (\Sigma^{(0,0)} \setminus A_{\zeta, \eta}) \cup \bd(\Delta_{\zeta}) \cup \bd(\Delta_{\eta}).$

From \cite[Theorem 2.1]{GolovachST19Hitting} (see also \cite{GolovachST20Hitting}) there are two functions $f_{1} \colon \Nbbb \to \Nbbb,$ $f_{2} \colon \Nbbb \to \Nbbb,$ where $f_{1}(x) = \Ocal((f_{2}(x))^2),$ and a function $\sigma^*$ mapping each edge $e = xy \in E_{\hat{H}}$ to a path $\sigma^*(e)$ of $\hat{M}'$ such that, given that $2\mu + 1 \geq f_1(|\hat{T}'|),$ 
\begin{itemize}
\item $M^* \coloneqq \bigcup_{e \in E_{\hat{H}}} \sigma^*(e)$ is a subgraph of $\hat{G}',$
\item $C_{\mu} \cap M^*$ is a subgraph of $C_{\mu} \cap \cupall \{P_{1}, \ldots, P_{f_2(|\hat{T}'|)} \},$ and 
\item $\overline{A}_{\zeta, \eta} \cap M^*$ is a subgraph of $\overline{A}_{\zeta, \eta} \cap \hat{M}'.$
\end{itemize}
The result of \cite[Theorem 2.1]{GolovachST19Hitting} follows directly from the ``linkage combing lemma'' of \cite[Theorem 5]{GolovachST22Combing}, according to which the growth of $f_{2}(x)$ (where $x$ is the linkage size) depends linearly on the dependency of the ``linkage theorem'' proved in \cite{RobertsonS12GraphMinorsXXII} (see also \cite{KawarabayashiW10Ashorter}).
This dependency is quite high for general graphs (see e.g., \cite{KawarabayashiW10Ashorter} for the best known bounds), however for planar graphs, that is the case of $\hat{G}',$ we may assume that $f_{2}(x) = 2^{\Ocal(x)}$ because of the main result of \cite{AdlerKKLST17Irrelevant}.
Therefore, in order to make the above result applicable, we may pick any $\xi$ such that
\begin{eqnarray}
\xi & \geq & (2k \bar{h} + 1) \cdot f_{1}(2k \bar{h})\label{@industrially}
\end{eqnarray}
as this will guarantee that $\xi \geq (|\hat{T}'| + 1) \cdot f_{1}(|\hat{T}'|).$

\medskip
We now use $M^*$ and $\sigma^*$ in order to modify $M'$ by replacing, for every edge $e \in E(\hat{H}),$ the path $\sigma(e)$ by the path $\sigma^*(e).$
Notice that this modification does not alter the fact that $(M',T')$ is an $H$-expansion of $G'$ nor the fact that every updated $M'_{i},$ $i \in [k],$ is $c_{i}$-invading.
However, if $U \coloneqq V(M') \cap V(Ω_{Δ_{\mu}}),$ then  $|U| ≤ f_2(|\hat{T}'|) ≤ f_2(2k \bar{h})$ and, moreover, we also have that the vertices of $U$ are also vertices of the paths $P_{1}, \ldots, P_{\xi},$ and that $U \cap T' = \emptyset.$

We define $\overline{\Delta}_{\mu} = (\Sigma^{(0,0)} \setminus {\Delta}_{\mu}) \cup \bd(\Delta_{\mu})$ and we set $G'_{\mu} = G' \cap \overline{\Delta}_{\mu}.$
We consider the $\overline{\Delta}_{\mu}$-decomposition $\delta'_{\mu}$ of $G' \cap {\overline{\Delta}}_{\mu},$ defined by removing from $\delta'$ all cells that are not subsets of $\overline{\Delta}_{\mu}.$

\medskip
Notice that the $\overline{\Delta}_{\mu}$-decomposition $\delta'_{\mu}$ and the corresponding drawing of $G' \cap \inG_{\delta'}({\overline{\Delta}}_{\mu})$ do not necessarily imply that the drawing of $M' \cap \inG_{\delta'}({\overline{\Delta}}_{\mu}),$ according to $\delta'_{\mu},$ is a $\overline{\Delta}_{\mu}$-embedding.
Even worse, it does not even imply that $M' \cap \inG_{\delta'}({\overline{\Delta}}_{\mu})$ is $\overline{\Delta}_{\mu}$-embeddable and this is due to the fact that there might exist a cell $c$ of $\delta'_{\mu}$ where $M' \cap \inG_{\delta'}(c)$ is not $(\pi_{\delta'}(\tilde{c}) \cap V(M'))$-disk embeddable.
Our next step is to further modify the part of $M'$ that is drawn inside $\overline{\Delta}_{\mu},$ according to $\delta',$ so that its new version, say $M^*,$ can substitute $M'$ as the graph of some $H$-expansion in $G'$ in a way that $M^* \cap \inG_{\delta'}(c)$ will be $(\pi_{\delta'}(\tilde{c}) \cap V(M^*))$-disk embeddable, for every cell $c$ of $\delta'_{\mu}.$

\medskip
For $i \in [k],$ let $B_i$ be the set of the cells $c$ of $\delta'_{\mu}$ where $\inG_{\delta'}(c) \cap M_{i}'$ is not $(\pi_{\delta'}(\tilde{c}) \cap V(M_{i}'))$-disk embeddable.
Similarly, we denote by $B$ be the set of all cells $c$ of $\delta'_{\mu}$ where $\inG_{\delta'}(c) \cap M'$ is not $(\pi_{\delta'}(\tilde{c}) \cap V(M_{i}'))$-disk embeddable.
Notice that $B = \bigcup_{i \in [k]} B_{i}.$
Notice also that every cell in $B$ is an $M'$-improper cell, therefore $|B| ≤ |T'|.$

Consider now some $c\in B$ and assume that $I^{c} \subseteq [k]$ is the set of indices for which $c$ is a member of $B_{i},$ $i \in I.$
Clearly $|I| \leq 3.$
For every $i \in I^{c},$ we denote $S^{c, i} = \{ s^{c, i}_{1}, \ldots s^{c, i}_{\ell_{c. i}} \} = A^{c}_{M'_{i}} \cap B^{c}_{M'_{i}}.$
Since, for every $i \in I^{c},$ $M'_{i}$ is $c_{i}$-invading and $A^{c_{i}}_{M'_{i}} \subseteq A^{c}_{M'_{i}},$ we may apply \cref{invading_host_parallel_sep} to obtain an $S^{c, i}$-disk embeddable complement $\langle Q^{c, i}, q^{c, i}_{1}, \ldots, q^{c, i}_{\ell_{c, i}} \rangle$ for $M'_{i}[A^{c}_{M'_{i}}].$
Let $\tilde{Q}^{c, i} = \{ q^{c, i}_{1}, \ldots, q^{c, i}_{\ell_{c, i}} \}.$
In particular, we have that
\begin{itemize}
\item $Q^{c, i}$ is $\tilde{Q}^{c, i}$-disk embeddable, $|V(Q^{c, i})| ≤ |V(H_i)|,$ and $M'''_{i} = \langle M'_{i}[A^{c}_{M'_{i}}], s^{c, i}_{1}, \ldots s^{c, i}_{\ell_{c. i}} \rangle \oplus \langle Q^{c, i}, q^{c, i}_{1}, \ldots, q^{c, i}_{\ell_{c, i}} \rangle$ contains $\nonplanar(H_{i})$ as a minor and
\item there exists an $\nonplanar(H_{i})$-model $\Xcal'''_{i}$ in $M''_{i}$ such that $A^{c_{i}}_{M'''_{i}}$ is the $\Xcal'''_{i}$-core side of $(A^{c_{i}}_{M'''_{i}}, B^{c_{i}}_{M'''_{i}}),$ where $(A^{c_{i}}_{M'''_{i}}, B^{c_{i}}_{M'''_{i}})$ is the separation of $M'''_{i}$ where $A^{c_{i}}_{M'''_{i}} = A^{c_{i}}_{M'_{i}}.$
\end{itemize}

Moreover, since for every $i \neq j \in I^{c},$ $M'_{i}$ and $M'_{j}$ are disjoint, we have that the graph $\bigcup_{i \in I^{c}} Q^{c, i}$ is $(\bigcup_{i \in I} \tilde{Q}^{c, i})$-disk embeddable.
Now, assume that $I^{c} = \{ i_{1}, \ldots, i_{|I^{c}|}\}.$
We now consider the graph
$$G'' = \langle G'[B^{c}], s^{c, i_{1}}_{1}, \ldots, s^{c, i_{1}}_{\ell_{c, i_{1}}}, \ \ldots \ , s^{c, i_{|I^{c}|}}_{1}, \ldots, s^{c, i_{|I^{c}|}}_{\ell_{c, i_{|I^{c}|}}} \rangle \oplus \langle \bigcup_{i \in I} Q^{c, i}, q^{c, i_{1}}_{1}, \ldots, q^{c, i_{1}}_{\ell_{c, i_{1}}}, \ \ldots \ , q^{c, i_{|I^{c}|}}_{1}, \ldots, q^{c, i_{|I^{c}|}}_{\ell_{c, i_{|I^{c}|}}} \rangle$$
and the $\Sigma$-decomposition $\delta''$ of $G''$ which we obtain from $\delta'$ by replacing $\sigma_{\delta'}(c) - \pi_{\delta'}(\tilde{c})$ and its drawing by a $(\bigcup_{i \in I^{c}} \tilde{Q}^{c, i})$-disk embedding of  $\bigcup_{i \in I^{c}} Q^{c, i}.$
By construction, for every $i \in I^{c},$ $M'''_{i}$ is an $\nonplanar(H_{i})$-host in $G''$ and $(M'''_{i}, T'''_{i})$ is an $\nonplanar(H_{i})$-expansion in $G'',$ where $T'''_{i} \coloneqq T' \cap V(M'''_{i}) \cup V(Q^{c, i}).$
Also, notice that the existence of the $\nonplanar(H_{i})$-model $\Xcal'''_{i}$ in $M'''_{i}$ implies that $M'''_{i}$ remains $c_{i}$-invading (with respect to $\delta''$).
Moreover, every $M'_{i},$ $i \in [k] \setminus I^{c}$ is unaffected in $G''$ and therefore retains all its properties.
Now, let $B' = B \setminus \{ c \}.$

\medskip
It follows that, if we iteratively repeat the process above for every cell in $B',$ we will eventually construct a graph $M^{*}$ and a set $T^{*} \subseteq V(M^{*})$ with the following properties.
We have that
$$|T^*| ≤ |T'| + \sum_{c \in B} |V(\bigcup_{i \in I^{c}} Q^{c, i})| ≤ k\bar{h}(1+ 3h_{\Zcal}).$$
Moreover, by construction, it follows that $(M^*,T^*)$ is an $H$-expansion and that there is a drawing of $M^*$ in $\Sigma^{(0,0)}$ such that $M^* \cap \overline{\Delta}_{\mu}$ is a $\overline{\Delta}_{\mu}$-embedding.
Also, keep in mind that $M^* \cap \inG_{\delta'}({\Delta}_{\mu})$ and $M' \cap \inG_{\delta'}({\Delta}_{\mu})$ are identical, as all modifications were applied to cells that are subsets of $\overline{\Delta}_{\mu}.$

\medskip
For the sake of simplicity, in the rest of the proof, we do not distinguish the embedding or the drawing of a graph in $\overline{\Delta}_{\mu}$ by the graph itself.
Similarly, we treat the drawing of a vertex/edge in $\overline{\Delta}_{\mu}$ as the vertex/edge itself.

We proceed with a final modification of $M^*$ by repetitively ``uncontracting'' every vertex $v$ of $T^* \cap \overline{\Delta}_{\mu}$ of degree at least four to an edge $x_v y_v,$ without harming the embeddability in $\overline{\Delta}_{\mu},$
until all vertices of (the updated) $T^*$ have degree $\leq 3.$
After each such operation we update $T^* := (T^* \setminus \{ v \}) \cup \{x_v,x_y\}.$
As all modifications were applied to the interior of $\overline{\Delta}_{\mu}$ (recall that $U \cap T' = \emptyset$), again we have that $M^* \cap \inG({\Delta}_{\mu})$ and $M' \cap \inG({\Delta}_{\mu})$ are identical also for the new $M^*.$
It is also easy to verify (using standard planarity arguments) that the new $T^*$ is at most three times the size of the old one.
Therefore
$$|T^*| ≤ 3k \bar{h}(1+ 3h_{\Zcal}).$$

Consider the graph $H^+$ created if we dissolve in $M^*$ all vertices not in $T^* \cup U.$
Recall that all vertices of $U$ have degree two in $M^{*},$ so they also have degree two in $H^{+}.$
Clearly $H≤H^+.$

Our next (and final) step is to find some $H^+$-expansion in $G'$ by redrawing $M^*\cap \overline{\Delta}_{\mu}$ inside $G' \cap \Delta_{\xi}.$
Let $Τ^*_{\mu} = T^* \cap \overline{\Delta}_{\mu}.$
Consider the graph $H_{\mu}^{+}$ obtained from $M^* \cap \overline{\Delta}_{\mu}$ if we dissolve all vertices not in $Τ^*_{\mu} \cup U.$
Notice that $H^+_{\mu}$ is a $\overline{\Delta}_{\mu}$-embedded graph on $|Τ^*_{\mu} \cup U|$ vertices where all vertices in $V(H^+_{\mu}) \setminus U$ have degree three and all vertices in $U$ have degree one. Observe that
$$|Τ^*_{\mu} \cup U| ≤ 3k\bar{h}(1 + 3h_{\Zcal}) + f_2(2k\bar{h}).$$

Now, let $\overline{H}$ be the graph obtained by taking the graph $H_{1} \setminus \nonplanar(H_{1}) + \ldots + H_{k} \setminus \nonplanar(H_{k}),$ which is clearly a planar graph, and then repeatedly ``uncontracting'' every vertex of degree at least four until every vertex has degree at most three without harming planarity, as we did before in the case of $M^{*}.$
Clearly, the graph $H + \overline{H} = H_{1} + \ldots + H_{k}.$
Moreover, it must be that $|V(\overline{H})| \leq 3kh_{\Zcal}.$

As the vertices of $U$ are also vertices of the paths in $\Pcal,$ there exists a function $f_3 \colon \Nbbb\to\Nbbb$ that makes it possible to find an $(H^+_{\mu} + \overline{H})$-expansion $(\tilde{M}_{\mu}, \tilde{T}_{\mu} \cup U)$ in the graph
$$\cupall \{C_{\mu+1}, \ldots, C_{\mu + f_3(|\tilde{Τ}_{\mu} \cup U|)} \} \cup \cupall \{P_{1} \cap \overline{\Delta}_{\mu}, \ldots, P_{\xi} \cap \overline{\Delta}_{\mu} \}$$
such that $|\tilde{T}_{\mu}| \leq |T^{*}_{\mu}| + |V(\overline{H})|.$
The construction of $(\tilde{M}_{\mu}, \tilde{Τ}_{\mu} \cup U)$ can be done using standard graph drawing techniques (see e.g., \cite[Lemma 16]{BasteST19Hitting}) by ensuring that $f_{3}(x) = \Ocal(x^2).$
In order to make the redrawing above possible inside $G' \cap \Delta_{\xi},$  and therefore also in $G,$ we must ensure that $\xi \geq f_3(|\tilde{Τ}_{\mu} \cup U|).$
Therefore, we pick $\xi$ so that 
\begin{eqnarray}
\xi &\geq & f_3(3k \bar{h}(1 + 3h_{\Zcal}) + f_2(2k \bar{h}) + 3kh_{\Zcal}).\label{@automobile}
\end{eqnarray}
Notice now that, as $f_{1}(x) = \Ocal((f_{2}(x))^2),$ $f_{2}(x)=2^{\Ocal(x)},$ and $f_{3}(x)=\Ocal(x^2)$ it is possible to choose $\mu$ and $\xi$ so that \eqref{@occasionally}, \eqref{@industrially}, and \eqref{@automobile} are all satisfied in a way that $\xi = 2^{\Ocal(k)\cdot\poly(h_{\Zcal})}.$

Let $U = \{u_{1}, \ldots, u_{q}\}.$ 
Recall that $M^* \cap \Delta_{\mu}$ and $M' \cap \Delta_{\mu}$ are identical and the same holds also for the sets $T^* \cap \Delta_{\mu}$ and $T' \cap \Delta_{\mu}.$
Let 
\begin{eqnarray*}
M^{\bullet} & = & \langle G^* \cap \Delta_{\mu}, u_{1}, \ldots, u_{q} \rangle \oplus \langle \tilde{M}_{\mu}, u_{1}, \ldots, u_{q} \rangle \mbox{ \ and}\\
T^{\bullet} & = & (T^* \cap \Delta_{\mu}) \cup \tilde{Τ}_{\mu} \cup U
\end{eqnarray*}
and observe that $(M^{\bullet},T^{\bullet})$ is an $(H^+ + \overline{H})$-expansion of $G - A.$
As $H ≤ H^+$ and $M^{\bullet}$ is a subgraph of $G-A,$ we can see $M^{\bullet}$ as the disjoint union of $k$ $\Zcal$-hosts in $G - A.$
This yields a mixed $\Zcal$-packing of $G$ of size $k,$ as required.

According to the proof above, the running time of finding a mixed $\Zcal$-packing of size $k$ in $\inG_{\delta}(\Delta)$ is dominated by the computation of the substitute $\Zcal$-expansion $(M^*, T^*),$ where $|T^*| = \Ocal(k\cdot \poly(h_{\Zcal}))$ which is essentially a disjoint paths problem on $\Ocal(k \cdot \poly(h_{\Zcal}))$ pairs of terminals in the planar graph $\hat{G}'$ where  we discard all vertices of $C_{\mu},$ except from those that are also vertices of $\cupall\{P_{1}, \ldots, P_{f_2(|\hat{T}'|)} \}.$
As the planar disjoint paths problem on $q$ pairs of terminals can be solved in time $2^{\Ocal(q)}|V(G)|$ (using the result of Cho, Oh, and Oh in~\cite{cho2023parameterized}), the substitution graph $(M^*, T^*)$ can be computed in $2^{\Ocal(k)\cdot\poly(h_{\Zcal})}|V(G)|$ steps.
\end{proof}

\section{The local structure theorem}
\label{@ilemigoddesscs}

This section deals with the proof of our local structure theorem (see \cref{@antiauthoritarian}).
The proof is split in four subsections.
\cref{@reciprocal} produces a vortex-free version of \cref{@duplicating}.
\cref{@abstractness} finds in the resulting schema, a flat vortex collection containing all $\Zcal$-red cells of the schema.
\cref{@conjecture} further refines this collection.
\cref{@translation} shows how to eliminate all $\Zcal$-red cells, by exploiting the properties of the refined collection, and finally produces the free local structure theorem that is free of $\Zcal$-red cells.

\subsection{Killing vortices in Dyck-walls}
\label{@reciprocal}

The goal of this section is to produce a vortex-free version of \autoref{@duplicating}.
This is made possible by the results of Thilikos and Wiederrecht \cite{ThilikosW22Killingconf} which state that the local structure of graphs which exclude a large shallow-vortex grid is vortex-free.
Since we are only dealing with shallow-vortex minors $H,$ the shallow-vortex grid contains our graph $H$ as a minor.
Moreover, it is easy to see that a large enough such grid even contains a packing of $H.$
Hence, if we ever find a large enough shallow-vortex grid, we may immediately return a packing.

We first introduce the definition of \textsl{tight nests}.

\paragraph{Tight nests.} Let $θ$ be a positive integer. Let $(ρ, G, Ω)$ be a rendition of a nest $\Ccal$ around a cell $c \in C(\rho)$ in a disk $Δ.$ We say that $(ρ, G, Ω)$ is \emph{$θ$-tight} if one of the following is true
\begin{enumerate}
   
  \item there exists a set $Z \subseteq V(G)$ such that $|Z| \leq θ$ and every $V(Ω)\mbox{-}\tilde{c}$-path in $G$ intersects $Z,$ or
  \item there exists a $ρ$-aligned disk $Δ' \subseteq Δ$  with $X = \bd(Δ') \cap N(ρ)$ such that
  \begin{itemize}
   
  \item $Δ'$ contains $c,$
  \item there exists a family $\Pcal$ of pairwise disjoint $V(Ω)\mbox{-}π(X)$-paths with $|\Pcal| = |π(X)| \geq θ,$ and
  \item if $Ω'$ is an ordering of $π_{\rho}(X)$ induced by $Δ',$ then the society $(\inG_{ρ}(Δ'), Ω')$ has depth at most $3θ.$
  \end{itemize}
\end{enumerate}

Let $ρ = (Γ, \mathcal{D}, c_{0})$ be a cylindrical rendition of a society $(G, Ω)$ and let $\Ccal = \langle C_{1}, \ldots, C_{s} \rangle$ as well as $\Ccal' = \langle C'_{1}, \ldots, C'_{s} \rangle$ be two nests around $c_{0}$ in $ρ.$
Let also   $\langle Δ_{C^c_{1}}, \dots,  Δ_{C^c_{s}}\rangle$ and $\langle Δ_{C^c_{1}}', \dots,  Δ_{C^c_{s}}'\rangle$ be the  disk sequences of $\Ccal$ and $\Ccal'$ respectively. 
We associate a vector $\mathbf{v}_{\Ccal}=\langle\mathsf{v}_0,\ldots,\mathsf{v}_{s-1}\rangle \in \mathbb{N}^{s}$ with $\Ccal$ as follows.
For each $i \in [0, s-1]$ let $\mathsf{v}_{i}$ be the number of nodes of $ρ$ which are contained in the disk $Δ_{C^c_{1}}.$
The vector $\mathbf{v}_{\Ccal'}=\langle\mathsf{v}_0',\ldots,\mathsf{v}_{s-1}'\rangle$ is defined analogously.
We now write $\Ccal < \Ccal'$ if there exists some $j\in [s]$\
such that $\mathbf{v}_{j}<\mathbf{v}_{j}'$ and for every $i\in[s]\setminus\{j\},$ $v_{i}≤v_{i}'.$
\smallskip

We extract the following lemma from the proof of Lemma 33 in \cite{thilikos2022killing}.
The statement of the lemma below differs from Lemma 33 in \cite{thilikos2022killing} as follows:
We drop the assumption that the exceptional cell of the rendition of our nest is a vortex of bounded depth.
However, when dropping this assumption now a large transaction that can be found may invade the exceptional cell.
If many of the paths invade the cell we have obtained the third outcome.
Otherwise most of the transaction avoids the exceptional cell and thus the arguments from the proof of Lemma 33 may be applied.
For additional insight, we provide a sketch of the proof below.

\begin{lemma}\label{@pessimistic}
\label{@considered}
Let $G$ be a graph and $Ω$ be the cyclic ordering 
of a subset of vertices of $G.$ 
There exists a function $f_{\ref{@considered}} : \Nbbb^{2} \to \Nbbb$ such that for every $q \in \Nbbb,$ and every $θ \in \Nbbb,$
if $(ρ=(Γ, \mathcal{D}, c), G, Ω)$ is a rendition of a nest $\Ccal$ of order $q$ around $c,$ then
\begin{itemize}
   
\item either $\mathcal{C}$ is $f_{\ref{@considered}}(q, θ)$-tight, or
\item there exists a nest ${\cal C}'$ of order $q$ around $c$ within $G$ where $\mathcal{C}' < \mathcal{C},$ or
\item there is a transaction in $(G,Ω)$ where at least $f_{\ref{@considered}}(q, θ)$ of its paths intersect $V(\inG_{ρ}(Δ_{c}))\setminus π_{ρ}(\tilde{c}).$
\end{itemize}
Moreover, there exists an algorithm that, given $\mathcal{C},$ $(\rho,\mathcal{D},c),$ $q$ and $\theta$ as above, finds one of the three ouctomes above in time $\mathcal{O}(f_{\ref{@considered}}(q,\theta)|V(G)|^2).$
\end{lemma}

\begin{proof}[Proof sketch for \cref{@pessimistic}]
We set $f_{\ref{@considered}}(q, θ)\coloneqq \max\{ 2q,\theta\}.$

The proof first asks for a maximum collection $\mathcal{Q}'$ of pairwise vertex-disjoint $V(\Omega)$-$\pi_{\rho}(\widetilde{c})$-paths in $G$ together with a hitting set $X$ for all paths in $\mathcal{Q}'.$
Since $\rho$ is a rendition of $\mathcal{C}$ in a disk $Δ$ it follows that $X$ defines a closed curve $\zeta$ in $Δ$ that separates $c$ from the boundary.
Moreover, $\zeta$ defines a $\rho$-aligned disk $Δ_{\zeta}$ which contains $c.$
Let $\mathcal{Q}$ be the collection of all $V(\Omega)$-$V(\Omega_{Δ_{\zeta}})$-subpaths of the paths in $\mathcal{Q}'.$

In case $(\Omega_{Δ_{\zeta}},\mathsf{inner}_{\rho}(\Omega_{Δ_{\zeta}}))$ has depth at most $3f_{\ref{@considered}}(q,\theta)$ we are in the first outcome and may terminate.

Otherwise we continue from here and ask for a transaction of order $3f_{\ref{@considered}}(q,\theta)+1$ on $(\Omega_{Δ_{\zeta}},\mathsf{inner}_{\rho}(\Omega_{Δ_{\zeta}})).$
If at least $f_{\ref{@considered}}(q,\theta)$ of these paths contain vertices or edges drawn in the interior of $c$ we have reached the third ourcome of the lemma.
Hence, we may assume that at least $2f_{\ref{@considered}}(q,\theta)+1$ paths avoid the interior of the exceptional cell.
Let $\mathcal{L}\subseteq\mathcal{Q}$ be the collection of these paths.
Notice that we may combine the paths in $\mathcal{L}$ with the paths in $\mathcal{Q}$ to obtain a transaction $\mathcal{R}$ of order at least $2f_{\ref{@considered}}(q,\theta)+1$ on $(G,\Omega)$ which avoids the interior of $c$ completely.
Hence, $\mathcal{R}$ is grounded and a planar transaction.

Let $i\in[q]$ be the smallest integer such that there exists some $V(C_i)$-subpath $R'$ of some path in $\mathcal{R}$ which is drawn completely inside the $\rho$-aligned disk bounded by the trace of $C_i.$
As $|\mathcal{R}|\geq 2f_{\ref{@considered}}(q,\theta)+1$ such an $i$ must exist.
Finally, notice that $R'$ may be used to substitute a subpath of $C_i$ in order to form a new cycle $C$ whose trace defines a $\rho$-aligned disk that still contains $c.$
From here it is straightforward to check that $\langle C_1,\dots,C_{i-1},C,C_{i+1},\dots,C_q\rangle<\mathcal{C}$ which is the desired second outcome of the lemma.
\end{proof}

We will also need the following ``vortex-killing'' lemma that will allow us to remove all of the vortices we obtain as a result of applying \cref{@duplicating} or find a large packing of any shallow-vortex minor we like.
We present here an algorithmic version of the original lemma from \cite{thilikos2022killing,ThilikosW22Killingconf}.
To see a proof with very similar arguments we refer the reader to \cref{@translation}, in particular \cref{@repression}.
However, for the sake of completeness, we include a rough sketch of the proof below.

\begin{lemma}[\!\! \cite{ThilikosW22Killingconf}]\label{@rockefeller} 
Let $t\leq \theta$ be positive integers. 
There exists a positive universal constant $\mathsf{c}$ such that, if $(\rho=(\Gamma,\mathcal{D},c),G,\Omega)$ is a \hyperref[@maidservants]{$\theta$-tight} rendition of a nest $\mathcal{C},$ with $|\mathcal{C}|\geq 12t^2+\mathsf{c},$ around a vortex $c$ of depth at most $\theta$ in a disk $Δ,$ then one of the following holds.
\begin{enumerate}
   
\item There exists a separation $(A,B)$ of order at most $12\theta(t-1)$ such that, if $\Omega'$ is the restriction of $\Omega$ to $A\setminus B$ and $V(\Omega)\cap B\subseteq A\cap B,$ then $(G[A\setminus B],\Omega')$ has a vortex-free rendition in the disk, or 
\item $G$ contains the shallow-vortex grid of order $t$ as a minor.
\end{enumerate}
Moreover, there exists an algorithm that, given a $\theta$-tight rendition $\rho$ of a nest $\mathcal{C}$ and an integer $t,$ finds one of the two outcomes as above in time $\mathcal{O}(t\cdot \theta\cdot|V(G)|^2).$
\end{lemma}

\begin{proof}[Proof sketch of \cref{@rockefeller}]
As $(\rho,G,\Omega)$ is a $\theta$ tight rendition of the nest $\mathcal{C}$ in a disk $Δ,$ we may assume that we are given a $\rho$-aligned disk $Δ'\subseteq Δ$ together with a $V(\Omega)$-$V(\Omega_{Δ'})$-linkage $\mathcal{P}$ of order $|V(\Omega_{Δ'})|.$
We can also compute such a linkage and disk in time $\mathcal{O}(|V(G)|^2)$ by using \cref{@pessimistic}.

Next we fix a linear ordering $\lambda$ of $V(\Omega)$ that agrees with $\Omega.$
That is, the cyclic permutation obtained from $\lambda$ by setting $\lambda^{-1}(|V(\Omega)|+1)\coloneqq\lambda^{-1}(1)$ is equal to $\Omega.$

We then start collecting sets of vertices $I_i$ which form consecutive segments of $\lambda,$ each with the property that, if we remove a minimum $V(I_i)$-$V(\Omega)\setminus I_i$ separator $X_i$ in $\mathsf{inner}_{\rho}(Δ_{\Omega_{Δ'}}),$ there exists a \textsl{cross} on the vertices of $I_i$ in $\mathsf{inner}_{\rho}(Δ_{\Omega_{Δ'}})-X_i.$
By \textsl{cross} we mean two vertex disjoint paths $P_1$ and $P_2,$ each with endpoints $s_j$ and $t_j,$ $j\in[2],$ such that these endpoints appear in $I_i$ in the order $s_1,$ $s_2,$ $t_1,$ $t_2$ as induced by $\lambda,$ and none of the $P_j$ has internal vertices in $I_i.$

Notice that we can find $X_i$ in time $\mathcal{O}(\theta \cdot |V(G)|).$
Moreover, such a cross can be found using the so called Two Paths Theorem (see for example \cite{jung1970verallgemeinerung} or the explanation and a new proof in \cite{kawarabayashi2020quickly}) in time $\mathcal{O}(|V(G)|^2).$

In case we find $t$ such segments, each with their own cross, we can use $\mathcal{P}$ to route these crosses to $V(\Omega)$ and construct the desired shallow-vortex grid using the resulting paths and a part of the nest.

Otherwise, by the minimality of the $I_i$ and the bounded depth of $(\mathsf{inner}_{\rho}(Δ_{\Omega_{Δ'}}),\Omega_{Δ'})$ it suffices to delete at most $12\theta(t-1)$ vertices to remove all crosses on $(\mathsf{inner}_{\rho}(Δ_{\Omega_{Δ'}}),\Omega_{Δ'}).$
A consequence of the Two Paths Theorem then yields the desired vortex-free rendition of what remains.
\end{proof}

In order to link the vortex-free $Σ$-decomposition we obtain from \cref{@duplicating} by applying the lemmas above to remove the vortices, we need a way to encode that the tangle of the Dyck wall which is part of the $Σ$-schema we construct is a truncation of the well-linked set the local structure theorem arises from.
For this, first notice that, given some Dyck wall $D$ of order $d$ and a separation $(A,B)$ of order less than $d,$ we can check in linear time if $(A,B)\in\mathcal{T}_D$ or $(B,A)\in\mathcal{T}_D$ by checking which of the two sides fully contains a cycle and a track of $D.$
During our proofs in the following sections we will show that we can always maintain the following invariant property, given that we start out with a large enough Dyck wall in the beginning.
\medskip

Given a $Σ$-schema $(Α,δ,D)$ of $G,$ a real number $α\in[2/3,1),$ and a set $S\subseteq V(G),$ we say that $(A,δ,D)$ is $α$-\emph{anchored at} $S$ if
\begin{quotation}
  \noindent $\mathcal{T}^A_D$ is the tangle of $G-A$ induced by $D,$ and for every $(B_1,B_2)\in\mathcal{T}^A_D$ it holds that $$|(B_1\cup A)\cap S|<(1-\alpha)|S|.$$
\end{quotation}
Notice that $(A,δ,D)$ being $α$-anchored at a set $S,$ in some sense, resembles the properties of so-called \textit{respectful tangles} from \cite{robertson1994graphXI,robertson1995graphXII}.
That is, all separations of $G$ which involve the apex set $A$ and whose separator contains less than $d$ vertices outside of $A,$ where $d$ is the order of $D,$ inherit their properties from the graph where $A$ is deleted.
In particular, this means that these separations ``respect'' the surface which is represented by $D.$

The following lemma ensures that, if we choose $D$ to be large enough, any carving we apply to our anchored $Σ$-schema will remain anchored.

\begin{lemma}\label{@repetitive}
  Let $d\geq 4$ be an integer, $\alpha\in[2/3,1),$ and $Σ$ be a surface.
  Let $(A,δ,D)$ be a $Σ$-schema of a graph $G$ which is $\alpha$-anchored at a set $S,$ where $D$ is of order $d.$
  Moreover, let $(B_1,B_2)$ be a separation of order less than $d-3$ in $G-A$ and let $(A\cup(B_1\cap B_2),δ',D')$ be a $Σ$-schema of $G$ obtained from a carving $(A\cup(B_1\cap B_2),δ',D'')$ of $(A,δ,D)$ at $(B_1\cup A,A\cup B_2)$ by selecting $D'$ to be a $(Σ;d')$-Dyck subwall of $D''$ where $d'\coloneqq d-|B_1\cap B_2|.$
  Then $(A\cup(B_1\cap B_2),δ',D')$ is $\alpha$-anchored at $S.$
  \end{lemma}
  
  \begin{proof}
  Let $\mathcal{T}^A_D$ be the tangle of $G-A$ which is induced by $D.$
  Let $A'\coloneqq A\cup (B_1\cap B_2)$ and let us denote by $\mathcal{T}'$ the tangle $\mathcal{T}^{A'}_{D'}$ of $G-A'$ which is induced by $D'.$
  Notice that $D'$ contains a cycle $C$ of $D$ and thus, there exist $d$ pairwise vertex-disjoint paths from $D'$ to every cycle and every track of $D.$
  Hence, $\mathcal{T}_{D'}\subseteq \mathcal{T}_D.$
  Moreover, for every separation $(X_1,X_2)\in\mathcal{T}'$ it holds that $(X_1\cup(B_1\cap B_2),(B_1\cap B_2)\cup X_2)\in\mathcal{T}_D^A.$
  Hence, as $(A,δ,D)$ is $\alpha$-anchored at $S,$ it follows that $|(X_1\cup A')\cap S|<(1-\alpha)|S|$ which proves our claim.
  \end{proof}

\paragraph{Finding a clique minor.}

We will also need the following subroutine to ensure that the graphs we hand over to our more complicated subroutine have a linear, in the number of vertices, number of edges.
To do this, we make use of the following seminal result of Robertson and Seymour \cite{robertson1995graph} to find a clique minor in a dense graph.

\begin{proposition}[\!\! \cite{robertson1995graph}]\label{@enlighlermicut}
There exists an algorithm that, for every integer $k\geq 3$ and every graph $G,$ either
\begin{itemize}
  \item determines that $|E(G)|\leq 2^k|V(G)|,$ or
  \item finds a minimal $K_t$-minor model in $G.$
\end{itemize}
Moreover, this algorithm runs in time $2^{\mathcal{O}(k)}|V(G)|^2.$
\end{proposition}

We are now ready for the main statement of this section.

\begin{theorem}\label{@enlightened}
There exist functions $f^1_{\ref{@enlightened}} \colon \Nbbb^5 \to \Nbbb,$ $f^2_{\ref{@enlightened}} \colon \Nbbb^5 \to \Nbbb,$ and an algorithm that, given a $\Zcal \in\Vbbb,$ a graph $G,$ three positive integers $k, t, d$, and an $(s,α)$-well-linked set $S$ of $G,$ where $α \in [2/3,1)$ and $s ≥ f^1_{\ref{@enlightened}}(\gamma_{\Zcal},h_{\Zcal},t,d,k)$ as input, either outputs
\begin{itemize}
   
\item an inflated copy of the Dyck grid $\mathscr{D}_{t}^{Σ'},$ for some $Σ' \in \sobs(\mathbb{S}_{\Zcal})$, or 
\item a mixed $\Zcal$-packing of size $k$,
\end{itemize}
or determines that $|E(G)|\leq 2^{\max\{kh_{\Zcal},4t(\gamma_{\Zcal}+2)\}}|V(G)|,$ and outputs
\begin{itemize}
   
\item a triple $(A, δ, D)$ such that
\begin{itemize}
   
\item  $|A| ≤ f^2_{\ref{@enlightened}}(\gamma_{\Zcal},h_{\Zcal},t, d, k),$ 
\item $(Α,δ,D)$ is a $Σ$-schema of $G,$  for some $Σ \in \mathbb{S}_{\Zcal},$
\item $(A, δ,D)$ is $α$-anchored at $S,$ and
\item  $D$ is a $(Σ;d)$-Dyck wall of $G-A$ that is grounded in $δ.$
\end{itemize}
\end{itemize}

Moreover, this algorithm runs in time $2^{\mathsf{poly}(2^{\mathsf{poly(t\cdot 2^{\mathcal{O}(\gamma_{\Zcal})})}}d\cdot k\cdot h_{\Zcal})}|V(G)|^3\log(|V(G)|),$ and the functions $f^1_{\ref{@enlightened}}$ and $f^2_{\ref{@enlightened}}$ are of the following orders:
\begin{align*}
  f^1_{\ref{@enlightened}}(\gamma_{\Zcal},h_{\Zcal},t,d,k) &= 2^{\mathsf{poly}(t\cdot 2^{\mathcal{O}(\gamma_{\Zcal})})}d^{20}k^{40}\mathsf{poly}(h_{\Zcal})\text{, and}\\
  f^2_{\ref{@enlightened}}(\gamma_{\Zcal},h_{\Zcal},t,d,k) &= 2^{\mathsf{poly}(t\cdot 2^{\mathcal{O}(\gamma_{\Zcal})})}dk^3\mathsf{poly}(h_{\Zcal}).
\end{align*}
\end{theorem}

\begin{proof}
We begin by calling \cref{@enlighlermicut} to check whether $G$ contains a clique minor of size $\max\{ kh_{\Zcal},2t(\gamma_{\Zcal}+2)\}.$
If we find such a clique minor, then we have also found $\mathcal{D}_t^{Σ'}$ for any $Σ'\in\mathsf{sobs}(\mathbb{S}_{\Zcal})$ and $kH$ for every $H\in\Zcal $ as a minor in $G$ and can output either of those.
Otherwise, the algorithm will conclude that $|E(G)|\leq 2^{\max\{kh_{\Zcal},4t(\gamma_{\Zcal}+2)\}}|V(G)|.$
Moreover, this algorithm runs in time $2^{\mathcal{O}(\max\{kh_{\Zcal},4t(\gamma_{\Zcal}+2)\})}|V(G)|^2.$

Hence, from now on we may assume that $G$ only has a linear number of edges with respect to $\Zcal $ and $t.$

Next we discuss the functions $f^1_{\ref{@enlightened}}$ and $f^2_{\ref{@enlightened}}.$

Our goal is to find a $k$-packing of some shallow-vortex minor $H\in\Zcal $ whenever we find a shallow-vortex grid of large enough order as a minor.
Let $h^*$ be the minimum order of a shallow-vortex grid to host every shallow-vortex minor on $h_{\Zcal}$ vertices as a minor.
By \cref{@childishness} we need a shallow-vortex grid of order $3kh^*$ to find $k$ pairwise disjoint shallow-vortex grids of order $h^*.$
By using standard graph drawing techniques (as one can see, for example, in Lemma 16 of \cite{BasteST19Hitting}) one can observe that there exists a universal constant $c_1$ such that $h^*\leq c_1h_{\Zcal}^2.$
it follows that we are satisfied whenever we find a shallow-vortex grid of order
$$a\coloneqq 3c_1kh_{\Zcal}^2.$$

Notice that, in order to apply \cref{@rockefeller} and ensuring that we find our shallow-vortex grid of order $a$ we must require that the nests of our vortices have order at least $12a^2+c_2$ where $c_2$ is the constant from \cref{@rockefeller}.
Hence, by including the requirements from \cref{@duplicating} we obtain the following requirement on the size of the nests of our vortices:
$$b\coloneqq 2^{\mathcal{O}(\gamma_{\Zcal})}t+12\cdot9c_1^2k^2h_{\Zcal}^4 +c_2.$$
Notice that \cref{@duplicating} requires as input a wall of order at least
$$2^{\mathsf{poly}(t\cdot 2^{\mathcal{O}(\gamma_{\Zcal})})}\cdot b+ \mathsf{poly}(t\cdot 2^{\mathcal{O}(\gamma_{\Zcal})})r$$
to guarantee a nest of size $b.$
However, we will choose a much larger wall to also be able to accommodate the Dyck wall of order $d$ as required and to ensure that we guarantee that the $Σ$-schema we are constructing will be $\alpha$-anchored at $S.$

To be exact, by choosing $c_3$ to be the first constant from \cref{thm_algogrid}, that we may set $f^1_{\ref{@enlightened}}$ to be as follows.
\begin{align*}
  f^1_{\ref{@enlightened}}(\gamma_{\Zcal},h_{\Zcal},t,d,k)&\coloneqq 36c_3(2^{\mathsf{poly}(t\cdot 2^{\mathcal{O}(\gamma_{\Zcal})})}\cdot(2^{\mathcal{O}(\gamma_{\Zcal})}(t+d)+12\cdot 9c_1^2k^2h_{\Zcal}^4+c_2))^{20}+3+ \mathsf{poly}(t\cdot 2^{\mathcal{O}(\gamma_{\Zcal})})r\\
  & = 2^{\mathsf{poly}(t\cdot 2^{\mathcal{O}(\gamma_{\Zcal})})}d^{20}k^{40}\mathsf{poly}(h_{\Zcal})+ \mathsf{poly}(t\cdot 2^{\mathcal{O}(\gamma_{\Zcal})})r.
\end{align*}
We will determine the value for $r$ later.

We now prepare for calling the algorithm from \cref{@duplicating}.
To do this, let us fix the necessary numbers as follows (to not confuse the letters used until now but to ensure better readability, we take the letters from \cref{@duplicating} and augment them with a ``\ $'$\ '')
\begin{align*}
t' &\coloneqq t,\\
d' &\coloneqq d,\\
q' &\coloneqq 2^{\mathcal{O}(\gamma_{\Zcal})}(t+d)+12\cdot9c_1^2k^2h_{\Zcal}^4 +c_2\text{, and}\\
r' &\coloneqq r\coloneqq f^3_{\ref{@duplicating}}(\gamma_{\Zcal},t',q')+f^4_{\ref{@duplicating}}(\gamma_{\Zcal},t',q')\cdot f^5_{\ref{@duplicating}}(\gamma_{\Zcal},t',q')\cdot a+d+1\\
& = 2^{\mathsf{poly}(t\cdot 2^{\mathcal{O}(\gamma_{\Zcal})})}dk^3\mathsf{poly}(h_{\Zcal})
\end{align*}
It follows from our choice of $r$ that indeed
\begin{align*}
  f^1_{\ref{@enlightened}}(\gamma_{\Zcal},h_{\Zcal},t,d,k) = 2^{\mathsf{poly}(t\cdot 2^{\mathcal{O}(\gamma_{\Zcal})})}d^{20}k^{40}\mathsf{poly}(h_{\Zcal})
\end{align*}
as required.

This ensures that \cref{thm_algogrid} will find a wall $W$ of the required size, such that $\mathcal{T}_W$ is a truncation of $\mathcal{T}_S,$ in time $$ 2^{\mathsf{poly}(2^{\mathsf{poly(t\cdot 2^{\gamma_{\Zcal}})}}d\cdot k\cdot h_{\Zcal})}|V(G)|^3\log(|V(G)|). $$
Notice that we may replace the dependency of the running time of the algorithm from \cref{thm_algogrid} on $|E(G)|$ by $|V(G)|$ since we have determined, for this stage, that $G$ does not contain a large clique minor.
\medskip

With the wall $W$ at hand, we may now forward everything to the algorithm of \cref{@duplicating} together with the integers $t',$ $d',$ $q',$ and $r'.$
\cref{@duplicating} has a two different possible outcomes.
\begin{itemize}
  \item it might produce an inflated copy of the Dyck grid $D^{Σ'}_{t'}$ for some $Σ'\in\mathsf{sobs}(\mathbb{S}_{\Zcal})$ in which case we are done, or
  \item it outputs 
  \begin{itemize}
    \item an apex set $A_0$ of size at most $f^3_{\ref{@duplicating}}(\gamma_{\Zcal},t',q'),$
    \item a $Σ$-decomposition $δ_0$ of $G-A_0$ for some $Σ\in\mathbb{S}_{\Zcal}$ of depth at most $f^5_{\ref{@duplicating}}(\gamma_{\Zcal},t',q'),$ and breadth at most $f^4_{\ref{@duplicating}}(\gamma_{\Zcal},t')$ such that every vortex has a private nest of order $q',$
    \item a $(Σ;d')$-Dyck wall $D$ which is disjoint from the interiors of the societies of the nests of all vortices, grounded in $δ_0,$ and for which $\mathcal{T}_D$ is a truncation of $\mathcal{T}_W,$ and
    \item an $r$-subwall $W'\subseteq W$ which is vertex-disjoint from the interiors of the societies of the nests of the vortices, grounded in $δ_0,$ vertex-disjoint from $D,$ and for which $\mathcal{T}_{W'}$ is a truncation of $\mathcal{T}_W.$
  \end{itemize}
\end{itemize}
All of this happens in time $q'^22^{\mathsf{poly(t\cdot 2^{\mathcal{O}(\gamma_{\Zcal})})}}|V(G)|^2.$
By our choice of $q'$ we may now apply, for each vortex withing the cylindrical rendition of its nest, first at most $|V(G)|$ times the algorithm from \cref{@pessimistic} to produce a tight nest.
In total this takes at most $\mathcal{O}(|V(G)|^3)$ time (surpressing the parameters).
We may than call \cref{@rockefeller} to either find a shallow-vortex grid of order $a$ as a minor, which allows us to output a minor model of $kH$ for some shallow-vortex minor $H\in\Zcal $ by choice of $a,$ or we move at most $f^5_{\ref{@duplicating}}(\gamma_{\Zcal},t',q')\cdot a$ many vertices to the apex set in order to fully remove the vortex from $δ_0.$
As this happens at most $f^4_{\ref{@duplicating}}(\gamma_{\Zcal},t',q')$ many times, if we do not terminate, we end up with a new apex set $A_1$ of size at most
\begin{align*}
& |A_0|+f^4_{\ref{@duplicating}}(\gamma_{\Zcal},t',q')\cdot f^5_{\ref{@duplicating}}(\gamma_{\Zcal},t',q')\cdot a\\
 <~&r\\
=~ &2^{\mathsf{poly}(t\cdot 2^{\mathcal{O}(\gamma_{\Zcal})})}dk^3\mathsf{poly}(h_{\Zcal}),
\end{align*}
and a vortex-free $Σ$-decomposition $δ_1$ which is a carving of $δ_0,$ and in which $D$ still is grounded.
We set
\begin{align*}
  f^2_{\ref{@enlightened}}(\gamma_{\Zcal},h_{\Zcal},t,d,k) &\coloneqq 2^{\mathsf{poly}(t\cdot 2^{\mathcal{O}(\gamma_{\Zcal})})}dk^3\mathsf{poly}(h_{\Zcal}).
\end{align*}

To complete the proof we have to show that $(A_1,δ_1,D)$ is $\alpha$-anchored at $S.$
First observe that $r>d$ and, since both $\mathcal{T}_D$ and $\mathcal{T}_{W'}$ are truncations of $\mathcal{T}_W,$ it follows that $\mathcal{T}_D$ is a truncation of $\mathcal{T}_{W'}.$

Now let $\mathcal{T}_D^{A_1}$ be the tangle of $G-A_1$ induced by $D$ and let $(B_1,B_2)\in\mathcal{T}^{A_1}_D$ be any separation.
Since $r>|A_1|+d$ by our choice of $r',$ it follows that $(B_1\cup A_1,A_1\cup B_2)\in\mathcal{T}_{W'}.$
As $\mathcal{T}_{W'}\subseteq\mathcal{T}_S$ we obtain that $|(B_1\cup A_1)\cap S|<(1-\alpha)|S|$ from the definition of $\mathcal{T}_S.$
Since our choice of $(B_1,B_2)$ was arbitrary, this completes the proof.
\end{proof}

\subsection{Finding a flat vortex collection}
\label{@abstractness}

In this subsection, we show how given a $Σ$-schema $(A, δ, D)$ of a graph $G$ in a surface $\Sigma,$ we can group all $\Zcal$-red cells of $\delta$ in a bounded number of areas, each with their own sufficiently large enough railed nest around, that will be used later in the proof to simulate models intersecting the $\Zcal$-red cells, within the area bounded by the railed nest.
This is made possible by exploiting the large enough Dyck-wall $D$ that is grounded in $\delta,$ guaranteed to be contained in $G - A,$ by \cref{@enlightened}.

\paragraph{Red bricks.} We first demonstrate how from $\Zcal$-red cells of $δ,$ we define $\Zcal$-red bricks of $D.$
By definition, all cycles of $D$ are grounded in $δ.$ 
Let $C$ be a cycle of $D.$
We call $C$ \emph{contractible} if the trace of $C$ bounds a disk in $Σ.$
We denote this disk by $Δ_{C}.$ 
We define the \emph{influence} of $C,$ denoted by $\mathsf{influence}_{δ}(C)$ as the set of cells of $δ$ that contain at least one edge of $C$ and those that are subsets of $Δ_{C}.$ 
We say that a contractible cycle $C$ of $D$ is $\Zcal$-\emph{red} if at least one of the cells in $\mathsf{influence}_{δ}(C)$ is $\Zcal$-red.

\paragraph{Layers of a wall and central walls/bricks.} We proceed with the definition of central walls/bricks, a concept central to the proofs of this subsection.
Let $r \in \Nbbb^{\mathsf{even}}.$ The \emph{layers} of an $r$-wall $W$ are recursively defined as follows.
The first layer of $W$ is its perimeter.
For $i = 2, \ldots, r/2,$ the $i$-th layer of $W$ is the $(i-1)$-th layer of the subwall $W'$ obtained from $W$ after removing from $W$ its perimeter and removing recursively all occurring vertices of degree one.
Note that the $r/2$-th layer is a brick.
We call this brick the \emph{central brick} of $W.$

Let $W$ be an $r$-subwall of $D.$
Given $\ell \in \Nbbb,$ we say that $W$ is \emph{$\ell$-central} if there is a $2(\ell + 2) + r$-wall $W'$ 
that is a subwall of $D$ such that the perimeter of $W$ is the $\ell + 2$-th layer of $W'.$
In the case where $B$ is a brick of $D,$ we say that $B$ is $\ell$-central if $B$ is the central brick of a $2(\ell + 2)$-wall $W''.$
We call ($W'$) (resp. $W''$) the \emph{$\ell$-region} of $W$ (resp. $B$).

\medskip
For the proofs that follow, we moreover require the following folklore result.

\begin{proposition}\label{@instructive} 
There exists an algorithm that, given $k \in \Nbbb$ and a graph $G$ such that $Δ(G) \leq d$ for some non-negative integer $d,$ outputs either an independent set of size at least $k$ or a dominating set of size at most $k-1,$ in which case $|V(G)| \leq (k-1)d.$
Moreover the previous algorithm runs in time $\Ocal(kd).$
\end{proposition}

The following lemma shows that in a sufficiently large Dyck wall, we either find $k$ distinct $\ell$-central bricks where the areas around them defined by their $\ell$-regions are pairwise disjoint, or we get an upper bound on the total number of these bricks.

\begin{lemma}\label{@institutions} 
There exist functions $f^{1}_{\ref{@institutions}} : \Nbbb^{2} \to \Nbbb,$ $f^{2}_{\ref{@institutions}} : \Nbbb^{2} \to \Nbbb$ and an algorithm that, given
\begin{itemize}
\item $k, \ell \in \Nbbb,$
\item a $(Σ; q)$-Dyck wall $D,$ where $d \geq f^{1}_{\ref{@institutions}}(k, \ell),$ and
\item a set $\Bcal$ of $\ell$-central bricks of $D,$
\end{itemize}
outputs one of the following:
\begin{itemize}
 \item a set $\{ B_{1}, \ldots, B_{k} \} \subseteq \Bcal,$ where for $i \neq j \in [k],$ $B_{i}$ and $B_{j}$ have disjoint $\ell$-regions, or
 \item that $|\Bcal| \leq f^{2}_{\ref{@institutions}}(k, \ell).$
\end{itemize} 
Moreover the previous algorithm runs in time $\poly(k,\ell)$ and the functions $f^{1}_{\ref{@institutions}}$ and $f^{2}_{\ref{@institutions}}$ are of the following orders:
\begin{align*}
	&f^{1}_{\ref{@institutions}}(k, \ell) = \Ocal(k  \ell)\text{, and}\\
	&f^{2}_{\ref{@institutions}}(k, \ell) = \Ocal(k  \ell^{2}).
\end{align*}
\end{lemma}
\begin{proof} We define $f^{1}_{\ref{@institutions}}(k, \ell) = k(2(2(\ell + 2) + 1))$ and $f^{2}_{\ref{@institutions}}(k, \ell) = (k - 1)(2(2(\ell + 2) + 1) - 1)^{2}.$

From $D$ we define an auxiliary graph $G$ as follows. 
Each vertex $x$ of $G$ corresponds to a brick $B_{x} \in \Bcal.$ 
Moreover, for two vertices $x$ and $y$ of $G,$ $x$ and $y$ are adjacent if and only if the $\ell$-regions of $B_{x}$ and $B_{y}$ intersect. 
We claim that $Δ(G) \leq r,$ for some $r \in \Nbbb.$ 

Indeed, let $x$ be a vertex of $G$ and $B_{x}$ be the corresponding brick in $\Bcal.$
Let $W_{x}$ be the $2(\ell + 2)$-wall in $D$ that is the $\ell$-region of $B_{x}.$
Moreover, assuming that $W_{x}$ is also $\ell$-central, let $W_{x}'$ be the $2(\ell + 2) + 2(\ell + 2)$-wall in $D$ that is the $\ell$-region of $W_{x}.$
Observe then, that for any brick $B$ of $D,$ that is $\ell$-central and does not intersect $W_{x}',$ $W_{x}$ and the $\ell$-region of $B$ are disjoint.
Since, we require that $B$ does not intersect $W_{x}',$ we need to remove an additional layer outside of $W_{x}'.$
Then, by definition of $G,$ $x$ is adjacent only to vertices in $G$ that correspond to bricks of the $2(\ell + 2) + 2(\ell + 2) + 2$-wall that has $B_{x}$ as its central brick.
Since the number of bricks of this wall are $(2(2(\ell + 2) + 1) - 1)^{2},$ we conclude that $r = (2(2(\ell + 2) + 1) - 1)^{2}.$

We continue by applying \cref{@instructive} on $G.$
Assume that $\{ x_{1}, \ldots, x_{k} \}$ is an independent set of $G.$ 
Then, by definition of $G,$ for the set of bricks $\{ B_{x_{1}}, \ldots, B_{x_{k}} \}$ we have that for $i \neq j \in [k],$ $B_{x_{i}}$ and $B_{x_{j}}$ have disjoint $\ell$-regions.
Otherwise, $|V(G)| \leq (k - 1) \cdot r,$ which implies that $|\Bcal| \leq (k - 1)(2(2(\ell + 2) + 1) - 1)^{2}$.
\end{proof}

\paragraph{Simple (resp. Exceptional) layers of a Dyck-wall.} We continue with the definition of the simple (resp. exceptional) layers of a Dyck-wall, that will be of use in the proof of \cref{@custodians}.
Let $D$ be a $(Σ; q)$-Dyck wall, for some surface $Σ.$ 
Observe that in the graph $D - \exceptional(D),$ after the removal of all vertices of degree one, there exists a unique cycle that does not bound a face in $D.$
In a slight abuse of terminology, we consider this cycle to be the exceptional face of that graph.
We define the \emph{exceptional layers} of $D$ as follows. 
The first exceptional layer of $D$ is $\exceptional(D).$
For $i \in [2, q],$ the $i$-th exceptional layer of $D$ is defined recursively as the $(i - 1)$-th exceptional layer of the graph obtained from $D$ by removing $\exceptional(D)$ and afterwards, removing all vertices of degree one.
Observe that the maximum cycle of $D$ that the $i$-th exceptional layer intersects is $C_{i}.$
Moreover, we define the \emph{simple layers} of $D$ as follows.
For $i \in [1, q]$ the $i$-th simple layer of $D$ corresponds to the cycle $C_{q+1 - i}.$

\medskip
The following lemma shows that in sufficiently large Dyck-wall with a set of marked bricks $\Bcal,$ we either find $k$ distinct $\ell$-central bricks of $\Bcal$ where the areas around them defined by their $\ell$-regions are pairwise disjoint, or we find a collection of at most $k-1$ $\ell$-central subwalls, $\ell$ consecutive simple layers and $\ell$ consecutive exceptional layers, with sufficiently large areas around them that are pairwise disjoint, which collection contains all bricks of $\Bcal.$

\begin{lemma}\label{@custodians}
There exists a function $f_{\ref{@custodians}} : \Nbbb^{2} \to \Nbbb$ and an algorithm than, given
\begin{itemize}
\item $k, \ell \in \Nbbb,$
\item a $(Σ; d)$-Dyck wall $D,$ where $d \geq f_{\ref{@custodians}}(k, \ell),$ and
\item a set of bricks $\Bcal$ of $D,$
\end{itemize}
outputs one of the following:
\begin{itemize}
\item a set $\{ B_{1}, \ldots, B_{k} \} \subseteq \Bcal$ of $\ell$-central bricks of $D,$ where for $i \neq j \in [k],$ $B_{i}$ and $B_{j}$ have disjoint $\ell$-regions, or
\item a collection $\mathcal{W} = \{ W_{i} \mid i \in [k'] \},$ $k' \leq k - 1$ of $\ell$-central walls of $D,$ 
a collection $\mathcal{S} = \{ S_{1}, \ldots, S_{\ell} \}$ of $\ell$ consecutive simple layers, and
a collection $\mathcal{E} = \{ E_{1}, \ldots, E_{\ell} \}$ of $\ell$ consecutive exceptional layers, such that,
\begin{enumerate}
\item the disks bounded by the perimeters of the $\ell$-regions of all walls in $\mathcal{W},$ $Δ_{S_{\ell}}$ and $Δ_{E_{\ell}}$ are pairwise disjoint, and
\item for every brick $B \in \Bcal,$ either there is $i \in [k']$ such that $B$ is a brick of $W_{i},$ or the vertices of $B$ are drawn in the interior of $Δ_{S_{\ell}}$ or in the interior of $Δ_{E_{\ell}}.$
\end{enumerate}
\end{itemize}
Moreover the previous algorithm runs in time $\poly(k,\ell)$ and the function $f_{\ref{@custodians}}$ is of order $f_{\ref{@custodians}}(k, \ell) = \Ocal(k^{2}\ell^{5}).$
\end{lemma}
\begin{proof} We define $f_{\ref{@custodians}}(k , \ell) = 2(k - 1)(\ell + 2)f^{2}_{\ref{@institutions}}(k, \ell) + 3(z + w)$, where $z = f^{2}_{\ref{@institutions}}(k, \ell) \cdot 2(\ell + 2) + \ell$ and $w = (\ell + 2) + (f^{2}_{\ref{@institutions}}(k, \ell) \cdot z + 1)$.

We begin by applying \cref{@institutions}. There are two outcomes. 
Either we get the desired set $\{ B_{1}, \ldots, B_{k} \} \subseteq \Bcal,$ in which case we are done, 
or we have that $|\Bcal'| \leq f^{2}_{\ref{@institutions}}(k, \ell).$
Given a non-negative integer $w \geq \ell,$
let $r$ be the smallest non-negative even integer such that there are no bricks of $\Bcal$ intersecting any of the simple and exceptional layers in the range $[r, r + w].$
Notice that, the layers of a $2(\ell + 2)$-wall in $D$ intersect at most $\ell + 2$ simple and/or exceptional layers.
Let $\Bcal'$ consist of the $\ell$-central bricks of $D$ in $\Bcal.$
Then, any brick of $\Bcal'$ is in the subgraph obtained from $D$ after the removal of the first $\ell + 2$ simple and exceptional layers.
This fact combined with the fact that $|\Bcal'| \leq f^{2}_{\ref{@institutions}}(k, \ell),$ 
implies that $r$ is at most $(\ell + 2) + (f^{2}_{\ref{@institutions}}(k, \ell) \cdot w + 1).$
We set $w \coloneqq f^{2}_{\ref{@institutions}}(k, \ell) \cdot 2(\ell + 2) + \ell.$

We define $\Bcal_{r} \subseteq \Bcal'$ as follows. $B \in \Bcal_{r}$ if $B$ does not intersect the first $r$ simple and exceptional layers of $D.$
Now, let $\Scal = \{ S_{1}, \ldots, S_{\ell} \}$ (resp. $\Ecal = \{ E_{1}, \ldots, E_{\ell} \}$) contain the first $\ell$ simple (resp. exceptional) layers in the range $[r, r + w].$
Notice that by definition of $r,$ the vertices of any brick in $\Bcal \setminus \Bcal_{r}$ are drawn either in the interior of the disk bounded by $S_{1}$ or by $E_{1}.$
We proceed with the description of an iterative scheme that gives us the desired set $\Wcal.$ 

\medskip
Set $\Wcal_{0}$ to be $\Bcal_{r}.$
By definition of an $\ell$-central brick, the $\ell$-regions of any brick in $\Wcal_{0}$ is even.
Set $i \coloneqq 1.$
As long as there exists an $x$-wall $W \in \Wcal_{i - 1}$ and a $y$-wall $W' \in \Wcal_{i-1}$, where both $x$ and $y$ are even,
such that $W$ and $W'$ intersect, do as follows.
Consider $W''$ to be the smallest even subwall in $D$ that contains both the $\ell$-region of $W$ and the $\ell$-region of $W',$ if it exists.
Notice that the size of $W''$ is at most $x + y.$
Moreover, observe that $W''$ is the $\ell$-region of $W''',$ where $W'''$ is a subwall of $D$ such that the perimeter of $W'''$ is the $\ell/2$ layer of $W''.$
Then, by definition it must be that $W'''$ contains all bricks of $\Bcal_{r}$ contained in $W$ and $W'.$
Set $\Wcal_{i} = (\Wcal_{i-1} \setminus \{ W, W' \}) \cup \{ W''' \},$ and $i \coloneqq i + 1.$
If there is no such pair $W$ and $W',$ we set $\Wcal \coloneqq \Wcal_{i - 1}.$

\medskip
Notice that the above scheme can be iterated at most $|\Bcal'| \leq f^{2}_{\ref{@institutions}}(k, \ell)$ times.
Hence the size of the $\ell$-region of any wall in $\Wcal$ is at most $f^{2}_{\ref{@institutions}}(k, \ell) \cdot 2(\ell + 2).$
Then, by definition of $\Bcal_{r},$ we have that in any iteration, $W''$ always exists.
Moreover, notice that, any brick in $\Bcal_{r}$ is a brick of some wall in $\Wcal,$ and that all walls in $\Wcal$ contain at least one brick of $\Bcal_{r}.$
Then, since by definition of $\Wcal,$ the $\ell$-regions of all walls in $\Wcal$ are disjoint, it holds that $|\Wcal| \leq k - 1.$
Finally, observe that by the choice of $w,$ the disks bounded by the perimeters of the $\ell$-regions of all walls in $\Wcal$ are disjoint from the disks bounded by $S_{\ell}$ and $E_{\ell}.$

For our previous arguments to work we have to guarantee that $d$ is large enough. 
Observe that for the $k-1$ disjoint $\ell$-regions of the walls in $\Wcal$ to be contained in $D$, it must be that $d \geq 2(k - 1)(\ell + 2)f^{2}_{\ref{@institutions}}(k, \ell).$
Moreover, by definition of $r,$ to ensure the existence of the $r + w$ simple (resp. exceptional) layers of $D$ we require, 
it must be that $d \geq r + w$ (resp. $d \geq 2(r+w)$ (since in $D - \exceptional(D),$ the maximum order of a Dyck subwall is reduced by two.)
Observe that by the assumptions on $d$ and the definition of $f_{\ref{@custodians}}$ this is achieved.
\end{proof}

\paragraph{Railed flat vortices.}
We continue with the definition of railed flat vortices, that aims to formalize the idea of the areas with a large railed nest around them that we briefly introduced earlier in the subsection.

\medskip
Let $\Delta = (\Gamma, \Dcal)$ be a vortex-free $\Sigma$-decomposition of a graph $G$ in a surface $\Sigma.$
Let $\Ccal = \langle C_{1}, \ldots, C_{q} \rangle$ be a nest of order $q,$ for some non-negative integer $q,$ around an arcwise connected set in $\Sigma$ and $\langle \Delta_{1}, \ldots, \Delta_{q} \rangle$ be the disk sequence of $\Ccal$.
Consider the society $(H \coloneqq G \cap \Delta_{q}, \Omega \coloneqq \Omega_{\Delta_{q}})$ and let $\rho = (\Gamma_{\rho}, \Dcal_{\rho}, c_{\rho})$ be a cylindrical rendition of $(H, \Omega)$ around a cell $c_{\rho}$ in $\Delta_{q}$ such that
\begin{itemize}
\item $\Gamma_{\rho} = \Gamma \cap \Delta_{q}$,
\item $\Dcal_{\rho}$ contains $c_{\rho}$ and all cells of $\delta$ contained in $\Delta_{q}$ that are not contained in $\Delta_{1}$, and
\item $c_{\rho}$ is the disk $\Delta_{1}$ minus the nodes of $\rho$ on the boundary of $\Delta_{1}$.
\end{itemize}

Additionally, consider a set of $q$ many paths $\Pcal$ in $G$ such that the pair $(\Ccal, \Pcal)$ forms a railed nest of order $q$ around $\Delta_{0}$ in $\delta.$
We call the pair $(Δ_{q}, ρ)$ a \emph{railed flat vortex} in $δ$ equipped with the railed nest $(\Ccal, \Pcal)$ of order $q \in \Nbbb.$
Additionally, we refer to the society $(H, \Omega)$ as the \emph{flat vortex society} of $(Δ_{q}, ρ)$ in $\delta$.
Moreover, given a $\Zcal \in \Kbbb^{-},$ we call $(Δ_{q}, ρ)$ $\Zcal$-\emph{progressive} if $H$ contains a $\Zcal$-host.

\medskip
A \emph{railed flat vortex collection} of $δ$ is a collection $\mathbf{V} = \{ (Δ_{i}, ρ_{i}) \mid i \in [k] \}$ where, for every $i \neq j \in [k],$ $Δ_{i}$ and $Δ_{j}$ are disjoint, and
for every $i \in [k],$ $(Δ_{i}, ρ_{i})$ is a railed flat vortex in $δ$ equipped with the railed nest $(\Ccal_{i}, \Pcal_{i})$ of order $q \in \Nbbb.$
We refer to $q$ as the \emph{order} of $\mathbf{V}.$
Also we refer to the collection of pairs $(\Ccal_{i}, \Pcal_{i}),$ $i \in [k],$ as \emph{the railed nests} of $\mathbf{V}.$
Given a $\Zcal \in \Kbbb^{-},$ we say that $\mathbf{V}$ is $\Zcal$-\emph{progressive} if
\begin{itemize}
   
\item for every $i \in [k],$ there is a $\Zcal$-red cell of $δ$ that is a subset of $c_{ρ_{i}},$ 
\item all $\Zcal$-red cells of $δ$ are subsets of $\bigcup_{i \in [k]} c_{ρ_{i}},$ 
\item for every $i \in [k],$ $(Δ_{i}, ρ_{i})$ is $\Zcal$-progressive.
\end{itemize}

\begin{observation}\label{@sprachphilosophischen}
Let $δ$ be a vortex-free $Σ$-decomposition of a graph $G$ in a surface $\Sigma$. 
Let $\Zcal \in \Kbbb^{-}$ and $\mathbf{V}$ be a $\Zcal$-progressive railed flat vortex collection in $δ.$ 
Then $G$ contains a mixed $\Zcal$-packing of size $|\mathbf{V}|.$
\end{observation}

By definition of railed flat vortices and the redrawing lemma, it follows almost immediately that the presence of a railed flat vortex of large enough order in the $\Sigma$-decomposition of some graph whose cell in the middle contains a $\Zcal$-red cell implies the existence of a ``local'' model of a graph in $\Zcal$ drawn within the flat vortex.
This is formalized in the next lemma.

\begin{lemma}\label{flat_vortex_contains_host} Let $\delta$ be a vortex-free $\Sigma$-decomposition of a graph $G$ in a surface $\Sigma$.
Let $\Zcal \in \Kbbb^{-}$ and $(\Delta, \rho = (\Gamma_{\rho}, \Dcal_{\rho}, c_{\rho}))$ be a railed flat vortex of order at least $f_{\ref{@horkheimer}}(1, h_{\Zcal})$ such that there exists a $\Zcal$-red cell in $\delta$ that is a subset of $c_{\rho}$.
Then $(\Delta, \rho)$ is $\Zcal$-progressive.
\end{lemma}
\begin{proof} Let $\Delta_{c_{\rho}}$ be the disk that corresponds to $c_{\rho}$, $c \in C(\delta)$ be a $\Zcal$-red cell that is a subset of $c_{\rho}$, and $(\Ccal, \Pcal)$ be the railed nest equipped to $(\Delta, \rho)$ that is of order at least $f_{\ref{@horkheimer}}(1, h_{\Zcal})$.
Since $c$ is a $\Zcal$-red cell in $\delta$ there exists a $c$-invading $\nonplanar(H)$-inflated copy $M$ in $G,$ for some $H \in \Zcal.$
Now, to prove that $(\Delta, \rho)$ is $\Zcal$-progressive it suffices to apply \cref{@horkheimer} on $\Zcal$, $\delta$, $(\Ccal, \Pcal)$, with $\Delta^{*}$ being $\Delta_{c_{\rho}}$, $c$, with $\Delta$ being $\Delta$, and $M$.
\end{proof}

For an $\ell$-central brick $B,$ given $i \in [\ell]$ we define the $δ$-aligned disk $Δ_{B, i}$ as the disk that bounds the trace of the $\ell+2-i$-th layer of the $\ell$-region of $B.$
Notice that, for every $c \in \mathsf{influence}_{δ}(B)$ we have that $c \subseteq Δ_{B,1},$ that $Δ_{B, 1} \subseteq \ldots \subseteq Δ_{B, \ell},$ and that $\bd(Δ_{B, \ell})$ is disjoint from $\simple(D)$ and $\exceptional(D)$.
We also define these disks in the case of an $\ell$-central wall.

\begin{lemma}\label{@misfortune} Let $q \geq 2f_{\ref{@horkheimer}}(1, h_{\Zcal})$ be an even integer, $(A, \delta, D)$ be a $\Sigma$-schema of a graph $G$ in a surface $\Sigma,$ $\Zcal \in \Kbbb^{-},$ and $W$ be a $q$-central wall of $D.$
Then there exists a railed flat vortex $(\Delta, \rho)$ in $\delta$ of order $q/2$ such that $\Delta$ is the disk that bounds the perimeter of the $q$-region of $W.$
Moreover, if the perimeter of $W$ is $\Zcal$-red then $(\Delta, \rho)$ is also $\Zcal$-progressive.
\end{lemma}
\begin{proof} Let $q' \coloneqq q/2.$
Observe that $B$ is a $2q'$-central brick of $D$.
Consider $\Ccal \coloneqq \langle C_{1}, \ldots, C_{q} \rangle$ to be the cycles in $D$ such that $Δ_{W, i}$ corresponds to the disk that bounds the trace of $C_{i}$, for every $i \in [q].$
By definition, $C_{q' + 1}$ corresponds to the $(q' + 1)$-th layer of the $q$-region of $W.$
Then, $C_{q'+1}$ corresponds to the perimeter of a $2(q'+1)$-wall $W'$ that contains $W$ and is a subwall of the $q$-region of $W.$
Since $\Perimeter(W') = 8(q'+1) + 4,$ it follows that in $W$ there exist $8(q'+1) + 4$ pairwise disjoint paths with one endpoint in $C_{q' + 1}$ and the other in $C_{q}.$
Let $\Pcal \coloneqq \{ P_{1}, \ldots, P_{q'} \}$ be a collection of $q'$ of these paths.
Since these paths are subpaths of tracks of $D$ they moreover are orthogonal to $\Ccal' \coloneqq \langle C_{q' + 1}, \ldots, C_{q} \rangle$ and hence the pair $(\Ccal', \Pcal)$ is a railed nest in $δ$ of order $q'$. 

Now, we show how to define a railed flat vortex in $\delta$ that is equipped with the previously defined railed nest.
Let $\Delta_{1}$ be the disk that bounds the trace of $C_{q' + 1}$ and $\Delta$ be the disk that bounds the trace of $C_{q},$ i.e., the perimeter of the $q$-region of $W.$
By definition $\Ccal$ can be seen as a nest around $\Delta_{1}$ in $\delta.$
Moreover, note that $\Delta_{1}$ fully contains the $\delta$-influence of the perimeter of $W.$
Hence it is straightforward to define a railed flat vortex $(\Delta, \rho)$ in $\delta$ of order $q'.$
Additionally, since $q' \geq f_{\ref{@horkheimer}}(1, h_{\Zcal}),$ if the perimeter of $W$ is $\Zcal$-red and as a result $\Delta_{1}$ contains a $\Zcal$-red cell of $\delta$ it follows that $(\Delta, \rho)$ is $\Zcal$-progressive by applying \autoref{flat_vortex_contains_host}.
\end{proof}

\begin{lemma}\label{@productive} Let $q \geq f_{\ref{@horkheimer}}(1, h_{\Zcal}),$ $(A, \delta, D)$ be a $\Sigma$-schema of a graph $G$ in a surface $\Sigma,$ $\Zcal \in \Kbbb^{-},$ and $\{ L_{1}, \ldots, L_{q} \}$ be a set of $q$ many simple (exceptional) layers of $D.$
Then there exists a railed flat vortex $(\Delta, L_{0}, \rho)$ in $\delta$ of order $q$ such that $\Delta$ is the disk that bounds the trace of $L_{q}.$
Moreover, if the disk that bounds the trace of $L_{1}$ contains a $\Zcal$-red cell then $(\Delta, \rho)$ is also $\Zcal$-progressive.
\end{lemma}
\begin{proof} Let $\Ccal \coloneqq \langle L_{1}, \ldots, L_{q} \rangle$ and $\Delta_{1}$ be the disk that bounds the trace of $L_{1}.$
It follows that $\Ccal$ is a nest around $\Delta_{1}$ in $\delta$ of order $q.$
Moreover, the existence of at least $q$ layers of $D$ implies that the order of $D$ is at least $q$ and hence there exists a collection $\Pcal$ of $q$ pairwise disjoint paths with one endpoint in $L_{1}$ and the other in $L_{q}$ defined as subpaths of the tracks of the annulus wall part of $D$ that is clear of handles and cross-caps and that are orthogonal to $\Ccal.$
Then the pair $(\Ccal, \Pcal)$ is a railed nest in $δ$ of order $q$ around $Δ_{1}.$

Now, we show how to define a railed flat vortex in $\delta$ that is equipped with the previously defined railed nest.
Let $\Delta$ be the disk that bounds the trace of $L_{q}.$
Now it is straightforward to define a railed flat vortex $(\Delta, \rho)$ of order $q.$
Additionally, since $q' \geq f_{\ref{@horkheimer}}(1, h_{\Zcal}),$ if $\Delta_{1}$ contains a $\Zcal$-red cell, it follows that $(\Delta, \rho)$ is $\Zcal$-progressive by applying \autoref{flat_vortex_contains_host}.
\end{proof}

We are now ready to present the main result of this subsection.

\begin{lemma}\label{@approached}
There exist functions $f^{1}_{\ref{@approached}} : \Nbbb \to \Nbbb$ and $f^{2}_{\ref{@approached}} : \Nbbb^{2} \to \Nbbb$ such that, 
for every $\Zcal \in \Kbbb^{-},$ there exists an algorithm that, given
\begin{itemize}
\item $k \in \Nbbb,$
\item $q \geq f^{1}_{\ref{@approached}}(h_{\Zcal}),$
\item a graph $G,$
\item a $Σ$-schema $(A, δ = (Γ, \Dcal), D)$ of $G$ in a surface $\Sigma$ that is $α$-anchored at a set $S,$ for some $α \in [2/3, 1),$ and $D$ is of order $d,$ where $d > f^{2}_{\ref{@approached}}(k, q) + 3,$
\end{itemize}
outputs either
\begin{itemize}
\item a mixed $\Zcal$-packing of size $k$ in $G - A,$ or 
\item a $Σ$-schema $(A, δ, D')$ that is $α$-anchored at $S,$ where $D'$ is a $(Σ; d')$-Dyck wall that is a subwall of $D$ of order $d' \geq d - f^{2}_{\ref{@approached}}(k, q),$ and a railed flat vortex collection $\mathbf{V}$ in $δ$ that is $\Zcal$-progressive, has size less than $k$ and order $q,$ such that for every $(Δ, ρ) \in \mathbf{V},$ $\inG_{δ}(Δ) \cap D'$ is the empty graph.
\end{itemize}
Moreover the above algorithm runs in time $\Ocal_{h_{\Zcal}}(|V(G)|^3)+\poly(q,k)\cdot |V(G)|$ and the functions $f^{1}_{\ref{@approached}}$ and $f^{2}_{\ref{@approached}}$ are of the following orders:
\begin{align*}
	f^{1}_{\ref{@approached}}(h_{\Zcal}) &= \Ocal(2^{\mathsf{poly}(h_{\Zcal})})\text{, and}\\
	f^{2}_{\ref{@approached}}(k, q) &= \Ocal(k^{2}q^{5}).
\end{align*}
\end{lemma}
\begin{proof} We define $f^{1}_{\ref{@approached}}(h_{\Zcal}) \coloneqq f_{\ref{@horkheimer}}(1, h_{\Zcal})$ and $f^{2}_{\ref{@approached}}(k, q) \coloneqq f_{\ref{@custodians}}(k, 2q).$

By \cref{@calculations}, in time $\Ocal_{h_{\Zcal}}(|V(G)|^3)$ we detect all $\Zcal$-red cells of $δ.$
Next, we apply \cref{@custodians} with $k,$ $\ell = 2q,$ and $\Bcal$ being the set of $\Zcal$-red bricks of $D.$
There are two outcomes.
Assume that the previous application returns a set of $\{ B_{1}, \ldots, B_{k} \} \subseteq \Bcal$ of $2q$-central bricks,
where for $i \neq j \in [k],$ $B_{i}$ and $B_{j}$ have disjoint $2q$-regions.
An application of \autoref{@misfortune} for every $i \in [k],$ to the $2q$-central brick $B_{i},$ implies the existence of a $\Zcal$-progressive flat vortex in $\delta$ such that the disks of all these railed flat vortices correspond to the disks that bound the trace of the perimeter of the corresponding $2q$-region if each brick $B_{i}$, and they are therefore pairwise disjoint.
This implies the existence of $\Zcal$-mixed packing in $G - A$ and we conclude.

This allows us to assume that the previous application returns a collection $\Wcal = \{ W_{i} \mid i \in [k]' \},$ $k' \leq k-1,$ of $2q$-central walls of $D,$ a collection $\Scal = \{ S_{1}, \ldots, S_{2q} \}$ of $2q$ consecutive simple layers of $D,$ and a collection $\Ecal = \{ E_{1}, \ldots, E_{2q} \}$ of $2q$ consecutive exceptional layers of $D,$ satisfying the specifications of \cref{@custodians}.
Now, by appropriately applying \autoref{@misfortune} to every wall in $\Wcal$ and \autoref{@productive} we may define a set $\mathbf{\Delta}$ that consists of all the implied railed flat vortices in $\delta$ of order $q.$

Observe that by the specifications of \cref{@custodians}, \cref{@misfortune}, and \cref{@productive} the disks corresponding to the railed flat vortices in $\mathbf{\delta}$ are pairwise disjoint and moreover all $\Zcal$-red cells of $\delta$ are contained within the cells of the cylindrical renditions of the railed flat vortices in $\mathbf{\delta}.$
Hence, the collection $\mathbf{\Delta}$ after we remove any of the railed flat vortices that are not $\Zcal$-progressive, defines a railed flat vortex collection in $δ$ of order $q.$
If the size of this collection is at least $k,$ then by \cref{@sprachphilosophischen}, we get a mixed $\Zcal$-packing of size $k$ in $G - A.$

To conclude, we consider the $Σ$-schema $(A, δ, D'),$ where $D'$ is any $(Σ; d')$-Dyck subwall of $D,$ where $d' \geq d - f^{2}_{\ref{@approached}}(k, q),$ such that $D'$ is disjoint from all disks in the final collection of railed flat vortices we obtain.
Notice that, by the specifications of \cref{@custodians} and the definition of $f^{2}_{\ref{@approached}}(k, q),$ $D'$ is well-defined.
Moreover, we can trivially apply \cref{@repetitive}, to show that $(A, δ, D')$ is $α$-anchored.
\end{proof}

\subsection{Refining flat vortices}
\label{@conjecture}

In this subsection, given a $Σ$-schema of some graph $G$ with a sufficiently large Dyck-wall, accompanied by a railed flat vortex collection of bounded size and sufficiently large order, we show how to further refine it in order to bound the depth of the flat vortex society of each of the railed flat vortices in our refined collection.
The outcome of this process will either succeed or produce a mixed $\Zcal$-packing of size $k$ in $G - A.$

\paragraph{Strip societies.}
Since we are dealing with railed-flat vortices whose corresponding society is cross-free, i.e., has a vortex-free rendition in a disk, we can immediately deduce that any transaction in such a society is planar and moreover satisfies some additional properties which we require later in our proofs.

\medskip
Let $G$ be a graph and $H$ be a subgraph of $G.$
An \emph{$H$-bridge} is either an edge $e$ with both endpoints in $H$ such that $e\notin E(H),$ or a subgraph of $G$ formed by a component $K$ of $G-V(H)$ along with all edges of $G$ with one endpoint in $V(K)$ and the other endpoint in $V(H).$
In the first case, we call the endpoints of $e$ the \emph{attachments} of the bridge, and in the second case the \emph{attachments} of the bridge are those vertices that do not belong to $K.$

Let $(G,\Omega)$ be a society and $\mathcal{P}$ be a planar transaction of order at least two in $( G,\Omega).$
The \emph{$\mathcal{P}$-strip society of $( G, \Omega)$} is defined as follows.
Let $\mathcal{P}=\{ P_1,P_2,\dots,P_{\ell}\}$ be ordered such that for each $i\in[\ell],$ $P_i$ has the endpoints $a_i$ and $b_i$ and $a_1,a_2,\dots,a_{\ell},b_{\ell},b_{\ell-1},\dots,b_1$ appear in $\Omega$ in the order listed.
We denote by $H$ the subgraph of $G$ defined by the union of the paths in $\mathcal{P}$ together with all vertices of $\Omega$ and by $H'$ the subgraph of $H$ consisting of $\mathcal{P}$ together with the vertices in $V(a_1\Omega a_{\ell})\cup V(b_{\ell}\Omega b_1).$
We consider the set $\mathcal{B}$ of all $H$-bridges of $G$ with at least one vertex in $V(H')\setminus V(P_1\cup P_{\ell}).$
For each $B\in\mathcal{B}$ let $B'$ be obtained from $B$ by deleting all attachments that do not belong to $V(H').$
Finally, let us denote by $G_1$ the graph defined as the union of $H'$ and all $B'$ where $B\in\mathcal{B}$ and let $\Omega_1\coloneqq a_1\Omega a_{\ell}\oplus b_{\ell}\Omega b_1.$
The resulting society $(G_1,\Omega_1)$ is the desired $\mathcal{P}$-strip society of $(G,\Omega).$
The paths $P_1$ and $P_{\ell}$ are called the \emph{boundary paths} of $(G_1, \Omega_1).$

We say that the $\mathcal{P}$-strip society $(G_1,\Omega_1)$ of $(G,\Omega)$ is \emph{isolated} if no edge of $G$ has one endpoint in $V(G_1) \setminus V(P_1\cup P_{\ell})$ and the other endpoint in $V(G)\setminus V(G_1).$

Moreover, an isolated $\mathcal{P}$-strip society is \emph{separating} if after deleting any non-boundary path of $\mathcal{P}$ there does not exist a path from one of the two resulting segments of $\Omega$ to the other.

Finally, a $\mathcal{P}$-strip society is \emph{rural} if it has a vortex-free rendition in a disk.

\medskip
The following lemma formalizes our previous remark that in a railed flat vortex any planar transaction is isolated, separating, and rural.

\begin{lemma}\label{transaction_isoseprur} Let $(\Delta, \rho)$ be a railed flat vortex of a vortex-free $\Sigma$-decomposition of a graph $G$ in a surface $\Sigma$. Then every transaction $\Pcal$ in the flat vortex society $(H, \Omega)$ of $(\Delta, \rho)$ is planar and the $\Pcal$-strip society of $(H, \Omega)$ is isolated, separating, and rural.
\end{lemma}
\begin{proof} The proof follows immediately from the fact that $\delta$ is vortex-free.
This implies the existence of a vortex-free rendition $\rho'$ of the flat vortex society $(H, \Omega)$ of $(\Delta, \rho)$.
In fact $\rho'$ can simply be defined as $\delta \cap \Delta$.
The fact that $\rho'$ is vortex-free, i.e., does not contain a cross, implies that any transaction $\Pcal$ in $(H, \Omega)$ is planar.
It follows that the $\Pcal$-strip society of $(H, \Omega)$ is isolated, separating, and rural.
\end{proof}

\paragraph{Compressing flat vortices.} There are two different ways we will ensure that we are making progress when refining our railed flat vortices. One of these ways is a \textsl{compression} trick whose idea originates from \cite{thilikos2022killing} and is very similar to the notion of tightness that we are already used in this paper.

Let $q \in \Nbbb$, $G$ be a graph, $\Sigma$ be a surface, and let $\delta$ be a vortex-free $\Sigma$-decomposition of $G$ in $\Sigma$.
Moreover, consider a railed flat vortex $(\Delta, \rho = (H, \Omega, c_{\rho}))$ of $\delta$ equipped with the railed nest $(\Ccal = \langle C_{1}, \ldots, C_{q} \rangle, \Rcal)$ in $\delta$ around $c_{\rho}$ of order $q$.
A railed flat vortex $(\Delta', \rho' = (H', \Omega', c_{\rho'}))$ of $\delta$ equipped with the railed nest $(\Ccal' = \langle C'_{1}, \ldots, C'_{q} \rangle, \Rcal')$ in $\delta$ around $c_{\rho'}$ of order $q$ is said to be a \emph{compression} of $(\Delta, \rho)$ if there exists $i \in [q]$ such that
\begin{itemize}
\item for every $j \in [i+1, q]$, $C_{j} = C'_{j}$ and
\item if $\Delta_{i}$ (resp. $\Delta'_{i}$) is the disk that bound the trace of $C_{i}$ (resp. $C'_{i}$) in $\rho$ (resp. $\rho'$) either $(H' \cap \Delta'_{i}) - V(C'_{i}) \subsetneq (H \cap \Delta_{i}) - V(C_{i})$ or $E(H' \cap \Delta'_{i}) \subsetneq E(H \cap \Delta_{i})$.
\end{itemize}

The compression of a railed flat vortex is essentially achieved by ``pushing'' one of the cycles of its nest closer to the cell in the middle of the flat vortex which might result in actually extending the nest within the area defined by the cell.
When this is the case we will make sure that it is safe to do so, i.e., only when no $\Zcal$-red cell is being pushed outside of the newly defined cell in the middle of our flat vortex.
Additionally, when we will ``push'' cycles to obtain a compression we will do so in a way that either maintains the orthogonality of the nest with its original linkage or in a more extreme case where this might not be possible we will obtain a new orthogonal linkage to be paired with our new nest.

\paragraph{Exposed transactions.} The second way we will be making progress when refining our flat vortices is by splitting the flat vortex in two new flat vortices.
To make sure this is possible we require our transactions to run through the cell in the middle of our flat vortex.

Let $(\Delta, \rho= (\Gamma_{\rho}, \Dcal_{\rho}, c_{\rho}))$ be a railed flat vortex in a $\Sigma$-decomposition $\delta$ of a graph $G$ in a surface $\Sigma.$
Let $(\Ccal, \Rcal)$ be the railed nest equipped to $(\Delta, \rho)$ and $(H, \Omega)$ be the flat vortex society of $(\Delta, \rho).$
A transaction $\mathcal{P}$ in $(H, \Omega)$ is said to be \emph{exposed} if $c_{\rho}$ intersects the drawing of every path in $\mathcal{P}.$

In the next lemma we use a  trick from \cite{thilikos2022killing} that allows us to find in every large enough transaction, either a big exposed transaction or a compression of our flat vortex that reduces the part of $G$ that belongs to its flat vortex society.

\begin{lemma}\label{exposure}
Let $s, p \in \Nbbb_{\geq 1}$, $(\Delta, \rho)$ be a railed flat vortex of order $s$ of a vortex-free $\Sigma$-decomposition of a graph $G$ in a surface $\Sigma$, and $(H, \Omega)$ be its flat vortex society.
Let $\Pcal$ be a transaction of order at least $2s + p$ in $(H, \Omega)$.
Then there exists
\begin{itemize}
    \item an exposed transaction $\mathcal{Q} \subseteq \mathcal{P}$ of order $p,$ or
    \item a compression of $(\Delta, \rho)$.
\end{itemize}
Moreover, there exists an algorithm that finds one of the two outcomes in time $\mathcal{O}(|G|).$
\end{lemma}
\begin{proof}

Assume that $\rho = (\Gamma_{\rho}, \Dcal_{\rho}, c_{\rho})$ and let $(\Ccal, \Rcal)$ be the railed nest in $\delta$ around $c_{\rho}$ of order $s$ equipped to $(\Delta, \rho)$.
If $\mathcal{P}$ contains an exposed transaction of order $p$ we can immediately conclude.
Moreover, we can check in linear time for the existence of such a transaction by checking for each path in $\mathcal{P}$ individually if it is exposed.
Hence, we may assume that there exists a linkage $\mathcal{Q} \subseteq \mathcal{P}$ of order $2s+1$ such that \textsl{no} path of $\mathcal{Q}$ is exposed.
It follows that $\mathcal{Q}$ is a planar transaction and no path in $\mathcal{Q}$ intersects the interior of $c_{\rho}$.
Notice that each member $Q$ of $\mathcal{Q}$ naturally separates $\Delta$ into two disks, exactly one of which intersects $c_{\rho}$ (in fact this disk must contain $c_{\rho}$).
Let us call the other disk the \emph{small side of $Q$}.
Given two members of $\mathcal{Q}$ then either the small side of one is contained in the small side of the other, or their small sides are disjoint.
It is straightforward to see that if we say that two members are \emph{equivalent} if their small sides intersect, this indeed defines an equivalence relation on $\mathcal{Q}.$
Moreover, there are exactly two equivalence classes, one of which, call it $\mathcal{Q}',$ must contain at least $s+1$ members.
Since $|\mathcal{C}|=s$ there must exist some $i\in[s]$ and some subpath $L$ of some path in $\mathcal{Q}'$ such that
\begin{itemize}
    \item both endpoints of $L$ belong to $V(C_i),$
    \item $L$ is internally disjoint from $\bigcup_{i\in[s]}V(C_i),$
    \item $L$ contains at least one edge that does not belong to $C_i,$ and
    \item $L$ is drawn in the disk that bounds $\mathsf{trace}(C_i)$ that is disjoint from the nodes corresponding to the vertices of $\Omega.$
\end{itemize}
Notice that $C_i\cup L$ contains a unique cycle $C'$ different from $C_i$ whose trace separates $c_{\rho}$ from the nodes corresponding to the vertices of $\Omega.$
Moreover, the disk $\Delta''$ that bounds $\mathsf{trace}(C')$ and that contains $c_{\rho}$, is properly contained in the disk that bounds $\mathsf{trace}(C_i)$ and contains $c_{\rho}$.
In particular, there exists an edge of $C_i$ which does not belong to $C_i',$ and there exists an edge that belongs to $C_i'$ but not to $C_i$ and therefore, the graph drawn in $\Delta''$ after deleting the vertices of $C_i'$ is properly contained in the graph drawn on $\Delta'$ after deleting the vertices of $C_i.$

We would now like to define the new nest $\mathcal{C}'\coloneqq \{ C_1,\dots,C_{i-1},C',C_{i+1},\dots,C_s\}.$
However, to fully obtain the desired compression of $(\Delta, \rho)$ we need to use the linkage $\mathcal{R}$ to obtain a linkage $\mathcal{R}'$ of the same order which is orthogonal to $\mathcal{C}'.$
Instead, we are going to change the cycle $C'$ once more.
Let us traverse along $C'$ in the clockwise direction.
For every member $R$ of $\mathcal{R}$ which does not intersect $L$ we do not have to do anything.
For every member $R$ of $\mathcal{R}$ which has a non-empty intersection with $L$ we may define $x_R$ to be the first vertex of $L$ on $R$ and $y_R$ to be the last such vertex.
Then there exists a unique $x_R$-$y_R$-subpath $R'$ of $R$ and a unique $x_R$-$y_R$ subpath $L'$ of $C'$ such that replacing $L'$ by $R'$ results in a new cycle whose trace separates the nodes corresponding to the vertices of $\Omega$ from $c_{\rho}.$
Moreover, this new cycle has now one potential path less that could break orthogonality.
Hence, iterating this process exhaustively finally yields the desired cycle $C''$ which allows us to define the nest $\mathcal{C}' \coloneqq \{ C_1,\dots,C_{i-1},C'',C_{i+1},\dots,C_s\}$ to obtain our desired compression of $(\Delta, \rho)$.
\end{proof}

\paragraph{Orthogonality.} Another important ingredient that we will need in order to split our flat vortices is the presence of an orthogonal exposed transaction in our flat vortex society.
For this purpose we develop a tool that allows us, given a flat vortex with a sufficiently large exposed transaction, to slightly change our flat vortex by finding an alternative nest of the same order and a new still large exposed transaction that is orthogonal to our new nest.

\begin{lemma}\label{orthogonal_transaction} Let $q, p  \in \Nbbb_{\geq 1},$ $(\Delta, \rho = (\Gamma_{\rho}, \Dcal_{\rho}, c_{\rho}))$ be a railed flat vortex of order $q$ of a vortex-free $\Sigma$-decomposition $\delta$ of a graph $G$ in a surface $\Sigma,$ and $(H, \Omega)$ be its flat vortex society.
Let $\Pcal$ be a planar exposed transaction of order at least $q \cdot (p + q + 2)$ in $(H, \Omega).$
Then there exists
\begin{itemize}
\item a railed flat vortex $(\Delta, \rho' = (\Gamma_{\rho}, \Dcal'_{\rho}, c'_{\rho}))$ of $\delta$ such that
\begin{itemize}
\item the flat vortex society of $(\Delta, \rho')$ is $(H, \Omega),$
\item $c_{\rho}$ is a subset of $c'_{\rho},$ and
\item $(\Delta, \rho')$ is equipped with the nest $\Ccal$ of order $q,$ and
\end{itemize}
\item a planar exposed transaction $\Qcal$ of order $p$ that is orthogonal to $\Ccal.$
\end{itemize}
Moreover, there exists an algorithm that finds the outcome above in time $\Ocal(q \cdot p \cdot |G|).$
\end{lemma}
\begin{proof} By \autoref{transaction_isoseprur} we have that the $\Pcal$-strip society $(H', \Omega')$ of $(H, \Omega)$ is isolated, separating, and rural.
It is straightforward to obtain a rendition $\rho^{*}$ of $(H, \Omega)$ by combining a vortex-free rendition of $(H', \Omega')$ and $\rho$ such that $c_{\rho}$ is divided into a set $\Ucal$ of vortex cells in $\rho^{*}.$
Note that, since $\Pcal$ is exposed it follows that $|\Ucal| \geq 2.$
Moreover, we can define $\rho^{*}$ as a \emph{restriction} of $\delta,$ i.e., so that every cell of $\rho^{*}$ is a cell of $\delta \cap \Delta,$ except its vortex cells and cells neighbouring the vortex cells that are traversed by the two boundary paths of $\Pcal,$ which may be subsets of their corresponding cells in $\delta \cap \Delta.$
For what follows, we work with the rendition $\rho^{*}$.
Notice that in $\rho^{*}$ the drawing of the paths in $\Pcal$ never cross one another and this is crucial for the arguments that follow.

Let $s \coloneqq q \cdot (p + q + 2).$
Let $\Lambda_{\Pcal}$ be an ordering of the paths in $\Pcal$ such that while traversing $\Omega$ we first encounter one of the two endpoints for all paths in $\Pcal$ and then all others.
Let $P_{1}, \ldots, P_{s}$ be the paths in $\Pcal$ respecting the ordering $\Lambda_{\Pcal}.$
For $i \in [q]$ and $j \in [s],$ let $P$ be a $V(P_{j})$-$V(P_{j})$-path $B$ that is a subpath of $C_{i}.$
Let $P_{B}$ be the subpath of $P_{i}$ that shares its endpoints with $B.$
Let $\Delta_{B}$ be the disk bounded by the trace of the cycle $B \cup P_{B}.$
We call $B$ a \emph{bend} of $C_{i}$ at $P_{j}$ if $\Delta_{B}$ does not intersect the drawing of $C_{i} \setminus B.$
Moreover, we call a bend $B$ \emph{shrinking} if $\Delta_{B}$ is a subset of the disk that bounds the trace of $C_{i}$ and \emph{expanding} otherwise.
We call an expanding bend $B$ \emph{loose} if $\Delta_{B}$ does not intersect the drawing of any cycle in $\Ccal$ that is not $C_{i}$ and $\emph{tight}$ if any expanding bend $B'$ of $C_{i}$ where $B'$ is a subpath of $B$ is not loose.
We also define the \emph{height} of a bend $B$ as the non-negative integer $t + 1,$ where $t$ is the maximum value such that $B$ intersects at least one of $P_{j + t}$ or $P_{j - t}$ and $j + t \leq s$ or $j - t \geq 1$ (if $j + t$ exceeds $s$ we consider $P_{j + t} = P_{s}$ and if $j - t$ drops below $1$ we similarly assume that $P_{j - t} = P_{1}$).
Finally, we call any minimal $V(P_{1})$-$V(P_{s})$ path contained in a cycle $C \in \Ccal$ a \emph{pillar} of $C.$

First, notice that an expanding bend $B$ of a cycle $C \in \Ccal$ that is neither loose nor tight must contain a subpath $B'$ that is a loose expanding bend of $C.$
Moreover, if $B$ is any expanding bend of a cycle $C \in \Ccal,$ then the graph obtained by replacing the subpath $B$ of $C$ with the path $P_{B}$ defines a cycle $C'$ and the disk that is bounded by the trace of $C'$ strictly contains the disk that bounds the trace of $C,$ hence the name expanding.
In the case of a shrinking bend the opposite happens, particularly we obtain a cycle that bounds a strictly smaller area.
Additionally, notice that for every $C \in \Ccal,$ since $\Pcal$ is an exposed transaction, $C$ contains at least one pillar and since $C$ is a cycle, it must contain an even number of pairwise disjoint pillars.

We proceed to prove the following statement by an induction on the number of cycles in $\Ccal$: If no cycle in $\Ccal$ contains a loose expanding bend then for every $i \in [q],$
\begin{itemize}
\item every expanding bend of $C_{i}$ has height at most $q - i + 1,$
\item $C_{i}$ contains exactly two disjoint pillars $R_{i},$ $i \in [2],$ and
\item every shrinking bend of $C_{i}$ that is a subpath of $R_{i},$ $i \in [2],$ has height at most $q - i + 1.$
\end{itemize}

Now, let $k \in [1, q-1]$ and $\Ccal^{k}$ be the subset of $\Ccal$ that consists of its $k$ outermost cycles.
Assume that the statement above holds for $\Ccal^{k}.$
We prove that it also holds for $\Ccal^{k + 1}.$
Assume that no cycle in $\Ccal^{k + 1}$ contains a loose expanding bend.
Let $B$ be any tight expanding bend of $C_{q - k}$ at $P \in \Pcal$ which is the innermost cycle contained in $\Ccal^{k+1}$ and the only cycle of $\Ccal^{k+1}$ not contained in $\Ccal^{k}.$
Notice that any subpath $B'$ of $B$ whose endpoints are the endpoints of a maximal subpath of a path in $\Pcal$ that is drawn in $\Delta_{B}$ is also an expanding bend of $C_{q - k}$ where $\Delta_{B'} \subseteq \Delta_{B}.$
Since $B$ is tight, any such subpath $B'$ must also be tight.
This implies that the disk $\Delta_{B'}$ for any such subpath $B'$ intersects the drawing of some other cycle of $\Ccal.$
Since $B$ is expanding and hence $\Delta_{B}$ is not a subset of the disk that bounds the trace of $C_{q - k},$ it follows that $\Delta_{B'}$ intersects the drawing of cycles only in $\Ccal^{k}.$
In fact, it follows that any such intersection corresponds to an expanding bend of some cycle in $\Ccal^{k}$ at $P$ that is contained in $\Delta_{P}.$
Then, by assumption, if all these expanding bends are tight, it must be that the maximum height among all of them is $k.$
Hence, it must be that the height of $B$ is $k + 1.$

Moreover, it is straightforward to observe that $C_{q - k}$ contains exactly two disjoint pillars.
Indeed, if there were at least four disjoint pillars in $C_{q - k}$ then it is easy to see that there would exist an expanding bend of $C_{q - k}$ on either $P_{1}$ and/or $P_{s}$ whose height would be larger than $k + 1,$ which would imply that it is a loose expanding bend, which by assumption cannot exist.

Finally, it is also straightforward to observe that the existence of a shrinking bend of $C_{q - k}$ that is a subpath of either of the two pillars of $C_{q - k}$ of height more than $k + 1,$ implies the existence of an expanding bend of the same height, which as we proved cannot exist.

We define a new set of cycles $\Ccal' \coloneqq \langle C'_{1}, \ldots, C'_{q} \rangle,$ where $C'_{i},$ $i \in [q]$ is obtained from $C_{i}$ by iteratively replacing $B$ by $P_{B}$ for an expanding bend $B$ of $C$ until there is not loose expanding bend of $C.$
We do this starting from the outermost cycle of $\Ccal$ and working towards the innermost.
This can be done in time $\Ocal(q \cdot p \cdot |G|).$
Then the implications of the arguments above hold for $\Ccal'.$
Note that $C'_{q} = C_{q}.$

We now define a transaction $\Pcal' \subseteq \Pcal$ of order $p + q$ by skipping the first and last $q$ paths in $\Pcal$ and then choosing the first path from every bundle of $q$ many consecutive paths from the remaining paths of $\Pcal.$
By the arguments above the span of every expanding bend of every cycle in $\Ccal'$ at $P_{1}$ or $P_{s}$ is at most $q$ and therefore by definition no such bend can intersect any of the paths in $\Pcal'.$
It follows that any other bend of any cycle must be a bend that is a subpath of a pillar of the given cycle.
By definition of $\Pcal',$ this implies that any such bend of any cycle can intersect at most one path in $\Pcal'.$
This implies that the transaction $\Pcal'$ which is a planar and exposed transaction in $(H, \Omega)$ is almost orthogonal to our new set of cycles $\Ccal'.$
We need to make one more adjustment to each of the paths in $\Pcal'$ to finally make them orthogonal to $\Ccal'.$
We define the final transaction $\Qcal$ where each path $Q \in \Qcal$ is a path obtained from a path in $P \in \Pcal'$ that shares the same endpoints with $P$ and is defined by starting from one of the two endpoints of $P$ and while moving towards the other endpoint of $P$ taking any possible ``shortcut'', that is by greedily following along any bend of any pillar of a cycle in $\Ccal'$ that appears until we reach the other endpoint of $P.$
Since as we already observed any bend of any pillar of a cycle can intersect at most one path in $\Pcal',$ it is implied that $\Qcal$ is a linkage.
Moreover, by construction it now follows that $\Qcal$ is orthogonal to $\Ccal'$ as desired.

Please notice that, since by definition, $\rho^{*}$ is a restriction of $\delta,$ it is implied that all cycles of the newly defined nest $\Ccal'$ are grounded in $\delta,$ since by definition, the disks of the expanding bends we used to define the cycles in $\Ccal'$ did not intersect the vortex cells in $\Ucal.$
With this in mind, we can now conclude by defining a railed flat vortex in $\delta$ whose flat vortex society is still $(H, \Omega)$ by using the newly defined nest $\Ccal'$ and by cropping $q$ of the $p + q$ many paths to obtain the radial linkage that is required for the railed nest of the flat vortex.
Now, since to define $C'_{1}$ we only consider expanding bends of $C_{1}$ it also follows that the newly created vortex cell for our flat vortex contains the one of $(\Delta, \rho).$
Finally, since the shortcuts we take are subpaths of pillars it follows that $\Qcal$ contains at least one edge from the interior of the disk bounded by the trace of $C'_{1}$ and therefore is exposed with respect to our new railed flat vortex and we can conclude.
\end{proof}

\paragraph{Clean transactions.} As we already mentioned one of the ways we seek to make progress when refining our flat vortices is splitting them in two. However, we have to be careful.
We have to make sure that after we have successfully split a flat vortex, all $\Zcal$-red cells of our $\Sigma$-decomposition that were previously located deep within the cell in the middle of our original flat vortex, can now be located in the new cells in the middle of the two resulting flat vortices.
To achieve this we want that the exposed planar transaction that is isolated, separating, and rural, that we will use in order to split, to also be ``clean'' of $\Zcal$-red cells.

Let $\Zcal \in \Kbbb^{-}$ and $\delta$ be a vortex-free $\Sigma$-decomposition of a graph $G$ in a surface $\Sigma$.
Let $s \in \Nbbb_{\geq 1}$, $(\Delta, \rho)$ be a railed flat vortex of $\delta$ equipped with the nest $\Ccal = \langle C_{1}, \ldots, C_{s} \rangle$, and $(H, \Omega)$ be the flat vortex society of $(\Delta, \rho)$.
Let $\mathcal{P}$ be a transaction in $(H, \Omega)$ that is planar and exposed.
A \emph{parcel} of $\mathcal{P}$ is a cycle $C$ that is the union of two vertex-disjoint subpaths of $C_s$, say $X_1,X_2$, and two paths $P_1 \neq P_2 \in \mathcal{P}$ such that each $X_{i}$ is a subpath of $P_{i},$ $i \in [2],$ and $C$ intersects no other member of $\mathcal{P}.$
Notice that all but the boundary paths of $\Pcal$ belong to exactly two parcels and the boundary paths belong to exactly one parcel each.
We call a parcel $O$ of $\Pcal$ $\Zcal$-\emph{red} if there exists a $\Zcal$-red cell of $\delta$ in the $\delta$-influence of $O$.
We call $\mathcal{P}$ $\Zcal$-\emph{clean} if no parcel of $\Pcal$ is $\Zcal$-red.

From this point forward, let $\mathsf{q} \colon \Nbbb \to \Nbbb$ be the function defined as $\mathsf{q}(x) \coloneqq 2(f_{\ref{@horkheimer}}(1, x) + 2)$, for $x \in \Nbbb$.

\begin{lemma}\label{find_clean_transaction} Let $\Zcal \in \Kbbb^{-}$, $k, p \in \Nbbb_{\geq 1}$, $q \geq \mathsf{q}(h_{\Zcal})$, $\delta$ be a vortex-free $\Sigma$-decomposition of a graph $G$ in a surface $\Sigma$, $(\Delta, \rho)$ be a railed flat vortex in $\delta$ equipped with the nest $\Ccal$, and $(H, \Omega)$ be the flat vortex society of $(\Delta, \rho)$.
Let $\Pcal$ be a planar and exposed transaction in $(H, \Omega)$ of order $k \cdot \mathsf{q}(h_{\Zcal}) \cdot p$ that is orthogonal to $\Ccal.$
Then there exists either
\begin{itemize}
\item a mixed $\Zcal$-packing of size $k$ in $H$ or
\item a $\Zcal$-clean transaction $\Pcal' \subseteq \Pcal$ of order $p.$
\end{itemize}
Moreover, there exists an algorithm that finds one of the two outcomes in time $\mathcal{O}(G).$
\end{lemma}
\begin{proof}

If there exists a $\Zcal$-clean transaction $\Pcal' \subseteq \Pcal$ of order $p$ we can immediately conclude.
Therefore we may assume that such a transaction does not exist.
Let $\Lambda_{\Pcal}$ be an ordering of the paths in $\Pcal$ such that while traversing $\Omega$ we first encounter one of the two endpoints for all paths in $\Pcal$ and then all others.
Let $r \coloneqq k \cdot \mathsf{q}(h_{\Zcal}).$
We define a transaction $\Qcal \subseteq \Pcal$ by selecting the first path from every bundle of $p$ many consecutive paths in $\Pcal$ respecting the ordering $\Lambda_{\Pcal}.$
Note that the order of $\Qcal$ is $r.$
Moreover, since there is no $\Zcal$-clean subtransaction of $\Pcal$ of order $p,$ it must be that every parcel of $\Qcal$ is $\Zcal$-red.
Our goal is to show that the graph $\cupall \Qcal \cup \cupall \Ccal$ contains $k$ pairwise disjoint $\mathsf{q}(h_{\Zcal})$-walls $W_{i},$ $i \in [k],$ whose central bricks are $\Zcal$-red as subgraphs.
Notice that in this case we can conclude by applying (the arguments in the proof of) \autoref{@misfortune} to every $W_{i},$ $i \in [k],$ and obtain as a result a mixed $\Zcal$-packing of size $k$ in $H.$
To see that this is the case, observe that in $\cupall \Qcal \cup \cupall \Ccal,$ since $|\Ccal| \geq \mathsf{q}(h_{\Zcal}),$ $|\Qcal| = k \cdot \mathsf{q}(h_{\Zcal}),$ and $\Qcal$ is exposed and orthogonal to $\Ccal,$ there is a system of $k \cdot \mathsf{q}(h_{\Zcal})$ (the paths in $\Qcal$) many paths that orthogonally cross $2 \cdot \mathsf{q}(h_{\Zcal})$ (obtained from $\Ccal$) many paths.
These can essentially act as the $k$ pairwise disjoint $\mathsf{q}(h_{\Zcal})$-walls that we are looking for.
\end{proof}

\paragraph{Making progress.} We are now ready to proceed with the main lemma on how we make progress in refining our railed flat vortices under the presence or a large enough $\Zcal$-clean transaction in their flat vortex society.

\begin{lemma}\label{split_vortex} Let $\Zcal \in \Kbbb^{-},$ $q \in \Nbbb_{\geq 1},$ $p \coloneqq 4q + 2,$ $\delta$ be a vortex-free $\Sigma$-decomposition of a graph $G$ in a surface $\Sigma,$ $(\Delta, \rho = (\Gamma_{\rho}, \Dcal_{\rho}, c_{\rho}))$ be a railed flat vortex in $\delta$ such that $c_{\rho}$ contains a $\Zcal$-red cell of $\delta$, and $(H, \Omega)$ be the flat vortex society of $(\Delta, \rho)$.
Let $\Pcal$ be a $\Zcal$-clean transaction in $(H, \Omega)$ of order $p$ that is orthogonal to $\Ccal.$
Then there exist
\begin{itemize}
\item two railed flat vortices $(\Delta^{1}, \rho^{1} = (\Gamma^{1}_{\rho}, \Dcal^{1}_{\rho}, c^{1}_{\rho}))$ and $(\Delta^{2}, \rho^{2} = (\Gamma^{2}_{\rho}, \Dcal^{2}_{\rho}, c^{2}_{\rho}))$ in $\delta$ of order $q$ such that
\begin{itemize}
\item both $\Delta^{1}$ and $\Delta^{2}$ are subsets of $\Delta,$
\item $\Delta^{1}$ and $\Delta^{2}$ are disjoint,
\item $c^{1}_{\rho}$ and $c^{2}_{\rho}$ both contain a $\Zcal$-red cell, and
\item all $\Zcal$-red cells of $\delta$ contained in $c_{\rho}$ are contained in $c^{1}_{\rho} \cup c^{2}_{\rho},$ or
\end{itemize}
\item a compression of $(\Delta, \rho)$ with the flat vortex society $(H', \Omega')$ such that $H'$ is a proper subgraph of $H.$
\end{itemize}
Moreover, there exists an algorithm that finds one of the two outcomes in time $\mathcal{O}(G^{2}).$
\end{lemma}
\begin{proof} By \autoref{transaction_isoseprur} we have that the $\Pcal$-strip society $(H', \Omega')$ of $(H, \Omega)$ is isolated, separating, and rural.
It is straightforward to obtain a rendition $\rho''$ of $(H, \Omega)$ by combining a vortex-free rendition of $(H', \Omega')$ and $\rho$ such that $c_{\rho}$ is divided into exactly two vortex cells $c^{1}_{\rho''}$ and $c^{2}_{\rho''}.$
Moreover, we can define $\rho''$ as a \emph{restriction} of $\delta,$ i.e., so that every cell of $\rho^{*}$ is a cell of $\delta \cap \Delta,$ except its two vortex cells and cells neighbouring the two vortex cells that are traversed by the two boundary paths of $\Pcal,$ which may be subsets of their corresponding cells in $\delta \cap \Delta.$
Note that the fact that $c_{\rho}$ is divided into exactly two vortex cells in $\rho''$ follows from the fact that $\Pcal$ is exposed and orthogonal to $\Ccal.$

Now, let $I_{1}$ and $I_{2}$ be the two minimal segments of $\Omega$ such that $\Pcal$ is a $V(I_{1})$-$V(I_{2})$ linkage in $G.$
We consider an ordering $\langle P_{0}, P_{1}, \ldots, P_{4q}, P_{p - 1} \rangle$ of the paths in $\Pcal$ such that the endpoints of the paths in $I_{1}$ appear before any endpoint in $I_{2}$ while traversing $\Omega.$
The two paths $P_{0}$ and $P_{p-1}$ are the boundary paths of $(H', \Omega').$
We consider a partition of all non-boundary paths in $\Pcal$ in four sets of $q$ many consecutive paths.
\begin{align*}
\Tcal^{1} &\coloneqq  \{ P_{1}, \ldots, P_{q} \} = \{ T^{1}_{q}, \ldots, T^{1}_{1} \} \\
\Qcal^{1} &\coloneqq \{ P_{q + 1}, \ldots, P_{2q} \} = \{ Q^{1}_{q}, \ldots, Q^{1}_{1} \} \\
\Qcal^{2} &\coloneqq \{ P_{2q + 1}, \ldots, P_{3q} \} = \{ Q^{2}_{1}, \ldots, Q^{2}_{q} \} \\
\Tcal^{2} &\coloneqq \{ P_{3q + 1}, \ldots, P_{4q} \} = \{ T^{2}_{1}, \ldots, T^{2}_{q} \}
\end{align*}

Let $\Ccal = \langle C_{1}, \ldots, C_{q} \rangle$ be the nest around $c_{\rho}$ in $\rho.$
We define two subdisks $\Delta^{1}$ and $\Delta^{2}$ of $\Delta$ as follows.
$\Delta^{1}$ is defined as the disk bounded by the trace of the cycle in $C_{q} \cup Q^{1}_{1}$ that avoids $Q^{2}_{1}$ while $\Delta^{2}$ is defined in a symmetric way as the disk bounded by the trace of the cycle in $C_{q} \cup Q^{2}_{1}$ that avoids $Q^{1}_{1}.$
Clearly each of the two vortex cells $c^{i}_{\rho''},$ $i \in [2],$ is a subset of exactly one of $\Delta^{i},$ $i \in [2].$
Therefore we can assume without loss of generality that $c^{1}_{\rho''}$ is a subset of $\Delta^{1}$ while $c^{2}_{\rho}$ is a subset of $\Delta^{2}$ respectively.
Moreover, since $c_{\rho}$ contains a $\Zcal$-red cell of $\delta,$ by construction at least one of $c^{i}_{\rho''},$ $i \in [2]$ must also contain a $\Zcal$-red cell and all $\Zcal$-red cells that are subsets of $c_{\rho}$ are subsets of $c^{1}_{\rho''} \cup c^{2}_{\rho''}.$
The disks $\Delta^{1}$ and $\Delta^{2}$ are our candidate disks for the two new railed flat vortices.

For $i \in [2],$
let $x^{i}$ be a point drawn in $c^{i}_{\rho''}.$
We proceed to define $q$ many concentric cycles in $\Delta^{i}$ ``around'' $x^{i}$ as follows.
For every $j \in [q],$ let $C^{i}_{j}$ be a cycle in $C_{j} \cup Q^{i}_{j}$ such that the disk bounded by the trace of $C^{i}_{j}$ contains $x^{i}.$
The fact that $\Pcal$ is orthogonal to $\Ccal$ implies that all these cycles are well-defined and pairwise disjoint.

We define $\Ccal^{i} \coloneqq \langle C^{i}_{1}, \ldots, C^{i}_{q} \rangle$.
Note that by definition $\Delta^{i}$ corresponds to the disk bounded by the trace of $\Ccal^{i}_{q}$ that contains $x^{i}.$
Moreover, let $\Delta^{i}_{1}$ be the disk bounded by the trace of $C^{i}_{1}.$
By construction it must be that $c^{i}_{\rho''}$ is contained within $\Delta^{i}_{1}.$
It follows that $\Ccal^{i}$ is a nest around $\Delta^{i}_{1}$ in $\delta$ of order $q$ and that $\Delta^{1}_{1} \cup \Delta^{1}_{2}$ contain all $\Zcal$-red cells of $\delta$ that are subsets of $c_{\rho}.$
Now, consider the society $(H^{i} \coloneqq G \cap \Delta^{i}, \Omega^{i} \coloneqq \Omega_{\Delta^{i}})$ and let $\rho^{i} = (\Gamma^{i}_{\rho}, \Dcal^{i}_{\rho}, c^{i}_{\rho})$ be a cylindrical rendition of $(H^{i}, \Omega^{i})$ around a cell $c^{i}_{\rho}$ in $\Delta^{i}$ such that
\begin{itemize}
\item $\Gamma^{i}_{\rho} = \Gamma \cap \Delta^{i},$
\item $\Dcal^{i}_{\rho}$ contains $c^{i}_{\rho}$ and all cells of $\delta$ contained in $\Delta^{i}$ that are not contained in $\Delta^{i}_{1},$ and
\item $c^{i}_{\rho}$ is the disk $\Delta^{i}_{1}$ minus the nodes of $\rho^{i}$ on the boundary of $\Delta^{i}_{1}.$
\end{itemize}

To conclude that $(\Delta^{1}, \rho^{1})$ and $(\Delta^{2}, \rho^{2})$ are the desired flat vortices, we need to find a radial linkage of order $q$ for both nests $\Ccal^{1}$ and $\Ccal^{2}.$
Towards this we make use of the transactions $\Tcal^{1}$ and $\Tcal^{2}.$
For every path $P \in \Tcal^{i},$ let $\mathsf{r}^{i}(P)$ be the minimal $V(I_{1})$-$V(C^{i}_{q})$ subpath of $P.$
For $i \in [2],$ let $\mathsf{r}^{i}(\Tcal^{i}) \coloneqq \{ \mathsf{r}^{i}(T^{i}_{j}) \mid j \in [q] \}.$
By definition $\mathsf{r}^{i}(\Tcal^{i})$ is a radial linkage in $\rho^{i}.$
Moreover, since $\Tcal^{i}$ is orthogonal to $\Ccal$ it follows by construction that $\mathsf{r}^{i}(\Tcal^{i})$ is orthogonal to $\Ccal^{i}.$
Hence we conclude that the pair $(\Ccal^{i}, \mathsf{r}^{i}(\Tcal^{i}))$ is a railed nest around $c^{\rho}_{i}$ of order $q$ as desired.

To finish the proof we observe the following.
If $c^{i}_{\rho}$ contains a $\Zcal$-red cell of $\delta$ for both $i \in [2]$ the proof is complete.
Otherwise, assume without loss of generality, that only $c^{2}_{\rho}$ contains a $\Zcal$-red cell of $\delta.$
In this case we argue that $(\Delta^{2}, \rho^{2})$ is a compression of $(\Delta, \rho).$
Indeed this follows immediately since by definition $\Delta^{2} \subsetneq \Delta$ and at least the paths in $\Tcal^{1}$ are not contained in $H^{2}.$
\end{proof}

\paragraph{Obtaining a bounded number of bounded depth railed flat vortices.}

We are now ready to inductively apply \autoref{split_vortex} to prove that given a flat vortex collection of sufficiently large order we can either find a large mixed $\Zcal$-packing in our graph or refine it to a railed flat vortex collection of bounded size where each flat vortex has bounded depth and moreover such that all $\Zcal$-red cells of the starting decomposition of our graph are cornered within the interior of the flat vortices of our refined collection.

\medskip
Let $\delta$ be a $\Sigma$-decomposition of a graph $G$ in a surface $\Sigma$ and $\Zcal \in \Kbbb^{-}.$
We call $\Zcal$-\emph{enclosing} a railed flat vortex collection $\mathbf{\Delta}$ in $\delta$ such that
\begin{itemize}
\item for every $(\Delta, \rho) = (\Gamma_{\rho}, \Dcal_{\rho}, c_{\rho})) \in \mathbf{\Delta},$ $c_{\rho}$ contains a $\Zcal$-red cell of $\delta$ and
\item all $\Zcal$-red cells of $\delta$ are subsets of $\bigcup_{(\Delta, \rho) \in \mathbf{\Delta}} c_{\rho}.$
\end{itemize}

\begin{lemma}\label{obtain_vortex_collection} Let $\Zcal \in \Kbbb^{-},$ $k \in \Nbbb_{\geq 1},$ $q \geq \mathsf{q}(h_{\Zcal})$, $\delta$ be a vortex-free $\Sigma$-decomposition of a graph $G$ in a surface $\Sigma.$
Let $\mathbf{\Delta}$ be a $\Zcal$-enclosing railed flat vortex collection of size less that $k$ in $\delta.$
Then there exists either
\begin{itemize}
\item a mixed $\Zcal$-packing of size $k$ in $G$ or
\item a $\Zcal$-enclosing railed flat vortex collection of size less than $k$ in $\delta$ such that for every $(\Delta, \rho) \in \mathbf{\Delta}',$ the flat vortex society of $(\Delta, \rho)$ has depth at most $2q + q \cdot (k \cdot \mathsf{q}(h_{\Zcal}) \cdot (4q + 2) + q) + 2q - 1$ and $\Delta$ is a subset of $\Delta'$ for some $(\Delta', \rho') \in \mathbf{V}.$
\end{itemize}
Moreover, there exists an algorithm that finds one of the two outcomes in time $\mathcal{O}(|G|^{3}).$
\end{lemma}
\begin{proof}

We show that that there exists an algorithm that given a $\Zcal$-enclosing railed flat vortex collection $\mathbf{\Delta}^{1}$ of size $s$ and order $q$ in $\delta$ outputs a $\Zcal$-enclosing railed flat vortex collection $\mathbf{\Delta}^{2}$ of the same order that is either of size at least $s + 1$ or such that for every $(\Delta, \rho) \in \mathbf{\Delta}',$ the flat vortex society of $(\Delta, \rho)$ has depth at most $2q + q \cdot (k \cdot \mathsf{q}(h_{\Zcal}) \cdot (4q + 2) + q) + 2q - 1,$ in time $\mathcal{O}(G^{3}).$

We first argue why the existence of such an algorithm suffices for the purposes of our proof.
Indeed by an iterative application of the algorithm above to the given railed flat vortex collection, in at most $k$ iterations, we conclude with a $\Zcal$-enclosing railed flat vortex collection $\mathbf{\Delta}'$ such that either the size of $\mathbf{\Delta}'$ is less than $k$ and for every $(\Delta, \rho) \in \mathbf{\Delta}',$ the flat vortex society of $(\Delta, \rho)$ has depth at most $2q + q \cdot (k \cdot \mathsf{q}(h_{\Zcal}) \cdot (4q + 2) + q) + 2q - 1,$ in which case we can conclude, or the size of $\mathbf{\Delta}'$ is at least $k.$
In this case, by the assumption that $q \geq \mathsf{q}(h_{\Zcal})$ and that $\mathbf{\Delta}'$ is $\Zcal$-enclosing, we can apply \autoref{flat_vortex_contains_host} to every railed flat vortex in the collection to conclude that it is $\Zcal$-progressive and hence by \autoref{@sprachphilosophischen}, conclude with a mixed $\Zcal$-packing of size $k$ in $G.$

We now explain how the algorithm above works.
For every railed flat vortex $(\Delta, \rho) \in \mathbf{\Delta}^{1}$ we do the following.
If the flat vortex society $(H, \Omega)$ of $(\Delta, \rho)$ has depth at most $2q + q \cdot (k \cdot \mathsf{q}(h_{\Zcal}) \cdot (4q + 2) + q) + 2q - 1$ we are done.
Otherwise let $\Pcal$ be a transaction of order at least $2q + q \cdot (k \cdot \mathsf{q}(h_{\Zcal}) \cdot (4q + 2) + q) + 2q$ in $(H, \Omega).$
By applying \autoref{exposure}, we obtain either a compression of $(\Delta, \rho),$ or an exposed transaction in $(H, \Omega)$ of order $q \cdot (k \cdot \mathsf{q}(h_{\Zcal}) \cdot (4q + 2) + q) + 2q.$
Then, by an application of \autoref{orthogonal_transaction}, we obtain an alternative railed flat vortex $(\Delta, \rho' = (\Gamma'_{\rho}, \Dcal'_{\rho}, c'_{\rho}))$ in $\delta$ of order $q$ having the same society $(H, \Omega)$ and a planar exposed transaction $\Qcal$ in $(H, \Omega)$ of order $k \cdot \mathsf{q}(h_{\Zcal}) \cdot (4q + 2)$ that is orthogonal to the nest equipped to $(\Delta, \rho').$
Notice that by the specifications of \autoref{orthogonal_transaction}, the collection $\mathbf{\Delta}^{1} \setminus \{ (\Delta, \rho) \} \cup \{ (\Delta, \rho') \}$ is also $\Zcal$-enclosing.
We continue with an application of \autoref{find_clean_transaction} and we either obtain a mixed $\Zcal$-packing of size $k$ in $H$ (and hence in $G$) and we conclude or  a $\Zcal$-clean transaction of order $4q + 2$ in $(H, \Omega)$ that is orthogonal to the nest equipped to $(\Delta, \rho').$
In the latter case, we then apply \autoref{split_vortex}, where we either obtain a compression of $(\Delta, \rho')$ which is also a compression of $(\Delta, \rho),$ or two railed flat vortices $(\Delta^{1}, \rho^{1} = (\Gamma^{1}_{\rho}, \Dcal^{1}_{\rho}, c^{1}_{\rho}))$ and $(\Delta^{2}, \rho^{2} = (\Gamma^{2}_{\rho}, \Dcal^{2}_{\rho}, c^{2}_{\rho}))$ in $\delta$ of order $q$ such that
\begin{itemize}
\item both $\Delta^{1}$ and $\Delta^{2}$ are subsets of $\Delta,$
\item $\Delta^{1}$ and $\Delta^{2}$ are disjoint,
\item $c^{1}_{\rho}$ and $c^{2}_{\rho}$ both contain a $\Zcal$-red cell, and
\item all $\Zcal$-red cells of $\delta$ contained in $c'_{\rho}$ are contained in $c^{1}_{\rho} \cup c^{2}_{\rho}.$
\end{itemize}

In the latter case, we can conclude with the railed flat vortex collection $\mathbf{\Delta}^{1} \setminus \{ (\Delta, \rho) \} \cup \{ (\Delta^{1}, \rho^{1}), (\Delta^{2}, \rho^{2}) \}.$
Hence if we never encounter a compression of $(\Delta, \rho),$ for none of the railed flat vortices $(\Delta, \rho) \in \mathbf{\Delta}^{1},$ our algorithm will terminate in time $\mathcal{O}(|G|^{2}).$
To finish the proof, notice that every time we find a compression as the outcome of one of the above steps, at least one of the cycles of the nest equipped to $(\Delta, \rho)$ is ``pushed'' slightly closer to the cell $c_{\rho}$.
That means, at least one edge of $G$ is ``pushed'' further to the outside, or fully outside, of our railed flat vortex in each of these occurrences.
Since there are at most $|E(G)|$ many edges and only $q$ many cycles in our nest, we cannot find a compression more than $q \cdot |E(G)|$ many times.
As a result, and since we may assume that $G$ excludes a graph of bounded size as a minor which means $|E(G)| \in \mathcal{O}(|G|),$ after $\mathcal{O}(|G|^3)$ many iterations of the procedure above we can conclude.
\end{proof}

We are now ready to present the main result of this subsection which is an immediate consequence of \autoref{obtain_vortex_collection} and \autoref{flat_vortex_contains_host}.

\begin{corollary}\label{@sensualism}
There exist a function $f_{\ref{@sensualism}} : \Nbbb^{3} \to \Nbbb$ such that,
for every $\Zcal \in \Kbbb^{-},$ there exists an algorithm that, given
\begin{itemize}
   
\item $k \in \Nbbb,$
\item $q \geq \mathsf{q}(h_{\Zcal}),$
\item a graph $G,$
\item a $Σ$-schema $(A, δ, D)$ of $G$ in a surface $Σ,$ that is $α$-anchored at some set $S,$ where $α \in [2/3, 1),$ and $D$ is of order $d,$ where $d \geq 3,$ and
\item a flat vortex collection $\mathbf{V}$ in $δ$ that is $\Zcal$-progressive has size less than $k$ and order $q,$ such that for every $(Δ, ρ) \in \mathbf{V},$ $\inG_{δ}(Δ) \cap D$ is the empty graph,
\end{itemize}
outputs either
\begin{itemize}
\item a mixed $\Zcal$-packing of size $k$ in $G - A$ or
\item a $\Zcal$-progressive railed flat vortex collection $\mathbf{V}'$ in $\delta$ of size less than $k$ and order $q$ such that for every $(Δ, ρ) \in \mathbf{V}',$ the flat vortex society $(H, \Omega)$ of $(\Delta, \rho)$ has depth at most $f_{\ref{@sensualism}}(k, q, h_{\Zcal})$ and $H \cap D$ is the empty graph.
\end{itemize}
Moreover the previous algorithm runs in time $\Ocal_{h_{\Zcal}}(kq^{2} \cdot |V(G)|^3)$ and the function $f_{\ref{@sensualism}}$ is of order
\begin{align*}
&f_{\ref{@sensualism}}(k, q, h_{\Zcal}) = \Ocal(k \cdot q^{2} \cdot 2^{\mathsf{poly}(h_{\Zcal})}).
\end{align*}
\end{corollary}

\subsection{Killing flat vortices}\label{@translation}

In this subsection we show, given a $Σ$-schema $(A, δ, D)$ of a graph $G$ in a surface $\Sigma$ and a railed flat vortex collection in $\delta,$ refined as in \autoref{@sensualism}, how to either find a mixed $\Zcal$-packing of size $k$ and conclude, or find a separation of small order which covers all inflated copies in $G$ of the non-planar components of graphs in $\Zcal$ that are invading through $\Zcal$-red cells located within the flat vortices of the given collection.
As we shall see, this suffices to prove our local structure theorem.

\begin{lemma}\label{@repression}
There exist functions $f^{1}_{\ref{@repression}} : \Nbbb^{2} \to \Nbbb$ and $f^{2}_{\ref{@repression}} : \Nbbb^{2} \to \Nbbb$ such that,
for every $\Zcal \in \Kbbb^{-},$ there exist an algorithm 
that, given  
\begin{itemize}
\item $k, θ \in \Nbbb_{\geq 1},$ 
\item $q \geq f^{1}_{\ref{@repression}}(k, h_{\Zcal}),$
\item a graph $G,$ 
\item a $Σ$-schema $(A, δ, D)$ of $G$ in a surface $\Sigma,$ that is $α$-anchored at a set $S,$ where $α \in [2/3, 1),$ and $D$ is of order $d > f^{2}_{\ref{@repression}}(k, θ) + 3,$ and
\item a $\Zcal$-progressive railed flat vortex $(Δ, \rho = (\Gamma_{\rho}, \Dcal_{\rho}, c_{\rho}))$ in $\delta$ of order $q$ such that the flat vortex-society $(H, \Omega)$ of $(\Delta, \rho)$ has depth at most $\theta$ and $H \cap D$ is the empty graph,
\end{itemize}
outputs either
\begin{itemize}
\item a mixed $\Zcal$-packing of size $k$ in $G-A,$ or
\item a separation $(X, Y)$ of $G - A$ of order at most $f^{2}_{\ref{@repression}}(k, θ)$ such that $X \subseteq V(H)$ and if $M$ is a $c$-invading $\nonplanar(Z)$-inflated copy in $G - A,$ where $c \subseteq c_{\rho}$ and $Z \in \Zcal,$ then $M$ contains a vertex of $X \cap Y.$
\end{itemize}
Moreover the algorithm above runs in time $2^{\Ocal(k)\cdot\poly(h_{\Zcal})}|V(G)|+\Ocal_{h_{\Zcal}}(|V(G)|^3)$ and the functions $f^{2}_{\ref{@repression}}$ and $f^{1}_{\ref{@repression}}$ are of the following orders:
\begin{align*}
f^{2}_{\ref{@repression}}(k, θ) &= \Ocal(kθ)\text{, and}\\
f^{1}_{\ref{@repression}}(k, h_{\Zcal}) &= 2^{\Ocal(k) \cdot \mathsf{poly}(h_{\Zcal})}.\\
\end{align*}
\end{lemma}
\begin{proof}
We define $f^{1}_{\ref{@repression}}(k, h_{\Zcal}) \coloneqq \max\{ f_{\ref{@horkheimer}}(k,h_{\Zcal}), \mathsf{q}(h_{\Zcal}) \}$ and $f^{2}_{\ref{@repression}}(k, θ) \coloneqq 2 k θ.$

Let $\langle X_1, X_2, \dots, X_n, v_1, v_2, \dots, v_n \rangle$ be a linear decomposition of $(H, Ω)$ of adhesion at most $2 \theta,$ according to \cref{@prefascist} (this is done in time $\Ocal(θ\cdot |V(G)|^2)$).
We also set $X_{-1} = X_{0} = X_{n+1} = \emptyset.$
For every $i, j \in [0,n], i ≤ j,$ we define the graph $H_{i,j} \coloneqq H[\bigcup_{\ell \in [i, \ldots, j]} X_{\ell}]$ (notice that $H_{0,0}$ is the empty graph) and we set $H_{i,j}^+ = (G \setminus H) \cup (H_{i,j} - X_{i-1}).$
We say that $H_{i,j}^+$ is \emph{$\Zcal$-met} if it contains a $c$-invading $\nonplanar(Z)$-inflated copy $M,$ where $c \in C(\rho)$ and $Z \in \Zcal$ (notice that if this happens, then $c$ is an $\Zcal$-red cell).

We apply the following procedure:

\begin{itemize}
\item[1.] Initialize $P\coloneqq\emptyset,$ $p_{0}\coloneqq 0,$ $r\coloneqq 1.$
\item[2.] If $H_{p_r+1,n}^+$ is not $\Zcal$-met, then stop, otherwise go to step 3.
\item[3.] 
Let $p_{r}$ be the minimum $h > p_{r-1}$ such that $H_{p_{r-1}+1,h}^+$ is $\Zcal$-met.
Let also $A_{r} = X_{p_{r}} \cap X_{p_{r}+1},$ $\bar{A}_{r} = X_{p_{r}} \cap X_{p_{r}-1},$ and set $P\coloneqq P \cup \{p_{r}\}.$
Observe that, by the minimality of the choice of $h,$ $H_{p_{r-1}+1, p_{r} -1}^+ - \bar{A}_{r}$ is not $\Zcal$-met.
\item[4.] if $p_{r} < n,$ then set $r \coloneqq r+1$ and go to Step 2.
\end{itemize}
It follows from the minor-checking algorithm of Kawarabayashi, Kobayashi, and Reed \cite{KawarabayashiKR12Thedisjoint} that takes $\Ocal_{h_{\Zcal}}(|V(G)|^{2})$ time, that the procedure above takes $\Ocal_{h_{\Zcal}}(|V(G)|^{3})$ time.
The fact that $(Δ, ρ)$ is a $\Zcal$-progressive railed flat vortex in $δ$ implies that at least one of the cells of $\rho$ that are subsets of $c_{\rho}$ is $\Zcal$-red and that there are no other $\Zcal$-red cells that are subsets of $Δ.$
Notice the above procedure produces a collection of $A_{i}$'s that ensures the following fact:
\begin{eqnarray}
\begin{minipage}{14cm}
If $M$ is a $c$-invading $\nonplanar(Z)$-inflated copy in $G - A,$ where $c \subseteq c_{\rho}$ and $Z \in \Zcal,$ then $M$ contains some vertex of $A_{i},$ for some $i \in [r].$
\end{minipage}\label{@quasimythical}
\end{eqnarray}
We distinguish two cases.

\medskip
\noindent\emph{Case 1.} $|P|<k.$ In this case we consider the sets $A'=\bigcup_{i\in[r]}(A_{i}\cup \bar{A}_{i})$ and $X=\bigcup_{i\in[r]}X_{p_i}$
and notice that $A'\subseteq X$ and $|A'|≤2kθ.$ Notice also that $(X, Y \coloneqq (V(G - A) \setminus (X\setminus A') ))$ is a separation of $G - A$ of order  $2kθ.$
Moreover, by the fact in \eqref{@quasimythical}, the separation $(X, Y)$ has the desired property.

\medskip
\noindent\emph{Case 2.} $|P|≥k.$ This implies that for every $i\in [k],$ $H_{p_{i-1}+1,p_{i}}^+$ contains a $c$-invading $\nonplanar(Z_{i})$-inflated copy $M_i,$ where $c \subseteq c_{\rho}$ and $Z_{i} \in \Zcal.$
By the choice of $H_{p_{i-1}+1,p_{i}}^+,i\in[2k]$  during the procedure above, it follows that, for every $(i, j) \in{[k]\choose 2},$ $M_i \cap H$ and $M_{j} \cap H$ are disjoint.
We now have all conditions required for applying \cref{@horkheimer}.
As a result we obtain in time $2^{\Ocal(k)\cdot\poly(h_{\Zcal})}|V(G)|,$ a mixed $\Zcal$-packing of size $k$ in $H$ that is also a mixed $\Zcal$-packing of size $k$ in $G - A$ as required.
\end{proof}

\begin{lemma}\label{@loudspeakers} 
There exist functions $f^{1}_{\ref{@repression}} \colon \Nbbb^{2} \to \Nbbb$ and $f_{\ref{@loudspeakers}} \colon \Nbbb^{2} \to \Nbbb$ such that, for every $\Zcal \in \Kbbb^{-},$ there exists an algorithm that, given
\begin{itemize}
\item $k, θ \in \Nbbb_{\geq 1},$
\item $q \geq f^{1}_{\ref{@repression}}(k,h_{\Zcal}),$
\item a graph $G,$ 
\item a $\Sigma$-schema $(A, δ, D)$ of $G$ in a surface $\Sigma$ that is $α$-anchored at a set $S,$ where $α \in [2/3, 1),$ and $D$ is of order $d,$ where $d > f_{\ref{@loudspeakers}}(k, θ) + 3,$ and
\item a $\Zcal$-progressive flat vortex collection $\mathbf{V}$ in $δ$ of size less than $k$ and order $q$ such that for every $(\Delta, \rho) \in \mathbf{\Delta}$ the flat vortex society $(H, \Omega)$ of $(\Delta, \rho)$ has depth at most $\theta$ and $H \cap D$ is the empty graph,
\end{itemize}
outputs one of the following:
\begin{itemize}
\item a mixed $\Zcal$-packing of size $k$  in $G - A,$ or
\item a $Σ$-schema $(A \cup A', δ', D')$ of $G$ such that
\begin{itemize}
\item $|A'| \leq f_{\ref{@loudspeakers}}(k, θ),$
\item $(A \cup A', δ', D')$ is $α$-anchored at $S,$
\item the order of $D'$ is at least $d - f_{\ref{@loudspeakers}}(k, θ),$ and
\item none of the cells of $δ'$ are $\Zcal$-red.
\end{itemize}
\end{itemize}
Moreover the previous algorithm runs in time $2^{\Ocal(k)\cdot\poly(h_{\Zcal})}|V(G)|+\Ocal_{h_{\Zcal}}(k \cdot |V(G)|^3)$ and $f_{\ref{@loudspeakers}}$ is of order
\begin{align*}
f_{\ref{@loudspeakers}}(k, θ) &= \Ocal(k^{2}θ).\\
\end{align*}
\end{lemma}
\begin{proof} We define $f_{\ref{@loudspeakers}}(k, θ) \coloneqq k \cdot f^{2}_{\ref{@repression}}(k, θ).$

The proof is a parallel application of \cref{@repression} to every railed flat vortex in $\mathbf{V}.$
Let $\mathbf{V} = \{ (\Delta_{i}, \rho_{i}) \mid i \in [r] \},$ $r < k.$
For every $i \in [r]$ we apply \cref{@repression} on the $\Zcal$-progressive railed flat vortex $(\Delta_{i}, \rho_{i} = (\Gamma_{\rho}^{i}, \Dcal_{\rho}^{i}, c_{\rho}^{i}))$ and its corresponding flat vortex society $(H_{i}, \Omega_{i})$ of depth at most $\theta.$
As a result we either get a mixed $\Zcal$-packing in $G' \coloneqq G - A$ of size $k$ and we conclude or, for every $i \in [r],$ we get a separation $(X_{i}, Y_{i})$ of $G'$ of order at most $2kθ$ such that $X \subseteq V(H_{i})$ and
\begin{eqnarray} 
\begin{minipage}{14cm}
if $M$ is a $c$-invading $\nonplanar(Z)$-inflated copy in $G',$ where $c \subseteq c_{\rho}^{i}$ and $Z \in \Zcal,$ then $M$ contains a vertex of $X_{i} \cap Y_{i}.$
\end{minipage}\label{@temptations}
\end{eqnarray}

Notice that $(X' \coloneqq \bigcup_{i \in [r]} X_{i}, Y' \coloneqq \bigcap_{i \in [r]} Y_{i})$ defines a separation in $G'$ of order at most $f_{\ref{@loudspeakers}}(k, θ),$ where 
by the assumptions on $\mathbf{V}$ and the fact that $X \subseteq V(H_{i})$ it is implied that $D \subseteq G'[Y' \setminus X'].$

We next consider the $\Sigma$-schema $(A\cup A', δ',D'')$ of $G$ that is obtained by carving $(A, δ,D)$ by the separation $(X' \cup A, Y' \cup A).$ 
Let $\Tcal^{A}_{D}$ be the tangle of $G'$ induced by the Dyck grid $D\subseteq G'.$
As $D \subseteq G'[Y' \setminus X']$ it follows that $D'' = D$ and we deduce that $(X', Y') \in \Tcal^{A}_{D}.$
We now set $d' = d - f_{\ref{@loudspeakers}}(k, θ).$
Let $D'$ be any $(\Sigma, d')$-Dyck subwall of $D''.$
Then we can apply \cref{@repetitive} and obtain that $(A \cup A', δ',D')$ is $α$-anchored at $S.$

Finally, assume towards a contradiction, that some cell $c' \in C(δ')$ is $\Zcal$-red in $\delta'.$
Since $\mathbf{V}$ is a $\Zcal$-progressive railed flat vortex collection, it must be that $c' \subseteq c^{i}_{\rho},$ for some $i \in [r].$
Notice that this implies the existence of a $c'$-invading $\nonplanar(Z)$-inflated copy $M$ in $G - (A \cup A'),$ where $Z \in \Zcal.$
As $M$ is also a subgraph of $G - A,$ we have that $M$ is also a $c$-invading $\nonplanar(Z)$-inflated copy in $G -A$ where $c$ is the precursor of $c'$ in $δ.$
As $c$ is the precursor of $c',$ it also holds that $c \subseteq c^{i}_{ρ}$ therefore contradicting \eqref{@temptations}.
\end{proof}

We are now ready to present our local structure theorem.

\begin{theorem}
\label{@antiauthoritarian}
There exist functions $f^1_{\ref{@antiauthoritarian}} : \Nbbb^{4} \to \Nbbb,$ $f^2_{\ref{@antiauthoritarian}} : \Nbbb^{4} \to \Nbbb$ such that, for every $\Zcal  \in \mathbb{H}^{-},$ there exists an algorithm that, given
\begin{itemize}
\item $k , t \in \Nbbb_{\geq 1},$
\item a graph $G,$ and
\item an $(s, α)$-well-linked set $S$ of $G,$ where $α \in [2/3, 1)$ and $s \geq f^1_{\ref{@antiauthoritarian}}(γ_{\Zcal},h_{\Zcal},t,k),$
\end{itemize}
outputs either
\begin{itemize}
\item an inflated copy of the Dyck grid $\mathscr{D}_{t}^{\Sigma'},$ where $\Sigma' \in \sobs(\mathbb{S}_{\Zcal}),$ or
\item a mixed $\Zcal$-packing of size $k$ in $G,$ or
\item a $Σ$-schema $(A,δ,D)$ of $G$ in a surface $\Sigma \in \mathbb{S}_{\Zcal}$ such that
\begin{itemize}
\item $|A| \leq f^{2}_{\ref{@antiauthoritarian}}(γ_{\Zcal},h_{\Zcal},t,k),$
\item $(A,δ,D)$ is $\alpha$-anchored at $S,$ and
\item none of the cells of $\delta$ is $\Zcal$-red.
\end{itemize}
\end{itemize}
Moreover the previous algorithm runs in time $2^{2^{\Ocal_{γ_{\Zcal}}(\mathsf{poly}(t)) + \Ocal_{h_{\Zcal}}(k)}}\cdot |V(G)|^3\cdot \log(|V(G)|)$ and the functions $f^1_{\ref{@antiauthoritarian}}$ and $f^2_{\ref{@antiauthoritarian}}$ are of order
\begin{align*}
f^{1}_{\ref{@antiauthoritarian}}(γ_{\Zcal},h_{\Zcal},t,k) &= 2^{\Ocal_{γ_{\Zcal}}(\mathsf{poly}(t)) + \Ocal_{h_{\Zcal}}(k)}\text{ and}\\
f^{2}_{\ref{@antiauthoritarian}}(γ_{\Zcal},h_{\Zcal},t,k) &= 2^{\Ocal_{γ_{\Zcal}}(\mathsf{poly}(t)) + \Ocal_{h_{\Zcal}}(k)}.\\
\end{align*}
\end{theorem}
\begin{proof}
We define 
\begin{align*}
f^1_{\ref{@antiauthoritarian}}(γ_{\Zcal},h_{\Zcal},t,k) &= f^{1}_{\ref{@enlightened}}(γ_{\Zcal},h_{\Zcal}, t, d_{1}, k) = 2^{\Ocal_{γ_{\Zcal}}(\mathsf{poly}(t)) + \Ocal_{h_{\Zcal}}(k)}\text{ and}\\
f^2_{\ref{@antiauthoritarian}}(γ_{\Zcal},h_{\Zcal},t,k) &= f^{2}_{\ref{@enlightened}}(γ_{\Zcal},h_{\Zcal}, t, d_{1}, k) + f_{\ref{@loudspeakers}}(k, θ) = 2^{\Ocal_{γ_{\Zcal}}(\mathsf{poly}(t)) + \Ocal_{h_{\Zcal}}(k)},
\end{align*}
where
\begin{align*}
q &\coloneqq \max\{ f_{\ref{@horkheimer}}(k,h_{\Zcal}), \mathsf{q}(h_{\Zcal}) \} = 2^{\Ocal_{h_{\Zcal}}(k)}\text{,}\\
θ &\coloneqq f_{\ref{@sensualism}}(k, q, h_{\Zcal}) = 2^{\Ocal_{h_{\Zcal}}(k)}\text{, and}\\
d_{1} &\coloneqq f^{2}_{\ref{@approached}}(k, q) + f_{\ref{@loudspeakers}}(k, θ) + 3 = 2^{\Ocal_{h_{\Zcal}}(k)}.
\end{align*}

We begin by applying \cref{@enlightened} on $\Zcal,$ $G,$ $k,$ $t,$ $d_{1}$ and $S.$
In case we get an inflated copy of the Dyck grid $\mathscr{D}^{Σ'}_{t},$ where $\Sigma' \in \sobs(\mathbb{S}_{\Zcal})$ we conclude.
In case we get a mixed $\Zcal$-packing of size $k$, we also conclude.
Otherwise, we get a $Σ$-schema $(A_{1}, δ_{1}, D_{1})$ of $G,$ for some $Σ \in \Sbbb_{\Zcal},$ where $|A_{1}| \leq f^{2}_{\ref{@enlightened}}(γ_{\Zcal},h_{\Zcal},t,k),$ that is $α$-anchored at $S,$ and $D_{1}$ is a $(Σ; d_{1})$-Dyck wall of $G - A$ that is grounded in $δ_{1}.$
This step concludes in time $2^{2^{\Ocal_{γ_{\Zcal}}(\mathsf{poly}(t)) + \Ocal_{h_{\Zcal}}(k)}}\cdot |V(G)|^3\cdot \log(|V(G)|).$

Next we apply \cref{@approached} on $\Zcal,$ $G,$ $k,$ $q,$ and $(A_{1}, δ_{1}, D_{1}).$
In case we get a mixed $\Zcal$-packing of size $k$ in $G - A$ we conclude.
Otherwise we get a $Σ$-schema $(A_{1}, δ_{1}, D_{2})$ that is $α$-anchored at $S,$ where $D_{2}$ is a $(Σ; d_{2})$-Dyck subwall of $D_{1}$ of order $d_{2} \geq d_{1} - f^{2}_{\ref{@approached}}(k, q),$ and a railed flat vortex collection $\mathbf{V}_{1}$ in $δ_{1}$ that is $\Zcal$-progressive has size less than $k$ and order $q$ such that for every $(Δ, ρ) \in \mathbf{V}_{1},$ $\inG_{δ_{1}}(Δ) \cap D_{2}$ is the empty graph.
This step concludes in time $\Ocal_{h_{\Zcal}}(|V(G)|^3)+\poly(q,k)\cdot |V(G)|.$

Next we apply \cref{@sensualism} on $\Zcal,$ $G,$ $k,$ $q,$ $(A_{1}, δ_{1}, D_{2}),$ and $\mathbf{V}_{1}.$
In case we get a mixed $\Zcal$-packing of size $k$ in $G - A$ we conclude.
Otherwise we get a $\Zcal$-progressive railed flat vortex collection $\mathbf{V}_{2}$ in $\delta_{1}$ of size less than $k$ and order $q$ such that for every $(Δ, ρ) \in \mathbf{V}_{2},$ the flat vortex society $(H, \Omega)$ of $(\Delta, \rho)$ has depth at most $\theta$ and $H \cap D$ is the empty graph.
This step concludes in time $\Ocal(kq^{2} \cdot |V(G)|^3).$

Finally, we apply \cref{@loudspeakers} on $\Zcal,$ $G,$ $k,$ $q,$ $θ,$ $(A_{1}, δ_{}, D_{2}),$ and $\mathbf{V}_{2}.$
In case we get a mixed $\Zcal$-packing of size $k$ in $G - A$ we conclude.
Otherwise we get a $Σ$-schema $(A_{1} \cup A_{2}, δ_{2}, D_{3})$ of $G,$ where $|A_{2}| \leq f_{\ref{@loudspeakers}}(k, θ),$ that is $α$-anchored at $S,$ the order of $D_{3}$ is $d_{3} \geq d_{2} - f_{\ref{@loudspeakers}}(k, θ),$ and none of the cells of $δ_{2}$ are $\Zcal$-red.
This step concludes in time $2^{\Ocal(k)\cdot\poly(h_{\Zcal})}|V(G)|+\Ocal_{h_{\Zcal}}(k \cdot |V(G)|^3).$

We conclude the proof with the $Σ$-schema $(A_{1} \cup A_{2}, δ_{2}, D_{3})$ of $G,$ where $|A_{1} \cup A_{2}| \leq f^{2}_{\ref{@enlightened}}(γ_{\Zcal},h_{\Zcal},t,k) + f_{\ref{@loudspeakers}}(k, θ) = f^2_{\ref{@antiauthoritarian}}(γ_{\Zcal},h_{\Zcal},t,k).$
\end{proof}

\section{The global structure theorem}
\label{@persecuting}

This section contains the last missing pieces of the main proof and concludes with the proof of \cref{@indescribably}, in the form of \cref{@indescribably} in this section.
These include the proof of a global structure theorem (\cref{@malcontents}) following standard balanced separator arguments, (\cref{@liabilities})
the derivation of the three outcomes of \cref{@indescribably}, exploiting the tree decomposition obtained from the global structure theorem, for the case of connected graphs in $\Zcal$ (\cref{@analogistic} and \cref{@restrained}), 
and, in conclusion,
the treatment of the disconnected case (\cref{@expurgated}). 

\subsection{From local to global}
\label{@liabilities}

The goal of this subsection is to prove a global version of the structure theorem we gave in \cref{@antiauthoritarian} in the form of \cref{@malcontents}.

\medskip
For this we first introduce the concept of a $\Zcal$-\textsl{local cover} for a given antichain $\Zcal \in \Hbbb^{-}.$

\paragraph{$\Zcal$-local covers.}
Given a graph $G$ and an antichain $\Zcal \in \Hbbb^{-},$ we say that a pair $(X, A)$ where $A \subseteq X \subseteq V(G)$ is a $\Zcal$-\emph{local cover} of $G$ if for every $\nonplanar(Z)$-inflated copy $M$ in $G,$ where $Z \in \Zcal,$
$$M \cap X \neq \emptyset \Rightarrow M \cap A \neq \emptyset.$$
The \emph{order} of a $\Zcal$-local cover $(X,A)$ is $|A|.$

\medskip
The proof of \cref{@malcontents} follows from standard balanced separator arguments and allows us to show that given an antichain $\Zcal \in \mathbb{H}^{-}$ and a graph $G$, in the absence of a large order Dyck grid minor corresponding to a surface in $\sobs(\mathbb{S}_{\Zcal})$ and of a large mixed $\Zcal$-packing in $G,$ we can compute a tree-decomposition $(T, \beta)$ of $G$ of bounded adhesion, where each bag $\beta(t)$ is accompanied by a set $\alpha(t)$ of bounded size, such that $(\beta(t), \alpha(t))$ is a $\Zcal$-local cover.
Formally we prove the following.

\begin{theorem}\label{@malcontents}
There exist functions $f^1_{\ref{@malcontents}} \colon \Nbbb^{4} \to \Nbbb,$ $f^2_{\ref{@malcontents}} \colon \Nbbb^{4} \to \Nbbb$ such that, for every $\Zcal \in \mathbb{H}^{-},$ there exists an algorithm that, given $k, t \in \Nbbb_{\geq 1}$ and a graph $G,$ outputs either
\begin{itemize}
\item an inflated copy of the Dyck grid $\mathscr{D}_{t}^{\Sigma'},$ where $\Sigma' \in \sobs(\mathbb{S}_{\Zcal}),$ or
\item a mixed $\Zcal$-packing of size $k$ in $G,$ or
\item a tree decomposition $(T, \beta)$ of $G$ with adhesion at most $f^{1}_{\ref{@malcontents}}(\gamma_{\Zcal},h_{\Zcal},t,k)$ and a function $\alpha \colon V(T) \to 2^{V(G)}$ such that, for every $t \in V(T),$ $(\beta(t), \alpha(t))$ is a $\Zcal$-local cover of order at most $f^2_{\ref{@malcontents}}(\gamma_{\Zcal}, h_{\Zcal}, t, k).$
\end{itemize}
Moreover the previous algorithm runs in time $2^{2^{\Ocal_{\gamma_{\Zcal}}(\mathsf{poly}(t)) + \Ocal_{h_{\Zcal}}(k)}} \cdot |V(G)|^4 \cdot \log(|V(G)|)$ and the functions $f^1_{\ref{@malcontents}}$ and $f^2_{\ref{@malcontents}}$ are of order
\begin{align*}
f^{1}_{\ref{@antiauthoritarian}}(\gamma_{\Zcal}, h_{\Zcal}, t, k) &= 2^{\Ocal_{γ_{\Zcal}}(\mathsf{poly}(t)) + \Ocal_{h_{\Zcal}}(k)}\text{ and}\\
f^{2}_{\ref{@antiauthoritarian}}(\gamma_{\Zcal}, h_{\Zcal}, t, k) &= 2^{\Ocal_{γ_{\Zcal}}(\mathsf{poly}(t)) + \Ocal_{h_{\Zcal}}(k)}.\\
\end{align*}
\end{theorem}
\begin{proof} We define $f^{1}_{\ref{@malcontents}}(\gamma_{\Zcal}, h_{\Zcal}, t, k) \coloneqq 3g(\gamma_{\Zcal}, h_{\Zcal}, t, k)$ and $f^{2}_{\ref{@malcontents}}(\gamma_{\Zcal}, h_{\Zcal}, t, k) \coloneqq 4g(\gamma_{\Zcal}, h_{\Zcal}, t, k) + 1,$ where we define the function $g \colon \Nbbb^{4} \to \Nbbb$ so that $g(\gamma_{\Zcal}, h_{\Zcal}, t, k) = \max\{ f^1_{\ref{@antiauthoritarian}}(γ_{\Zcal},h_{\Zcal},t,k), f^2_{\ref{@antiauthoritarian}}(γ_{\Zcal},h_{\Zcal},t,k) + 3 \}.$

We prove the following stronger statement that immediately implies our claim.
We show that for every set $X \subseteq V(G)$ such that $|X| \leq 3g(\gamma_{\Zcal}, h_{\Zcal}, t, k) + 1$ there exists either
\begin{itemize}
\item an inflated copy of the Dyck grid $\mathscr{D}_{t}^{\Sigma'},$ where $\Sigma' \in \sobs(\mathbb{S}_{\Zcal}),$ or
\item a mixed $\Zcal$-packing of size $k$ in $G,$ or
\item a rooted tree decomposition $(T, r, \beta)$ of $G$ with adhesion at most $3g(\gamma_{\Zcal}, h_{\Zcal}, t, k) + 1$ such that $X \subseteq \beta(r)$ and a function $\alpha \colon V(T) \to 2^{V(G)}$ such that, for every $t \in V(T),$ $(\beta(t), \alpha(t))$ is a $\Zcal$-local cover of order at most $4g(\gamma_{\Zcal}, h_{\Zcal}, t, k) + 1.$
\end{itemize}
We prove the statement above by induction on $|V(G) \setminus X|.$

In the trivial case where $|V(G)| \leq 3g(\gamma_{\Zcal}, h_{\Zcal}, t, k) + 1$ we can immediately conclude with a tree decomposition with a single bag $\beta(r)$ that contains every vertex of $G$ and where $\alpha(t) = \beta(t).$
Then, we can assume that $|V(G)| > 3g(\gamma_{\Zcal}, h_{\Zcal}, t, k) + 1.$
Also, in the case that $|X| < 3g(\gamma_{\Zcal}, h_{\Zcal}, t, k) + 1,$ since $|V(G)| > 3g(\gamma_{\Zcal}, h_{\Zcal}, t, k) + 1,$ there is a vertex $v \in V(G) \setminus X.$
In this case we can set $X' \coloneqq X \cup \{ v \}$ and apply induction since $|V(G) \setminus X'| < |V(G) \setminus X|.$
Therefore we can assume that $|X| = 3g(\gamma_{\Zcal}, h_{\Zcal}, t, k) + 1$ and that $|V(G)| > 3g(\gamma_{\Zcal}, h_{\Zcal}, t, k) + 1.$
We now distinguish between the two main cases that will guide how we build our rooted tree decomposition.

\paragraph{There exists a ``small'' $\nicefrac{2}{3}$-balanced separator for $X$.}

Assume that there exists a $\nicefrac{2}{3}$-balanced separator $S$ for $X$ in $G$ such that $|S| \leq g(\gamma_{\Zcal}, h_{\Zcal}, t, k)$ which we can compute in time $2^{\Ocal(g(\gamma_{\Zcal}, h_{\Zcal}, t, k))}\cdot |G|$ (see \cite{Reed92finding}).
In this case we can define the desired rooted tree decomposition $(T, r, \beta)$ as follows.
We define $\beta(r) \coloneqq X \cup S$ and $\alpha(r) \coloneqq \beta(r).$
Then, let $J$ be any connected component in $G - S.$
Let $X_{J} \coloneqq V(J) \cap X$ and observe that since $S$ is a $\nicefrac{2}{3}$-balanced separator for $X$ it holds that $|X_{J}| \leq 2g(\gamma_{\Zcal}, h_{\Zcal}, t, k).$
Moreover, let $X_{J}' \coloneqq X_{J} \cup S$ and note that $|X_{J}'| \leq 3g(\gamma_{\Zcal}, h_{\Zcal}, t, k).$
Observe that, if $V(J) \setminus X_{J}' = \emptyset$ we can immediately conclude as all vertices of $J$ are already part of the root bag.
Therefore, there exists a vertex $v_{J} \in V(J) \setminus X'_{J}$.
Let $X''_{J} \coloneqq X'_{J} \cup \{ v_{J} \}.$
Then, $|X''_{J}| \leq 3g(\gamma_{\Zcal}, h_{\Zcal}, t, k) + 1$ and observe that $|(V(J) \cup S) \setminus X''_{J}| < |V(G) \setminus X|$ and we can apply induction on $G[V(J) \cup S]$ and $X_{J}''$ to obtain a rooted tree-decomposition $(T_{J}, r_{J}, \beta_{J})$ of $G[V(J) \cup S]$ with the desired properties.
We can now define our rooted tree-decomposition for $G$ by making the root nodes $r_{J}$ adjacent to $r,$ for every connected component $J$ of $G - S.$
Also, note that by definition, $\beta(r) \cap \beta(r_{J}) = X'_{J}$ and therefore the adhesion constraint is satisfied.

\paragraph{There is no ``small'' $\nicefrac{2}{3}$-balanced separator for $X$.}

Otherwise, we are in the case where there is no $\nicefrac{2}{3}$-balanced separator for $X$ in $G$ of size at most $g(\gamma_{\Zcal}, h_{\Zcal}, t, k)$ and therefore, by definition, $X$ is a $(g(\gamma_{\Zcal}, h_{\Zcal}, t, k), \nicefrac{2}{3})$-well-linked set of $G.$
Since $g(\gamma_{\Zcal}, h_{\Zcal} \geq f^1_{\ref{@antiauthoritarian}}(γ_{\Zcal},h_{\Zcal},t,k),$ we are now in the position to call upon \cref{@antiauthoritarian} with $k,$ $t,$ $G,$ and $X.$
There are three outcomes.
In case we obtain either an inflated copy of the Dyck grid $\mathscr{D}^{\Sigma'}_{t},$ where $\Sigma' \in \sobs(\Sbbb_{\Zcal}),$ or a mixed $\Zcal$-packing of size $k$ in $G$ we can immediately conclude.
Therefore we are left with the case where we obtain a $\Sigma$-schema $(A, \delta, D)$ of $G$ in a surface $\Sigma \in \Sbbb_{\Zcal}$ such that
\begin{itemize}
\item $|A| \leq f^{2}_{\ref{@antiauthoritarian}}(\gamma_{\Zcal}, h_{\Zcal}, t, k) \leq g(\gamma_{\Zcal}, h_{\Zcal}, t, k),$
\item $(A,δ,D)$ is $\nicefrac{2}{3}$-anchored at $S,$ and
\item none of the cells of $\delta$ is $\Zcal$-red.
\end{itemize}
In this case we can define the desired rooted tree decomposition $(T, r, \beta)$ as follows.
We define $\beta(r) \coloneqq \ground(\delta) \cup A \cup X$ and $\alpha(r) \coloneqq A \cup X.$
Notice that $|\alpha(r)| \leq 4g(\gamma_{\Zcal}, h_{\Zcal}, t, k) + 1$ as required.
Then, let $J^{c}$ be the graph $\sigma_{\delta}(c),$ for any cell $c \in C(\delta).$
Let $X_{J^{c}} \coloneqq V(J^{c}) \cap X.$
Since $|A \cup \pi_{\delta}(\tilde{c})| \leq f^{2}_{\ref{@antiauthoritarian}}(\gamma_{\Zcal}, h_{\Zcal}, t, k) + 3 = g(\gamma_{\Zcal}, h_{\Zcal}, t, k)$ and $(A, \delta, D)$ is $\nicefrac{2}{3}$-anchored it follows that $|X_{J^{c}}| < \nicefrac{1}{3}|X| \leq g(\gamma_{\Zcal}, h_{\Zcal}, t, k).$
Let $X'_{J^{c}} \coloneqq X_{J^{c}} \cup A \cup \pi_{\delta}(\tilde{c})$ and notice that $|X'_{J^{c}}| \leq 2g(\gamma_{\Zcal}, h_{\Zcal}, t, k).$
Now, as before, if $V(J^{c}) \setminus X'_{J^{c}} = \emptyset$ we can immediately conclude as all vertices of $J^{c}$ are already part of the root bag.
Therefore, there exists a vertex $v_{J^{c}} \in V(J^{c}) \setminus X'_{J^{c}}.$
Let $X''_{J^{c}} \coloneqq X'_{J^{c}} \cup \{ v_{J^{c}} \}.$
Then, $|X''_{J^{c}}| \leq 2g(\gamma_{\Zcal}, h_{\Zcal}, t, k) + 1$ and observe that $|(V(J^{c}) \cup A) \setminus X''_{J^{c}}| < |V(G) \setminus X|$ and we can apply induction on $G[V(J^{c}) \cup A]$ and $X_{J^{c}}''$ to obtain a rooted tree-decomposition $(T_{J^{c}}, r_{J^{c}}, \beta_{J^{c}})$ of $G[V(J^{c}) \cup A]$ with the desired properties.
We can now define our rooted tree-decomposition for $G$ by making the root nodes $r_{J^{c}}$ adjacent to $r,$ for every subgraph $J^{c},$ $c \in C(\delta).$
Also, note that by definition, $\beta(r) \cap \beta(r_{J^{c}}) = X'_{J^{c}}$ and therefore the adhesion constraint is satisfied.

To conclude it remains to show that $(\beta(r), \alpha(r))$ is a $\Zcal$-local cover.
Since no cell of $\delta$ is $\Zcal$-red, by definition this implies that there is no $c$-invading $\nonplanar(Z)$-inflated copy in $G - A,$ for any cell $c \in C(\delta)$ and any $Z \in \Zcal.$
Then, by \cref{@primitives}, for every $\nonplanar(Z)$-inflated copy in $G - A,$ where $Z \in \Zcal,$ it must be the case that $V(M) \cap \ground(\delta) = \emptyset.$
This implies that every $\nonplanar(Z)$-inflated copy in $G,$ where $Z \in \Zcal,$ that intersects $\beta(t)$ must, by definition of $\beta(t)$ intersect the set $A \cup X = \alpha(r).$
Therefore $(\beta(r), \alpha(r))$ is a $\Zcal$-local cover as desired.

\medskip
As for the running time of the claimed algorithm, notice that it is dominated by the number of times we call upon \cref{@antiauthoritarian} which in the worst case may be $|V(G)|$ many times.
Therefore the desired running times is satisfied.
\end{proof}

\subsection{Auxiliary algorithmic results}
\label{@analogistic}

In this subsection we present some known algorithmic results that will be used as subroutines in our final algorithm.

\medskip
The following states that every planar graph $H$ can be found as a minor of a wall of size quadratic in the size of $H.$
It can be seen as a special case of \cref{@headstrong}.

\begin{proposition}[\!\! \cite{RobertsonST94Quickly}]\label{@enlightening}
There exists some constant $c$  
such that every planar graph $H$ on $h$ vertices is a minor of a $(c \cdot |V(H)|^2 \times c \cdot |V(H)|^2)$-wall.
\end{proposition}

\begin{lemma}\label{@horklieimer}
There exists a function $f_{\ref{@horklieimer}} \colon \Nbbb^2 \to \Nbbb$ and an algorithm that, given an $n$-vertex graph $G,$ an $h$-vertex planar graph $H,$ and $k \in \Nbbb,$ outputs either
\begin{itemize}
\item an $H$-packing of size $k$ in $G$ or 
\item a tree decomposition of $G$ of width at most $f_{\ref{@horklieimer}}(h,k).$
\end{itemize} 
Moreover, the algorithm above runs in time $2^{\poly(h,k)} |V(G)| \log(|V(G)|$ and $f_{\ref{@horklieimer}}(h,k) = \Ocal(h^{20} \cdot k^{5}).$
\end{lemma}
\begin{proof}
Let  $g = \lceil \sqrt{k} \rceil \cdot c_1 \cdot |V(H)|^2,$ where $c_{1}$ is the constant of \cref{@enlightening} and observe that if we have a $g$-wall in $G,$ then we can also find a $(k \cdot H)$-\major in $G.$
We now run the algorithm of \cref{@reestablish} with input $G$ and $g.$
This permits us to either output a $(k\cdot H)$-\major in $G$ or to obtain a tree decomposition of $G$ of width at most $c_2 \cdot g^{10},$ where $c_2$ is the constant of \cref{@reestablish}. 
\end{proof}

\begin{lemma}\label{@exonerates}
There exists an algorithm that, given an $n$-vertex graph $G,$ an $h$-vertex graph $H,$ and a tree decomposition of $G$ of width $w$ outputs either an $H$-packing in $G$ of size $k$ or an $H$-cover of $G$ of size $wk$ in time $2^{\poly_{h}(k,w)} \cdot |V(G)|.$
\end{lemma}
\begin{proof}
With such a tree decomposition at hand, we may find, if it exists, a $(k\cdot H)$-\major of $G$ in time $2^{\Ocal(w)}(hk)^{2w} 2^{\Ocal(hk)} \cdot |V(G)|,$ using, e.g., the minor-checking algorithm of \cite{AdlerDFST11Faster}.
If this is not the case, then it is known that $G$ has a $H$-cover of size at most $wk$ and a minimal such cover can be found in time $2^{\Ocal_{h}(w \log(w))} \cdot |V(G)|$ using the algorithm of~\cite{BasteST20acomp,BasteST19Hitting}. \end{proof}

A direct consequence of \cref{@horklieimer} and \cref{@exonerates} is the following.
\begin{corollary}\label{@flattering}
There exists an algorithm that, given a graph $G,$ an $h$-vertex planar graph $H,$ and $k \in \Nbbb,$ outputs either an $H$-packing of size $k$ in $G$ or an $H$-cover of $G$ of size $k \cdot f_{\ref{@horklieimer}}(h,k)$ in time $2^{\poly_{h}(k)} \cdot |V(G)|.$
\end{corollary}

\subsection{From $\Zcal$-local covers to $\Zcal$-local cover separations}\label{@restrained}

In this subsection we show how to transform a tree decomposition as the one in the second output of \cref{@malcontents} to a separation $(X,Y),$ where $(X, X\cap Y)$ is a $\Zcal$-local cover and $G[Y\setminus X]$ is $\Zcal$-minor free, i.e., $\Zcal \not\leq G$.

\begin{lemma}\label{@contradictory}
For every antichain $\Zcal \in \Hbbb^{-},$ there is an algorithm that, given $k, d \in \Nbbb,$ a graph $G,$ a tree decomposition $(T,\beta)$ of adhesion at most $z,$ and a function $\alpha \colon V(T) \to 2^{V(G)},$ where for every $t \in V(T),$ $(\beta(t), \alpha(t))$ is a $\Zcal$-local cover of $G$ of order at most $d,$ outputs either
\begin{enumerate}
\item a mixed $\Zcal$-packing of size $k$ in $G$ or
\item a $\Zcal$-local cover $(X, A)$ of $G$ such that 
\begin{itemize}
\item $G - X$ is $\Zcal$-minor free,
\item $|A| ≤ (k-1)\cdot (z+d),$
\item for every connected component $C$ of $G-X$ it holds that $|N_{G}(V(C))| \leq z$
\end{itemize}
\end{enumerate}
Moreover, the algorithm above runs in time $\Ocal_{h_{\Zcal}}(|V(G)|^3).$
\end{lemma}
\begin{proof}
For two adjacent nodes $t_{1}$ and $t_2$ of $T,$ we denote by $T_{t_1}$ and $T_{t_2}$ the components of $T-{t_1t_2}$ containing $t_1$ and $t_2$ respectively. 

Then, for $i \in [2],$ $G_{t_i}$ denotes the subgraph induced by $V_{t_{i}} = \cupall \{\beta(t') \mid t'\in V(T_{t_{i}})\}\setminus \beta(t_{3-i}).$
We define, for $i \in [2],$ $\beta_{t_{i}} = (T_{t_{i}}, \beta_{t_{i}})$ where $\beta_{t_{i}}(t) = \beta(t) \cap V_{t_{i}},$ for $t \in V(T_{t_{i}}),$ and define $\alpha_{t_{i}} \colon V(T_{t_{i}})\to 2^{V_{t_{i}}}$ such that $\alpha_{t_{i}}(t) = (\alpha(t) \cap V_{t_{i}}),$ for $t \in V(T_{t_{i}}).$

Initially, we set $X \coloneqq A \coloneqq \emptyset.$
If $(T, \beta)$ contains two adjacent nodes $t_1$ and $t_2$ such that each of the two subgraphs $G_{t_1}$ and $G_{t_2}$ contains a $\Zcal$-\major $M,$ then we set $X \coloneqq A \coloneqq A \cup (\beta(t_1) \cap \beta(t_2))$ and recurse, for $i \in [2],$ in the tree decomposition $(T_{t_i}, \beta_{t_i})$ of $G_{t_{i}}$ along with the function $\alpha_{t_{i}}.$
Eventually, we either find a mixed $\Zcal$-packing of $G$ of size $k$ or we obtain at most $k - 1$ pairwise disjoint subgraphs $G_1, \dots G_{\ell}$ of $G - X,$ together with their corresponding tree decompositions $(T^1, \beta^1),$ \dots, $(T^{\ell}, \beta^{\ell})$ and their corresponding functions in $\{ \alpha^i \colon V(T^i) \to 2^{V(G_i)} \mid i \in [\ell] \}.$
Observe that after the above procedure completes, we have $|A| \leq (k-1) \cdot z$ and that, for every $i \in [\ell],$ $\Zcal \leq G_{i}$.

For every $i \in [\ell],$ we now process the tree decomposition $(T^i, \beta^i)$ as follows.
Suppose that $(T^i, \beta^i)$ contains an adhesion $\beta^i(a) \cap \beta^i(b)$ for some $ab \in E(T^i),$ such that $\Zcal \not \leq G_i - (\beta^i(a) \cap \beta^i(b)).$
Then we update $A \coloneqq A \cup (\beta^i(a) \cap \beta^i(b))$ and $X \coloneqq X \cup V(G_i).$
Otherwise, if such an adhesion does not exist, this implies that for every edge $ab \in E(T^i)$ only one of $G^i_a$ or $G^i_{b}$ contains some $H \in \Zcal$ as a minor.
Therefore we may consider an orientation of the edges of $T^{i}$ where every each is oriented towards the side that contains some $H \in \Zcal$ as a minor.
This implies that in this orientation of $(T^i,\beta^i)$, there there must be a sink node $t^i$ such that for every adhesion $\beta^i(a^i) \cap \beta^i(t^i),$ with $a^i t^i \in E(T^i),$ the subgraph $G^i_{a^i}$ of $G^i$ is $\Zcal$-minor free.
Οbserve that $(\beta^{i}(t^i), \alpha^{i}(t^i))$ is a $\Zcal$-local cover for $G^i$ and update $A \coloneqq A \cup \alpha^{i}(t^i)$ and $X \coloneqq X\cup \beta^{i}(t^i).$
It follows that the resulting pair $(X,A)$ is a $\Zcal$-local cover of $G$ and that $G - X$ is $\Zcal$-minor free.
Moreover, by construction of $A$ and $X,$ we have that $|A| \leq (k-1)\cdot (z+d).$
Finally, for every connected component $C$ of $G - X,$ $N_G(V(C)) \setminus A$ is a subset of some adhesion of $(T, \beta)$ implying that $|N_{G}(V(C))| \leq z.$

The time complexity follows from the fact that, in the initial step, every adhesion set is processed once.
When processing an adhesion set $\beta(a) \cap \beta(b),$ we call the minor checking algorithm for every $H \in \Zcal$ in $G_a$ and in $G_b.$
By \cite{KawarabayashiKR12Thedisjoint}, this requires $\Ocal_{h_{\Zcal}}(|V(G)|^2)$ time.
If the minor checking succeeds on $G_{t_1}$ and $G_{t_2},$ then we move $\beta(t_{1}) \cap \beta(t_{2})$ and update the tree decomposition to ``separate'' $(T_{t_1}, \beta_{t_2})$ from $(T_{t_{1}}, \beta_{t_2}).$
This requires $\Ocal(|V(G)|)$ time.
Otherwise the adhesion set $\beta(a) \cap \beta(b)$ is declared ``dead'' and it will never be used again.
In the second step, either we identify an adhesion set or a bag.
Again for each tree decomposition $(T^i, \beta^i)$ this requires $\Ocal(|V(G)|)$ calls to the minor checking algorithm for each $H \in \Zcal$ in $G^i.$
It follows that the overall running time is $\Ocal_{h_{\Zcal}}(|V(G)|^3).$
\end{proof}

Now, given an antichain $\Zcal\in\Hbbb^-,$ we define $\Kcal_{\Zcal} \coloneqq \{\nonplanar(H)\mid H\in\Zcal\}$ and observe that $|\Kcal_{\Zcal}|≤|\Zcal|.$

\begin{lemma} \label{@unruliness}
For every $\Zcal\in\Hbbb^-,$ there exists an algorithm that, given a graph $G,$ a separation $(A,B)$ of $G$ of order $s,$ a mixed $\Kcal_\Zcal$-packing $\{M_{1}^{A}, \ldots, M_{k}^{A}$\} of $G[A \setminus B],$ and a mixed $\Kcal_\Zcal$-packing $\{M_{1}^{B}, \ldots, M_{k}^{B}\}$ of $G[B\setminus A]$ outputs either
\begin{itemize}
\item a mixed $\Zcal$-packing of size $k$ in $G$ or 
\item a $\Zcal$-cover of $G$ of size $f_{\ref{@horklieimer}}(h_{\Zcal},k) + s = \Ocal(h^{20}\cdot k^{5}) + s.$
\end{itemize}
Moreover, the algorithm above runs in time $2^{\poly_{h_{\Zcal}}(k,s)} \cdot |V(G)|\cdot \log(|V(G)|).$
\end{lemma}
\begin{proof}
Let $P_1^A, \dots, P_k^A$ be planar graphs such that for every $1 \leq i \leq k,$ $M_i^B + P^A_i = H_i \in \Zcal.$
Consider also the planar graph $P^{A} = P_1^A + \dots + P_k^A.$
By applying the algorithm of \cref{@horklieimer}, we either find a $P^{A}$-\major in $G[A \setminus B]$ that along with $M_1^B + \dots + M_k^B,$ forms a mixed $\Zcal$-packing of $G$ of size $k,$ or returns a tree decomposition of $G[A \setminus B]$ of width $f_{\ref{@horklieimer}}(h, k) = \Ocal(h^{20}\cdot k^{5}).$
We work in the same way on $G[A \setminus B]$ and again we either find a mixed $\Zcal$-packing of $G$ or a tree decomposition of $G[B \setminus A]$ of width $\Ocal(h^{20}\cdot k^{5}).$
Now,  we can clearly combine the two tree decomposition in an a tree decomposition of $G$ of width $\Ocal(h^{20}\cdot k^{5}) + s$ and we can conclude by applying \autoref{@exonerates} to either find an $H$-\major with $H = H_1 + \dots + H_k$ or an $H$-cover of size $\Ocal(h^{20}\cdot k^{5}) + s.$
\end{proof}

Given a graph $G$ and a separation $(X,Y),$ we say that $(X,Y)$ is a $\Zcal$-\emph{local cover separation} of $G$ if $(X, X\cap Y)$ is a $\Zcal$-local cover and $\Zcal \not\leq G[Y\setminus X].$

\begin{lemma}\label{@wrongfully} 
For every $\Zcal \in \Hbbb^-,$ there exists an algorithm that, given $k, d \in \Nbbb,$ a graph $G,$ and a $\Zcal$-local cover $(X,A)$ of $G$ such that
\begin{itemize}
\item $G - X$ is $\Zcal$-minor free and
\item for every connected component $C$ of $G-X$ it holds that $|N_{G}(V(C))| \leq d,$
\end{itemize}
outputs either
\begin{enumerate}
\item a mixed $\Zcal$-packing of size $k$ in $G,$
\item a $\Zcal$-cover of $G$ of size $f_{\ref{@horklieimer}}(h_{\Zcal},k)+ kd = \Ocal(h^{20}\cdot k^{5})+kd,$ or
\item a $\Zcal$-local cover separation $(X', Y')$ of $G$ of order at most $|A| + (2k -1) \cdot d.$
\end{enumerate}
Moreover, the algorithm above runs in time $2^{\poly_{h_{\Zcal}}(k,d)} \cdot |V(G)| \log(|V(G)|).$
\end{lemma}
\begin{proof}
Let $\Ccal$ be the set of connected components of $G - X.$
We compute a partition $\{\Ccal^0, \Ccal^1\}$ of $\Ccal$ where $\Ccal^{1}$ contains the connected components in $\Ccal$ that contain some graph in $\Kcal_\Zcal$ as a minor.
The partition $\{\Ccal^0, \Ccal^1\}$ can be computed in time $\Ocal_{h_{\Zcal}}(V(G)|^2)$ using the quadratic minor testing algorithm of Kawarabayashi, Kobayashi, and Reed \cite{KawarabayashiKR12Thedisjoint}.
If $|\Ccal^1| \geq 2k$ then $G$ contains a separation $(A,B)$ of order at most $k \cdot d$ such that both $G[A \setminus B]$ and $G[B \setminus A]$ contain a mixed $\Kcal_\Zcal$-packing of size $k.$
So the algorithm of \cref{@unruliness} applies for $s = k \cdot d$ and we are done.
 
In the case that $|\Ccal^1| \leq 2k-1,$ we set $X' = X \bigcup_{C\in\Ccal^{0}}V(C),$ $A' = A \cup \bigcup_{C\in\Ccal^1}N_{G}(V(C))$ and $Y'=A'\cup (V(G)\setminus X').$
Observe that $(X',Y')$ is a $\Zcal$-local cover separation of $G$ of order at most $|A| + (2k-1) \cdot d.$
The fact that $G[Y' \setminus X'] = G - X'$ is $\Zcal$-minor free follows directly from the fact that $G-X$ is $\Zcal$-minor free.
\end{proof}

\subsection{Proof of the upper bound}\label{@expurgated}

In this subsection we conclude the proof of \cref{@indescribably}.

\medskip
By combining \cref{@malcontents}, \cref{@contradictory}, and \cref{@wrongfully} we have the following.

\begin{corollary}\label{@disreputable}
There exists a function $f_{\ref{@disreputable}} \colon \mathbb{N}^4 \to \mathbb{N}$ and an algorithm as follows:

\medskip
\noindent {\sl Algorithm} ${\bf Alg \mbox{-}Main}(\Zcal, G, t, k)$

\noindent {\sl Input}: $\Zcal\in\Hbbb^-,$ a graph $G,$ and two integers $t, k \in \Bbb{N}.$

\noindent {\sl Output}: either
\begin{itemize}
\item a $\mathscr{D}^{Σ}_{t}$-\major of $G,$ for some $Σ \in \sobs(\mathbb{S}_{\Zcal}),$
\item a mixed $\Zcal$-packing in $G$ of size at most $k,$
\item a $\Zcal$-cover of $G$ of size at most $f_{\ref{@disreputable}}(γ_{\Zcal},h_{\Zcal},t,k),$ or
\item a $\Zcal$-local cover separation $(X,Y)$ of $G$ of order at most $f_{\ref{@disreputable}}(\gamma_{\Zcal}, h_{\Zcal}, t, k).$
\end{itemize}
Moreover, the previous algorithm runs in time $2^{2^{\Ocal_{γ_{\Zcal}}(\mathsf{poly}(t)) + \Ocal_{h_{\Zcal}}(k)}}\cdot |V(G)|^4 \cdot \log(|V(G)|)$ and  $f_{\ref{@disreputable}}(\gamma_{\Zcal}, h_{\Zcal}, t, k)$ is of order $2^{\Ocal_{\gamma_{\Zcal}}(\mathsf{poly}(t)) + \Ocal_{h_{\Zcal}}(k)}.$
\end{corollary}
\begin{proof}
By calling upon \cref{@malcontents} and given that we do not have one of its two first outcomes as a result, we may assume to obtain a tree decomposition $(T, \beta)$ of adhesion at most $ z \coloneqq f^{1}_{\ref{@malcontents}}(\gamma_{\Zcal}, h_{\Zcal}, t, k)$ and a function $\alpha \colon V(T) \to 2^{V(G)},$ where for every $t \in V(T),$ $(\beta(t), \alpha(t))$ is a $\Zcal$-local cover of $G$ of order at most $d \coloneqq f^{2}_{\ref{@malcontents}}(\gamma_{\Zcal}, h_{\Zcal}, t, k).$
We next apply \cref{@contradictory}.
In case we obtain its first outcome, we can conclude.
Otherwise, we obtain a $\Zcal$-local cover $(X, A)$ of $G$ such that $G - X$ is $\Zcal$-minor free, $|A| \leq (k-1) \cdot (z+d),$ and, moreover, for every connected component $C$ of $G - X,$ it holds that $|N_{G}(V(C))| \leq z.$

Applying now the algorithm of \cref{@wrongfully} we obtain either
\begin{enumerate}
\item a mixed $\Zcal$-packing of size $k$ in $G,$
\item a $\Zcal$-cover of $G$ of size $f_{\ref{@horklieimer}}(h_{\Zcal},k)+ kd = \Ocal(h_{\Zcal}^{20}\cdot k^5) + k \cdot d = 2^{\Ocal_{γ_{\Zcal}}(\mathsf{poly}(t)) + \Ocal_{h_{\Zcal}}(k)},$ or
\item a $\Zcal$-local cover separation $(X',Y')$ of $G$ of order at most $(k-1) \cdot (z+d) + (2k-1) \cdot d = 2^{\Ocal_{γ_{\Zcal}}(\mathsf{poly}(t)) + \Ocal_{h_{\Zcal}}(k)}.$
\end{enumerate}
By choosing $f_{\ref{@disreputable}} \colon \mathbb{N}^4 \to \mathbb{N}$ so that $f_{\ref{@disreputable}}(γ_{\Zcal},h_{\Zcal},t,k) = \max\{ f_{\ref{@horklieimer}}(h_{\Zcal}, k) + kd, (k-1) \cdot (z+d) + (2k-1) \cdot d \},$ we finally have that $f_{\ref{@disreputable}}(γ_{\Zcal},h_{\Zcal},t,k) = 2^{\Ocal_{γ_{\Zcal}}(\mathsf{poly}(t)) + \Ocal_{h_{\Zcal}}(k)}.$

The running time follows by observing that the running time of \cref{@malcontents} dominates those of \cref{@contradictory}, and \cref{@wrongfully}.
\end{proof}

\begin{observation}\label{@unequivocally}
If $H$ is a \textsl{connected} non-planar Kuratowski-connected graph and $(X,Y)$ is an $H$-local cover separation of $G,$ then $X \cap Y$ is an $H$-cover of $G,$ for every $H \in \Zcal.$
\end{observation}

Notice that \cref{@wrongfully} and \cref{@unequivocally} already implies the upper bound of \cref{@admiration} (that is \cref{@indescribably}) in the case where $\Zcal$ contains only connected graphs.
The next part is dedicated to the general (non-necessarily connected) case.

\begin{proof}[Proof of \cref{@indescribably}.]
Let $\hat{H}$ be a graph in $\Zcal$ that is non-planar, Kuratowski-connected, and a shallow-vortex minor.
Let $\hat{P}_{1}, \ldots, \hat{P}_{z}$ be the planar components of $\hat{H},$ i.e., $\hat{H} = \nonplanar(\hat{H}) + \hat{P}_{1} + \cdots + \hat{P}_{z}.$
We also define
$$\hat{\Zcal} = \cupall \big\{ \nonplanar(\hat{H}) + \{\sum_{i\in I}\hat{P}_{i}\mid i\in I\}\mid I\in 2^{[z]} \big\}.$$
Notice that all graphs in $\hat{\Zcal}$ are shallow-vortex minors.

\medskip
We define the following recursive algorithm:

\medskip
\noindent {\sl Procedure} ${\bf Rec\mbox{-}EP}(G,H,k)$

\smallskip
\noindent {\sl Input}: a graph $G,$ a non-planar Kuratowski-connected graph $H,$ and $k \in \Nbbb.$
We demand that if $ H \not\in \hat{\Zcal},$ then $G$ is morever $\hat{H}$-minor free.
We also define $c_{H}$ to be the number of planar connected components of $H.$
\smallskip

\noindent {\sl Output}: either
\begin{itemize}
\item a $\mathscr{D}^{Σ}_{t}$-\major of $G,$ for some $\Sigma \in \sobs(\mathbb{S}_{\{H\}}),$
\item an $H$-packing in $G$ of size at most $k,$ or
\item an $H$-cover in $G$ of size at most $2^{c_{H} \choose 2} \cdot f_{\ref{@disreputable}}(\gamma_{\Zcal}, h_{\Zcal}, t, k) + f_{\ref{@disreputable}}(\gamma_{\Zcal}, h_{\Zcal}, t, k).$
\end{itemize}

\smallskip
\noindent {\sl Description}:
If $H \in \hat{\Zcal}$ then we set $\Zcal^+ = \{ H \},$ otherwise we set $\Zcal^+ = \{H, \hat{H}\}.$
In either case, we have that $\Zcal^+ \in (\Kbbb \cap \Vbbb) \cap \Pbbb.$
This allows for the application of the first step of {\bf Rec\mbox{-}EP} that is to call ${\bf Alg\mbox{-}Main}(\Zcal^+,G,t,k)$ from \cref{@disreputable}.
If one of the three first outcomes are returned, then we conclude.
Here, we stress that, in the second case where a $\Zcal^+$-packing is returned for the case where $\Zcal^+ = \{H, \hat{H}\},$ this must be an $H$-packing as, by assumption, $G$ is $\hat{H}$-minor free.
So we may assume that the output is an $\Zcal^+$-local cover separation $(X,Y)$ of $G$ of order at most $f_{\ref{@disreputable}}(h_{\Zcal^+}, t, k).$

In the base case where $H$ is connected we are done as, from \cref{@unequivocally}, $X \cap Y$ is an $H$-cover of $G.$
Suppose now that $H$ is not connected and let $P_{1}, \ldots, P_{c_{H}}$ be the planar connected components of $\tilde{H}.$
Given some $I \in 2^{[c_{H}]},$ we denote $\bar{I} \coloneqq [c_{H}] \setminus I.$
Given an $I\in 2^{[c_{H}]} \setminus \{[c_{H}]\},$ we denote $P_{I} = \cupall \{P_{i} \mid i \in I\}$ and $P_{\bar{I}} = \cupall \{P_{i} \mid i \not\in I \}.$

The next step of the algorithm is to call ${\bf Rec\mbox{-}EP}(G[Y\setminus X], \nonplanar(H) + P_{I}, k),$ for every $I \in 2^{[c_{H}]} \setminus \{[c_{H}]\}$ and keep in mind that each such call is performed in a graph with strictly less planar connected components than $H.$
If, in some of these calls, the first outcome appears, we are done and the algorithm stops.
If this does not happen, then the outcomes of the above calls partition $2^{[c_{H}]} \setminus \{[c_{H}]\}$ into two sets $\Ical^{\mathsf{p}}$ and $\Ical^{\mathsf{c}}$ such that $\Ical^{\mathsf{p}}$ contains the $I$'s for which the algorithm returns some $(k \cdot (\nonplanar(H) + P_{I}))$-\major $M_{I}$ of $G[Y/X]$ and $\Ical^{\mathsf{c}}$ contains all the $I$'s for which the algorithm returns some $(\nonplanar(H)+ P_{I})$-cover $S_{I}$ of $G[Y/X]$ of size $2^{(c_{H}-1) \choose 2} \cdot f_{\ref{@disreputable}}(h_{P_{I}}, t, k).$

The next step is to partition $2^{[c_{H}]} \setminus \{[c_{H}]\}$ into two sets $\bar{\Ical}^{\mathsf{p}}$ and $\bar{\Ical}^{\mathsf{c}}$ where $\bar{\Ical}^{\mathsf{p}}$ contains the $I$'s for which $G[X \setminus Y]$ contains some $(k \cdot P_{\bar{I}})$-\major $\bar{M}_{I}$ and $\bar{\Ical}^{\mathsf{c}}$ contains all the $I$'s for which $G[X \setminus Y]$ contains some $P_{\overline{I}}$-cover $S_{\overline{I}}$ of size at most $k \cdot f_{\ref{@horklieimer}}(h,k) ≤ f_{\ref{@disreputable}}(h_{P_{\overline{I}}}, t, k).$
This can be done because $P_{\overline{I}}$ is planar and using \cref{@flattering} in time $2^{\poly_{h_{\Zcal}}(k)}\cdot |V(G)|.$

Now, notice that if there exists some $I \in {\Ical}^{\mathsf{p}} \cap \bar{\Ical}^{\mathsf{p}}$ then the algorithm returns $\bar{M}_{I} \cup M_{I}$ as a $(k\cdot H)$-\major of $G$ and stops.
Otherwise, it returns $S \coloneqq (X\cup Y) \cup (\bigcup_{I\in \Ical^{\mathsf{c}}} S_{I}) \cup (\bigcup_{I \in \bar{\Ical}^{\mathsf{c}}} \bar{S}_{I})$ as an $H$-cover of $G.$

This completes the description of the algorithm {\bf Rec\mbox{-}EP}.

\medskip
Notice now that in each recursive call, the set of the vertices added in $S,$ that are different than those from $(X \cup Y)$ (that are always the same), during its at most $2^{c_{H}}$ recursive calls is upper bounded by $2^{c_{H}} \cdot \big(2^{(c_{H}-1)\choose 2} \cdot f_{\ref{@disreputable}}(h_{P_{I}}, t, k) + f_{\ref{@horklieimer}}(h,k)\big) ≤ f_{\ref{@disreputable}}(h_{P_{\overline{I}}}, t, k) ≤ 2^{c_{H}\choose 2}\cdot f_{\ref{@disreputable}}(h_{H},t,k),$ as required.

\medskip
The main algorithm of the lemma works as follows.

\begin{itemize}
\item Let $\{ H_{1}, \ldots, H_{μ} \}$ be an ordering of $\Zcal$ where $H_{1} = \hat{H}.$
\item Set $S \coloneqq \emptyset,$ $i \coloneqq 1.$
\item While $i ≤ μ,$ do the following:
\begin{itemize}
\item run ${\bf Rec\mbox{-}EP}(G,H_{i},k).$ 
\item If ${\bf Rec\mbox{-}EP}(G,\hat{H},k)$ returns the first or second outcome, then output accordingly and stop.
\item If  ${\bf Rec\mbox{-}EP}(G,\hat{H},k)$ returns an $H$-cover $S_{H}$ in $G,$ then update $S \coloneqq S \cup S_{H}$ and $G \coloneqq G - S.$
\item Set $i \coloneqq i + 1.$
\end{itemize}
\item Output $S.$
\end{itemize}
 
The fact that the shallow-vortex minor $\hat{H}$ is processed first in the above procedure is important.
It guarantees for the rest of the recursive calls that $G$ will be $\hat{H}$-minor free as required in the specifications of algorithm {\bf Rec\mbox{-}EP}.
 
Notice that if we eventually obtain a $\Zcal$-cover in $G,$ then this is upper bounded by
$$f_{\ref{@indescribably}}(\gamma_{\Zcal}, h_{\Zcal}, t, k) \coloneqq |\Zcal| \cdot \big(2^{c_{H}\choose 2} \cdot f_{\ref{@disreputable}}(\gamma_{\Zcal}, h_{\Zcal}, t, k) + f_{\ref{@disreputable}}(\gamma_{\Zcal}, h_{\Zcal}, t, k)\big) = 2^{\Ocal_{γ_{\Zcal}}(\mathsf{poly}(t)) + \Ocal_{h_{\Zcal}}(k)}.$$

Notice that the algorithm above makes at most $|\Zcal|$ calls to ${\bf Rec\mbox{-}EP}(G,H,k)$ for each of the graphs in $\Zcal.$
If $T(|V(G)|,|V(H)|)$ is the running time of ${\bf Rec\mbox{-}EP}(G,H,k),$ then
$$T(|V(G)|,|V(H)|) ≤ 2^{|V(H)|-1} \big(T(|V(G)|,|V(H)|-1) + 2^{\poly_{h_{\Zcal}}(k)} \cdot |V(G)| + T^*(\gamma_{\Zcal},h_{\Zcal},t,k) \big),$$
where $T^*(\gamma_{\Zcal}, h_{\Zcal}, t, k)$ is the running time of \cref{@disreputable}.
The recursion above implies that the overall algorithm runs in time whose expression is the same as the one of $T^*(\gamma_{\Zcal},h_{\Zcal},t,k)$ in \cref{@disreputable}.
\end{proof}

\section{Proof of the main theorem}
\label{@homerische}

In this section we conclude the proof of our main result, which is \cref{@admiration}.
Given that we have the proof of  \cref{@indescribably}, it only remains to prove the lower bound of \cref{@admiration}.

\subsection{Proof of the lower bound}
\label{@instigated}

In this subsection we provide a proof of the lower bound of \cref{@admiration}.
We require the following fact about the additivity of the Euler genus of a graph (see~\cite{MoharT01Graphs}). We use the notation $\mathsf{cc}(G)$ for the set of connected components of the graph $G$.
 
\begin{proposition}\label{@skepticism} Let $G$ be a graph. It holds that $\eg(G) = \sum_{C \in \mathsf{cc}(G)} \eg(C)$.
\end{proposition} 

We prove two lemmata that showcase the behaviour of packing versus covering of a non-planar graph, in a Dyck-grid which contains it as a minor.

\begin{lemma}\label{@automobiles} Let $H$ be a non-planar graph that is $Σ$-embeddable, for some surface $Σ.$
Then, there exists a non-negative integer $c$ such that, for every $k \in \mathbb{N},$ $\pack_{\{ H \}}(\mathscr{D}^{Σ}_{k}) \leq c.$ Moreover, $c = 1 + \eg(Σ) - \eg(H).$
\end{lemma}
\begin{proof} Since $H$ is $Σ$-embeddable, we have that $\eg(Σ) \geq \eg(H).$ Since $H$ is non-planar $\eg(H) \geq 1.$ 
Then, by \autoref{@skepticism}, for every $k \geq 2,$ $\eg(k \cdot H) = k \cdot \eg(H) > \eg(H).$ 
Let $c = 1 + \eg(Σ) - \eg(H) \geq 1.$ 
Then, $(c + 1) \cdot H$ is not $Σ$-embedabble and hence, by \cref{@headstrong}, for every $k \in \mathbb{N},$ $\pack_{\{ H \}}(\mathscr{D}^{Σ}_{k}) \leq c.$
\end{proof}

\begin{lemma}\label{@predominantly} Let $H$ be a graph that is $Σ$-embeddable, for some surface $Σ.$
Then, there exists a polynomial function $f : \mathbb{N} \to \mathbb{N}$ such that for every $k \in \mathbb{N},$ $\cover_{\{ H \}}(\mathscr{D}^{Σ}_{k}) > f(k).$
\end{lemma}
\begin{proof} By \cref{@headstrong}, let $c_{H} = f_{\ref{@headstrong}}(\eg(Σ), h).$ We have that $H \leq \mathscr{D}^{Σ}_{c_{H}}.$
By this fact and \autoref{@monopolism}, for every $k \in \mathbb{N},$ $\mathscr{D}^{Σ}_{k \cdot c_{H}}$ contains a half-integral packing of $k$ copies of $\mathscr{D}^{Σ}_{c_{H}}$ and thus of $H.$ Then, by definition of a half-integral packing, an $H$-cover in $\mathscr{D}^{Σ}_{k \cdot c_{H}}$ has size at least $k/2.$
Hence, observe that for $f(k) = \lfloor k/(2c_{H}) \rfloor,$ we have that for every $k \in \mathbb{N},$ $\cover_{\{ H \}}(\mathscr{D}^{Σ}_{k}) > f(k).$
\end{proof}

\subsection{Proof of \cref{@admiration}}
\label{@prophetesses}

\begin{proof}[Proof of \cref{@admiration}]
Recall that $\text{\scriptsize\textsf{EP}}_{\Zcal}\coloneqq \mathsf{p}_{\mathfrak{D}_{\Zcal}},$ where $\mathfrak{D}_{\Zcal}\coloneqq \{ \mathscr{D}^{Σ} \mid Σ\in\mathsf{sobs}(\mathbb{S}_{\Zcal}) \}.$
Let $\Gcal$ be a graph class.

For the upper bound, assume that $\textsf{EP}_{\Zcal}$ is bounded in $\Gcal,$ therefore, there is some $t\in\Nbbb$ such that $\mathsf{p}_{\mathfrak{D}_{\Zcal}}(G)<t.$
This implies that every graph $G$ in $\Gcal$ excludes $\mathscr{D}^{Σ}_{t}$ as a minor, for every $Σ\in\sobs(\mathbb{S}_{\Zcal}).$
Assume now that $\pack_{\Zcal}(G)≤k.$
From \cref{@indescribably}, the only possible output is an $\Zcal$-cover of size at most $2^{\Ocal_{γ_{\Zcal}}(\poly(t))+\Ocal_{h_{\Zcal}}(k)}$  in $G,$ therefore $\Zcal$ has the EP-property in $\Gcal.$

Assume now that $\textsf{EP}_{\Zcal}$ is not bounded in $\Gcal.$
This means that, for some $Σ\in\mathsf{sobs}(\mathbb{S}_{\Zcal}),$  all graphs in $\langle \mathscr{D}^{Σ}_{k}\rangle_{k\in\Nbbb}$ belong to $\Gcal.$
Since $Σ \in \sobs(\mathbb{S}_{\Zcal}),$ every $H \in \Zcal$ is $Σ$-embeddable. 
Let $c_{\Zcal} = \max \{ \eg(Σ) - \eg(H) \mid H \in \Zcal  \}.$ 
Then, by \autoref{@automobiles}, for every $k \in \mathbb{N},$ $\pack_{\Zcal}(\mathscr{D}^{Σ}_{k}) \leq c_{\Zcal}.$ Moreover, by \autoref{@predominantly}, there exists a polynomial function $g : \mathbb{N} \to \mathbb{N}$ such that for every $k \in \mathbb{N},$ $\cover_{\Zcal}(\mathscr{D}^{Σ}_{k}) > g(k).$ 
Then, since for every $k \in \mathbb{N},$ $\mathscr{D}^{Σ}_{k} \in \Gcal,$ clearly there is no function $f : \mathbb{N} \to \mathbb{N}$ such that $\forall G \in \Gcal : \cover_{\Zcal}(G) ≤ f(\pack_{\Zcal}(G))$ and hence $\Zcal$ does not have the EP-property in $\Gcal.$
\end{proof}

\section{Conclusion and open problems}\label{@regressing}

\paragraph{Obstructions of graph classes.} 
The (minor)-\emph{obstruction set} of a graph class $\Gcal,$ denoted by $\obs(\Gcal),$ consists of the minor-minimal elements of $\gall \setminus \Gcal.$ 
Clearly $\obs(\Gcal)$ is an antichain.
Moreover, it is finite by Robertson's \& Seymour's theorem.
Obstruction sets permit the following equivalent statement of \cref{@admiration}.

\begin{theorem}\label{@incomplete}
Let $\Zcal$ be an antichain in $\Hbbb^-$ and let $\Gcal$ be a minor-closed graph class. 
$\Zcal$ has the Erdős-Pósa property in $\Gcal$ if and only if, for every surface $Σ \in \sobs(\Sbbb_{\Zcal}),$ there exists an obstruction in $\obs(\Gcal)$ which is $Σ$-embeddable. 
\end{theorem}

\subsection{Universal obstructions}

Let $\p \colon \gall\to\Nbbb$ be a minor-monotone graph parameter.
We say that a set $\mathfrak{H}$ of minor-monotone parametric graphs is a (minor-)\emph{universal obstruction} for $\p$ if $\p \sim \p_{\mathfrak{H}}$ (recall \eqref{@liberating} for the definition of $\p_{\mathfrak{H}}$). 
Universal obstruction may serve as canonical representations of graph parameters. (For more on the foundation of universal obstructions of parameters, see~\cite{paul2023graph,paul2023universal}.)
From this point of view, \cref{@admiration} can be restated follows:
 
\begin{theorem}
\label{@confession}
For every $\Zcal\in\Hbbb^-,$ the set of parametric graphs $\mathfrak{D}_{\Zcal}=\{ \mathscr{D}^{Σ} \mid Σ\in\mathsf{sobs}(\mathbb{S}_{\Zcal}) \}\cup\{\langle k\cdot H\rangle_{k\in\Nbbb}\mid H\in\Zcal\}$ is a universal obstruction for both $\cover_{\Zcal}$ and  $\nicefrac{1}{2}\text{-}\pack_{\Zcal}.$
\end{theorem}

Given some $\Zcal\in\Hbbb,$  for every $k\in\Nbbb,$ we define $\Ccal_{k}^{\Zcal}=\{G\mid \cover_{\Zcal}(G)≤k\}.$ \cref{@confession} (or the equivalent  \cref{@incomplete})  gives us some valuable information about the obstructions in  $\obs(\Ccal_{k}^{\Zcal}),$ for every $k.$

Certainly, the simplest antichain in $\Hbbb^-$ is the one consisting of the two Kuratowski graphs $\Kcal=\{K_{5},K_{3,3}\}.$
The parameter $\cover_{\Kcal}$ is the \textsl{planarizer number} that is the minimum number of vertices whose removal can make a graph planar. 
The obstruction $\obs(\Ccal_{k}^{\Kcal})$ is unknown for every positive value of $k$ and its size is expected to grow rapidly as a function of $k$ (see \cite{Dinneen97} for  an exponential lower bound and \cite{SauST23apices} for a triply exponential upper bound).
The identification of $\obs(\Ccal_{k}^{\Kcal})$ is a non-trivial problem even for small values of $k.$
In particular, it has been studied extensively for the case where $k=1$ in~\cite{LiptonMMPRT16sixv,Mattman16forb,Yu06more}. 
In this direction, Mattman and Pierce conjectured that $\obs(\Ccal_{k}^{\Kcal})$ contains the $Y\Delta Y$-families of $K_{n+5}$ and $K_{3^2,2^n}$ and provided evidence towards this in~\cite{MattmanP16thea}.
Recently, Jobson and Kézdy identified \textsl{all} graphs in $\obs(\Ccal_{1}^{\Kcal})$ of connectivity two in \cite{JobsonK21allm}, where they also reported that $|\obs(\Ccal_{1}^{\Kcal})| \geq 401.$

It is easy to see that $\{(k+1)\cdot K_{5},(k+1)\cdot K_{3,3}\}\subseteq \obs(\Ccal_{k}^{\Kcal}),$ for every $k\in\Nbbb.$
Our results, together with the fact that $\sobs(\Sbbb_{\Kcal})=\{Σ^{(1,0)},Σ^{(0,1)}\},$ provide the following extra information about $\obs(\Ccal_{k}^{\Kcal})$: for every $k\in \Nbbb,$ it contains some graph, say  $G_{k}^{\sf t},$ embeddable in the torus and some graph, say  $G_{k}^{\sf p},$ embeddable in the projective plane. 
Most importantly, our results indicate, that the four-member subset $\{(k+1)\cdot K_{5},(k+1)\cdot K_{3,3},G_{k}^{\sf t},G_{k}^{\sf p}\}$ of $\obs(\Ccal_{k}^{\Kcal})$ is sufficient to determine the approximate behaviour of the planarizer number.

Similar implications can be derived for every $\Zcal\in\Hbbb^-.$
For instance, if $\Pcal$ is the Petersen family, we again have that $\sobs(\Sbbb_{\Pcal})=\{Σ^{(1,0)},Σ^{(0,1)}\}.$
Therefore the parameter defined as the minimum number of vertices to remove so as to make a graph linkless, is approximately characterized by picking only nine graphs of $\obs(\Ccal_{k}^{\Pcal}),$ for every $k\in\Nbbb.$

Other examples of surface obstructions corresponding to known graphs that are Kuratowski-connected and shallow-vortex minors are $\sobs(\Sbbb_{\{K_{5}\}})=\sobs(\Sbbb_{\{K_{6}\}})=\sobs(\Sbbb_{\{M_{2n}\}})=\{Σ^{(1,0)},Σ^{(0,1)}\},$ 
where $M_{2n}$ is the $2n$-Möbius ladder,\footnote{The \emph{Möbius ladder} $M_{2n}$ is formed if we consider a cycle on $2n$ vertices and then connect by edges the $n$ anti-diametrical pairs of vertices. Notice that $M_{6}=K_{3,3}.$
$M_{8}$ is called the \emph{Wagner Graph}.} for $n\in\Nbbb_{≥3}.$
Two other examples are $\sobs(\Sbbb_{\{K_{4,4}\}})=\{Σ^{(1,0)},Σ^{(0,2)}\}$ and $\sobs(\Sbbb_{\{K_{7}\}})=\{Σ^{(1,0)}\}.$

Another implication of our results is the following.

\begin{theorem}
\label{@superstruc}
For every closed and proper set of surfaces $\Sbbb,$ the set of parametric graphs $\mathfrak{V}_{\Zcal}=\{ \mathscr{D}^{Σ} \mid Σ\in\mathsf{sobs}(\mathbb{S})\}\cup\{\langle \mathscr{V}_{k}\rangle_{k\in\Nbbb}\}$ is a universal obstruction for $\Sbbb\mbox{-}\tw_{\mathsf{apex}}.$
\end{theorem}

The above is a direct consequence of \cref{@enlightened} that is a local structure theorem for graphs excluding the parametric graphs in $\mathfrak{V}_{\Zcal},$ i.e., graphs where $\p_{\mathfrak{V}_{\Zcal}}$ is bounded.
The parameter $\Sbbb\mbox{-}\tw_{\mathsf{apex}}$ (defined in \eqref{@ambivalent}) corresponds to the global version of this theorem.
In particular, the equivalence between $\Sbbb\mbox{-}\tw_{\mathsf{apex}}$ and $\p_{\mathfrak{V}_{\Zcal}}$ follows directly by \cite[Theorem 5.18]{thilikos2023approximating}.
Notice that $\Sbbb\mbox{-}\tw_{\mathsf{apex}}$ can be seen as a parametric extension of graph embeddability and that the exclusion of shallow-vortex minors is pivotal for its definition.
The potential algorithmic applications of $\Sbbb\mbox{-}\tw_{\mathsf{apex}}$ are open to investigate.

Notice that for both the equivalences in \cref{@confession} and in \cref{@superstruc} we have a single exponential gap which, in turn, determines the gap of our \FPT-approximations.
Is it possible to reduce this to a polynomial one?
Certainly, this requires a polynomial dependency on $k$ and $t$ in \cref{@indescribably} (the lower bound in \cref{@predominantly} is polynomial). 
There are two sources of exponentiality in the proof of \cref{@indescribably}.
The first is in the exclusion of the Dyck grid $\mathscr{D}_{t}^{Σ'},$ for  $Σ'\in\sobs(\mathbb{S})$ in \cref{@duplicating}, where we have an exponential dependency on $t.$
This dependency already emerges from the bounds in \cite{kawarabayashi2020quickly}.
On the other hand the exponentially dependency on $k$ emerges from the redrawing lemma \cref{@horkheimer}, where the exponential bound comes from the dependencies of the planar linkage theorem in \cite{AdlerKKLST17Irrelevant}.
Avoiding these two sources of exponentiality appears to be a hard task. 
An alternative approach is to try instead to ``enlarge'' the size of the universal obstructions to obtain a polynomial parametric graph. This would also be desirable  for the purposes of better \FPT-approximation algorithms.

\subsection{Going further than $\Kbbb\cap \Vbbb$.}

The central question proposed by this work is to chart the threshold of half-integrality when covering and packing graphs as minors.
In this paper we resolved this question for every antichain in $\Kbbb\cap \Vbbb.$
The wide open question is whether and how this can be done for more general families of antichains. 
For this, one needs to prove structure theorems on  the exclusion of parameterized graphs of unbounded genus, as those in \cref{@collection}.
The challenges that have to be met for this, when going beyond $\Kbbb,$ are different from those encountered when going beyond $\Vbbb.$ 
We believe that the proof strategy of our paper can serve as a starting point for both directions towards the general case.
The resolution of the general case is highly non-trivial and requires new tools and ideas.

We conclude with a conjecture. 
Our guess is that when we insist on universal obstructions of bounded genus, then we cannot go much further than the horizon of $\Kbbb\cap \Vbbb.$
Let $\Bbbb$ be the set of all antichains consisting of graphs where each of them can be embedded in both the torus and the projective plane.
As an example, observe that $\{K_{3,4}\}\in \Bbbb\setminus\Kbbb,$ while $\{K_{3,5}\}\not\in \Bbbb.$
We conjecture the following.

\begin{conjecture}
Let $\Zcal$ be an antichain and let $\text{\scriptsize\textsf{EP}}_\Zcal \colon \gall\to\Nbbb$ be a graph parameter such that $\Zcal$ has the Erdős-Pósa property in a minor-closed graph class $\mathcal{G}$ if and only if $\text{\scriptsize\textsf{EP}}_\Zcal$ is bounded in $\mathcal{G}$.
Then $\Zcal \in (\Kbbb \cap \Vbbb) \cup \Bbbb$ if and only there exists some $g_{\Zcal}$ such that all universal obstructions of $\text{\scriptsize\textsf{EP}}_\Zcal$ consist of parametric graphs of Euler genus  $\leq  g_{\Zcal}.$
\end{conjecture}

\paragraph{Acknowledgements:} The authors wish to thank the anonymous reviewers for their remarks and suggestions on  earlier versions of this paper.

\end{document}